\documentclass[10pt]{amsbook}
\usepackage{amsmath,epsf}
\usepackage{mathrsfs}
\usepackage{hyperref}

\usepackage{bm}
\usepackage{amssymb,stmaryrd,mathtools}
\usepackage{enumitem}
\usepackage{tikz} \usetikzlibrary{matrix,arrows,decorations.pathmorphing}
\usepackage{tikz-cd}

\usepackage[LGR,T1]{fontenc}

\newcommand\Sti{\begingroup\fontencoding{LGR}\selectfont\char7\endgroup} 

\newcommand\stig{\text{\Sti}} 

\renewcommand*\thepart{\Roman{part}}

\makeatletter
    \def\subsection{\@startsection{subsection}{2}%
      \z@{.5\linespacing\@plus.7\linespacing}{-.5em}%
      {\normalfont\bfseries}}
\makeatother

\makeatletter
\def\part{\cleardoublepage \thispagestyle{empty}%
  \null\vfil  \markboth{}{}\secdef\@part\@spart}
\def\@part[#1]#2{%
  \ifnum \c@secnumdepth >-2\relax \refstepcounter{part}%
    \addcontentsline{toc}{part}{\partname\ \thepart.
        \protect\enspace\protect\noindent#1}%
  \else
    \addcontentsline{toc}{part}{#1}\fi
  \begingroup\centering
  \ifnum \c@secnumdepth >-2\relax
       {\fontsize{\@xviipt}{22}\bfseries
         \partname\ \thepart} \vskip 20\p@ \fi
  \fontsize{\@xxpt}{25}\bfseries
      #1\vskip 15\p@  \fontsize{\@xiipt}{15}\bfseries #2\vfil\vfil\endgroup \newpage\thispagestyle{empty}}
\def\@spart#1{\addcontentsline{toc}{part}{\protect\noindent#1}%
  \begingroup\centering
  \fontsize{\@xxpt}{25}\bfseries
     #1\vfil\vfil\endgroup \newpage\thispagestyle{empty}}
\makeatother

\theoremstyle{plain}
\newtheorem{theorem}[subsection]{Theorem}
\newtheorem{corollary}[subsection]{Corollary}
\newtheorem{proposition}[subsection]{Proposition}
\newtheorem{lemma}[subsection]{Lemma}

\theoremstyle{definition}
\newtheorem{definition}[subsection]{Definition}
\newtheorem{example}[subsection]{Example}

\theoremstyle{remark}
\newtheorem{remark}[subsection]{Remark}

\newtheorem{notation}[subsection]{Notation}

\providecommand{\hoch}{Hochschild }

\providecommand{\sym}{\mathcal{S}}
\providecommand{\Z}{\mathbb{Z}}
\providecommand{\C}{\mathbb{C}}

\providecommand{\F}{\mathbb{F}}
\providecommand{\N}{\mathbb{N}}
\providecommand{\Nsucc}{N_{+}}
\providecommand{\Zsucc}{Z_{+}}
\providecommand{\bbA}{\mathbb{A}}
\providecommand{\bbG}{\mathbb{G}}

\providecommand{\bbO}{\mathbb{O}}
\providecommand{\bbV}{\mathbb{V}}
\providecommand{\bbT}{\mathbb{T}}

\providecommand{\bbS}{\mathbb{S}}
\providecommand{\bbD}{\mathbb{D}}

\providecommand{\A}{\mathfrak{a}}

\providecommand{\p}{\mathfrak{p}}

 \providecommand{\cF}{\mathscr{F}}
\providecommand{\cG}{\mathscr{G}}
\providecommand{\cA}{\mathscr{A}}
\providecommand{\cB}{\mathscr{B}}
\providecommand{\cC}{\mathscr{C}}
\providecommand{\hC}{\widehat{\mathscr{C}}}

\providecommand{\cD}{\mathscr{D}}
\providecommand{\hD}{\widehat{\mathscr{D}}}
\providecommand{\cE}{\mathscr{E}}
\providecommand{\cI}{\mathscr{I}}
\providecommand{\cO}{\mathscr{O}}
\providecommand{\cM}{\mathscr{M}}
\providecommand{\cN}{\mathscr{N}}
\providecommand{\cP}{\mathscr{P}}
\providecommand{\cS}{\mathscr{S}}
\providecommand{\cT}{\mathscr{T}}
\providecommand{\cV}{\mathscr{V}}
\providecommand{\cW}{\mathscr{W}}

\providecommand{\bA}{\mathbf{A}}
\providecommand{\bB}{\mathbf{B}}
\providecommand{\bC}{\mathbf{C}}
\providecommand{\bF}{\mathbf{F}}
\providecommand{\bE}{\mathbf{E}}
\providecommand{\bG}{\mathbf{G}}
\providecommand{\bF}{\mathbf{F}}
\providecommand{\bH}{\mathbf{H}}
\providecommand{\bI}{\mathbf{I}}
\providecommand{\bX}{\mathbf{X}}
\providecommand{\bY}{\mathbf{Y}}
\providecommand{\bZ}{\mathbf{Z}}
\providecommand{\bT}{\mathbf{T}}
\providecommand{\bO}{\mathbf{O}}
\providecommand{\bP}{\mathbf{P}}
\providecommand{\OO}{\mathbf{O}}
\providecommand{\bM}{\mathbf{M}}
\providecommand{\bN}{\mathbf{N}}
\providecommand{\V}{\mathbf{V}}
\providecommand{\W}{\mathbf{W}}
\providecommand{\K}{\mathbf{K}}

\providecommand{\Do}{\mathcal{D}}
\providecommand{\cH}{\mathscr{H}}
\providecommand{\cR}{\mathscr{R}}

\providecommand{\h}{\mathbf{h}}

\providecommand{\bsi}{{\bm{\sigma}}}

\providecommand{\spec}{\mathop{\rm Spec}}
\providecommand{\Sp}{\mathop{\rm Sp}}
\providecommand{\speczar}{\mathop{\rm Spec.Zar}}
\providecommand{\specet}{\text{Spec.\'Et}}

\providecommand{\Pic}{\text{Pic}}
\providecommand{\uPic}{\text{\underline{Pic}}}

\providecommand{\et}{\text{\'et}}
\providecommand{\Ob}{\mathop{\rm Ob}}
\providecommand{\id}{\mathop{\rm id}\nolimits}
\providecommand{\eq}[1]{#1\text{-eq}}
\providecommand{\Aut}{\mathop{\rm Aut}\nolimits}
\providecommand{\Hom}{\mathop{\rm Hom}\nolimits}
\providecommand{\Ext}{\mathop{\rm Ext}\nolimits}
\providecommand{\Isom}{\mathop{\rm Isom}\nolimits}
\providecommand{\End}{\mathop{\rm End}\nolimits}
\providecommand{\uAut}{\mathop{\underline{\rm Aut}}\nolimits}
\providecommand{\uHom}{\mathop{\underline{\rm Hom}}\nolimits}
\providecommand{\uIsom}{\mathop{\underline{\rm Isom}}\nolimits}
\providecommand{\uEnd}{\mathop{\underline{\rm End}}\nolimits}
\providecommand{\uExt}{\mathop{\underline{\rm Ext}}\nolimits}
\providecommand{\uC}{\mathop{\underline{\rm C}}\nolimits}
\providecommand{\Ch}{\text{C}}
\providecommand{\ch}{\mathbf{Ch}}
\providecommand{\Ex}{\mathbf{E}}
\providecommand{\eend}{{\textstyle\int}}
\providecommand{\coend}{{\textstyle\int}}
\providecommand{\mono}{\rightarrowtail}

\providecommand{\HH}{\mathop{\rm H}\nolimits}
\providecommand{\uH}{\mathop{\underline{\rm H}}\nolimits}
\providecommand{\tH}{\text{\rm H}}

\providecommand{\ter}{e}
\providecommand{\ini}{0}
\providecommand{\op}{{\mathop{\rm op}\nolimits}}

\providecommand{\ind}{\mathop{\rm Ind}\nolimits}
\providecommand{\be}{\underline{\mathbf{e}}}
\providecommand{\I}{\text{I}}
\providecommand{\jj}{\text{J}}
\providecommand{\Fix}{\mathop{\rm Fix}\nolimits}
\providecommand{\Quo}{\mathop{\rm Quo}\nolimits}
\renewcommand{\ker}{\mathop{\rm ker}\nolimits}
\providecommand{\Ker}{\mathop{\rm Ker}\nolimits}
\providecommand{\im}{\mathop{\rm im}\nolimits}
\renewcommand{\Im}{\mathop{\rm Im}\nolimits}
\providecommand{\coim}{\mathop{\rm coim}\nolimits}
\providecommand{\Coim}{\mathop{\rm Coim}\nolimits}
\providecommand{\coker}{\mathop{\rm coker}\nolimits}
\providecommand{\Coker}{\mathop{\rm Coker}\nolimits}
\providecommand{\vect}{\mathop{\rm vec}\nolimits}
\providecommand{\Eq}{\mathop{\rm Eq}\nolimits}
\providecommand{\coeq}{\mathop{\rm coeq}\nolimits}
\providecommand{\cone}{\mathop{\rm cone}\nolimits}
\providecommand{\coh}{\mathbf{H}}
\providecommand{\der}{\mathbf{R}}

\providecommand{\colim}{\mathop{\rm colim}}

\providecommand{\copinf}{\amalg_\infty}
\providecommand{\copz}{\amalg_\Z}
\providecommand{\prodz}{\Pi_\Z}
\providecommand{\prodinf}{\Pi_\infty}
\providecommand{\produinf}{\Pi^\infty}
\providecommand{\sprod}{{\textstyle\prod}}
\providecommand{\ssum}{{\textstyle\sum}}

\providecommand{\Gal}{\mathop{\rm Gal}}

\providecommand{\Split}{\text{\rm Split}}

\providecommand{\Sub}{\text{\rm Sub}}
\providecommand{\cat}{\mathbf{cat}}
\providecommand{\Sch}{\mathbf{Sch}}

\providecommand{\aff}{\text{\rm -Aff}}
\providecommand{\Set}{\mathbf{Set}}
\providecommand{\Rng}{\mathbf{Rng}}
\providecommand{\Tors}{\mathbf{Tors}}
\providecommand{\uTors}{\underline{\mathbf{Tors}}}
\providecommand{\Gp}{\text{\rm Gr}}
\providecommand{\daGp}{\text{\rm -Gr}}
\providecommand{\rng}{\text{\rm -Rng}}
\providecommand{\da}{\text{\rm-}}
\providecommand{\Ab}{\text{\rm Ab}}
\providecommand{\ab}{\text{\rm -Ab}}
\providecommand{\CGp}{{(\cC\text{-}\Gp)}}

\providecommand{\hCGp}{{(\hC\text{-}\Gp)}}
\providecommand{\hCGps}{{\hC\text{-}\Gp}}
\providecommand{\Omod}{{(\bO\text{-Mod})}}
\providecommand{\Mod}{\text{\rm-Mod}}
\providecommand{\Comod}{\text{\rm-Comod}}

\providecommand{\Alg}{\text{\rm-Alg}}
\providecommand{\GOmod}{{(\bG\text{-}\bO\text{\rm-Mod})}}
\providecommand{\Vab}{\cV\da\Ab}

\providecommand{\Top}{\mathfrak{Top}}

\providecommand{\Diff}{\mathbf{Diff}}
\providecommand{\diff}{\bsi\text{-}}
\providecommand{\diffinf}{\bsi_\infty\text{-}}
\providecommand{\difuinf}{\bsi^\infty\text{-}}
\providecommand{\diffN}{\bsi_\N\text{-}}
\providecommand{\dif}[1]{\bsi#1\text{-}}
\providecommand{\difd}[1]{\bsi_{#1}\text{-}}
\providecommand{\difu}[1]{\bsi^{#1}\text{-}}
\providecommand{\diffZ}{\bsi_\Z\text{-}}
\providecommand{\diffst}{\bsi^{*}\text{-}}

\providecommand{\Nat}{\mathbf{Nat}}
\providecommand{\nat}{\text{-}\Nat}

\providecommand{\psh}{\mathbf{PSh}}
\providecommand{\sh}{\mathbf{Sh}}
\providecommand{\diffpsh}{\diff\psh}

\providecommand{\hdifC}{\diffpsh(\diff\cC)}

\providecommand{\ov}{{\kern-1pt\mathord{/}\kern-.7pt}}
\providecommand{\Ov}{{\kern-1pt\sslash\kern-.7pt}}

\providecommand{\diff}{\text{\it Diff}}

\providecommand{\forg}[1]{\lfloor#1\rfloor}
\providecommand{\ssig}[1]{\lceil#1|}
\providecommand{\psig}[1]{|#1\rceil}
 
\providecommand{\real}[1]{\llbracket#1\rrbracket}
 
\providecommand{\assoc}[1]{\mathring{#1}}

\providecommand{\lexp}[2]{{\vphantom{#2}}^{#1}{\kern-.1ex#2}}
\providecommand{\lsub}[2]{{\vphantom{#2}}_{#1}{\kern-.1ex#2}}

\providecommand{\lrexp}[3]{{\vphantom{#2}}^{#1}{\kern-.1ex#2^#3}}
\providecommand{\ldexp}[3]{{\vphantom{#2}}^{#1}{\kern-.1ex#2_#3}}

\begin{document}

\frontmatter

\subjclass[2000]{Primary . Secondary .}
\title{A topos-theoretic view of difference algebra}
\date{\today}

\author{Ivan Toma{\v s}i{\'c}} 
\address{Ivan Toma{\v s}i{\'c}\\
         School of Mathematical Sciences\\
  	Queen Mary University of London\\
         London, E1 4NS\\
        United Kingdom}
\email{i.tomasic@qmul.ac.uk}


\maketitle



\tableofcontents

\section*{Introduction}

\pagestyle{plain}

\subsection{The origins of difference algebra}

Difference algebra can be traced back to considerations involving recurrence relations, recursively defined sequences, rudimentary dynamical systems, functional equations and the study of associated \emph{difference equations}. 

Let  $k$ be a commutative ring with identity, and let us write
$$
R=k^\N
$$
for the ring ($k$-algebra) of $k$-valued sequences, and let $\sigma:R\to R$ be the \emph{shift} endomorphism given by 
$$
\sigma(x_0,x_1,\ldots)=(x_1,x_2,\ldots).
$$
The \emph{first difference} operator $\Delta:R\to R$ is defined as
$$
\Delta=\sigma-\id,
$$
and, for $r\in\N$, the \emph{$r$-th difference} operator $\Delta^r:R\to R$ is the $r$-th compositional power/iterate of $\Delta$, i.e.,
$$
\Delta^r=(\sigma-\id)^r=\sum_{i=0}^r\binom{r}{i}(-1)^{r-i}\sigma^i.
$$
Classically, a \emph{difference equation} is a polynomial expression involving an unknown $x$ for a sequence in $R$, and its iterated differences. More precisely, given a polynomial  $F\in k[X_0,X_1,\ldots,X_r]$, an expression
$$
F(x,\Delta(x),\ldots,\Delta^r(x))=0
$$
is a polynomial difference equation in an unknown $x$. In view of the fact that iterated difference operators can be expressed in terms of iterated shifts, the above equation can be rewritten as
$$
\tilde{F}(x,\sigma(x),\ldots,\sigma^r(x))=0,
$$
for some polynomial $\tilde{F}\in k[X_0,X_1,\ldots,X_r]$. 

Historically, the terminology shifted in favour of the shift endomorphisms, so now we most commonly refer to the latter form as a polynomial difference equation, and its solutions are sought in an abstract \emph{difference $k$-algebra}
$$
(R,\sigma),
$$ 
consisting of a $k$-algebra $R$ and an endomorphism $\sigma:R\to R$. 

\subsection{Algebraic structures with operators}\label{class-diffalg}

Abstract difference and differential algebra were founded by Ritt in 1930s as the study of algebraic structures equipped with difference and differential operators. In this classical (or \emph{strict}) point of view, the operators are superimposed onto classical algebraic structures. 

In particular, a \emph{difference ring} is a pair
$$
(R,\sigma_R),
$$
where $R$ is a commutative ring with identity and $\sigma:R\to R$ is an endomorphism.

A \emph{morphism} 
$f:(R,\sigma_R)\to (S,\sigma_S)$ is
a commutative diagram
 \begin{center}
 \begin{tikzpicture} 
\matrix(m)[matrix of math nodes, row sep=2em, column sep=2em, text height=1.5ex, text depth=0.25ex]
 {
 |(1)|{R}		& |(2)|{S} 	\\
 |(l1)|{R}		& |(l2)|{S} 	\\
 }; 
\path[->,font=\scriptsize,>=to, thin]
(1) edge node[above]{$f$} (2) edge node[left]{$\sigma_R$}   (l1)
(2) edge node[right]{$\sigma_S$} (l2) 
(l1) edge  node[above]{$f$} (l2);
\end{tikzpicture}
\end{center}
or ring homomorphisms. We obtain the category 
$$
\diff\Rng
$$
of difference rings. 

If $R\in\diff\Rng$, a (difference) \emph{$R$-module} is a pair 
$$(M,\sigma_M),$$ 
where $M$ is a module for the underlying ring of $R$, and $\sigma_M:M\to M$ is an additive endomorphism such that, for $r\in R$ and $m\in M$, 
$$
\sigma_M(r.m)=\sigma_R(r).\sigma_M(m).
$$
A \emph{morphism $f:(M,\sigma_M)\to (N,\sigma_N)$} is a homomorphism $f$ of the underlying modules satisfying $$ f\circ\sigma_M=\sigma_N\circ f.$$
We write 
$$
R\Mod
$$
for the category of difference modules over the difference ring $R$.

Similarly, we define difference analogues of other types of mathematical structures, including sets, groups, abelian groups, algebras over a difference ring, etc.

A  \emph{differential ring} is a pair
$$
(R,\delta)
$$
consisting of a  ring $R$ with a derivation $\delta:R\to R$.
There exist suitable notions of differential modules and other algebraic objects. 

It is possible to consider multiple difference or differential operators, and the two types of operators can be combined into difference-differential structures, where the interaction between the operators is often prescribed.

The differential side of the story developed much quicker. 
Most notably, it was formulated in a manner analogous to the algebraic geometry of the time by Kolchin, who studied differential varieties over differential fields, developed the Galois theory of differential equations, and studied difference algebraic groups (\cite{kolchin}).

A comparable level of development of difference algebra had to wait until Cohn's monograph \cite{cohn}, where the difference analogue of commutative algebra needed for formulating `difference algebraic geometry' was established. The Galois theory of difference equations was described by van der Put and Singer in \cite{vdP-Singer}.

The ultimate exposition of the classical view of difference algebra is Levin's recent book \cite{levin}, which covers Cohn's material from a contemporary point of view and expands it with a selection of important developments established since.

\subsection{Model theory of fields with operators}

Model theory has been highly successful in the study of fields with operators, both difference and differential. To focus just on the difference side of the story, existentially closed difference fields were axiomatised by Macintyre \cite{angus} and Chatzidakis-Hrushovski \cite{zoe-udi} by the theory called ACFA, but the classification of definable sets was made difficult by the lack of full quantifier elimination for ACFA. 

Nevertheless, Zilber's trichotomy has been established by Chatzidakis, Hrushovski and Peterzil in  \cite{zoe-udi} and \cite{acfa2}, stating that a one-dimensional definable set is either in correspondence with an
algebraic curve over a definable one-dimensional field, or it is locally modular (it arises from a definable group $G$ such that every definable subset of $G^n$ is a finite Boolean combination of cosets of definable groups), or it is trivial in the sense that all definable relations on X are reducible to binary relations.

These results, coupled with methods and ideas of stability/simplicity theory, found spectacular applications in number theory, such as Hrushovski's model-theoretic proof of the Manin-Mumford conjecture \cite{udi-MM}, and more recently in algebraic dynamics, such as in the work of Medvedev-Scanlon \cite{MeSc}.

The long-standing conjecture by Macintyre that ACFA is the
first-order theory of difference fields 
$$
(\bar{\F}_p,x\mapsto x^{p^n})
$$
 consisting of the algebraic closure of a finite field $\F_p$ 
 equipped with a power of the Frobenius automorphism was settled by Hrushovski in \cite{udi-frob}.

\subsection{Difference algebraic geometry}\label{old-diff-ag}

The scheme-theoretic development of difference algebraic geometry originates in Hrushovski's paper \cite{udi-frob} with the definition of the fixed-point spectrum 
$$
{\spec}^\sigma(A)
$$
of a difference ring $(A,\sigma)$. Cohn's monograph provides a wealth of information needed for studying such schemes, but a number of new results had to be developed from first principles. This context was further explored by Wibmer and the author. The drawback of this approach is that it is not possible to recover the difference ring $(A,\sigma)$ by taking the global sections of the structure sheaf of  ${\spec}^\sigma(A)$ in general. 

More recently, Wibmer and the author adopted the position that it is beneficial to adopt a functorial view, where a difference scheme associated to the difference ring $(A,\sigma)$ is treated as a representable functor from the category of difference rings to the category of sets
$$
R\mapsto \diff\Rng(A,R),
$$ 
to be thought of as the realisation functor of the difference scheme $\spec(A,\sigma)$. By Yoneda's lemma, the difference ring $(A,\sigma)$ can be recovered from its realisation functor, which parallels the situation in classical algebraic geometry. 

\subsection{In search of difference cohomology}

The goal of this manuscript is to identify the right context for introducing cohomology into difference algebra and geometry. The author took the following path to enlightenment in his quest for difference cohomology.


{\bf Cartesian closed categories.}
Classical algebra is founded in the category of sets, where
we have the natural `currying' isomorphism
$$
\Set(X\times Y,Z)\simeq \Set(X,\Set(Y,Z)).
$$
Homological algebra is rooted in a few basic dualities that stem from it, such as the following \emph{hom-tensor duality}. Given a commutative ring with identity $R$, and $R$-modules $A$, $B$, $C$, the object $R\Mod(B,C)$ is again an $R$-module, and we have a natural isomorphism
$$
R\Mod(A\otimes_R B,C)\simeq R\Mod(A,R\Mod(B,C)).
$$

Any attempt to develop homological algebra in the strict difference context described in \ref{class-diffalg} encounters the following insurmountable obstacles. 

If $Y$ and $Z$ are difference sets, the object
$$
\diff\Set(Y,Z)
$$
is a bare set, and not a difference set, so we cannot hope for a difference analogue of the currying isomorphism in the strict difference context.

Similarly, if $R\in\diff\Rng$ is a difference ring, and $B$, $C$ are difference $R$-modules,
the set of difference $R$-module homomorphisms
$$
R\Mod(B,C)
$$ 
is not a difference $R$-module in general. Hence, we cannot hope for a difference analogue of the hom-tensor duality in the confines of the strict difference context.


The above properties can be restated by saying that the category of sets is \emph{cartesian closed,} and that the category of modules over a commutative unital ring is \emph{monoidal closed.} 

We now show  how to make difference sets into a cartesian closed category, and how to make the category of difference modules over a difference ring into a monoidal closed category.

{\bf Enriching difference algebra.}
Consider the difference set $$\Nsucc=(\N,i\mapsto i+1).$$
Given difference sets $X$ and $Y$, we  define the \emph{internal hom object} 
$$
[X,Y]\in\diff\Set
$$
by stipulating
$$
[X,Y]=\diff\Set(\Nsucc\times X,Y)\simeq \{(f_i)_{i\in\N}: f_i\in\Set(X,Y), f_{i+1}\circ \sigma_X=\sigma_Y\circ f_i\},
$$
together with the shift $s:[X,Y]\to [X,Y]$,
$$
s(f_0,f_1,\ldots)=(f_1,f_2,\ldots).
$$
We obtain a natural isomorphism
$$
\diff\Set(X\times Y,Z)\simeq\diff\Set(X,[Y,Z]),
$$ 
and hence $\diff\Set$ is a cartesian closed category. 

If $R$ is a difference ring, and $A$, $B$ and $C$ are difference $R$-modules, we can define the internal hom object
$$
[B,C]_R\in R\Mod
$$
analogously, and we obtain the difference hom-tensor duality
$$
R\Mod(A\otimes_R B,C)\simeq R\Mod(A,[B,C]_R),
$$
and thus $R\Mod$ is a monoidal closed category.

Our 
motto becomes the following.
\begin{quote}
Difference algebraic geometry and homological algebra must be developed in the framework of \emph{enriched category theory}, where the relevant categories are enriched over difference sets and difference modules.
\end{quote}

In view of the fact that 
$$\diff\Set(X,Y)=\Fix([X,Y],s),$$ 
we observe that the ordinary difference categories are the \emph{underlying categories} of the relevant $\diff\Set$-enriched categories. Hence, the classical difference algebra only sees the tip of the iceberg.  
 
{\bf A topos-theoretic view.}
Writing $\bsi$ for the category associated to the monoid $\N$, we have equivalences of categories
$$
\diff\Set\simeq \bB\N \simeq [\bsi,\Set],
$$
between difference sets, the category $\bB\N$ of sets with an action of the monoid $(\N,+)$ and the category of presheaves on $\bsi^\op\simeq\bsi$. By definition, they are all Grothendieck topoi, and $\bB\N$ is often referred to as the \emph{classifying topos of $\N$.}

We update our basic principle to the following.
\begin{quote}
Difference algebra is the study of algebraic objects \emph{internal} in the topos $\diff\Set$.
\end{quote}
 Indeed, in terminology of categorical logic, the theories of groups, abelian groups, rings, and modules are \emph{algebraic theories} in the sense that the relevant axioms can be expressed by the commutativity of various diagrams involving algebraic operations treated as morphisms, so they can be interpreted in any category with finite limits.

 In particular, difference rings are ring objects in the category $\diff\Set$, i.e.,
 $$
 \diff\Rng\simeq \mathbf{Rng}(\diff\Set).
 $$
 Moreover, if $R\in\diff\Rng$ is a ring object in the topos $\diff\Set$, the category of difference $R$-modules is equivalent to the category of modules in the ringed topos $(\diff\Set,R)$, i.e.,
 $$
 R\Mod\simeq \mathbf{Mod}(\diff\Set,R).
 $$
 The enriched structure mentioned above is automatic, since the internal homs of difference sets are the internal homs from topos theory, and the internal homs of difference modules are the topos-theoretic ones that stem from the operations on modules in ringed topoi.  

{\bf Changing the universe.}
An established principle in topos theory states that the universe of sets can be replaced by an arbitrary base topos, and that it should be possible to reprise  interesting  topics and chapters of classical mathematics in the new context. 

We cling to this principle as a kind of Ariadne's thread, guiding us out of the labyrinth of guesswork on the path of applying the vast machinery of topos theory and categorical logic in the case of difference sets. There is no need to wonder how to define appropriate difference analogues of classical objects, a predicament often encountered by a researcher in difference algebra. 

Remarkably, this special case of a possibly simplest example of a non-boolean topos hides a wealth of previously uncovered objects, and takes us on a path through a significant part of the Elephant \cite{elephant1}, \cite{elephant2}.

\subsection{Difference algebraic geometry as relative algebraic geometry}

Inspired by the above philosophy of working in a new universe, we proclaim the following.

\begin{quote}
Difference algebraic geometry is algebraic geometry over the \emph{base topos} $\diff\Set$.
\end{quote}

For a ringed topos $(\cE,A)$, Hakim \cite{hakim} defines the \emph{Zariski spectrum}
$\speczar(\cE,A)$
as a locally ringed topos equipped with a morphism of ringed topoi
$$
\speczar(\cE,A)\to (\cE,A)
$$
so that the 2-functor $\speczar$ is a 2-right adjoint to the inclusion 2-functor of locally ringed topoi into ringed topoi, as explained in Section~\ref{rel-AG}.

Moreover, the inclusion 2-functor of \emph{strictly locally ringed} topoi  into locally ringed topoi into admits a right adjoint $\specet$, giving rise to the \emph{\'etale spectrum}
$$
\specet(\cE,A)\to (\cE,A).
$$

Hence, the \emph{affine difference scheme} associated to a difference ring $A$ is the locally ringed topos 
$$
(X,\cO_X)=\speczar(\diff\Set,A),
$$ 
and its \'etale topos is
$$
(X_\text{\'et},\cO_{X_\text{\'et}})=\specet(X,\cO_X),
$$
as we explain in Section~\ref{diff-AG}. This approach resolves all foundational issues mentioned in \ref{old-diff-ag}. 

{\bf \'Etale fundamental group of a difference scheme.}
Under certain conditions, the structure morphism $X_\text{\'et}\to\diff\Set$ is a locally connected geometric morphism.

If we have a point $\bar{x}:\diff\Set\to X_\text{\'et}$, the 
\emph{difference \'etale fundamental group} of $X$ is defined as the prodiscrete $\diff\Set$-localic group
$$
\pi_1({X_\text{\'et}}_{\ov\diff\Set},\bar{x})
$$
as in \cite{bunge-04}, building on the earlier work of \cite{sga1, joyal-tierney, bunge-92, bunge-moerdijk}. In the absence of a point, we work with a localic difference groupoid given by this theory.

We also adapt Janelidze's pure Galois theory (\cite{borceux-janelidze}, \cite{janelidze}) to the context of difference rings and compare the resulting profinite groupoid to the fundamental groupoid. 

Let us indicate why we require such advanced technology to develop difference Galois theory. In algebraic geometry, one can always reduce to the case of a connected base scheme, one can always find some geometric points in the base scheme and associated fibre functors, and Grothendieck's Galois theory exploits the fact that the target category $\Set$ of the fibre functors is a boolean topos. In the difference case, we have to deal with the cases where the base difference scheme is topologically totally disconnected yet indecomposable as a difference object, the inability to find points in the base difference scheme,  and the fact that the topos $\diff\Set$ is not boolean.

The classical theory is summarised in Section~\ref{rel-galois}, and our difference versions are given in Section~\ref{diff-gal}.

{\bf Cohomology of difference schemes.} If $(X,\cO_X)$ is a difference scheme as above, the global sections geometric morphism factors as
$$
\gamma_\text{Zar}:X\xrightarrow{\pi_\text{Zar}}\diff\Set\to\Set, 
$$
where $\pi_\text{Zar}$ is the structure geometric morphism of the Zariski spectrum.

If $M$ is an $\cO_X$-module, we define the Zariski cohomology and the relative Zariski cohomology as the usual topos cohomology groups
$$
\tH^i(X,M)=R^i\gamma_{\text{Zar}*}(M), \ \ \text{ and }\ \ \tH^i(X/\diff\Set,M)=R^i\pi_{\text{Zar}*}(M),
$$
the latter being a difference abelian group. 

Similarly, if $M$ is an 
abelian group in $X_\et$, we define the \'etale cohomology and the relative \'etale cohomology groups as the usual topos cohomology groups
$$
\tH^i(X_\et,M)=R^i\gamma_{\et,*}(M),\ \ \text{ and }\ \ \tH^i({X_\text{\'et}}/{\diff\Set},M)=R^i\pi_{\et,*}(M),
$$
the latter being a difference abelian group.

For any reasonable scheme topology $\tau$,  we define $\tau$-spectra of a difference scheme (\ref{external-spectra}, \ref{tau-spectra-diff}), and $\tau$-cohomology analogously to the above (\ref{coh-rel-sch}, \ref{diff-cohom}). In particular, we occasionally use the fppf topology. 

As explained in Section~\ref{top-coh}, the advantage of placing our definitions firmly within the classical context of topos theory is that we have a vast machinery at our disposal, and many results are automatic. 
In particular, all spectra in sight are ringed Grothendieck topoi, so the relevant categories of modules and abelian groups are known to be abelian with enough injectives, and classical homological algebra applies, as well as all results from Grothendieck's SGA or \cite{stacks-project} developed at the generality of ringed sites or topoi.

Hence, we immediately know the relationship between the first cohomology and the Picard group (\ref{picard}, \ref{h90-diff}), and the fact that the first cohomology group classifies torsors (\ref{group-torsors}, \ref{diff-tors}, \ref{gen-diff-tors}).

Grothendieck-Leray spectral sequence gives a relationship between the absolute and relative cohomology groups, which becomes particularly simple in the difference case \ref{diff-cohom}.

Following Hakim's general result \ref{coh-quasicoh}, we obtain the difference analogues of the vanishing theorems in cohomology of quasi-coherent sheaves \ref{coh-quasicoh-diff}.

Upon establishing a comparison \ref{comp-diff} between the relative difference cohomology and the classical cohomology of the underlying scheme, we are able to perform a number of fairly explicit calculations, such as computing  the \'etale cohomology groups of a difference field in \ref{etcoh-diff-field} and that of a difference curve in \ref{etcoh-diff-curve}.

 {\bf Cohomology of difference algebraic groups.} 
 
In \cite{michael-habil}, Wibmer adopts a categorical view of difference algebraic geometry. Given a base difference ring $k$, the work is done in the diagram category
$$
[k\Alg,\Set],
$$
and a difference varieties over $k$ correspond to functors represented by $k$-algebras of finite $\sigma$-type.  In particular, a difference algebraic group is a functor represented by a difference Hopf algebra over $k$. 

In Section~\ref{diff-gp-coh}, we work the context of categories enriched over $\diff\Set$ and the \emph{enriched} functor category
$$
[k_\infty\Alg,\diff\Set].
$$ 
This context naturally leads to the notion of representations of affine difference algebraic groups as difference comodules of the associated difference Hopf algebra, allowing us to compare to the previous work of \cite{michael-alexey} and \cite{piotr}.

\subsection{Organisation of the material}

In the first part of the manuscript, we aim to review the material from enriched category theory and topos theory that will be used in the second part in the context of difference algebra and over the base topos $\diff\Set$. 
We mostly follow the classics such as \cite{sga4.1}, \cite{borceux-1}, \cite{borceux-2}, \cite{borceux-3}, \cite{olivia-book}, \cite{johnstone}, \cite{kelly},  \cite{elephant1}, \cite{elephant2}, \cite{maclane-moerdijk}. 

The only original contributions in this part are in Section~\ref{enr-homol-alg}, where we show that one can obtain  enriched derived functors of an enriched functor, in Section~\ref{rel-AG}, where we indirectly show that Tierney's construction of the Zariski spectrum from \cite{tierney-spec} works,  and perhaps in Sections~\ref{rel-galois}, Section~\ref{coh-rel-sch} and Section~\ref{s:rel-gp-coh}, where we combine known methods for possibly novel applications.

In the second part of the manuscript, we have a surprisingly nontrivial task of unravelling the general concepts in the difference context, whilst interpreting the emerging objects in a way that gives an impression of doing `difference algebraic geometry'.  Most of it is completely unseen in difference algebra.

Whenever possible, we address an issue in the utmost generality in the first part in order to extract it as a special case in the second. On the other hand, certain concepts become much simpler in the difference case, and we are able to pursue them much further than in full generality. In particular, several calculations from  Section~\ref{diff-gal} and Section~\ref{diff-coh} are specific. 

\subsection{Generalisations}

Our basic premise is that most facts about difference algebra can be derived from topos theory over the base topos 
$$\bB \N.$$
We can replace $\N$ by any other monoid, quandle or a category, and, by following the template we outlined, obtain the relevant  `equivariant' geometry.

\subsection{Acknowledgements} I am a relative novice in category theory, and only started thinking of the role of enriched categories in difference homological algebra in 2017, and I discovered topos theory and categorical logic even later. Consequently, I was reliant on help from a number of category theorists. 

I  would like to express my gratitude to Olivia Caramello, Nicola Gambino, Martin Hyland, Peter Johnstone, Edmund Robinson, Ross Street and Gavin Wraith, for generously sharing their knowledge and insights with me.  

A number of topics, including the motivation for Section~\ref{diff-gp-coh} and Corollary~\ref{diff-galois-gp} as motivation for Section~\ref{diff-gal} stem from joint work and fruitful discussions with Michael Wibmer, and will lead to joint publications. I thank him for his patience. 

There is a number of independent papers that are related to parts of this manuscript, such as \cite{michael-anette}, \cite{piotr} and especially the recent preprint \cite{piotr2}. This shows that difference algebra is decidedly entering the age of cohomology. We explain the connections in Section~\ref{appendix1}.

\subsection{Disclaimer}

This is a working manuscript and therefore somewhat incomplete. It lacks examples and precision in certain parts. In spite of occasionally using internal language, and developing some categorical logic for difference sets, it lacks a general treatment of categorical logic. We will strive to address these issues in future versions. 

This version is intended as a `proof of concept' that the outlined programme of bringing relative algebraic geometry and difference algebraic geometry on par with classical algebraic geometry is possible, and we intend to carry out that programme as far as possible.

 \mainmatter 

\pagestyle{headings}

\part[$\cE$GA]{Algebraic geometry in enriched categories and topoi}

\section{Category theory essentials}

\subsection{Representable functors}

Let $\cC$ be a category and let
$$
\hC=\mathbf{Hom}(\cC^\circ,\Set)
$$
be the category of contravariant functors from $\cC$ to the category of sets. There is a canonical functor $\h:\cC\to\hC$, associating to each $X\in\Ob(\cC)$ the functor $\h_X$ defined by
$$
\h_X(S)=\Hom_{\cC}(S,X).
$$
By Yoneda, for every $X\in\Ob(\cC)$ and $\bF\in\Ob(\hC)$, we have a bijection
$$
\Hom_{\hC}(\h_X,\bF)\stackrel{\sim}{\longrightarrow}\bF(X).
$$
In particular, the functor $\h$ is  \emph{fully faithful}, i.e., for every pair $X,Y\in\Ob(\cC)$, the canonical map
$$
\Hom_{\cC}(X,Y)\stackrel{\sim}{\longrightarrow} \Hom_{\hC}(\h_X,\h_Y) 
$$
is bijective. It yields an equivalence between $\cC$ and the full subcategory of \emph{representable functors} of $\hC$, which allows us to identify $X$ and $\h_X$ in the sequel.

\subsection{Limits}

The category $\hC$ admits arbitrary (small) limits and the functor $\h$ commutes with limits.
We discuss some special cases.

The category $\hC$ has a terminal object $\be$ satisfying $\be(S)=\{\emptyset\}$. It is representable if and only if $\cC$ has a terminal object $e$, in which case $\be=\h_e$;

Projective limits in $\hC$ are computed argument-wise, i.e., 
$$
(\varprojlim_i \bF_i)(S)=\varprojlim_i \bF_i(S)
$$

Given morphisms $\bF\to \bG$ and $\bF'\to\bG$ in $\hC$, the fibre product is defined by
$$
\left(\bF\times_{\bG}\bF'\right)(S)=\bF(S)\times_{\bG(S)}\bF'(S).
$$
We define the product of $\bF, \bF'\in\Ob(\hC)$ as
$$
\bF\times\bF'=\bF\times_{\be}\bF'.
$$
In particular, given $X,X',Y\in \Ob(\cC)$, $\h_X\times_{\h_Y}\h_{X'}$ (resp.\ $\h_X\times\h_{X'}$) is representable if and only if $X\times_YX'$ (resp.\ $X\times X'$) exists in $\cC$, and in that case
$$
\h_X\times_{\h_Y}\h_{X'}=\h_{X\times_YX'} \ \ \ \left(\text{resp. } \h_X\times\h_{X'}=\h_{X\times X'}\right).
$$

\subsection{Adjoints and representability}\label{general-representab}

Let $R:\cC\to\cD$ be a functor which is right adjoint to $L:\cD\to\cC$. We define a functor 
$$
\widehat{L}:\widehat{\cC}\to\widehat{\cD}, \ \ \ \ \ \widehat{L}(\bF)=\bF\circ L.
$$
For any $X\in\Ob(\cC)$,
$$
\widehat{L}(\h_X)=\h_{R(X)}.
$$
The diagram of categories
 \begin{center}
 \begin{tikzpicture} 
\matrix(m)[matrix of math nodes, row sep=2em, column sep=3em, text height=1.9ex, text depth=0.25ex]
 {
 |(1)|{\cC}		& |(2)|{\hC} 	\\
 |(l1)|{\cD}		& |(l2)|{\widehat{D}} 	\\
 }; 
\path[->,font=\scriptsize,>=to, thin]
(1) edge node[above]{$\h$} (2) edge node[left]{$R$}   (l1)
(2) edge node[right]{$\widehat{L}$} (l2) 
(l1) edge  node[above]{$\h$} (l2);
\end{tikzpicture}
\end{center}
provides a manner of extending the functor $R$ from $\cC$ to $\widehat{\cC}$ via the Yoneda embedding. By a slight abuse of notation, we may write $R=\widehat{L}:\hC\to\hD$. 

Generally speaking, we know that left and right adjoints of $\widehat{L}$ are obtained as left and right Kan extensions $\text{Lan}_L\mathord{-}$, $\text{Ran}_L\mathord{-}$ along $L$, if they exist. With the above assumption on the existence of a right adjoint $R$, we see that
$$
\text{Ran}_L\mathord{-}\simeq\widehat{R},
$$
and so $\widehat{L}$ has a right adjoint $\widehat{R}$.

\subsection{The section functor}

We define a functor $\Gamma:\hC\to\Set$ as follows. For $\bF\in\Ob(\hC)$, let
$$
\Gamma(\bF)=\Hom(\be,\bF).
$$
For $X\in\Ob(\cC)$, we let $\Gamma(X)=\Gamma(\h_X)$. When $\cC$ has a terminal object $e$, we have an isomorphism $\Gamma(X)\simeq\Hom(e,X)$.

\subsection{Relative categories and slices}

Let $S\in\Ob(\cC)$. The  category $\cC_{\ov S}$ of objects of $\cC$ \emph{over $S$} (or a \emph{slice category}) is the category whose objects are arrows $f:T\to S$ of $\cC$, and, given another arrow $f':T'\to S$, 
$$
\Hom_{\cC_{\ov S}}(f,f')=\{u\in \Hom_{\cC}(T,T'): f=f'\circ u\}.
$$

Let $f:T\to S$ be an object of $\cC_{\ov S}$. The identity $S\to S$ is a terminal object of $\cC_{\ov S}$ so
$$
\Gamma(f)=\Gamma(T/S)=\{u\in\Hom(S,T): fu=\id_S\}
$$
is the set of \emph{sections} of $T$ over $S$.

The category $(\cC_{\ov S})_{\ov f}=(\cC_{\ov S})_{\ov T}$ is endowed with a canonical isomorphism
$$
\cC_{\ov T}\simeq (\cC_{\ov S})_{\ov T}.
$$


\subsection{Base change}\label{bc-abstr}

Let $S\in\Ob(\cC)$ and let $f:T\to S$ be an object of $\cC_{\ov S}$. We have canonical functors
$$
i_S=\ssum_S:\cC_{\ov S}\to\cC, \ \ \text{ and }\ \ \ i_f=i_{T/S}=\ssum_f:(\cC_{\ov S})_{\ov T}\to\cC_{\ov S},
$$
defined by $i_S(S'\to S)=S'$ and analogously for $i_{T/S}$.
Identifying $(\cC_{\ov S})_{\ov T}$ with $\cC_{\ov T}$, we have
$$
i_S\circ i_{T/S}=i_T,
$$
and we may write $i_f=\ssum_f:\cC_{\ov T}\to\cC_{\ov S}$,
$$\ssum_f(Z\xrightarrow{h}T)=(Z\xrightarrow{h}T\xrightarrow{f}S).$$
For $X\in\Ob(\cC)$ (resp.\ $Y\in\Ob(\cC_{\ov S})$, let $p_S(X)$ (resp.\ $p_f(Y)=p_{T/S}(Y)$) be the object
$X\times S$ (resp.\ $Y\times_S T$), if it exists, equipped with its second projection,
$$
X\times S\xrightarrow{p_S(X)} S, \ \ \ \text{ resp. }\ \ \ Y\times_ST\xrightarrow{p_{T/S}(Y)}T.
$$
We shall also write
$$
p_S(X)=S^*X=X_S\ \ \ \text{ and }\ \ \ p_{T/S}(Y)=f^*Y=Y_T.
$$
The resulting (partially defined) \emph{base change} functors
$$
p_S=S^*:\cC\to \cC_{\ov S}\ \ \ \text{ resp. }\ \ \ p_f=p_{T/S}=f^*:\cC_{\ov S}\to\cC_{\ov T}
$$
are right adjoint to $i_S$ (resp.\ $i_{T/S}$), i.e., we have a familiar relation
$$
\ssum_f\dashv f^*.
$$
Indeed, if $g:U\to S$ is an object of $\cC_{\ov S}$ (resp.\ $h:V\to T$ is an object of $\cC_{\ov T}$, then $i_S(g)=U$ (resp.\ $i_{T/S}(h)=f\circ h:V\to S$), so by the definition of product (resp.\ fibre product),
$$
{\cC_{\ov S}}(U,X\times S)\simeq{\cC}(U,X),\ \ \ \text{ resp. }\ \ \ {\cC_{\ov T}}(V,X\times_ST)\simeq{\cC_{\ov S}}(V,X).
$$ 
The restriction functors
$$
i_{S}^*:\hC\to\widehat{\cC_{\ov S}},\ \ \ \text{ resp. }\ \ \ i_{T/S}^*:\widehat{\cC_{\ov S}}\to\widehat{\cC_{\ov T}}
$$
are obtained by precomposition with functors $i_S$ (resp.\ $i_{T/S}$). We define the base changes of $\bF\in\hC$ and $\bG\in\widehat{\cC_{\ov S}}$ as
$$
\bF_S=i_S^*(F)=F\circ i_S, \ \ \ \text{ resp. }\ \ \ \bG_T=i_{T/S}^*(\bG)=\bG\circ i_{T/S}.
$$
Clearly,
$$
i_{T/S}^*\circ i_S^*=i_T^*,
$$
so for any $\bF\in\hC$,
$$
(\bF_S)_T=\bF_T.
$$
The notation for base change of objects of $\cC$ and of $\hC$ is consistent because, for $X\in\Ob(\cC)$, the functor $(\h_X)_S$ is representable if and only if the product $X\times S$ exists and in that case
$$
(\h_X)_S\simeq\h_{X_S}.
$$
Note that 
$$
\Gamma(\bF_S)\simeq\cC(\h_S,\bF)\simeq\bF(S),
$$
so in particular
$$
\Gamma(X_S)\simeq\cC(S,X)\simeq X(S).
$$

\subsection{The Grothendieck construction}\label{ord-groth-constr}

Given a presheaf $\bF\in\hC$, the Grothendieck construction forms the \emph{category of elements of $\bF$}, denoted
$$
\cC_{\ov\bF}
$$
as the category whose objects are elements of the disjoint union
$$
\coprod_{S\in\cC}\bF(S),
$$
and, given objects $x\in\bF(S)$, $y\in\bF(T)$, a morphism
$$
f:x\to y
$$
is given by a morphism $f\in\cC(S,T)$ with $\bF(f)(x)=y$. 

It is well-known that there is an equivalence of categories
$$
\hC_{\ov\bF}\simeq \widehat{\cC_{\ov\bF}}.
$$

In particular, given and object $S\in\cC$, the categories $\hC_{\ov \h_S}$ and $\widehat{\cC_{\ov S}}$ are equivalent. 

Moreover, for an object $f:T\to S$ of $\cC_{\ov S}$, $\h_f:\h_T\to\h_S$ is an object of $\hC_{\ov \h_S}$ and we have that
$$
\Gamma(\h_f)\simeq\Gamma(\h_T/\h_S)\simeq\Gamma(T/S)\simeq\Gamma(f).
$$

There is a canonical functor
$$
i_\bF:\cC_{\ov\bF}\to\cC
$$
defined by $i_\bF(x)=S$, if $x\in\bF(S)$. It yields the functor
$$
i_\bF^*: \hC\to \widehat{\cC_{\ov\bF}}\simeq\hC_{\ov\bF}, \bG\mapsto \bG_\bF=\bG\circ i_\bF,
$$
which coincides with the base change functor, i.e.,
$$
i_\bF^*\simeq \bF^*.
$$
In other words,
$$
\bG_\bF=(\bG\times\bF\to \bF),
$$
in perfect analogy with the considerations of \ref{bc-abstr}.

\subsection{Objects $\uHom$, $\uAut$}\label{uhom-psh}

For $\bF, \bG\in\Ob(\hC)$, we define an object $\uHom(\bF,\bG)$ of $\hC$ via
$$
\uHom(\bF,\bG)(S)=
\Hom_{\widehat{\cC_{\ov S}}}(\bF_S,\bG_S)\simeq
\Hom_{{\hC_{\ov \h_S}}}(\bF\times\h_S,\bG\times\h_S)\simeq
\Hom_{\hC}(\bF\times\h_S,\bG).
$$
This construction enjoys the following properties:
\begin{enumerate}
\item $\uHom(\be,\bG)\simeq\bG$;
\item $\uHom$ commutes with base change, i.e., 
$$
\uHom(\bF_S,\bG_S)\simeq\uHom(\bF,\bG)_S;
$$ 
\item the assignment $(\bF,\bG)\mapsto\uHom(\bF,\bG)$ is a bifunctor contravariant in $\bF$ and covariant in $\bG$.
\end{enumerate}

Let $\bE, \bF, \bG, \bH\in\Ob(\hC)$. We have a natural bijection
$$
\Hom(\bE\times\bF,\bG)\simeq \Hom(\bE,\uHom(\bF,\bG)).
$$
Moreover, there is an isomorphism of functors
$$
\uHom(\bE\times\bF,\bG)\simeq \uHom(\bE,\uHom(\bF,\bG)).
$$
Swapping the factors in the above product, we deduce
\begin{align*}
\Hom(\bE,\uHom(\bF,\bG)) & \simeq\Hom(\bF,\uHom(\bE,\bG)),\\
\uHom(\bE,\uHom(\bF,\bG)) & \simeq\uHom(\bF,\uHom(\bE,\bG)).
\end{align*}
Substituting $\bE=\be$ in the above, we see that
$$
\Gamma(\uHom(\bF,\bG))\simeq\Hom(\bF,\bG).
$$

The composition of $\Hom$ argument-wise yields a functorial morphism
$$
\uHom(\bF,\bG)\times\uHom(\bG,\bH)\to\uHom(\bF,\bH).
$$

We define a sub-object $\uIsom(\bF,\bG)$ of $\uHom(\bF,\bG)$ by
$$
\uIsom(\bF,\bG)(S)=\Isom(\bF_S,\bG_S).
$$   
In particular, we write
$$
\uEnd(\bF)=\uHom(\bF,\bF)\ \ \ \text{ and }\ \ \ \uAut(\bF)=\uIsom(\bF,\bF).
$$
Clearly,
$$
\Gamma(\uIsom(\bF,\bG))=\Isom(\bF,\bG), \ \Gamma(\uEnd(\bF))=\End(\bF), \ \Gamma(\uAut(\bF))=\Aut(\bF).
$$

All the above constructions apply to the special case of representable functors. Let $X,Y,Z\in\Ob(\cC)$. In the case when $\uHom(\h_X,\h_Y)$ is representable by an object of $\cC$, we denote that object
$$
\uHom(X,Y).
$$
If $Z\times X$ exists, then
$$
\Hom(Z,\uHom(X,Y))\simeq\Hom(Z\times X,Y),
$$
and that property characterises it when $\cC$ admits products.
We define objects $\uIsom(X,Y)$, $\uEnd(X)$, $\uAut(X)$ analogously, if they exist.

It is important to emphasise that the above subsection applies without effort to relative categories $\cC_{\ov S}$, and the following notation reflects the role of the base: if $T,T'\in\Ob(\cC_{\ov S})$, we will write 
$$
\uHom_S(T,T')
$$  
for the object $\uHom_{\cC_{\ov S}}(T/S,T'/S)$.

\subsection{Restriction of scalars}\label{weilres-alg}

Let $S\in\Ob(\cC)$ and $\bX,\bY,\bZ\in\Ob(\hC)$, where $\bX$ and $\bY$ are over $\bZ$, and $\bZ$ is over $S$. 

We define a sub-object $\uHom_{\bZ/S}(\bX,\bY)$ of $\uHom_S(\bX,\bY)$ in $\widehat{\cC_{\ov S}}$ by
$$
\uHom_{\bZ/S}(\bX,\bY)(S')=\Hom_{\bZ_{S'}}(\bX_{S'},\bY_{S'})\simeq\Hom_{\bZ}(\bX\times_SS',\bY),
$$
for  an object $S'$ of $\cC_{\ov S}$.

For every object $\bT$ in $\hC$ over $S$, we have natural bijections
\begin{align*}
\Hom_S(\bT,\uHom_{\bZ/S}(\bX,\bY)) & \simeq \Hom_{\bZ}(\bX\times_S\bT,\bY)\\
& \simeq \Hom_{\bZ}(\bZ\times_S\bT,\uHom_{\bZ}(\bX,\bY))\\
& \simeq \Hom_{\bZ}(\bX,\uHom_{\bZ}(\bZ\times_S\bT,\bY)),
\end{align*}
as well as analogous isomorphisms of $S$-functors with $\uHom$ in place of $\Hom$.

In the special case $\bX=\bZ$, we write
$$
\prod_{\bZ/S}\bY=\uHom_{\bZ/S}(\bZ,\bY).
$$
Thus
$$
\left(\prod_{\bZ/S}\bY\right)(S')=\Hom_{\bZ}(\bZ\times_SS',\bY)\simeq\Gamma(\bY_{S'}/\bZ_{S'}).
$$
The restriction of scalars functor
$$
\prod_{\bZ/S}:\widehat{\cC_{\ov \bZ}}\to\widehat{\cC_{\ov S}}
$$ 
is right adjoint to the base change functor from $S$ to $\bZ$, i.e., for an $S$-functor $\bT$, we have a natural bijection
$$
\Hom_S(\bT,\prod_{\bZ/S}\bY)\simeq\Hom_{\bZ}(\bT\times_S\bZ,\bY).
$$
Note the relation
$$
\uHom_{\bZ/S}(\bX,\bY)\simeq\uHom_{\bX/S}(\bX,\bY\times_{\bZ}\bX)=\prod_{\bX/S}(\bY\times_{\bZ}\bX),
$$
so, in particular, for $\bZ=S$ we derive
$$
\uHom_S(\bX,\bY)\simeq\prod_{\bX/S}\bY_{\bX}.
$$
We make the unit
$$
\eta:\id\to\prod_{\bZ/S}p_{\bZ/S}
$$
of the above adjunction explicit. For an $S$-functor $\bT$ and $S'\in\Ob(\cC_{\ov S})$,
$$
\left(\prod_{\bZ/S}\bT_{\bZ}\right)(S')=\Hom_{\bZ}(\bZ\times_SS',\bT_{\bZ})=\Hom_{\bZ}(S'_{\bZ},\bT_{\bZ})=\bT_{\bZ}(S'_{\bZ}),
$$
and the map
$$
\eta_{\bT}(S'):\bT(S')\to \bT_{\bZ}(S'_{\bZ})
$$
is the base change $u\mapsto u_{\bZ}$, for $u\in \bT(S')$.

The functor $\bY\mapsto\uHom_{\bZ/S}(\bX,\bY)$ commutes with fibre product,
$$
\uHom_{\bZ/S}(\bX,\bY\times_{\bZ}\bY')\simeq\uHom_{\bZ/S}(\bX,\bY)\times_S \uHom_{\bZ/S}(\bX,\bY).
$$
Consequently, if $\bY$ is a $\bZ$-group (resp. ring, or any reasonable algebraic structure), then $\uHom_{\bZ/S}(\bX,\bY)$ is an $S$-group (resp. ring, etc). 


\subsection{Constant objects}

Let $\cC$ be a category with coproducts 
and fibre products such that coproducts commute with base change. For a set $E$ and an object $S$ in $\cC$, we define the \emph{constant object}
$$
E_S=\coprod_{i\in E}S_i,
$$
where $S_i\in\Ob(\cC)$ are isomorphic copies of $S$. It is characterised by the property that, for all $T\in\Ob(\cC)$,
$$
\Hom_\cC(E_S,T)=\Hom_\Set(E,\Hom_\cC(S,T)).
$$
The object $E_S$ has a canonical projection to $S$, which makes it an object of $\cC_{\ov S}$. If $S'\to S$ is a morphism in $\cC$, we have
$$
E_{S'}=(E_S)_S',
$$
since we assumed that coproducts commute with base changes.

The assignment $E\mapsto E_S/S$ gives a functor $\Set\to\cC_{\ov S}$, which commutes with finite products. If we assume that $S\times_{S\coprod S}S$ equals the initial object $\emptyset$ in $\cC$, then this functor commutes with finite limits (including fibre products).


\section{Topoi}

This section contains an overview of basic topos theory, mostly following \cite{johnstone}. 
\subsection{Sites}

Let $\cC$ be a small category. 

A \emph{sieve} on an object $U\in\cC$ is a family $R$ of $\cC$-morphisms with codomain $U$ which is a left ideal for the composition  in $\cC$ in the sense that, for any $W\xrightarrow{\beta} V$ and any $V\xrightarrow{\alpha} U\in R$, we have that  $W\xrightarrow{\alpha\beta}U\in R$.

A \emph{Grothendieck topology} on $\cC$ specifies, for each $U\in\cC$, a set $J(U)$ of sieves on $U$, such that
\begin{enumerate}
\item for each $U$, the maximal sieve consisting of all morphisms with codomain $U$ belongs to $J(U)$;
\item if $R\in J(U)$ and $V\xrightarrow{f}U$ is a morphism in $\cC$, then the sieve
$$
f^*(R)=\{W\xrightarrow{\alpha}V: f\alpha\in R\}
$$
is in $J(V)$;
\item if $R\in J(U)$ and $S$ is a sieve on $U$ such that, for each $V\xrightarrow{f}U$ in $R$,
we have $f^*(S)\in J(V)$, then $S\in J(U)$. 
\end{enumerate}

A \emph{site} is a pair
$$
(\cC,J)
$$
consisting of a small category $\cC$ equipped with a Grothendieck topology $J$.

Note that a sieve $R$ on $U$ can be identified with the sub-presheaf 
$$
V\mapsto\{\alpha\in R: \text{domain}(\alpha)=V\}
$$
of $\h_U$.

A presheaf $F\in\hC$ is a \emph{sheaf} for the topology $J$, if, for every $U\in\cC$ and every $R\in J(U)$, each $\hC$-morphism $R\to F$ has a unique extension to a morphism $\h_U\to F$.

We write
$$
\sh(\cC,J)
$$
for the full subcategory of $\hC$ whose objects are $J$-sheaves.

Note, in the \emph{minimal topology}, where the only covering sieves are the maximal sieves, every presheaf is a sheaf.

\subsection{Grothendieck topoi}\label{groth-top}

A \emph{Grothendieck topos} is a category equivalent to the category of sheaves on a site.

Hence, by taking the minimal topology on a small category $\cC$, any presheaf category $\hC$ is a Grothendieck topos. 

The inclusion functor $\sh(\cC,J)\to \hC$ has a left-exact left adjoint
$$
L:\hC\to \sh(\cC,J),
$$
called the \emph{associated sheaf} functor.

\begin{theorem}[Giraud]\label{giraud-th}
A category $\cE$ is a Grothendieck topos if and only if the following conditions are satisfied:
\begin{enumerate}
\item $\cE$ has finite limits;
\item $\cE$ has all set-indexed coproducts, and they are disjoint and universal;
\item equivalence relations in $\cE$ have universal coequalisers and they are effective;
\item epimorphisms in $\cE$ are coequalisers;
\item $\cE$ has small hom-sets, i.e., for any $X,Y\in\cE$, $\cE(X,Y)$ is a set;
\item $\cE$ has a set of generators.
\end{enumerate}

\end{theorem}

\subsection{Elementary topoi}\label{elem-top}

A category $\cE$ is an \emph{elementary topos} if
\begin{enumerate}
\item $\cE$ has finite limits (equivalently, all pullbacks and a terminal object $I$);
\item $\cE$ is \emph{cartesian closed}, i.e., for each $X\in\cE$, the functor $\mathord{-}\times X$ has a right adjoint 
$$
[X,\mathord{-}]:\cE\to \cE;
$$
\item $\cE$ has a \emph{subobject classifier}, i.e., an object $\Omega$ and a morphism $I\xrightarrow{t}\Omega$ such that, for each monomorphism $Y\xrightarrow{u} X$ in $\cE$, there is a unique morphism $\chi_u:X\to\Omega$ (called the \emph{classifying map} of $u$) making 
\begin{center}
 \begin{tikzpicture} 
\matrix(m)[matrix of math nodes, row sep=2em, column sep=2em, text height=1.5ex, text depth=0.25ex]
 {
 |(1)|{Y}		& |(2)|{I} 	\\
 |(l1)|{X}		& |(l2)|{\Omega} 	\\
 }; 
\path[->,font=\scriptsize,>=to, thin]
(1) edge node[above]{} (2) edge node[left]{$u$}   (l1)
(2) edge node[right]{$t$} (l2) 
(l1) edge  node[above]{$\chi_u$} (l2);
\end{tikzpicture}
\end{center}
\end{enumerate}
a pullback diagram.

\begin{proposition}
Every Grothendieck topos is elementary.
\end{proposition}
\begin{proof}
Let $\cE\simeq \sh(\cC,J)$. Arbitrary limits in $\hC$ are defined argument-wise so they clearly exist, and it is readily verified that finite limits of $J$-sheaves are also sheaves. We have shown in \ref{uhom-psh} that $\hC$ is cartesian closed, and we can directly verify that $[X,Y]$ is a $J$-sheaf if $X$, $Y$ are. 

In $\hC$, the subobject classifier $\Omega$ is the presheaf
$$
\Omega(U)\simeq\hC(\h_U,\Omega)\simeq\{\text{sub-presheaves of }\h_U\}\simeq\{\text{sieves on }U\}.
$$
In $\cE$, $\Omega(U)$ is the set of subobjects of the associated sheaf $L(\h_U)$ which are sheaves.
\end{proof}

\subsection{Geometric morphisms}\label{geom-morphisms}

A \emph{geometric morphism} 
$$
\cF\xrightarrow{f}\cE
$$
consists of a pair of functors $f_*:\cF\to \cE$ and $f^*:\cE\to \cF$ such that
$$
f^*\dashv f_*
$$
and $f^*$ is left exact.

A \emph{natural transformation} 
$$
\eta: f\to g
$$
between geometric morphisms  \begin{tikzcd}[cramped,sep=small]
\cF \ar[yshift=2pt]{r}{f} \ar[yshift=-2pt]{r}[swap]{g} & \cE
\end{tikzcd}
is a natural transformation $f^*\xrightarrow{\eta} g^*$.

Topoi, geometric morphisms and natural transformations give rise to a 2-category denoted
$$
\Top.
$$

A geometric morphism $\cF\xrightarrow{f}\cE$ is \emph{essential,} if it stems from a triple of adjoint functors
$$
f_!\dashv f^*\dashv f_*.
$$

A geometric morphism $\cF\xrightarrow{f}\cE$ is 
\begin{enumerate}
\item an \emph{inclusion}, if the counit $f^*f_*\to\id$ of the adjunction $f^*\dashv f_*$ is an isomorphism;
\item a \emph{surjection}, the the unit $\id\to f_*f^*$ of $f^*\dashv f_*$ is a monomorphism.
\end{enumerate}

A functor between topoi is \emph{logical}, if it preserves finite limits, internal homs/exponentials, and the subobject classifier.

\begin{remark}
For a site $(\cC,J)$, the pair of functors  
\begin{tikzcd}[cramped,sep=small]
\sh(\cC,J)  \arrow[r, hook, yshift=-2pt] & \arrow[l,yshift=2pt ,"L" above] \hC
\end{tikzcd}
shows that a Grothendieck topos admits a geometric embedding into a presheaf category. 

In fact, a category is a Grothendieck topos if and only if it admits a geometric embedding into a presheaf category.
\end{remark}

A geometric morphism $\cF\xrightarrow{f}\cE$ satisfies 
$$
f_*[f^*Y,X]\simeq[Y,f_*X],
$$
naturally in $X\in\cF$ and $Y\in\cE$. Moreover, we have morphisms
\begin{align*}
\Phi_f:\, & f^*[Y,Y']\to [f^*Y,f^*Y'],  \\
\Psi_f:\, & f_*[X,X']\to [f_*X,f_*X'], \\
\Xi_f:\, & [f_!X,Y]\to f_*[X,f^*Y], & \\
\Lambda_f:\, & f_!(X\times f^*Y)\to f_!X\times Y, 
\end{align*}
natural in $X,X'\in\cF$ and $Y,Y'\in\cE$, 
where the last two morphisms are defined when $f$ is essential.

\begin{lemma}[{\cite[Expos\'e~IV, 10.6]{sga4.1}}]\label{geom-m-iso}
The morphism $\Psi_f$ is an isomorphism for all $X'$ if and only if the adjunction morphism $f^*f_*X\to X$ is an isomorphism. Hence, $\Psi_f$ is an isomorphism for all $X,X'$ if and only if $f$ is in inclusion. 

If $\cF$, $\cE$ are Grothendieck topoi (or it is given that $f$ is essential), the following conditions are equivalent:
\begin{enumerate}
\item $\Phi_f$ is an isomorphism for all $Y,Y'$;
\item $f$ is essential and $\Xi_f$ is an isomorphism for all $X,Y$;
\item $f$ is essential and $\Lambda_f$ is an isomorphism for all $X,Y$.
\end{enumerate}
\end{lemma}


\subsection{The fundamental theorem of topos theory}\label{fund-th-topos}

\begin{theorem}[Lawvere-Tierney]\label{fund-th-top}
Let $X\xrightarrow{f}Y$ be a morphism in a topos $\cE$. Then
\begin{enumerate}
\item $\cE_{\ov X}$ is a topos;
\item the base change/pullback functor $f^*:\cE_{\ov Y}\to\cE_{\ov X}$ is logical, and it has a right adjoint
$$
\sprod_f:\cE_{\ov X}\to\cE_{\ov Y}.
$$
\end{enumerate}
\end{theorem}
In particular, a topos is \emph{locally cartesian closed}, in the sense that each slice $\cE_{\ov X}$ is cartesian closed. We write $$[\mathord{-},\mathord{-}]_X$$ for the internal hom in $\cE_{\ov X}$.
\begin{corollary}
For any morphism $X\xrightarrow{f}Y$ in $\cE$, the adjunctions
$$
\ssum_f\dashv f^*\dashv \sprod_f
$$
give rise to an essential geometric \emph{localisation morphism}
$$
\cE_{\ov X}\xrightarrow{i_f}\cE_{\ov Y},
$$
where we take $(i_f)_!=f_!=\ssum_f$, $i_f^*=f^*$ and $(i_f)_*=f_*=\sprod_f$.
\end{corollary}

The geometric morphism $i_f$ is an inclusion if and only if $f$ is a monomorphism, and $i_f$ is a surjection if and only if $f$ is an epimorphism.

With the above notation,
the \emph{Beck-Chevalley condition} states that a diagram
\begin{center}
 \begin{tikzpicture} 
\matrix(m)[matrix of math nodes, row sep=2em, column sep=2em, text height=1.5ex, text depth=0.25ex]
 {
 |(1)|{\tilde{X}}		& |(2)|{\tilde{Y}} 	\\
 |(l1)|{X}		& |(l2)|{Y} 	\\
 }; 
\path[->,font=\scriptsize,>=to, thin]
(1) edge node[above]{$\tilde{f}$} (2) edge node[left]{$\tilde{p}$}   (l1)
(2) edge node[right]{$p$} (l2) 
(l1) edge  node[above]{$f$} (l2);
\end{tikzpicture}
\end{center}
is a pullback if and only if 
$$
\tilde{p}_!\,\tilde{f}^*\simeq f^*\, p_!.
$$
Using \ref{geom-m-iso}, together with the fact that $f^*$ is logical, or directly using Beck-Chevalley condition, we obtain that $\Phi_{i_f}$, $\Xi_{i_f}$ and $\Lambda_{i_f}$ are isomorphisms. In particular, the adjunction $f_!\dashv f^*$ is \emph{$\cE_{\ov Y}$-enriched} in the sense that, for any $A\xrightarrow{a} X\in\cE_{\ov X}$ and $B\xrightarrow{b}Y\in\cE_{\ov Y}$, we have the \emph{Wirthm\"uller} isomorphism
$$
f_*[a,f^*b]_X\simeq[f_!a,b]_Y.
$$
It follows, by taking $a=\id_X$, that
$$
f_*f^*b=[f,b]_Y.
$$
In particular, if $X,Y,S\in\cE$, writing $i_S:\cE_{\ov S}\to \cE$ for the localisation morphism, we have that
$$
{i_S}_*i_S^*Y=[S,Y],
$$
and we obtain the familiar relation
$$
\cE(S,[X,Y])\simeq \cE(S\times X,Y)\simeq \cE_{\ov S}(i_S^*X,i_S^*Y).
$$

\subsection{Image factorisations}\label{img-fact}

\subsection{Logical connectives on the subobject classifier} 

Logical operations on the subobject classifier $\Omega$ of an elementary topos $\cE$ are the following. 

The map `true' is the map 
$$\ter\xrightarrow{t}\Omega$$ introduced alongside $\Omega$ in \ref{elem-top}, and it can be viewed as the classifying map of $\ter\to\ter$. 

On the other hand, the `false' map
$$
\ter\xrightarrow{f}\Omega
$$
is the classifying map of $\ini\mono\ter$.

Conjunction 
$$
\land:\Omega\times\Omega\to\Omega
$$
is the classifying map of $\ter\xrightarrow{(t,t)}\Omega\times\Omega$.

Disjunction 
$$
\lor:\Omega\times\Omega\to\Omega
$$
is the classifying map of the union of subobjects $\Omega\times I\xrightarrow{\id\times t}\Omega\times\Omega$ and $I\times\Omega\xrightarrow{t\times\id}\Omega\times\Omega$.

Negation $$\lnot:\Omega\to\Omega$$ is the classifying map of $\ter\xrightarrow{f}\Omega$.

We make $\Omega$ into a poset with the order relation 
$$\Omega_1=\Eq(\land,\pi_1)\subseteq\Omega\times\Omega,$$
and the implication
$$
\Rightarrow:\Omega\times\Omega\to\Omega
$$ 
is the classifying map of $\Omega_1\to \Omega\times\Omega$.
 
The above structure makes the poset $\Omega$ into an internal Heyting algebra in $\cE$. 

A topos $\cE$ is called \emph{Boolean,} if $\lnot\lnot=\id_\Omega$, and \emph{De Morgan,} if $\lnot \circ\land=\lor\circ \lnot\times\lnot$.

\subsection{Subobjects}\label{subobj}

Given an object $X$ of a topos $\cE$, arguing through characteristic maps of subobjects,  the Heyting algebra structure of $\Omega$ makes
$$
\Sub(X)
$$ 
a Heyting algebra.

Given a morphism $f:X\to Y$ in $\cE$,  we can describe the functors
\begin{center}
 \begin{tikzpicture} 
 [cross line/.style={preaction={draw=white, -,
line width=3pt}}]
\matrix(m)[matrix of math nodes, minimum size=1.7em,
inner sep=0pt, 
row sep=0em, column sep=4em, text height=1.5ex, text depth=0.25ex]
 { 
  |(dc)|{\Sub(X)}	 &
 |(c)|{\Sub(Y)} 	      \\ };
\path[->,font=\scriptsize,>=to, thin]
(c) edge[draw=none] node (mid) {} (dc)
(c) edge node (fo) [pos=0.2,above=-1pt]{$f^{-1}$}  (dc)
(dc) edge [bend right=35] node (ss) [below]{$\forall_f$} (c) edge [bend left=35] node (ps) [above]{$\exists_f$} (c)
(mid) edge[draw=none] node[sloped]{$\vdash$} (ss)
(mid) edge[draw=none] node[sloped]{$\dashv$} (ps)
;
\end{tikzpicture}
\end{center}
as follows. 

The functor $f^{-1}$ is defined as the restriction of the pullback functor 
$f^*:\cE_{\ov Y}\to\cE_{\ov X}$ to $\Sub(Y)$.

The functor $\forall_f$ is the restriction of $\sprod_f:\cE_{\ov X}\to\cE_{\ov Y}$ to $\Sub(X)$.

The functor $\ssum_f$ does not restrict to $\Sub(X)$, so we define
$$
\exists_f(U\mono X)=\im(U\mono X\xrightarrow{f}Y).
$$

\subsection{Power objects}\label{power-ob}

Given an object $X$ in a topos $\cE$, its power object
$$
PX=[X,\Omega]
$$
inherits the Heyting algebra structure from $\Omega$. 

The membership relation $$\text{\large$\in$}_X\mono PX\times X$$
 is the subobject classified by the evaluation map
$[X,\Omega]\times X\to \Omega$.

For any other object $Y$ and a subobject $R\mono Y\times X$, there is a unique morphism $Y\xrightarrow{r}PX$ such that
\begin{center}
 \begin{tikzpicture} 
\matrix(m)[matrix of math nodes, row sep=2em, column sep=3em, text height=1.5ex, text depth=0.25ex]
 {
 |(1)|{R}		& |(2)|{\text{\large$\in$}_X} 	\\
 |(l1)|{Y\times X}		& |(l2)|{PX\times X} 	\\
 }; 
\path[->,font=\scriptsize,>=to, thin]
(1) edge node[above]{} (2) edge[>->] node[left]{}   (l1)
(2) edge node[right]{} (l2) 
(l1) edge  node[above]{$r\times \id$} (l2);
\end{tikzpicture}
\end{center}
is a pullback.

DISCUSS the functor 
$$
P=\cE^\op\to \cE.
$$
!!!!!!!!!!!!!!!!!!

Given a morphism $X\xrightarrow{f} Y$ in $\cE$, the above construction applied to subobjects $\exists_{\id\times f}(\text{\large$\in$}_X)\mono PX\times Y$ and 
$\forall_{\id\times f}(\text{\large$\in$}_X)\mono PX\times Y$  yields morphisms
\begin{center}
 \begin{tikzpicture} 
 [cross line/.style={preaction={draw=white, -,
line width=3pt}}]
\matrix(m)[matrix of math nodes, minimum size=1.7em,
inner sep=0pt, 
row sep=0em, column sep=4em, text height=1.5ex, text depth=0.25ex]
 { 
  |(dc)|{PX}	 &
 |(c)|{PY} 	      \\ };
\path[->,font=\scriptsize,>=to, thin]
(c) edge[draw=none] node (mid) {} (dc)
(c) edge node (fo) [pos=0.5,above=-1pt]{$Pf$}  (dc)
(dc) edge [bend right=35] node (ss) [above]{$\forall f$} (c) edge [bend left=35] node (ps) [above]{$\exists f$} (c)
;
\end{tikzpicture}
\end{center}
such that $\exists f$ is left adjoint, and $\forall f$ is right adjoint to $Pf$ in the sense of posets $PX$ and $PY$.  We observe that the functors from \ref{subobj} are externalisations of these morphisms.

\section{Enriched category theory}

In this section, we adapt the expositions  of enriched category theory from \cite{borceux-2}, \cite{kelly} and a few other sources to suit our subsequent applications.

\subsection{Enriched categories and functors}

Let $\cV$ be a \emph{monoidal category}, with \emph{tensor product} bifunctor $\otimes:\cV\times\cV\to\cV$ and \emph{unit object} $I\in\Ob(\cV)$.  We write 
$$
V=\cV(I,\mathord{-}):\cV\to\Set
$$
for the functor of points on $\cV$.

A \emph{$\cV$-category $\cA$} 
consists of 
\begin{enumerate}
\item a class $\Ob(\cA)$ of `objects';
\item for every pair $X,Y\in\Ob(\cA)$, a `hom' object $\cA(X,Y)$ of $\cV$;
\item for every triple $X,Y,Z\in\Ob(\cA)$, a `composition' morphism in $\cV$,
$$
c_{XYZ}:\cA(X,Y)\otimes\cA(Y,Z)\to\cA(A,Z);
$$
\item for every $X\in\Ob(\cA)$, a `unit' morphism in $\cV$,
$$
u_X:I\to\cA(X,X),
$$
\end{enumerate}
such that, for any $X,Y,Z,T\in\cA$, the diagrams
\begin{center}
 \begin{tikzpicture} 
\matrix(m)[matrix of math nodes, row sep=2em, column sep=3.5em, text height=1.5ex, text depth=0.25ex]
 {
 |(1)|{(\cA(X,Y)\otimes\cA(Y,Z))\otimes\cA(Z,T)}		& |(2)|{\cA(X,Z)\otimes\cA(Z,T)} 	\\
 |(m1)|{\cA(X,Y)\otimes(\cA(Y,Z)\otimes\cA(Z,T))}     &    \\
 |(l1)|{\cA(X,Y)\otimes\cA(Y,T)}		& |(l2)|{\cA(X,T)} 	\\
 }; 
\path[->,font=\scriptsize,>=to, thin]
(1) edge node[above]{$c_{XYZ}\otimes\id$} (2) edge node[left]{$a$}   (m1)
(m1) edge node[left]{$\id\otimes c_{YZT}$}   (l1)
(2) edge node[right]{$c_{XZT}$} (l2) 
(l1) edge  node[above]{$c_{XYT}$} (l2);
\end{tikzpicture}
\end{center}
and
\begin{center}
\begin{tikzpicture} 
\matrix(m)[matrix of math nodes, row sep=2em, column sep=2em, text height=1.5ex, text depth=0.25ex]
 {
|(0)|{I\otimes\cA(X,Y)}  &  |(1)|{\cA(X,Y)}		& |(2)|{\cA(X,Y)\otimes I} 	\\
|(l0)|{\cA(X,X)\otimes\cA(X,Y)} & |(l1)|{\cA(X,Y)}		& |(l2)|{\cA(X,Y)\otimes\cA(Y,Y)} 	\\
 }; 
\path[->,font=\scriptsize,>=to, thin]
(0) edge node[above]{$l$} (1)
(l0) edge node[above]{$c_{XXY}$} (l1)
(0) edge  node[left]{$u_X\otimes\id$} (l0)
(2) edge node[above]{$r$} (1) 
(1) edge node[right]{$\id$}   (l1)
(2) edge node[right]{$\id\otimes u_Y$} (l2) 
(l2) edge node[above]{$c_{XYY}$} (l1) 
;
\end{tikzpicture}
\end{center}
are commutative, where we wrote $a$, $l$ and $r$ for the associativity, left and right unit isomorphisms in $\cV$.

We also say that $\cA$ is \emph{enriched over $\cV$}.

When $\cV$ is symmetric monoidal, we can form the \emph{opposite} $\cV$-category $\cA^\op$ by setting
$$
\cA^\op(X,Y)=\cA(Y,X).
$$

Given $\cV$-categories $\cA$, $\cB$, a \emph{$\cV$-functor} $F:\cA\to\cB$ consists of
\begin{enumerate}
\item for every $X\in\Ob(\cA)$, an object $F(X)\in\Ob(\cB)$;
\item for every pair $X,Y\in\Ob(\cA)$, a morphism in $\cV$,
$$
F_{XY}:\cA(X,Y)\to\cB(F(X),F(Y)),
$$
\end{enumerate}
so that, for all $X,Y,Z\in\cA$, the diagrams 
\begin{center}
 \begin{tikzpicture} 
\matrix(m)[matrix of math nodes, row sep=2em, column sep=3.5em, text height=1.5ex, text depth=0.25ex]
 {
 |(1)|{\cA(X,Y)\otimes\cA(Y,Z)}		& |(2)|{\cA(X,Z)} 	\\
 |(l1)|{\cB(FX,FY)\otimes\cB(FY,FZ)}		& |(l2)|{\cB(FX,FZ)} 	\\
 }; 
\path[->,font=\scriptsize,>=to, thin]
(1) edge node[above]{$c_{XYZ}$} (2) edge node[left]{$F_{XY}\otimes F_{YZ}$}   (l1)
(2) edge node[right]{$F_{XZ}$} (l2) 
(l1) edge  node[above]{$c_{FX,FY,FZ}$} (l2);
\end{tikzpicture}
 \begin{tikzpicture} 
\matrix(m)[matrix of math nodes, row sep=2em, column sep=2em, text height=1.5ex, text depth=0.25ex]
 {
 |(1)|{I}		& |(2)|{\cA(X,X)} 	\\
		& |(l2)|{\cB(FX,FX)} 	\\
 }; 
\path[->,font=\scriptsize,>=to, thin]
(1) edge node[above]{$u_X$} (2) edge node[left,pos=.7]{$u_{FX}$}   (l2)
(2) edge node[right]{$F_{XX}$} (l2) 
;
\end{tikzpicture}
\end{center}
commute.

If $F,G:\cA\to\cB$ are two $\cV$-functors, a \emph{$\cV$-natural transformation} 
$\alpha: F\Rightarrow G$ is a collection of $\cV$-morphisms
$$
\alpha_X:I\to\cB(F(X),G(X)),
$$
indexed by $X\in\Ob(\cA)$, such that, for all $X,Y\in\cA$, the diagram
\begin{center}
\begin{tikzpicture} 
\matrix(m)[matrix of math nodes, row sep=2em, column sep=1.5em, text height=1.5ex, text depth=0.25ex]
 {
|(l0)|{I\otimes\cA(X,Y)} & 	 |(1)|{\cA(X,Y)}		& |(l2)|{\cA(X,Y)\otimes I} 	\\
|(d0)|{\cB(FX,GX)\otimes\cB(GX,GY)} & 	|(b1)|{\cB(FX,GY)}	& |(d2)|{\cB(FX,FY)\otimes\cB(FY,GY)} 	\\
 }; 
\path[->,font=\scriptsize,>=to, thin]
(1) edge node[above]{$l^{-1}$} (l0)
(1) edge node[above]{$r^{-1}$} (l2)
(l0) edge node[left]{$\alpha_X\otimes G_{XY}$} (d0) 
(l2) edge node[right]{$F_{XY}\otimes\alpha_Y$} (d2) 
(d0) edge node[above,pos=0.4]{$c$} (b1) 
(d2) edge node[above,pos=0.4]{$c$} (b1) 
;
\end{tikzpicture}
\end{center}
commutes.

Given a monoidal category $\cV$, the small $\cV$-categories, together with the $\cV$-functors and the $\cV$-natural transformations, form a 2-category denoted $$\cV\da\cat.$$

\subsection{Monoidal closed categories}

A symmetric monoidal category $\cV$ is \emph{closed} if, for each $B\in\Ob(\cV)$, the functor 
$\mathord{-}\otimes B:\cV\to\cV$ has a right adjoint denoted
$[B,\mathord{-}]$, i.e., we have bijections
$$
\cV(A\otimes B,C)\simeq\cV(A,[B,C]),
$$
natural in $A, C\in\cV$.

The \emph{internal hom objects} 
$$
[B,C]=\uHom(B,C)\in\Ob(\cV)
$$ 
for $B,C\in\Ob(\cV)$ endow the category $\cV$ with the structure of a $\cV$-category.

We recall that a category $\cV$ is \emph{cartesian closed}, if it has products, and it is monoidal closed for the monoidal structure given by direct product. In that case, the unit object $I$ is also a terminal object and we denote it by $e$.

%

\subsection{End/coend calculus}

Let $\cV$ be a symmetric monoidal closed category, let $\cA$ be a $\cV$-category, and consider a $\cV$-functor
$$
T:\cA^{\op}\otimes\cA\to\cV.
$$
Suppose that there exists a pair $(K,\lambda)$ consisting of an object $K\in\cV$ and a family $\lambda_X:K\to T(X,X)$ indexed by $X\in\cA$ which is \emph{$\cV$-natural} in the sense that, for each $X,X'\in\cA$, the diagram
\begin{center}
 \begin{tikzpicture} 
\matrix(m)[matrix of math nodes, row sep=2em, column sep=3em, text height=1.5ex, text depth=0.25ex]
 {
 |(1)|{\cA(X,X')}		& |(2)|{[T(X,X),T(X,X')]} 	\\
 |(l1)|{[T(X',X'),T(X,X')]}		& |(l2)|{[K,T(X,X')]} 	\\
 }; 
\path[->,font=\scriptsize,>=to, thin]
(1) edge node[above]{$T(X,\mathord{-})$} (2) edge node[left]{$T(\mathord{-},X')$}   (l1)
(2) edge node[right]{$[\lambda_X,\id]$} (l2) 
(l1) edge  node[above]{$[\lambda_{X'},\id]$} (l2);
\end{tikzpicture}
\end{center}
commutes, and \emph{universal} in the sense that any other $\cV$-natural family $\lambda'_X:K'\to T(X,X)$ is given by $\lambda'_X=\lambda_X\circ f$ for a unique morphism $f:K'\to K$. We write 
$$
K=\eend_{X\in\cA}T(X,X)
$$
and call it the \emph{end} of $T$. 

By adjunction, the $\cV$-naturality condition can be expressed as the commutativity of
\begin{center}
 \begin{tikzpicture} 
\matrix(m)[matrix of math nodes, row sep=2em, column sep=3em, text height=1.5ex, text depth=0.25ex]
 {
 |(1)|{K}		& |(2)|{T(X,X)} 	\\
 |(l1)|{T(X',X')}		& |(l2)|{[\cA(X,X'),T(X,X')]} 	\\
 }; 
\path[->,font=\scriptsize,>=to, thin]
(1) edge node[above]{$\lambda_X$} (2) edge node[left]{$\lambda_{X'}$}   (l1)
(2) edge node[right]{$\rho_{XX'}$} (l2) 
(l1) edge  node[above]{$\sigma_{XX'}$} (l2);
\end{tikzpicture}
\end{center}
where $\rho_{XX'}$ is the transform of $T(X,\mathord{-})_{XX'}$ and $\sigma_{XX'}$ is the transform of $T(\mathord{-},X')_{X'X}$.

Hence, when $\cA$ is small and $\cV$ is complete, the end of any $T$ exists and it is given as the equaliser
\begin{center}
 \begin{tikzpicture} 
\matrix(m)[matrix of math nodes, row sep=0em, column sep=1.7em, text height=1.5ex, text depth=0.25ex]
 {
|(0)|{\eend_{X\in\cA}T(X,X)} & |(1)|{\prod_{X\in\cA}T(X,X)}		& |(2)|{\prod_{X,X'\in\cA}[\cA(X,X'),T(X,X')]} 	\\
 }; 
\path[->,font=\scriptsize,>=to, thin,yshift=12pt]
(0) edge node[above]{$\lambda$} (1)
([yshift=2pt]1.east) edge node[above]{$\rho$} ([yshift=2pt]2.west) 
([yshift=-2pt]1.east)edge node[below]{$\sigma$}   ([yshift=-2pt]2.west) 
;
\end{tikzpicture}
\end{center}

The universal property of ends can be enriched to the relation
$$
[V,\eend_XT(X,X)]\simeq \eend_X[V,T(X,X)]
$$
for any $V\in\cV$. 

 The following facts are well-known in enriched category theory, and can be found in \cite{kelly}. For simplicity, we assume that $\cV$ is complete so that all the relevant ends exist.

Formation of ends is \emph{functorial} in the sense that $\eend_{X\in\cA}$ yields a functor from the category of $\cV$-functors $\cA^\op\otimes\cA\to \cV$ to $\cV$.

Ends commute with limits in the following sense. If $T_i:\cA^\op\otimes\cA\to\cV$ are $\cV$-functors, then
$$
\eend_X \lim_i T_i(X,X)=\lim_i \eend_X T_i(X,X).
$$

Let $\cA$ and $\cB$ be $\cV$-categories, and let $T:(\cA\otimes\cB)^\op\otimes(\cA\otimes\cB)\to\cV$ be a $\cV$-functor. The \emph{Fubini Theorem} for ends states that 
\begin{align*}
\eend_{(X,Y)\in\cA\otimes\cB}T(X,Y,X,Y)& \simeq \eend_{Y\in\cB}\eend_{X\in\cA}T(X,Y,X,Y)\\ & \simeq \eend_{X\in\cA}\eend_{Y\in\cB}T(X,Y,X,Y).
\end{align*}

Given a $\cV$-functor $T:\cA^\op\otimes\cA\to\cV$, its \emph{coend} $$\coend^{X\in\cA}T(X,X)$$ is dual to the notion of end, and it is determined by the universal property
$$
[\coend^X T(X,X),V]\simeq \eend_X[T(X,X),V],
$$
for all $V\in\cV$.

\subsection{Enriched functor categories}\label{enr-fun-cat}
Let $\cV$ be a complete symmetric monoidal closed category and let $\cA$ and $\cB$ be $\cV$-categories with $\cA$ small. The category $$[\cA,\cB]=\cV[\cA,\cB]$$
of $\cV$-functors $\cA\to\cB$ has a structure of a $\cV$-category as follows. Given $\cV$-functors $F,G:\cA\to\cB$, the internal hom object
$$
[\cA,\cB](F,G)\in\Ob(\cV),
$$
is given as the end
$$
[\cA,\cB](F,G)=\eend_{X\in\cA}\cB(FX,GX).
$$
More explicitly, it is the equaliser 
\begin{center}
 \begin{tikzpicture} 
\matrix(m)[matrix of math nodes, row sep=0em, column sep=1.7em, text height=1.5ex, text depth=0.25ex]
 {
|(0)|{[\cA,\cB](F,G)} & |(1)|{\displaystyle\prod_{X\in\cA}\cB(FX,GX)}		& |(2)|{\displaystyle\prod_{X,X'\in\cA}[\cA(X,X'),\cB(FX,GX')]} 	\\
 }; 
\path[->,font=\scriptsize,>=to, thin,yshift=12pt]
(0) edge node[above]{} (1)
([yshift=2pt]1.east) edge node[above]{} ([yshift=2pt]2.west) 
([yshift=-2pt]1.east)edge node[below]{}   ([yshift=-2pt]2.west) 
;
\end{tikzpicture}
\end{center}
in $\cV$, where the two parallel arrows express naturality in this context. A $\cV$-natural transformation $\alpha:F\Rightarrow G$ corresponds to a point 
$$\alpha\in V([\cA,\cB](F,G)),$$ i.e.,
to a $\cV$-morphism 
$\alpha:I\to[\cA,\cB](F,G)$.

\subsection{Enriched Yoneda}

Let $\cC$ be a $\cV$-category. The category of \emph{$\cV$-presheaves on $\cC$} is the $\cV$-category
$$
\hC=\cV[\cC^\circ,\cV],
$$
where $\cV$ is a $\cV$-category when considered with internal homs.

Clearly, an object $X\in\Ob(\cC)$ yields the $\cV$-presheaf
$$
\h_X:\cC^\circ\to\cV,\ \ \ \ \ \h_X(S)=\cC(S,X).
$$
A $\cV$-presheaf is called \emph{representable}, if it isomorphic to a presheaf of the form $\h_X$ for some object $X$ in $\cC$.

The enriched Yoneda lemma states that, for a complete symmetric monoidal closed category $\cV$, a small $\cV$-category $\cC$, an object $X\in\Ob(\cC)$ and a $\cV$-preseheaf $F:\cC^\circ\to\cV$, the object 
$\hC(\h_X,F)$ exists and there is an isomorphism 
$$
\hC(\h_X,F)\simeq F(X)
$$
in $\cV$, which is $\cV$-natural both in $F$ and in $X$.

Consequently, we obtain a fully faithful $\cV$-functor
$$
\h:\cC\to \hC, 
\ \ \ \ \ X\mapsto \h_X
$$
called the \emph{$\cV$-Yoneda embedding}.

Another useful formula for our (co)end calculus is provided by the \emph{ninja Yoneda lemma} stating that, for $F\in\hC$,
$$
F\simeq \eend_{X\in\cC}[\cC(X,\mathord{-}),F(X)]\simeq \coend^{X\in\cC}F(X)\otimes\cC(\mathord{-},X). 
$$

\subsection{Day convolution}\label{day-conv}

Let $\cV$ be a complete cocomplete symmetric monoidal category, and let $\cC$ be a small $\cV$-monoidal category. The monoidal structure of $\cC$ gives rise to the \emph{Day convolution tensor product} on the $\cV$-functor category $[\cC,\cV]$, 
$$
*:[\cC,\cV]\otimes [\cC,\cV]\to [\cC,\cV],
$$
defined by
$$
F*G(Z)=\coend^{(X,Y)\in\cC\times\cC}\cC(X\otimes_\cC Y,Z)\otimes_\cV F(X)\otimes_\cV G(Y).
$$
The monoidal $\cV$-category $([\cC,\cV],*)$ admits a biclosed structure, i.e., the $\cV$-functor
$\mathord{-}*F$ has a right adjoint $\mathord{-}/F$ given by
$$
(G/F)(X)=\eend_{Y\in\cC}[F(Y),G(X\otimes Y)],
$$
while the $\cV$-functor $F*\mathord{-}$ has a right adjoint $F\backslash\mathord{-}$ given by
$$
(F\backslash G)(X)=\eend_{Y\in\cC}[F(Y),G(Y\otimes X)].
$$
A symmetry for $\otimes_\cC$ yields an isomorphism $G/F\simeq F\backslash G$ and in that case $([\cC,\cV],*)$ is symmetric monoidal closed.

\subsection{Change of base along monoidal morphisms}

Let $F:\cV\to\cW$ be a morphism of monoidal categories, given by 
\begin{enumerate}
\item a functor $F:\cV\to\cW$;
\item for each $X,Y\in\Ob(\cV)$, a $\cW$-morphism $F(X)\otimes F(Y)\to F(X\otimes Y)$;
\item a $\cW$-morphism $J\to F(I)$, where $I$ is the unit of $\cV$ and $J$ is the unit of $\cW$.
\end{enumerate}

 It induces a 2-functor 
$$
F_{*}:\cV\text{-}\mathbf{Cat}\to \cW\text{-}\mathbf{Cat}
$$
called the \emph{base change} functor as follows. 
Given a $\cV$-category $\cA$, the $\cW$-category $F_{*}(\cA)$ has the same objects as $\cA$, and, for $X,Y\in\Ob(\cA)$, 
$$
F_{*}(\cA)(X,Y)=F(\cA(X,Y)).
$$
If $T:\cA\to\cB$ is a $\cV$-functor, the $\cW$-functor $F_{*}(T):F_{*}(\cA)\to F_{*}(\cB)$ acts as $T$ on objects, and for $X,X'\in\Ob(\cA)$, $$(F_{*}(T))_{XX'}=F(T_{XX'}).$$

If $\alpha:T\Rightarrow T'$ is a $\cV$-natural transformation, the $\cW$-natural transformation $F_{*}(\alpha):F_{*}(T)\Rightarrow F_{*}(T')$ is given by the collection $F_{*}(\alpha)_X$, $X\in\Ob(\cA)$, where $F_{*}(\alpha)_X$ is the composite
$$
J\longrightarrow F(I)\stackrel{F(\alpha_X)}{\longrightarrow}F(\cA(T(X),T'(X))).
$$

Note, if $F$ is left adjoint to $G:\cW\to\cV$, then $F_{*}$ is left adjoint to $G_{*}$. Given a $\cV$-category $\cA$ and a $\cW$-category $\cB$, and a $\cW$-functor $T:F_{*}\cA\to \cB$, the collection of $\cV$-morphisms associated to $\cW$-morphisms $T_{XX'}:F_{*}\cA(X,X')\to\cB(TX,TX')$ via adjunctions
$$
\Hom_\cW(F\cA(X,X'),\cB(HX,HX'))\simeq\Hom_\cV(\cA(X,X'),G\cB(HX,HX'))
$$
yields a $\cV$-functor $\cA\to G_{*}\cB$.

Additionally, by \cite[Prop.~2.3.3]{john-gray-closed-categories-lax-limits-etc} if $G$ is normal and $F$ is $\cV$-left adjoint to $G$ in the sense that
$$
G\cW[FX,Y]\simeq \cV[X,GY], 
$$
then $F_{*}$ is $\cV\text{-}\mathbf{Cat}$-left adjoint to $G_{*}$, i.e.,
$$
G_{*} \cW[F_{*}\cA,\cB]\simeq \cV[\cA,G_{*}\cB].
$$

\subsection{The underlying category 2-functor}\label{und-cat}

For any monoidal category $\cV$, a special case of base change functor is the `underlying category' 2-functor 
$$
(\mathord{-})_0:\cV\text{-}\mathbf{Cat}\to\mathbf{Cat}
$$ 
obtained by base change via the points functor $V:\cV\to\Set$. 

In particular, when $\cV$ is symmetric monoidal closed, the functor $V:\cV\to\Set$ admits a left adjoint, so the above construction gives a left adjoint to the forgetful functor  $(\mathord{-})_0$ 
called the `associated free $\cV$-category' functor. 

\begin{notation}\label{V-is-V-cat}
If $\cV$ is symmetric monoidal closed, it is a $\cV$-category when considered with internal homs. We use the notation $\cV_0$ when we wish to refer to the ordinary category structure of $\cV$, and in that case $\cV_0$ is indeed the underlying category of the $\cV$-category $\cV$.
\end{notation}

\subsection{The presheaf associated with an enriched presheaf}\label{und-presh}

Let $\cV$ be a symmetric monoidal closed category, let $\cC$ be a $\cV$-category, and let
$$
F:\cC^\circ\to \cV
$$
be a $\cV$-presheaf. Applying the underlying category 2-functor yields a functor
$F_0:\cC_0^\circ\to \cV_0$, where $\cV_0$ is simply the category $\cV$. The composite
$$
\assoc{F}=V\circ F_0:\cC^\circ_0\to\Set
$$
is called the presheaf associated with the $\cV$-presheaf $F$. This defines a functor
$$
\assoc{V}:(\hC)_0\to \widehat{\cC_0}, \ \ \ \assoc{V}(\bF)=\assoc{\bF}.
$$
Indeed, given a $\cV$-natural transformation $\varphi:\bF\to\bG$ of two $\cV$-presheaves,
the underlying category 2-functor gives a natural transformation $\varphi_0:\bF_0\to\bG_0$, and then $V\varphi_0$ is a natural transformation from $V\circ \bF_0=\assoc{\bF}$ to $V\circ\bG_0=\assoc{\bG}$.

Note, if $F$ is $\cV$-represented by an object $X$ in $\cC$, then $\assoc{F}$ is represented by $X$ considered as an object of $\cC_0$, since
$$
\assoc{F}(S)=V F(S)=V\cC(S,X)=\cC_0(S,X).
$$

\subsection{Kan extensions}

Let  $L:\cA\to\cC$ be a $\cV$-functor, and let $F\in \widehat{\cA}$. When $\cA$ is small and $\cV$ is complete, the \emph{right Kan extension of $F$ along $L$} is an element of $\hC$ obtained as
$$
\text{Ran}_LF(C)=\eend_{A\in\cA}[\cC(LA,C),FA].
$$
When $\cA$ is small and $\cV$ is cocomplete, the \emph{left Kan extension of $F$ along $L$} is an element of $\hC$ obtained as
$$
\text{Lan}_LF(C)=\coend^{A\in\cA}\cC(C,LA)\otimes FA.
$$
Precomposing with $L$ gives rise to a $\cV$-functor
$$
\widehat{L}:\hC\to \widehat{\cA}, \ \ \ \widehat{L}(F)=F\circ L.
$$
By the \emph{Theorem of Kan adjoints} (\cite[Theorem~4.50]{kelly}), $\widehat{L}$ has a left adjoint $\text{Lan}_L$ if $\text{Lan}_F$ exists for every $F\in\widehat{A}$, and $\widehat{L}$ has the right adjoint $\text{Ran}_L$ if $\text{Ran}_LF$ exists for each $F\in\widehat{\cA}$.

\subsection{Enriched adjointness and representability}\label{adj-rep-enrich}

Let $L:\cA\to\cB$ and $R:\cB\to\cA$ be $\cV$-functors so that $L$ is left $\cV$-adjoint to $R$, i.e., for $X\in\cA$, $Y\in\cB$, we have natural $\cV$-isomorphisms
$$
\cB(LX,Y)\simeq\cA(X,RY).
$$
Then, for $Y\in\cB$,
$$
\widehat{L}(\h_Y)=\h_Y\circ L\simeq \h_{RY},
$$
so the functor $R$ naturally extends  from $\cB$ to $\widehat{L}$ on the $\cV$-presheaf category $\widehat{\cB}$ 
via the enriched Yoneda embedding. 

If $\cA$ is small and $\cV$ is complete, the theorem of Kan adjoints yields that the right adjoint to $\widehat{L}$ is given as $\text{Ran}_L$. Using the adjointness of $L$ and $R$, as well as the ninja Yoneda lemma, we obtain
$$
\text{Ran}_LF(S)=\eend_{A\in\cA}[\cB(LA,S),FA]\simeq\eend_A[\cB(A,RS),FA]\simeq F(RS)=\widehat{R}F(S),
$$
so $\widehat{R}$ is right adjoint to $\widehat{L}$.


\subsection{Pullbacks and pushforwards along ordinary arrows}\label{push-pull}

Let $\cV$ be a $\cV$-category. For $X,Y\in\cC$, let 
$$
f\in\cC_0(X,Y)=V(\cC(X,Y))=\cV(I,\cC(X,Y))
$$
be an ordinary arrow. We sometimes simply write $f:X\to Y$. For $Z\in\cC$, we define the $\cV$-morphism 
$$
f_*:\cC(Z,X)\to \cC(Z,Y)
$$
as the composite
$$
\cC(Z,X)\to\cC(Z,X)\otimes I\xrightarrow{\id\otimes f}\cC(Z,X)\otimes\cC(X,Y)\xrightarrow{c_{ZXY}}\cC(Z,Y).
$$
Similarly, we define
$$
f^*:\cC(Y,Z)\to \cC(X,Z)
$$
as the composite
$$
\cC(Y,Z)\to I\otimes\cC(Y,Z)\xrightarrow{f\otimes\id}\cC(X,Y)\otimes\cC(Y,Z)\xrightarrow{c_{XYZ}}\cC(X,Z).
$$

\subsection{The arrow category of an enriched category}\label{arrow-cat}

Let $\cV$ be symmetric monoidal closed with pullbacks, and let $\cC$ be a $\cV$-category.
The \emph{arrow} $\cV$-category $$\cC^{\to}$$ has arrows $f:X\to Y$ from the underlying category $\cC_0$ as objects, and the internal hom object $\cC^{\to}(X\xrightarrow{f}Y,X'\xrightarrow{f'}Y')$ is defined as the pullback
 \begin{center}
 \begin{tikzpicture} 
\matrix(m)[matrix of math nodes, row sep=2em, column sep=3em, text height=1.5ex, text depth=0.25ex]
 {
 |(1)|{\cC^{\to}(f,f')}		& |(2)|{\cC(Y,Y')} 	\\
 |(l1)|{\cC(X,X')}		& |(l2)|{\cC(X,Y')} 	\\
 }; 
\path[->,font=\scriptsize,>=to, thin]
(1) edge node[above]{} (2) edge node[left]{}   (l1)
(2) edge node[right]{$f^*$} (l2) 
(l1) edge  node[above]{$f'_*$} (l2);
\end{tikzpicture}
\end{center}
where the $\cV$-morphisms $f^*$ and $f'_*$ have been defined in \ref{push-pull}.

\subsection{Tensors and cotensors}\label{tens-cotens}

Let $\cV$ be a  symmetric monoidal category, and let $\cC$ be a $\cV$-category.

We say that $\cC$ is \emph{tensored} over $\cV$, if for every $E\in\cV$, there exists a $\cV$-functor $E\otimes\mathord{-}:\cC\to \cC$ and $\cV$-isomorphisms
$$
\cC(E\otimes X,Y)\simeq [E,\cC(X,Y)],
$$
which are $\cV$-natural in $Y$.

We say that $\cC$ is \emph{cotensored} over $\cV$, if for every $E\in\cV$, there exists a $\cV$-functor $[E,\mathord{-}]:\cC\to \cC$ and $\cV$-isomorphisms
$$
\cC(Y,[E,X])\simeq [E,\cC(Y,X)],
$$
which are $\cV$-natural in $Y$.

Note, if $\cC$ is tensored and cotensored over $\cV$, the $\cV$-functors $E\otimes\mathord{-}:\cV\to \cV$ and $[E,\mathord{-}]:\cV\to\cV$ are $\cV$-adjoint,
$$
\cC(E\otimes X,Y)\simeq\cC(X,[E,Y]).
$$

\subsection{Tensors and cotensors in the opposite category}\label{tens-cotens-op}

Let $\cV$ be a  symmetric monoidal category. If $\cC$ is a tensored $\cV$-category, then 
$\cC^\op$ is a cotensored $\cV$-category with
$$
[E,X^\op]^\op=(E\otimes X)^\op,
$$
for $E\in\cV$ and $X\in\cC$, and we write $X^\op$ for the corresponding object of $\cC^\op$.
Indeed, 
$$
[E,\cC^\op(Y^\op,X^\op)]=[E,\cC(X,Y)]\simeq \cC(E\otimes X,Y)\simeq \cC^\op(Y^\op,(E\otimes X)^\op).
$$
Dually, if $\cC$ is a cotensored $\cV$-category, then $\cC^\op$ is a tensored $\cV$-category with
$$
E\otimes^{\op}X^\op=[E,X]^\op.
$$
Hence, if $\cC$ is both tensored and contensored over $\cV$, so it $\cC^\op$.

\subsection{Enriched presheaves are tensored and cotensored}\label{enr-psh-tens-cotens}

Let $\cV$ be a complete symmetric monoidal category, let $\cC$ be a small $\cV$ category, and let $\hC=[\cC^\circ,\cV]$ be the $\cV$-category of $\cV$-presheaves on $\cC$.

Then $\hC$ is both tensored and cotensored over $\cV$, with the following structure. For $E\in\cV$, and $\bF\in\hC$, we define the $\cV$-presheaf $E\otimes \bF\in\hC$ by
$$
(E\otimes \bF)(S)=E\otimes \bF(S),
$$
and the $\cV$-presheaf $[E,\bF]$ by
$$
[E,\bF](S)=[E,\bF(S)].
$$
Let $\cC$ be a tensored and cotensored $\cV$-category. For representable functors, the above structure agrees with the usual tensored and cotensored structure,
$$
E\otimes\h_X\simeq \h_{E\otimes X}, \ \ \ \ \ [E,\h_X]\simeq\h_{[E,X]}.
$$
Moreover, we see that
$$
[E,\bF](S)\simeq [E,\bF(S)]\simeq [E,[\h_S,\bF]]\simeq[E\otimes\h_S,\bF]\simeq \bF(E\otimes S).
$$
This is consistent with the extension of the cotensored structure of $\cC$ to $\hC$ obtained using
 the fact that, given $E\in\cV$, $E\otimes\mathord{-}$ is left $\cV$-adjoint to $[E,\mathord{-}]$,  and the principle \ref{adj-rep-enrich}.

\subsection{Enriching ordinary limits and colimits}\label{enr-limits}

\begin{remark}\label{tens-pres-lim}
Let $\cV$ be a symmetric monoidal category and let $\cC$ be tensored over $\cV$. Then, for any $X\in\cC$ the functor 
$$
\mathord{-}\otimes X:\cV_0\to \cC_0
$$
is left adjoint to 
$$
\cC(X,\mathord{-}):\cC_0\to \cV_0.
$$
Hence, if $\cV$ and $\cC_0$ are complete, $\cC(X,\mathord{-})$
preserves ordinary limits, 
$$
\cC(X,\lim_i Y_i)\simeq\lim_i\cC(X,Y_i),
$$
and, if $\cV$ and $\cC_0$ are cocomplete, $\mathord{-}\otimes X$ preserves ordinary colimits,
$$
\colim_j E_j\otimes X\simeq \colim_j (E_j\otimes X).
$$
\end{remark}

\begin{remark}
If $\cV$ is complete cartesian closed, \ref{tens-pres-lim} shows that limits in $\cV_0$ are automatically enriched. In particular, the category $\cV$ has $\cV$-products and equalisers in the sense that
$$
[A,\prod_i B_i]\simeq\prod_i[A,B_i]
$$
and
$$
\begin{tikzcd}[column sep=1em]
{[A,\Eq(B}\ar[yshift=2pt]{r}{} \ar[yshift=-2pt]{r}[swap]{} & {C)]}\simeq {\Eq([A,B]}\ar[yshift=2pt]{r}{} \ar[yshift=-2pt]{r}[swap]{} &{[A,C]}).
\end{tikzcd}
$$
\end{remark}

\begin{remark}\label{cotens-pres-lim}
Let $\cC$ be cotensored over a symmetric monoidal category $\cV$. Then, for all $E\in\cV$, $X,Y\in \cC$, we have
$$
\cC_0(Y,[E,X])\simeq\cV_0(E,\cC(Y,X)),
$$
so the functor 
$$
[\mathord{-},X]:\cV_0\to \cC_0^\circ
$$
is left adjoint to the functor
$$
\cC(\mathord{-},X):\cC_0^\circ\to\cV_0.
$$
Hence, if $\cV$ and $\cC_0$ are complete and cocomplete, both functors take colimits to limits,
$$
[\colim_iE_i,X]\simeq\lim_i[E_i,Y], \ \ \ \ \cC(\colim_j X_j,Y)\simeq\lim_j\cC(X_i,Y).
$$
\end{remark}

\begin{remark}\label{tens-cotens-lim}
Let $\cC$ be tensored and cotensored over a symmetric monoidal category $\cV$. Then, for any $E\in\cV$, the functor
$$
E\otimes\mathord{-}:\cC_0\to\cC_0
$$
is left adjoint to the functor
$$
[E,\mathord{-}]:\cC_0\to\cC_0.
$$
Hence, if $\cV$ and $\cC_0$ are cocomplete, the functor $E\otimes\mathord{-}$ commutes with colimits,
$$
E\otimes\colim_j X_j\simeq \colim_j E\otimes X_j,
$$
and, if $\cV$ and $\cC_0$ are complete, the functor $[E,\mathord{-}]$ commutes with limits,
$$
[E,\lim_i X_i]\simeq \lim_i[E,X_i].
$$
\end{remark}

\subsection{Enriched presheaves and cartesian closedness}

Let $\cV$ be a complete cartesian closed category, let $\cC$ be a small $\cV$-category, and let $\hC=[\cC^\circ,\cV]$ be the $\cV$-category of $\cV$-presheaves on $\cC$. We will show that $\hC$ is complete cartesian closed. 

\begin{definition}
For $\bF,\bG$ in $\hC$, we define 
$$
\bF\times\bG\in \hC   
$$
by setting, for $S\in \cC$,
$$
(\bF\times \bG) (S)=\bF(S)\times \bG(S).
$$
Moreover, for $S,S'\in\cC$, the $\cV$-morphism
$$
(\bF\times \bG)_{S'S}:\cC(S',S)\to [\bF(S)\times \bG(S),\bF(S')\times \bG(S')]
$$
is obtained by adjunction from the composite of
\begin{multline*}
\cC(S',S)\times \bF(S)\times \bG(S)\xrightarrow{\Delta\times\id\times\id}\cC(S',S)\times\cC(S',S)\times \bF(S)\times \bG(S)    \\ 
\xrightarrow{\bF_{S'S}\times \bG_{S'S}\times\id\times\id}[\bF(S),\bF(S')]\times[\bG(S),\bG(S')]\times \bF(S)\times \bG(S)\\
\simeq [\bF(S),\bF(S')]\times \bF(S)\times[\bG(S),\bG(S')]\times \bG(S)\xrightarrow{\mathop{\rm ev}\times\mathop{\rm ev}}\bF(S')\times \bG(S').
\end{multline*}
\end{definition}

The unit for the above product is the $\cV$-presheaf
$$
\be\in \hC,
$$
mapping every $S\in\cC$ to the terminal object $e$ of $\cV$.

\begin{definition}
We define the $\cV$-presheaf 
$$
[\bF,\bG]:\cC^\circ\to \cV
$$
by the rule, for $S\in\cC$,
$$
[\bF,\bG](S)=\hC(\bF\times\h_S,\bG).
$$
For $S,S'\in \cC$, the $\cV$-morphism
$$
[\bF,\bG]_{S'S}:\cC(S',S)\to [\hC(\bF\times\h_S,\bG),\hC(\bF\times\h_{S'},\bG)]
$$
is defined as follows. Consider the diagram
\begin{center}
 \begin{tikzpicture} 
\matrix(m)[matrix of math nodes, row sep=2em, column sep=1em, text height=1.5ex, text depth=0.25ex]
 {
 |(1)|{\displaystyle\cC(S',S)\times\prod_T[\bF(T)\times\cC(T,S),\bG(T)]}		&[1em] |(2)|{\displaystyle\cC(S',S)\times\prod_{T,T'}[\cC(T',T),[\bF(T)\times\cC(T,S),\bG(T')]} 	\\
 |(l1)|{\displaystyle\prod_T[\bF(T)\times\cC(T,S'),\bG(T)]}		& |(l2)|{\displaystyle\prod_{T,T'}[\cC(T',T),[\bF(T)\times\cC(T,S'),\bG(T')]} 	\\
 }; 
\path[->,font=\scriptsize,>=to, thin]
([yshift=2pt]1.east) edge 
([yshift=2pt]2.west) 
([yshift=-1pt]1.east) edge 
([yshift=-1pt]2.west) 
(1) edge 
 (l1)
(2) edge 
(l2) 
([yshift=2pt]l1.east) edge 
([yshift=2pt]l2.west) 
([yshift=-1pt]l1.east) edge 
([yshift=-1pt]l2.west) 
;
\end{tikzpicture}
\end{center}
where the right vertical arrow is obtained by adjunction from the composite of
\begin{multline*}
\bF(T)\times\cC(T,S')\times\cC(S',S)\times[\bF(T)\times\cC(T,S),\bG(T')]\\\xrightarrow{\id\times\circ\times\id}\bF(T)\times\cC(T,S)\times [\bF(T)\times\cC(T,S),\bG(T')]\xrightarrow{\mathop{\rm ev}} \bG(T'),
\end{multline*}
and the left vertical arrow is even more straightforward to define. The following lemma yields a morphism from the equaliser of the top row to the equaliser of the bottom row,  i.e.,  
a $\cV$-morphism $\cC(S',S)\times\hC(\bF\times\h_S,\bG)\to \hC(\bF\times\h_{S'},\bG)$, whence we obtain the desired morphism by adjunction.
\end{definition}
 
\begin{lemma}\label{eq-lemma}
Suppose we have a solid arrow diagram
 \begin{center}
 \begin{tikzpicture} 
\matrix(m)[matrix of math nodes, row sep=2em, column sep=2em, text height=1.5ex, text depth=0.25ex]
 {
|(0)|{\Eq(p,q)}  &  |(1)|{A}		&[1em] |(2)|{B} 	\\
|(l0)|{\Eq(p',q')} & |(l1)|{A'}		& |(l2)|{B'} 	\\
 }; 
\path[->,font=\scriptsize,>=to, thin]
(0) edge (1)
(l0) edge (l1)
(0) edge[dashed]  node[left]{$\bar{f}$} (l0)
([yshift=2pt]1.east) edge node[above]{$p$} ([yshift=2pt]2.west) 
([yshift=-1pt]1.east) edge node[below]{$q$} ([yshift=-1pt]2.west) 
(1) edge node[left]{$f$}   (l1)
(2) edge node[right]{$g$} (l2) 
([yshift=2pt]l1.east) edge node[above]{$p'$} ([yshift=2pt]l2.west) 
([yshift=-1pt]l1.east) edge node[below]{$q'$} ([yshift=-1pt]l2.west) 
;
\end{tikzpicture}
\end{center}
in a category where the equalisers $\Eq(p,q)$ and $\Eq(p',q')$ exist, which commutes in the sense that $g\circ p=p'\circ f$ and $g\circ q=q'\circ f$. Then there exists a unique dashed morphism $\bar{f}$ that renders the whole diagram commutative.

Moreover, if the right hand side of the diagram
 \begin{center}
 \begin{tikzpicture} 
\matrix(m)[matrix of math nodes, row sep=2em, column sep=2em, text height=1.5ex, text depth=0.25ex]
 {
|(0)|{\Eq(p,q)}  &  |(1)|{A}		&[1em] |(2)|{B} 	\\
|(l0)|{\Eq(p',q')} & |(l1)|{A'}		& |(l2)|{B'} 	\\
|(ll0)|{\Eq(p'',q'')} & |(ll1)|{A''}		& |(ll2)|{B''} 	\\
 }; 
\path[->,font=\scriptsize,>=to, thin]
(0) edge (1)
(l0) edge (l1)
(ll0) edge (ll1)
(0) edge[dashed]  node[left]{$\bar{f}$} (l0)
([yshift=2pt]1.east) edge node[above]{$p$} ([yshift=2pt]2.west) 
([yshift=-1pt]1.east) edge node[below]{$q$} ([yshift=-1pt]2.west) 
(1) edge node[left]{$f$}   (l1)
(2) edge  (l2) 
([yshift=2pt]l1.east) edge node[above]{$p'$} ([yshift=2pt]l2.west) 
([yshift=-1pt]l1.east) edge node[below]{$q'$} ([yshift=-1pt]l2.west) 

([xshift=-1.5pt]l0.south) edge[dashed] node[left]{$\bar{h}$}   ([xshift=-1.5pt]ll0.north)
([xshift=1.5pt]l0.south) edge[dashed] node[right]{$\bar{h}'$}   ([xshift=1.5pt]ll0.north)

([xshift=-1.5pt]l1.south) edge node[left]{$h$}   ([xshift=-1.5pt]ll1.north)
([xshift=1.5pt]l1.south) edge node[right]{$h'$}   ([xshift=1.5pt]ll1.north)

([xshift=-1.5pt]l2.south) edge    ([xshift=-1.5pt]ll2.north)
([xshift=1.5pt]l2.south) edge    ([xshift=1.5pt]ll2.north)

([yshift=2pt]ll1.east) edge node[above]{$p''$} ([yshift=2pt]ll2.west) 
([yshift=-1pt]ll1.east) edge node[below]{$q''$} ([yshift=-1pt]ll2.west) 
;
\end{tikzpicture}
\end{center}
commutes and the relevant equalisers exist, the above yields morphisms $\bar{f}$, $\bar{h}$, $\bar{h}'$ which make the whole diagram commutative,
and we obtain a morphism 
$$
\Eq(p,q)\to \Eq(\bar{h},\bar{h}').
$$
\end{lemma}

\begin{proposition}\label{enrich-presh-duality}
With the above cartesian structure, $\hC$ is cartesian closed. Even more, for $\bE,\bF,\bG\in\hC$, we have an isomorphism
$$
[\bE\times \bF,\bG]\simeq [\bE,[\bF,\bG]].
$$
\end{proposition}
\begin{proof}
It suffices to show that for all $\bE,\bF,\bG\in\hC$, we have a natural $\cV$-isomorphism
$$
\hC(\bE\times \bF,\bG)\simeq\hC(\bE,[\bF,\bG]).
$$
This follows as a consequence \cite[Example~5.2]{brian-day-on-closed-categories-of-functors} of a much more general framework, so we sketch a simpler proof in our special case. Using the (co)end calculus,
\begin{multline*}
\hC(\bE\times\bF,\bG)=\int_{T\in\cC}[\bE(T)\times\bF(T),\bG(T)]
=\int_{T}\left[\int^{S}\bE(S)\times\cC(T,S)\times \bF(T),\bG(T)\right]\\
=\int_T\int_S\left[\bE(S),[\cC(T,S)\times\bF(T),\bG(T)]\right]
=\int_S\left[\bE(S),\int_T[\cC(T,S)\times\bF(T),\bG(T)]\right]\\
=\int_S\left[\bE(S),[\bF,\bG](S)\right]=\hC(\bE,[\bF,\bG]).
\end{multline*}

\end{proof}

Given $\bF,\bG\in \hC$, it is sometimes convenient to write 
$$
\uHom(\bF,\bG)=[\bF,\bG],\ \  \uEnd(\bF)=[\bF,\bF].
$$

\subsection{Constant enriched presheaves and global sections}\label{enrich-global-sect}

Let $\cV$ be complete cartesian closed, and let $\cC$ be a $\cV$-category. The \emph{constant $\cV$-presheaf} functor
$$
\cC^*: \cV\to \hC, \ \ \ \cC^*(E)=E\otimes \be
$$
has a right $\cV$-adjoint 
$$
\Gamma=\hC(\be,\mathord{-}):\hC\to \cV,
$$
called the \emph{global sections functor}. 

Indeed, using the fact that $\hC$ is tensored over $\cV$, 
$$
\hC(\cC^*(E),\bF)=\hC(E\otimes \be,\bF)\simeq [E,\hC(\be,\bF)].
$$

\begin{corollary}\label{uhom-to-hom}
For $\bF,\bG\in\hC$, we have
$$
\Gamma[\bF,\bG]\simeq \hC(\bF,\bG).
$$
\end{corollary}
\begin{proof}
Using \ref{enrich-presh-duality}, we obtain that
$$
\Gamma[\bF,\bG]=\hC(\be,[\bF,\bG])\simeq\hC(\be\times \bF,\bG)\simeq\hC(\bF,\bG).
$$
\end{proof}

\begin{lemma}\label{Gamma-exact}
If $\cC$ is tensored over a cocomplete $\cV$ and $\cC_0$ has a terminal object $e_\cC$, then $\Gamma$ is exact.
\end{lemma}
\begin{proof}
Using \ref{tens-pres-lim}, for every $X\in\cC$, $\cC(X,e_\cC)=e_\cV$, so 
$$
\h_{e_\cC}\simeq \be=\cC^*(e_\cV).
$$
Hence, 
$$
\Gamma(\bF)=\hC(\be,\bF)\simeq\hC(\h_{e_\cC},\bF)\simeq\bF(e_\cC), 
$$
so $\Gamma$ preserves small colimits using the fact 
that small colimits exist pointwise in $\hC$ (\cite[3.3]{kelly}), i.e., that evaluation of presheaves at an object preserves them.
\end{proof}

\subsection{Internal hom and presheaf hom}

Suppose $\cV$ is a cartesian closed category.
For an object $X$ of $\cV$, we write
$\h_X(\mathord{-})=[\mathord{-},X]$ for the $\cV$-presheaf represented by $X$, and we write 
$\assoc{\h}_X(\mathord{-})=\cV_0(\mathord{-},X)$ for the associated presheaf.

Then, for $X,Y\in\cV$, 
$$
\uHom(\assoc{\h}_X,\assoc{\h}_Y)\simeq\assoc{\h}_{[X,Y]}.
$$

If $\cV$ is a complete cartesian closed category, we have even more,
$$
[\h_X,\h_Y]\simeq\h_{[X,Y]}.
$$
Indeed, 
$$
[\h_X,\h_Y](S)=\widehat{\cV}(\h_X\times\h_S,\h_Y)\simeq\widehat{\cV}(\h_{X\times S},\h_Y)\simeq[X\times S,Y]\simeq[S,[X,Y]]=\h_{[X,Y]}(S).
$$

Moreover, if $\cC$ is a tensored and cotensored $\cV$-category, then for $E\in\cV$, $X\in\cC$,
$$
[E_e,\h_X]\simeq\h_{[E,X]}\simeq[E,\h_X].
$$

\subsection{Internal and presheaf hom of associated presheaves}

Let $\cV$ be a complete cartesian closed category, let $\cC$ be a $\cV$-category and let $\hC=[\cC^\circ,\cV]$ be the category of $\cV$-presheaves on $\cC$. We write $\cC_0$ for the underlying category of $\cC$, and $\widehat{\cC_0}$ for the category of (ordinary) presheaves on $\cC_0$.

Recall, for an enriched presheaf $\bF\in\hC$, its underlying presheaf $\assoc{\bF}\in\widehat{\cC_0}$ is defined as the composite $V\circ\bF_0$, where $\bF_0$ is the underlying functor of $\bF$.

\begin{lemma}\label{inthom-assoc-psh}
For $X\in\cC$ and $\bF\in\hC$, 
$$
\widehat{\cC_0}(\assoc{\h}_X,\assoc{\bF})=V(\hC(\h_X,\assoc{\bF})), \ \ \ \text{ and }\ \ \ \ 
\uHom_{\widehat{\cC_0}}(\assoc{\h}_X,\assoc{\bF})=\assoc{[\h_X,\bF]}.
$$
\end{lemma}
\begin{proof}
Using ordinary and enriched Yoneda,
$$
\widehat{\cC_0}(\assoc{\h}_X,\assoc{\bF})=\assoc{\bF}(X)=V(\bF(X))=V(\hC(\h_X,\assoc{\bF})).
$$
Moreover, for $S\in\cC$, using the above,
\begin{multline*}
\uHom(\assoc{\h}_X,\assoc{\bF})(S)=\widehat{\cC_0}(\assoc{\h}_X\times\assoc{\h}_S,\assoc{\bF})=\widehat{\cC_0}(\assoc{\h}_{X\times S},\assoc{\bF})\\
=V(\hC(\h_{X\times S},\bF))=
V(\hC(\h_X\times\h_S,\bF)=V([\h_X,\bF](S)),
\end{multline*}
whence we conclude that $\uHom(\assoc{\h}_X,\assoc{\bF})$ is the  presheaf associated to the enriched presheaf $[\h_X,\bF]$. 
\end{proof}

\subsection{Internal isomorphisms and automorphisms}\label{int-isom}

Let $\cV$ be a cartesian closed category with pullbacks. For objects $X, Y$ in $\cV$, we define the internal isomorphism object $$\uIsom(X,Y)\in\cV$$ as the pullback
\begin{center}
 \begin{tikzpicture} 
\matrix(m)[matrix of math nodes, row sep=2em, column sep=5em, text height=1.9ex, text depth=0.25ex]
 {
 |(1)|{\uIsom(X,Y)}		& |(2)|{I} 	\\
 |(l1)|{[X,Y]\times[Y,X]}		& |(l2)|{[X,X]\times[Y,Y]} 	\\
 }; 
\path[->,font=\scriptsize,>=to, thin]
(1) edge node[above]{} (2) edge node[left]{}   (l1)
(2) edge node[right]{$(u_X,u_Y)$} (l2) 
(l1) edge  node[below]{$(c_{XYX},c_{YXY})$} (l2);
\end{tikzpicture}
\end{center}
and we write
$$
\uAut(X)=\uIsom(X,X).
$$


\section{Internal category theory}

Our exposition of internal category theory mostly follows \cite{johnstone, elephant1, elephant2, borceux-1}, but we put a slant on the enriched aspects of the theory, in particular on enriching slices in \ref{enr-int-slices} and enriched Grothendieck construction in \ref{enrich-groth-constr}. 

\subsection{Internal categories}\label{int-cats}

Let $\cV$ be a category with finite limits. A \emph{category internal in $\cV$} (or a \emph{$\cV$-internal category}) is a tuple
$$
\C=\left( C_0,C_1,C_1\xrightarrow{d_0}C_0, C_1\xrightarrow{d_1}C_0, C_0\xrightarrow{i}C_1, C_2\xrightarrow{c} C\right)
$$
of objects and morphisms of $\cV$, where we adopt a convention that $C_n$ is the n-fold pullback 
$$
C_1\times_{C_0}C_1\times_{C_0}\cdots\times_{C_0}C_1
$$
where $C_1$ appearing on the left of the symbol $\times_{C_0}$ is taken with the structure morphism $d_1$, while $C_1$ appearing on the right is taken with the structure morphism $d_0$. 

The object $C_0$ is called the \emph{object of objects} of $\C$, while $C_1$ is the \emph{object of morphisms}, $d_0$ and $d_1$ are the \emph{source} and \emph{target} maps, and $i$ is the \emph{identity} morphism. 
In particular, $C_2=C_1\times_{C_0}C_1$ is the pullback
 \begin{center}
 \begin{tikzpicture} 
\matrix(m)[matrix of math nodes, row sep=2em, column sep=2em, text height=1.5ex, text depth=0.25ex]
 {
 |(1)|{C_2}		& |(2)|{C_1} 	\\
 |(l1)|{C_1}		& |(l2)|{C_0} 	\\
 }; 
\path[->,font=\scriptsize,>=to, thin]
(1) edge node[above]{$\pi_2$} (2) edge node[left]{$\pi_1$}   (l1)
(2) edge node[right]{$d_0$} (l2) 
(l1) edge  node[above]{$d_1$} (l2);
\end{tikzpicture}
\end{center}
representing \emph{composable pairs} of morphisms, and $c$ is the \emph{composition} map. These morphisms must make the diagrams
\begin{center}
 \begin{tikzpicture} 
\matrix(m)[matrix of math nodes, row sep=2em, column sep=2em, text height=1.5ex, text depth=0.25ex]
 {
 |(1)|{C_0}		& |(2)|{C_1} 	\\
 |(l1)|{C_1}		& |(l2)|{C_0} 	\\
 }; 
 \path[-,font=\scriptsize,thin]
 ([yshift=.6pt,xshift=.6pt]1.south east) edge 
([yshift=.6pt,xshift=.6pt]l2.north west) 
([yshift=-.6pt,xshift=-.6pt]1.south east) edge 
([yshift=-.6pt,xshift=-.6pt]l2.north west);
\path[->,font=\scriptsize,>=to, thin]
(1) edge node[above]{$i$} (2) edge node[left]{$i$}   (l1)
(2) edge node[right]{$d_0$} (l2) 
(l1) edge  node[above]{$d_1$} (l2);
\end{tikzpicture}
\hskip4em
\begin{tikzpicture} 
\matrix(m)[matrix of math nodes, row sep=2em, column sep=2em, text height=1.5ex, text depth=0.25ex]
 {
|(0)|{C_1}  &  |(1)|{C_2}		& |(2)|{C_1} 	\\
|(l0)|{C_0} & |(l1)|{C_1}		& |(l2)|{C_0} 	\\
 }; 
\path[->,font=\scriptsize,>=to, thin]
(1) edge node[above]{$\pi_1$} (0)
(l1) edge node[above]{$d_0$} (l0)
(0) edge  node[left]{$d_0$} (l0)
(1) edge node[above]{$\pi_2$} (2) 
(1) edge node[right]{$c$}   (l1)
(2) edge node[right]{$d_1$} (l2) 
(l1) edge node[above]{$d_1$} (l2) 
;
\end{tikzpicture}
\end{center}
\begin{center}
\begin{tikzpicture} 
\matrix(m)[matrix of math nodes, row sep=2em, column sep=2em, text height=1.5ex, text depth=0.25ex]
 {
|(0)|{C_0\times_{C_0}C_1}  &  |(1)|{C_2}		& |(2)|{C_1\times_{C_0}C_0} 	\\
 & |(l1)|{C_1}		& 	\\
 }; 
\path[->,font=\scriptsize,>=to, thin]
(0) edge node[above]{$i\times\id$} (1)
(0) edge  node[left]{} (l1)
(2) edge node[above]{$\id\times i$} (1) 
(1) edge node[right]{$c$}   (l1)
(2) edge node[right]{} (l1) 
;
\end{tikzpicture}
\hskip2em
\begin{tikzpicture} 
\matrix(m)[matrix of math nodes, row sep=2em, column sep=2em, text height=1.5ex, text depth=0.25ex]
 {
 |(1)|{C_3}		& |(2)|{C_2} 	\\
 |(l1)|{C_2}		& |(l2)|{C_1} 	\\
 }; 
\path[->,font=\scriptsize,>=to, thin]
(1) edge node[above]{$\id\times c$} (2) edge node[left]{$c\times\id$}   (l1)
(2) edge node[right]{$c$} (l2) 
(l1) edge  node[above]{$c$} (l2);
\end{tikzpicture}
\end{center}
commutative.

An \emph{internal functor} (or \emph{morphism of internal categories}) 
$$
f:\C\to\bbD
$$
is given by a pair of $\cV$-morphisms $C_0\xrightarrow{f_0}D_0$, $C_1\xrightarrow{f_1}D_1$ making the diagrams
\begin{center}
\begin{tikzpicture} 
\matrix(m)[matrix of math nodes, row sep=2em, column sep=2em, text height=1.5ex, text depth=0.25ex]
 {
|(0)|{C_0}  &  |(1)|{C_1}		& |(2)|{C_0} 	\\
|(l0)|{D_0} & |(l1)|{D_1}		& |(l2)|{D_0} 	\\
 }; 
\path[->,font=\scriptsize,>=to, thin]
(1) edge node[above]{$d_0$} (0)
(l1) edge node[above]{$d_0$} (l0)
(0) edge  node[left]{$f_0$} (l0)
(1) edge node[above]{$d_1$} (2) 
(1) edge node[right]{$f_1$}   (l1)
(2) edge node[right]{$f_0$} (l2) 
(l1) edge node[above]{$d_1$} (l2) 
;
\end{tikzpicture}
 \begin{tikzpicture} 
\matrix(m)[matrix of math nodes, row sep=2em, column sep=2em, text height=1.5ex, text depth=0.25ex]
 {
 |(1)|{C_0}		& |(2)|{C_1} 	\\
 |(l1)|{D_0}		& |(l2)|{D_1} 	\\
 }; 
\path[->,font=\scriptsize,>=to, thin]
(1) edge node[above]{$i$} (2) edge node[left]{$f_0$}   (l1)
(2) edge node[right]{$f_1$} (l2) 
(l1) edge  node[above]{$i$} (l2);
\end{tikzpicture}
 \begin{tikzpicture} 
\matrix(m)[matrix of math nodes, row sep=2em, column sep=2em, text height=1.5ex, text depth=0.25ex]
 {
 |(1)|{C_2}		& |(2)|{C_1} 	\\
 |(l1)|{D_2}		& |(l2)|{D_1} 	\\
 }; 
\path[->,font=\scriptsize,>=to, thin]
(1) edge node[above]{$c$} (2) edge node[left]{$f_1\times f_1$}   (l1)
(2) edge node[right]{$f_1$} (l2) 
(l1) edge  node[above]{$c$} (l2);
\end{tikzpicture}
\end{center}
commutative.

An \emph{internal natural transformation} $\alpha$ between two internal functors $f, g:\C\to\bbD$ is given by a $\cV$-morphism $C_0\xrightarrow{\alpha} D_1$ rendering the diagrams
\begin{center}
\begin{tikzpicture} 
\matrix(m)[matrix of math nodes, row sep=2em, column sep=2em, text height=1.5ex, text depth=0.25ex]
 {
 &  |(1)|{C_0}		& 	\\
|(l0)|{D_0} & |(l1)|{D_1}		& |(l2)|{D_0} 	\\
 }; 
\path[->,font=\scriptsize,>=to, thin]
(1) edge node[above,pos=0.6]{$f_0$} (l0)
(1) edge node[above,pos=0.6]{$g_0$} (l2)
(1) edge node[right]{$\alpha$}   (l1)
(l1) edge node[above,pos=0.4]{$d_0$} (l0) 
(l1) edge node[above,pos=0.4]{$d_1$} (l2) 
;
\end{tikzpicture}
\begin{tikzpicture} 
\matrix(m)[matrix of math nodes, row sep=2em, column sep=2em, text height=1.5ex, text depth=0.25ex]
 {
 &  |(1)|{C_1}		& 	\\
|(l0)|{D_2} & |(l1)|{D_1}		& |(l2)|{D_2} 	\\
 }; 
\path[->,font=\scriptsize,>=to, thin]
(1) edge node[above left,pos=0.6]{$(\alpha\circ d_0,g_1)$} (l0)
(1) edge node[above right,pos=0.6]{$(f_1,\alpha\circ d_1)$} (l2)

(l0) edge node[above,pos=0.4]{$c$} (l1) 
(l2) edge node[above,pos=0.4]{$c$} (l1) 
;
\end{tikzpicture}
\end{center}
commutative.

We obtain the 2-category of $\cV$-internal categories
$$
\cat(\cV).
$$
It is classical (\cite[Proposition~7.2.2]{jacobs}) that, since $\cV$ has finite limits, $\cat(\cV)$ also has finite limits, and, moreover, if $\cV$ is cartesian closed, then so it $\cat(\cV)$. 

\subsection{Extended Yoneda embedding}\label{ext-yoneda}

If $\cV$ is a category with finite limits, we have the extended Yoneda embedding 2-functor
$$
\cat(\cV)\to [\cV^\op,\cat].
$$
Indeed, given an internal category $\C=(C_0,C_1)\in\cat(\cV)$, the functor 
$$
\h_\C: \cV^\op\to \cat
$$
assigns to an object $U\in\cV$ the category $\h_\C(U)$ with the set of objects $\cV(U,C_0)$ and the set of morphisms $\cV(U,C_1)$, while a morphism $V\to U$ in $\cV$ naturally yields a functor 
$
\h_\C(U)\to \h_\C(V)
$.
An internal functor $f:\C\to \bbD$ gives rise to a natural transformation $\h_f:\h_\C\to \h_\bbD$, and an internal natural transformation $\alpha$ between internal functors $f,g:\C\to \bbD$ yields, for all $U\in\cV$, a natural transformation $\h_\alpha(U):\h_f(U)\Rightarrow \h_g(U)$ between the functors $\h_f(U), \h_g(U):\h_\C(U)\to \h_\bbD(U)$. Hence, $\h$ is a 2-functor.

\subsection{Internal presheaves}\label{int-presh}

Let $\C\in\cat(\cV)$. An \emph{internal presheaf} on $\C$ is a pair
$$
F=\left(F_0\xrightarrow{\gamma_0}C_0, F_1\xrightarrow{e} F_0\right)
$$
of $\cV$-morphisms, where, for $n\geq 1$, we define 
$$
F_n=C_n\times_{C_0}F_0, 
$$
so that the diagrams 
\begin{center}
\begin{tikzpicture} 
\matrix(m)[matrix of math nodes, row sep=2em, column sep=2em, text height=1.5ex, text depth=0.25ex]
 {
|(0)|{F_0}  &  |(1)|{F_1}		& |(2)|{F_0} 	\\
|(l0)|{C_0} & |(l1)|{C_1}		& |(l2)|{C_0} 	\\
 }; 
\path[->,font=\scriptsize,>=to, thin]
(1) edge node[above]{$e$} (0)
(l1) edge node[above]{$d_0$} (l0)
(0) edge  node[left]{$\gamma_0$} (l0)
(1) edge node[above]{$\pi_2$} (2) 
(1) edge node[right]{$\pi_1$}   (l1)
(2) edge node[right]{$\gamma_0$} (l2) 
(l1) edge node[above]{$d_1$} (l2) 
;
\end{tikzpicture}
 \begin{tikzpicture} 
\matrix(m)[matrix of math nodes, row sep=2em, column sep=2em, text height=1.5ex, text depth=0.25ex]
 {
 |(1)|{C_0\times_{C_0}F_0}		& |(2)|{F_1} 	\\
		& |(l2)|{F_0} 	\\
 }; 
\path[->,font=\scriptsize,>=to, thin]
(1) edge node[above]{$i\times\id$} (2) edge node[left]{$\pi_2$}   (l2)
(2) edge node[right]{$e$} (l2) 
;
\end{tikzpicture}
 \begin{tikzpicture} 
\matrix(m)[matrix of math nodes, row sep=2em, column sep=2em, text height=1.5ex, text depth=0.25ex]
 {
 |(1)|{F_2}		& |(2)|{F_1} 	\\
 |(l1)|{F_1}		& |(l2)|{F_1} 	\\
 }; 
\path[->,font=\scriptsize,>=to, thin]
(1) edge node[above]{$\id\times e$} (2) edge node[left]{$c\times\id$}   (l1)
(2) edge node[right]{$e$} (l2) 
(l1) edge  node[above]{$e$} (l2);
\end{tikzpicture}
\end{center}
commute, where the right square in the first diagram is a pullback.

A \emph{morphism of internal presheaves} $f:F\to G$ is a morphism $f_0:F_0\to G_0$ over $C_0$ which makes the diagram
$$
 \begin{tikzpicture}
[cross line/.style={preaction={draw=white, -,
line width=4pt}}]
\matrix(m)[matrix of math nodes, row sep=.9em, column sep=.5em, text height=1.5ex, text depth=0.25ex]
{			& |(x0)| {F_0}	&				& |(x1)| {F_1} 	&			& |(x0s)| {F_0}	\\   [.2em]
|(y0)|{G_0} &			& |(y1)|{G_1} 	&			&  |(y0s)| {G_0}&			\\  [.4em]
			& |(s0)|{C_0} 		&			& |(s1)|{C_1} 	&			& |(s0s)|{C_0} 				\\};
\path[->,font=\scriptsize,>=to, thin]
(x1) edge node[above]{$e$} (x0) edge node[above]{$\pi_2$} (x0s) 
	edge node[left,pos=0.3]{$d_1^*f_0$} (y1) edge node[right,pos=0.8]{$\pi_1$} (s1)
(x0) edge node[left,pos=0.3]{$f_0$} (y0) edge node[right,pos=0.8]{$\gamma_0$} (s0)
(x0s) edge node[left,pos=0.3]{$f_0$} (y0s) edge node[right,pos=0.6]{$\gamma_0$} (s0s)
(y0) edge node[left,pos=0.6]{$\gamma_0$} (s0)
(y0s) edge node[left,pos=0.6]{$\gamma_0$}(s0s)
(y1) edge [cross line] node[above,pos=0.2]{$e$} (y0) edge [cross line] node[above,pos=0.8]{$\pi_2$} (y0s)  edge node[left,pos=0.6]{$\pi_1$}  (s1)
(s1) edge node[below]{$d_0$} (s0) edge node[below]{$d_1$} (s0s) 
;
\end{tikzpicture}
$$ 
commutative (and includes two appropriate pullback squares on the right).

The resulting category of internal presheaves on $\C$ is denoted
$$
[\C^\op,\cV].
$$

\subsection{Discrete fibrations}\label{discr-fibs}

Let $F\in [\C^\op,\cV]$ be an internal presheaf. In the notation of \ref{int-presh}, writing $F_n=C_n\times_{C_0}F_0$ and $\gamma_n=\pi_1:F_n\to C_n$ for $n\geq 1$,  we can  define an internal category $\F$ with the source map $e$ and the target map $\pi_2$ in such a way that $\gamma$ becomes an internal functor $\F\to\C$. In other words, we obtain a functor
$$
[\C^\op,\cV]\to \cat(\cV)_{\ov\C}, \ \ \ F\mapsto (\F\to\C).
$$
Clearly, an object $\F\xrightarrow{\gamma}\C$ of $\cat(\cV)_{\ov\C}$ is isomorphic to an object in the image of the above functor if and only if the square
$$
 \begin{tikzpicture} 
\matrix(m)[matrix of math nodes, row sep=2em, column sep=2em, text height=1.5ex, text depth=0.25ex]
 {
 |(1)|{F_1}		& |(2)|{F_0} 	\\
 |(l1)|{C_1}		& |(l2)|{C_0} 	\\
 }; 
\path[->,font=\scriptsize,>=to, thin]
(1) edge node[above]{$d_1$} (2) edge node[left]{$\gamma_1$}   (l1)
(2) edge node[right]{$\gamma_0$} (l2) 
(l1) edge  node[above]{$d_1$} (l2);
\end{tikzpicture}
$$
is a pullback, i.e., if it is a \emph{discrete fibration}. Hence we obtain an equivalence of categories 
$$
[\C^\op,\cV]\simeq \mathbf{dFib}(\cV)_{\ov\C},
$$
where $\mathbf{dFib}(\cV)$ stands for the category of $\cV$-internal categories and discrete fibrations between them. 

Consequently, if $F\in[\C^\op,\cV]$ is an internal presheaf associated with a discrete fibration $\F\to\C$, we have an equivalence of categories
$$
[\C^\op,\cV]_{\ov F}\simeq [\F^\op,\cV].
$$
This stament can be viewed as a (largely trivial) internal Grothendieck construction. 

Let $f:\C\to\bbD$ be a morphism in $\cat(\cV)$. The natural pullback functor $f^*:\cat(\cV)_{\ov\bbD}\to \cat(\cV)_{\ov\C}$ preserves discrete fibrations, so it induces a functor
$$
f^*:[\bbD^\op,\cV]\to [\C^\op,\cV].
$$

\subsection{Internal presheaves as monad algebras}\label{mon-alg}

Let $\cV$ be a category with finite limits, and let $\C\in\cat(\cV)$. The forgetful functor 
$$
U:[\C^\op,\cV]\to \cV_{\ov C_0}, \ \ \ \ F\mapsto (F_0\xrightarrow{\gamma_0}C_0)
$$
is monadic. Indeed, we have a monad
$$
\bbT_\C=(T_\C,\eta,\mu),
$$
where $T_\C=\sum_{d_0}\, d_1^*:\cV_{\ov C_0}\to \cV_{\ov C_0}$, and the natural transformations $$\eta:\id\to T_\C$$
iand $\mu:T_\C\, T_\C\to T_\C$ are defined 
using the identity and composition of $\C$
as follows. For $X\xrightarrow{\gamma}C_0$ in $\cV_{\ov C_0}$, the $\cV_{\ov C_0}$-morphism
$\eta_\gamma:\gamma\to T_\C(\gamma)$ is $i \times\id:X\to C_1\times_{C_0}X$,
and the $\cV_{\ov C_0}$-morphism $\mu_\gamma:T_\C T_\C(\gamma)\to T_\C(\gamma)$
is $c\times\id:C_1\times_{C_0}C_1\times_{C_0}\times X\to C_1\times_{C_0}X$.

The diagram in \ref{int-presh} shows that  
the action $e:F_1\to F_0$ of an internal presheaf $F\in[\C^\op,\cV]$ can be equivalently given by a morphism $F_1=d_1^*F_0\to d_0^*F_0$ in $\cV_{\ov C_1}$, or, via the adjunction $\sum_{d_0}\dashv d_0^*$, through a morphism
$$
T_\C(\gamma_0)\to\gamma_0.
$$
Interpreting the other two diagrams in this context yields that $F$ is a $\bbT_\C$-algebra.

Thus, we obtain an equivalence 
$$
[\C^\op,\cV]\simeq \cV^{\bbT_\C}
$$
between the category of internal presheaves on $\C$ and the category of $\bbT_\C$-algebras. 

Moreover, we obtain a functor
$$
R:\cV_{\ov C_0}\to [\C^\op,\cV]
$$  
which is left adjoint to $U$ as follows. For an object $X\xrightarrow{\gamma}C_0$ of $\cV_{\ov C_0}$,  $R(\gamma)$ is the internal presheaf corresponding to the discrete fibration over $\C$ with source and target maps $\mu_\gamma,\pi_{23}:T_\C T_\C(\gamma)\to T_\C(\gamma)$. The internal presheaf $R(\gamma)$ is called the \emph{representable functor} on $\gamma$, and the adjunction
$$
R\dashv U
$$
is the internal version of the Yoneda lemma.



\subsection{Internal limits and colimits}\label{int-lim-colim}

Assume $\cV$ has finite limits and reflexive coequalisers, and let $\C\in\cat(\cV)$. 
Then $\cV$ is \emph{internally cocomplete} in the sense that the functor
$$
{\textstyle\varinjlim_\C}:[\C^\op,\cV]\to \cV, \ \ \ \ 
F\mapsto\coeq(\begin{tikzcd}[cramped,sep=small]
F_1 \ar[yshift=2pt]{r}{e} \ar[yshift=-2pt]{r}[swap]{\pi_2} & F_0
\end{tikzcd})
$$
%
is left adjoint to the functor
$$
\C^*:\cV\to [\C^\op,\cV]
$$ 
sending an object $X\in\cV$ to the constant presheaf $\C\times X\xrightarrow{\pi_1}\C$.

The link to representable functors is provided by the relation 
$$
{\textstyle\varinjlim_\C}(R(X\xrightarrow{\gamma}C_0))\simeq X.
$$

If $\cV$ is cartesian closed, then it is \emph{internally complete} in the sense that $\C^*$ has a right adjoint
$$
\Gamma={\textstyle\varprojlim_\C}:[\C^\op,\cV]\to \cV
$$
defined as follows. Given an $F\in [\C^\op,\cV]$, denote by $\F\xrightarrow{\gamma}\C$ the associated discrete fibration, and let $\varprojlim_\C(F)$ be the equaliser of
$$
\begin{tikzpicture} 
\matrix(m)[matrix of math nodes, row sep=2em, column sep=1em, text height=1.5ex, text depth=0.25ex]
 {
|(0)|{\sprod_{C_0}(\gamma_0)}  &  |(0h)| {} & |(1)|{[C_0,F_0]}	& |(1h)|{}	& |(2)|{[C_1,F_0]} 	\\
|(l0)|{} & |(l0h)|{\sprod_{C_1}(\gamma_1)} &  |(l1)|{}	&|(l1h)|{[C_1,F_1]}	& |(l2)|{} 	\\
 }; 
\path[->,font=\scriptsize,>=to, thin]
(0) edge node[above]{} (1)
(l0h) edge node[above]{} (l1h)
(0) edge  node[below left]{$h$} (l0h)
(1) edge node[above]{$[d_0,F_0]$} (2) 
(l1h) edge node[below right]{$[C_1,e]$} (2) 
;
\end{tikzpicture}
$$
where the morphism $h$ is obtained by adjunction from the composite
$$
C_1^*\sprod_{C_0}(\gamma_0)\simeq d_1^*C_0^*\sprod_{C_0}(\gamma_0)\xrightarrow{d_1^*\beta}d_1^*(\gamma_0)\simeq\gamma_1,
$$
where we wrote $\beta$ for the counit of $C_0^*\dashv\sprod_{C_0}$.

The diagram
$$
\begin{tikzpicture} 
\matrix(m)[matrix of math nodes, row sep=3em, column sep=4em, text height=1.9ex, text depth=0.25ex]
 {
 |(1)|{[\C^\op,\cV]}		& |(2)|{\cV_{\ov C_0}} 	\\
 |(l1)|{\cV}		& |(l2)|{\cV} 	\\
 }; 
\path[->,font=\scriptsize,>=to, thin]
(1) edge node[above]{$U$} (2) 

(1) edge[draw=none] node (1mid) {} (l1)
(l1) edge node (fo) [pos=0.65,right=-2pt]{$\C^*$}  (1)
(1) edge [bend right=45] node(i) [left]{$\textstyle\varinjlim_\C$}   (l1)
(1) edge [bend left=45] node(fi) [right]{$\textstyle\varprojlim_\C$}   (l1)
(i) edge[draw=none] node{$\dashv$} (1mid)
(1mid) edge[draw=none] node{$\dashv$} (fi)

(2) edge[draw=none] node (2mid) {} (l2)
(l2) edge node (fo) [pos=0.70,right=-3pt]{$C_0^*$}  (2)
(2) edge [bend right=45] node(qo) [left=-3pt]{$\ssum_{C_0}$} (l2) 
(2) edge [bend left=45] node(hi) [right]{$\sprod_{C_0}$}   (l2)
(l1) edge  node[above]{$\id$} (l2)
(qo) edge[draw=none] node{$\dashv$} (2mid)
(2mid) edge[draw=none] node{$\dashv$} (hi)
;
\end{tikzpicture}
$$
summarises the above discussion.

\subsection{Comonad coalgebras in a topos}\label{lawvere-tierney-thm}

The theorem of Lawvere-Tierney states that, if $\bbG=(G,\epsilon,\delta)$ is a comonad in a topos $\cE$ with $G$ left-exact, then the category $\cE_\bbG$ of $\bbG$-coalgebras is a topos, and there is a geometric morphism $\cE\to \cE_\bbG$.

The category $\cE_\bbG$ has finite limits by left-exactness of $G$. 

Given coalgebras 
$X\xrightarrow{\theta} GX$ and $Y\xrightarrow{\phi}GY$, the internal hom (exponential)
$$
[(X,\theta),(Y,\phi)]
$$
is constructed as the equaliser of
$$
\begin{tikzpicture} 
\matrix(m)[matrix of math nodes, row sep=2em, column sep=1em, text height=1.5ex, text depth=0.25ex]
 {
|(0)|{G[X,Y]}  &  |(0h)| {} & |(1)|{}	& |(1h)|{}	& |(2)|{G[X,GY]} 	\\
|(l0)|{} & |(l0h)|{GG[X,Y]} &  |(l1)|{}	&|(l1h)|{G[GX,GY]}	& |(l2)|{} 	\\
 }; 
\path[->,font=\scriptsize,>=to, thin]
(l0h) edge node[above]{$G\rho$} (l1h)
(0) edge  node[below left]{$\delta_{[X,Y]}$} (l0h)
(0) edge node[above]{$G[X,\phi]$} (2) 
(l1h) edge node[below right]{$G[\theta,GY]$} (2) 
;
\end{tikzpicture}
$$
where $\rho:G[X,Y]\to[GX,GY]$ is obtained by adjunction from the composite
$$
G[X,Y]\times GX\simeq G([X,Y]\times X)\xrightarrow{G(\mathop{\rm ev})}GY.
$$

The subobject classifier $\Omega_\bbG$ of $\cE_\bbG$ is the equaliser of 
\begin{center}
\begin{tikzpicture} 
\matrix(m)[matrix of math nodes, row sep=2em, column sep=2em, text height=1.5ex, text depth=0.25ex]
 {
|(0)|{G\Omega}  &  	& |(2)|{G\Omega} 	\\
 & |(l1)|{GG\Omega}		& 	\\
 }; 
\path[->,font=\scriptsize,>=to, thin]
(0) edge node[above]{$\id_{G\Omega}$} (2)
edge  node[below left]{$\delta_\Omega$} (l1)
(l1) edge node[below right]{$G\tau$} (2) 
;
\end{tikzpicture}
\end{center}
where $\tau$ is the classifying map of $G(t):e\simeq G(e)\mono G\Omega$.

\subsection{Internal presheaves as comonad coalgebras}\label{comon-coalg}

Let $\cV$ be a locally cartesian closed category, let $\C\in\cat(\cV)$. Then the forgetful functor 
$$
U:[\C^\op,\cV]\to \cV_{\ov C_0}
$$
is comonadic. 

Indeed, the functor part $T_\C=\ssum_{d_0}d_1^*:\cV_{\ov C_0}\to\cV_{\ov C_0}$ of the monad $\bbT_\C$ we defined in \ref{mon-alg} has a right adjoint
$$
G_\C=\sprod_{d_1}d_0^*
$$
so by Eilenberg-Moore \cite[0.14]{johnstone} gives rise to a unique comonad $\bbG_\C=(G_\C,\varepsilon,\delta)$ so that $\cV^{\bbT_\C}\simeq\cV_{\bbG_\C}$. By \ref{mon-alg}, we obtain an equivalence
$$
[\C^\op,\cV]\simeq \cV_{\bbG_\C}
$$
between the category of internal presheaves on $\cC$, and the category of $\bbG_\C$-coalgebras.

We can now apply the methods of \ref{lawvere-tierney-thm} to the comonad $\bbG_\C$.
Noting that the construction of internal homs only uses the structure of a locally cartesian closed category and not the full topos structure, we conclude that
$$[\C^\op,\cV]$$
is cartesian closed.

Moreover, if $\cV=\cE$ is a topos, we get that
$$
[\C^\op,\cE]
$$
is a topos, and the diagram at the end of \ref{int-lim-colim} yields the canonical geometric morphism
$$
[\C^\op,\cE]\to\cE.
$$

The above comonad structure can be made explicit as follows, using the diagrams for internal categories set out in \ref{int-cats}. 

The counit $\varepsilon:G_\C\to\id$ is the composite
$$
\sprod_{d_1}d_0^*\to \sprod_{d_1}\sprod_i\, i^*d_0^*\simeq \sprod_{d_1 i}(d_0i)^*\simeq\id,
$$
using the properties $d_0 i=d_1 i=\id_{C_0}$. 

The comultiplication $\delta:G_\C\to G_\C G_\C$ is the composite
$$
\sprod_{d_1}d_0^*\to \sprod_{d_1}\sprod_c\, c^* d_0^*\simeq \sprod_{d_1}\sprod_{\pi_2}\pi_1^* d_0^*\to \sprod_{d_1} d_0^*\,\sprod_{d_1}d_0^*,
$$  
where the last step is obtained using the Beck-Chevalley morphism for the pullback square from the definition of $C_2$ in \ref{int-cats}.

Note that 
$$
G_\C\circ P=P\circ \ssum_{d_1}d_0^*
$$
as functors on $\cE_{\ov C_0}$, so $G_\C(C_0^*\Omega))\simeq G_\C(P(\id_{C_0}))\simeq P(d_1)$, and $G_\C G_\C(C_0^*\Omega)=P(\ssum_{d_1} d_0^*(d_1))$. 

Hence, writing
$\bar{d}_2=\ssum_{d_1} d_0^*(d_1):C_2\to C_0$ and $\bar{c}:\bar{d}_2\to d_1$ for the morphism corresponding to composition $c:C_2\to  C_1$, a suitable modification of the diagram from \ref{lawvere-tierney-thm} yields that the subobject classifier
$$
\Omega_{[\C^\op,\cE]}
$$
is the equaliser of the $\cE_{\ov C_0}$-diagram
\begin{center}
\begin{tikzpicture} 
\matrix(m)[matrix of math nodes, row sep=2em, column sep=2em, text height=1.5ex, text depth=0.25ex]
 {
|(0)|{P(d_1)}  &  	& |(2)|{P(d_1)} 	\\
 & |(l1)|{P(\bar{d}_2)}		& 	\\
 }; 
\path[->,font=\scriptsize,>=to, thin]
(0) edge node[above]{$\id$} (2)
edge  node[below left]{$P(\bar{c})$} (l1)
(l1) edge node[below right]{$\forall\pi_2$} (2) 
;
\end{tikzpicture}
\end{center}
where the morphism $\forall\pi_2$ is associated to $\pi_2:\bar{d}_2\to d_1$ as in \ref{power-ob}.

\subsection{Functoriality of forming internal presheaves}\label{funct-int-presh}


As discussed in \ref{discr-fibs}, a morphism $f:\C\to\bbD$ of $\cE$-internal categories induces a functor 
$$
f^*:[\bbD^\op,\cE]\to[\C^\op,\cE].
$$
If $\cE$ has reflexive coequalisers and the pullback functors in $\cE$ preserve them, then $f^*$ has a left adjoint
$$
\textstyle\varinjlim_{f}:[\C^\op,\cE]\to [\bbD^\op,\cE].
$$

If $\cE$ is a topos (or at least a locally cartesian closed category with Beck-Chevalley property), then $f^*$ has a right adjoint
$$
\textstyle\varprojlim_{f}:[\C^\op,\cE]\to [\bbD^\op,\cE].
$$

Hence, if $\cE$ is a topos, the resulting diagram
$$
\begin{tikzpicture} 
\matrix(m)[matrix of math nodes, row sep=3em, column sep=4em, text height=1.9ex, text depth=0.25ex]
 {
 |(1)|{[\C^\op,\cE]}		& |(2)|{\cE_{\ov C_0}} 	\\
 |(l1)|{[\bbD^\op,\cE]}		& |(l2)|{\cE_{\ov D_0}} 	\\
 }; 
\path[->,font=\scriptsize,>=to, thin]
(1) edge node[above]{$U$} (2) 

(1) edge[draw=none] node (1mid) {} (l1)
(l1) edge node (fo) [pos=0.65,right=-2pt]{$f^*$}  (1)
(1) edge [bend right=45] node(i) [left]{$\textstyle\varinjlim_f$}   (l1)
(1) edge [bend left=45] node(fi) [right]{$\textstyle\varprojlim_f$}   (l1)
(i) edge[draw=none] node{$\dashv$} (1mid)
(1mid) edge[draw=none] node{$\dashv$} (fi)

(2) edge[draw=none] node (2mid) {} (l2)
(l2) edge node (fo) [pos=0.70,right=-3pt]{$f_0^*$}  (2)
(2) edge [bend right=45] node(qo) [left=-3pt]{$\ssum_{f_0}$} (l2) 
(2) edge [bend left=45] node(hi) [right]{$\sprod_{f_0}$}   (l2)
(l1) edge  node[above]{$U$} (l2)
(qo) edge[draw=none] node{$\dashv$} (2mid)
(2mid) edge[draw=none] node{$\dashv$} (hi)
;
\end{tikzpicture}
$$
illustrates the fact that the assignment $\C\mapsto [\C^\op,\cE]$, $f\mapsto(\varprojlim_f,f^*)$ defines a 2-functor
$$
\cat(\cE)\to \mathfrak{Top}_{\ov\cE}.
$$

\subsection{Base change for internal presheaves}\label{bc-int-psh}

Let $f:\cF\to\cE$ be a geometric morphism and $\C\in\cat(\cE)$. Then we have a geometric morphism $f^{\C^\op}$ such that the diagram
$$
 \begin{tikzpicture} 
\matrix(m)[matrix of math nodes, row sep=2em, column sep=2em, text height=1.5ex, text depth=0.25ex]
 {
 |(1)|{[f^*\C^\op,\cF]}		& |(2)|{[\C^\op,\cE]} 	\\
 |(l1)|{\cF}		& |(l2)|{\cE} 	\\
 }; 
\path[->,font=\scriptsize,>=to, thin]
(1) edge node[above]{$f^{\C^\op}$} (2) edge 
(l1)
(2) edge 
(l2) 
(l1) edge  node[above]{$f$} (l2);
\end{tikzpicture}
$$
is a pullback in $\Top$, and satisfies the Beck condition
$$
(\C^{\op})^*\, f_*\simeq (f^{\C^\op})_*\, (f^*\C^\op)^*.
$$
Thinking of internal presheaves in $[\C^\op,\cE]$ as discrete fibrations over $\cC$ in $\cE$,
the functor $$(f^{\C^\op})^*$$ simply applies $f^*$ to a discrete fibration.

On the other hand, 
$$
(f^{\C^\op})_* =\eta_\C^*\circ f_*,
$$
i.e., it applies $f_*$ to a discrete fibration over $f^*\C$, and then takes the pullback along the unit $\eta_\C:\C\to f_*f^*\C$ of adjunction $f^*\dashv f_*$.

If $f$ is essential with $f_!$ left exact, then $f^{\C^\op}$ is essential, the functor 
$$
(f^{\C^\op})_!=\ssig_{\epsilon_\C}\circ f_!,
$$
where $\epsilon_\C:f_!f^*\C\to \C$ is the counit of adjunction $f_!\dashv f^*$.


\subsection{Internal sites}\label{int-sites}

Let $\C\in\cat(\cE)$ be an internal category in a topos. The subobject classifier 
$$
\Omega_{[\C^\op,\cE]}\xrightarrow{\omega}C_0
$$
can be viewed as the \emph{object of sieves} of $\C$. Indeed, its description from \ref{comon-coalg} shows that it is the subobject of realisations
$$
\real{\mathop{\rm Sv}}\mono P(d_1)
$$ 
of the formula
\begin{align*}
\mathop{\rm Sv}(S:P(d_1)) & \equiv \forall z:\bar{d}_2\ \pi_2(z)\in S\Rightarrow \bar{c}(z)\in S\\
& \equiv  \forall (a',a):\bar{d}_2\ a\in S\Rightarrow a\circ a'\in S.
\end{align*}

An \emph{internal coverage} is a $\cE_{\ov C_0}$-morphism
\begin{center}
\begin{tikzpicture} 
\matrix(m)[matrix of math nodes, row sep=2em, column sep=2em, text height=1.5ex, text depth=0.25ex]
 {
|(0)|{T}  &  	& |(2)|{\Omega_{[\C^\op,\cE]}} 	\\
 & |(l1)|{C_0}		& 	\\
 }; 
\path[->,font=\scriptsize,>=to, thin]
(0) edge node[above]{$c$} (2)
edge  node[below left]{$b$} (l1)
(2) edge node[below right]{$\omega$} (l1) 
;
\end{tikzpicture}
\end{center}
%
%
satisfying the following internal analogue of axiom (C). If we form the subobject 
$$
Q\mono T\times_{C_0}C_1\times_{C_0}T
$$
of realisations of the formula
$$
\{(t',a,t):T\times_{C_0}C_1\times_{C_0}T\,|\, \forall a':d_1\ a'\in c(t')\Rightarrow a\circ a'\in c(t)\},
$$
then the composite
$$
Q\mono T\times_{C_0}C_1\times_{C_0}T\xrightarrow{\pi_{23}}C_1\times_{C_0}T
$$
is required to be an epimorphism. 

An \emph{internal site} in $\cE$ is an internal category endowed with an internal coverage.

\subsection{Internal coverages on a poset}\label{int-cov-poset}

Assume $\C\in\cat(\cE)$ is an internal poset in a topos. Bearing in mind that a sieve in a poset identifies with a lower set/ideal of domains of its morphisms, an internal coverage on $\cC$ is determined by a span
$$
C_0\xleftarrow{b} T\xrightarrow{c} \mathop{\rm Idl}(\C),
$$
where the object of lower sets
$$
\mathop{\rm Idl}(\C)\mono PC_0
$$
is the object of realisations of the formula 
$$
\bar{\phi}(S:PC_0)\equiv \forall x:C_0\ \forall x':C_0\ x\in S\land x'\leq x \Rightarrow x'\in S. 
$$
The above data is required to satify the formula
$$
\forall t:T\ \forall x:C_0\ x\in c(t)\Rightarrow x\leq b(t),
$$
and the poset variant of axiom (C),
$$
\forall t:T\ \forall x:C_0\ x\leq b(t) \Rightarrow \exists t':T\ \ b(t')=x \land \forall x':C_0\ x'\in c(t')\Rightarrow x'\in c(t).
$$

\subsection{Internal sheaves}\label{int-sh}

Let $\C$ be an internal category in a topos $\cE$, and let $F\in[\C^\op,\cE]$ be an internal presheaf. 

Using the notation from \ref{int-presh}, the formula 
$$
\text{comp}_F(s:[d_0,\gamma_0],
S:\Omega_{[\C^\op,\cE]})
\equiv \forall (g,f):C_2\ \ f\in S\Rightarrow s(f\circ g)=e(g,s(f))
$$
expresses that a system $s$ of sections of $F$ is compatible with a sieve $S$.

The formulae
\begin{multline*}
\text{glue}_F(S:\Omega_{[\C^\op,\cE]})\equiv
\forall s:[d_0,\gamma_0]\ \text{comp}_F(s,S)\Rightarrow \\
\exists \tilde{s}:F_0 \ \gamma_0(\tilde{s})=\omega(S) \land \forall f:C_1\ f\in S\Rightarrow e(f,\tilde{s})=s(f)
\end{multline*}
and
\begin{multline*}
\text{uniq}_F(S:\Omega_{[\C^\op,\cE]}) \equiv \forall (\tilde{s},\tilde{s}'):F_0\times F_0\ 
\left(\forall f:C_1\ f\in S\Rightarrow e(f,\tilde{s})=e(f,\tilde{s}')\right)\Rightarrow \tilde{s}=\tilde{s}'
\end{multline*}
express the gluing property and the uniqueness of gluing of $F$ with respect to a sieve $S$. Hence, the formula
$$
\text{sheaf}_F(S:\Omega_{[\C^\op,\cE]})\equiv \text{glue}_F(S)\land \text{uniq}_F(S)
$$
expresses that the presheaf $F$ satisfies the presheaf axiom for a sieve $S$. 

If $(\C,T)$ is an internal site as in \ref{int-sites}, we say that $F$ is a $T$-sheaf if
$$
\models \forall t:T\  \text{sheaf}_F(c(t)),
$$
or, equivalently, if
\begin{multline*}
\models \forall t:T \ \forall s:[d_0,\gamma_0]\ \text{comp}_F(s,c(t))\Rightarrow \\
\exists! \tilde{s}:F_0 \ \gamma_0(\tilde{s})=b(t)  \land \forall f:C_1\ f\in c(t)\Rightarrow e(f,\tilde{s})=s(f).
\end{multline*}

We write 
$$
\sh_\cE(\C,T)
$$
for the full subcategory of $[\C^\op,\cE]$  of internal sheaves. In fact, it is a subtopos of the topos of internal presheaves, see \cite[C2.4]{elephant2}. 

\subsection{Externalising internal sheaves}\label{semidir}

Let $\cE=\sh(\cC,J)$ be a Grothendieck topos, and let $(\bbD,K)$ be an internal site in $\cE$. We present the construction from \cite{moerdijk-cont-fibr} and \cite[C2.5.4]{elephant2} 
of a category $\cC\rtimes\bbD$ and a coverage $J\rtimes K$ on $\cC\rtimes\bbD$ such that
$$
\sh_\cE(\bbD,K)\simeq \sh(\cC\rtimes\bbD,J\rtimes K),
$$
allowing us to reduce a consideration of internal sheaves to ordinary sheaves.

Considering $\bbD$ a functor $\cC^\op\to\cat$, we define the category
$$
\cC\rtimes\bbD
$$
whose objects are pairs $(U,V)$ with $U\in\cC$ and $V\in \bbD(U)$, and morphisms
$(U',V')\to (U,V)$ are pairs $(a,b)$, where $U'\xrightarrow{a}U\in\cC$, and $V'\xrightarrow{b}\bbD(a)(V)\in\bbD(U')$.

It follows immediately (\cite[C2.5.3]{elephant2}) that
$$
[\bbD^\op,[\cC^\op,\Set]]\simeq [(\cC\rtimes\bbD)^\op,Set].
$$

Note that, given $(U,V)\in\cC\rtimes\bbD$, if $S\in K(U)$ with $b(S)=V$, then
\begin{multline*}
c(C)\in \Hom_{\ov D_0}(h_U,P(d_1))\simeq  \Hom_{\ov D_0}(h_U,[d_1,\Omega_{D_0}])\\\simeq \Hom_{\ov D_0}(h_U\times_{D_0}D_1,\Omega_{D_0})\simeq \Sub_{\ov D_0}(h_U\times_{D_0}D_1),
\end{multline*}
so for $U'\in\cC$, $c(S)(U')$ is identified with a set of pairs $(f,g)$ with $f:U'\to U$ and $g:V'\to \bbD(f)(V)$ in $\bbD(U')$, i.e., with a set of morphisms
$$
(U',V')\to (U,V)
$$
in $\cC\rtimes\bbD$.

We define a coverage $J\rtimes K$ on $\cC\rtimes\bbD$ by saying that a sieve $\tilde{R}$ on an object $(U,V)$ in $\cC\rtimes\bbD$ is a $J\rtimes K$-covering, if there exists a $J$-covering sieve $R$ of $U$, and, for every $f\in R$, there is an $S_f\in K(\mathop{\rm dom}(f))$ with $b(S_f)=V$, such that $\tilde{R}$ contains the set
$$
\{(f,g)\,:\, f\in R, (f,g)\in S_f\}.
$$

\subsection{Bounded morphisms}\label{bdd-morphi}

We say that a geometric morphism $f:\cF\to\cE$ is \emph{bounded}, if there exists an object $G\in \cF$ which \emph{generates $\cF$ over $\cE$} in the sense that, for every $X\in\cF$ there exists a $Y\in\cE$, a subobject $S\to f^*(Y)\times G$ and an epimorphism $S\to X$.

Let $\cG\xrightarrow{g}\cF\xrightarrow{f}\cE$ be two geometric morphisms. By \cite[B3.1.10]{elephant1}, if both $f$ and $g$ are bounded, then $fg$ is bounded. If $fg$ is bounded, then $g$ is bounded.

Any geometric inclusion is bounded. Moreover, if $\C\in \cat(\cE)$, the canonical geometric morphism $[\C^\op,\cE]\to \cE$ is bounded (\cite[B3.2.1]{elephant1}).
\begin{theorem}[Giraud-Mitchell-Diaconescu, {\cite[B3.3.4]{elephant1}}]
Any bounded morphism $\cF\xrightarrow{f} \cE$ can be decomposed (up to natural isomorphism) as
$$
\cF\xrightarrow{i} [\C^\op,\cE]\to \cE,
$$
where $i$ is an inclusion, and $\C$ is some category internal in $\cE$. 

Moreover, $f$ is equivalent to the canonical morphism
$$
\sh_\cE(\C,K)\to \cE 
$$
for some internal site $(\C,K)$ in $\cE$.
\end{theorem}

Giraud's Theorem~\ref{giraud-th} can be reformulated (\cite[C2.2.8]{elephant2}) by saying that a topos $\cE$ is Grothendieck if and only if it admits a bounded geometric morphism 
$\cE\to\Set$.

Let $\cF\xrightarrow{f} \cE$ be a geometric morphism between Grothendieck topoi, and let $\cE\simeq\sh(\cC,J)$. Then $f$ is bounded and there exists an internal site $(\bbD,K)$ in $\cE$ so that $\cF\simeq\sh_\cE(\bbD,K)$. It follows from construction \ref{semidir} (more preciselu, by \cite[C2.5.4]{elephant2}) that  
$$
\cF\simeq \sh_\cE(\bbD,K)\simeq \sh(\cC\rtimes\bbD,J\rtimes K),
$$
and that $f$ is induced by the projection functor $\cC\rtimes\bbD\to \cC$.

%

\subsection{Enriched structure of internal presheaves}\label{enr-int-presh}

Let $\cE$ be a topos, and let $\C\in\cat(\cE)$. We write 
$$
\widehat{\C}=[\C^\op,\cE]
$$
for the category of internal presheaves on $\C$, considered as a $\cE$-enriched category with internal hom objects defined as
$$
\widehat{\C}(F,F')=\Gamma[F,F']=\sprod_{C_0}U[F,F'],
$$
where $[F,F']\in[\C^\op,\cE]$ is the internal hom constructed in \ref{comon-coalg}.

More explicitly, let $F=(F_0,\alpha)$ be given by a $\cE_{\ov C_1}$-morphism $\alpha:d_1^*F_0\to d_0^*F_0$ (as in \ref{mon-alg}), and similarly for $F'=(F_0',\alpha')$. Then 
$\widehat{\C}(F,F')$ is the equaliser 
\begin{center}
 \begin{tikzpicture} 
\matrix(m)[matrix of math nodes, row sep=0em, column sep=1.7em, text height=1.5ex, text depth=0.25ex]
 {
|(0)|{\widehat{\C}(F,F')} & |(1)|{\sprod_{C_0}[F_0,F_0']_{C_0}}		& |(2)|{\sprod_{C_1}[d_1^*F_0,d_0^*F_0']_{C_1}} 	\\
 }; 
\path[->,font=\scriptsize,>=to, thin,yshift=12pt]
(0) edge node[above]{} (1)
([yshift=2pt]1.east) edge node[above]{} ([yshift=2pt]2.west) 
([yshift=-2pt]1.east)edge node[below]{}   ([yshift=-2pt]2.west) 
;
\end{tikzpicture}
\end{center}
where the top arrow is the adjoint of the composite
\begin{multline*}
C_1^*\sprod_{C_0}[F_0,F_0']_{C_0}\simeq d_0^*C_0^*\sprod_{C_0}[F_0,F_0']_{C_0}\xrightarrow{d_0^*(\epsilon)}d_0^*[F_0,F_0']_{C_0}\\ \simeq [d_0^*F_0,d_0^*F_0']_{C_1}\xrightarrow{[\alpha,d_0^*F_0']_{C_1}}[d_1^*F_0,d_0^*F_0']_{C_1},
\end{multline*}
where we wrote $\epsilon$ for the counit of $C_0^*\dashv\sprod_{C_0}$, while the bottom arrow is obtained analogously using $\alpha'$.

\subsection{Geometric morphisms and enrichment}\label{enrich-via-geom}

The considerations of \ref{enr-int-presh} can be generalised, as noted in \cite{mitchell-on-topoi-as-closed-cats}. 

If $f:\cF\to\cE$ is a geometric morphism, then $\cF$ can be considered a tensored and cotensored $\cE$-category with internal hom objects
$$
\cF_{\ov\cE}(X,Y)=f_*[X,Y]_\cF\in \cE.
$$
The natural isomorphism 
$
f_*[f^*E,X]\simeq[E,f_*X]
$,
seen in \ref{geom-morphisms}, shows that the adjunction $f^*\dashv f_*$ is $\cE$-enriched, i.e., we have 
$$
\cF_{\ov\cE}(f^*E,X)\simeq [E,f_*X]_\cE.
$$

The tensored structure is given by the rule
$$
E\otimes X=f^*E\times X,
$$
where $E\in\cE$ and $X\in \cF$.

The cotensored structure is defined as
$$
[E,X]=[f^*E,X].
$$

\subsection{Enriching slices}\label{enr-int-slices}

Let $\cE$ be a topos, and let $S\in\cE$. Then $\cE_{\ov S}$ is also a topos and hence cartesian closed, so we can make it into an $\cE$-enriched category with internal hom objects
$$
\cE_{\ov S}(X\to S,Y\to S)=\sprod_S[X\to S,Y\to S]_S.
$$


Let $\C\in\cat(\cE)$, let $F\in\widehat{\C}$ be an internal presheaf 
on $\C$, and let $f:\F\to\C$ be the associated discrete fibration. Then the equivalence of categories from \ref{discr-fibs} can be enriched to the $\cE$-equivalence
$$
\widehat{\C}_{\ov F}\simeq \widehat{\F},
$$
by letting, for $G\to F, H\to F\in\widehat{\C}_{\ov F}$, 
$$
\widehat{\C}_{\ov F}(G\to F,H\to F)=\widehat{\F}(G,H).
$$

\begin{lemma}\label{adj-int-slices}
The functor $\varinjlim_f:\widehat{\C}_{\ov F}\simeq\widehat{\F}\to\widehat{\C}$ is $\cE$-left adjoint to $f^*:\widehat{\C}\to\widehat{\F}$ in the sense that, for $G\in\widehat{\F}$ and $H\in\widehat{\C}$,
$$
\widehat{\C}(\varinjlim_f G,H)\simeq \widehat{\F}(G,f^*H).
$$ 
\end{lemma}
\begin{proof}
We apply the functor $\varprojlim_{\C}$ to both sides of the relation
$$
[\varinjlim_f G,H]_{\widehat{\C}}\simeq \varprojlim_f[G,f^*H]_{\widehat{\F}}
$$
and use the fact that $\varprojlim_{\C}\varprojlim_f=\varprojlim_{\F}$.
\end{proof}

\subsection{Enriched and internal Grothendieck construction}\label{enrich-groth-constr}

Let $\cE$ be a topos, 
let $\cC$ be a small $\cE$-enriched category and let $\hC=[\cC^\op,\cE]$ be the $\cE$-category of $\cE$-presheaves on $\cC$. 

Given a $\cE$-presheaf $\bF\in\hC$, let
$$
\F=\cC_{\Ov\bF}
$$ 
be the $\cV$-internal category 
with the object of objects
$$
F_0=\coprod_{S\in\cC}\bF(S),
$$ 
and the object of morphisms
$$
F_1=\coprod_{S,S'\in\cC}\cC(S',S)\times\bF(S),
$$
where $d_0:F_1\to F_0$ is obtained by pasting together the actions
$$
\cC(S',S)\times\bF(S)\xrightarrow{e_{S'S}} \bF(S')
$$
obtained by adjunction from the morphisms $\bF_{S'S}:\cC(S',S)\to [\bF(S),\bF(S')]$,
and $d_1:F_1\to F_0$ is obtained from the projections
$$
\cC(S',S)\times\bF(S)\xrightarrow{p_{S'S}} \bF(S).
$$

Given an object $S\in\cC$, we define the $\cE$-internal category
$$
\cC_{\Ov S}=\cC_{\Ov\h_S},
$$
and, in particular, the internal category
$$
\C=\cC_{\Ov I}
$$
is the \emph{internalisation} of $\cC$.

On the other hand, the slice category $\hC_{\ov\bF}$ is $\cE$-enriched as follows. For $\bG\xrightarrow{g}\bF$ and $\bH\xrightarrow{h}\bF$, the internal hom object
$$
\hC_{\ov\bF}(g,h)=\hC_{\ov\bF}(\bG,\bH)
$$
is defined as the equaliser
\begin{center}
 \begin{tikzpicture} 
\matrix(m)[matrix of math nodes, row sep=0em, column sep=1em, text height=1.5ex, text depth=0.25ex]
 {
|(0)|{\hC_{\ov\bF}(g,h)} & |(1)|{\displaystyle\prod_{S\in\cC}\sprod_{\bF(S)}[g_S,h_S]_{\bF(S)}}		& |(2)|{\displaystyle\prod_{S,S'\in\cC}\sprod_{\cC(S',S)\times\bF(S)}[p_{S'S}^*(g_S),e_{S'S}^*(h_{S'})]_{\cC(S',S)\times\bF(S)}} 	\\
 }; 
\path[->,font=\scriptsize,>=to, thin,yshift=12pt]
(0) edge node[above]{} (1)
([yshift=2pt]1.east) edge node[above]{} ([yshift=2pt]2.west) 
([yshift=-2pt]1.east)edge node[below]{}   ([yshift=-2pt]2.west) 
;
\end{tikzpicture}
\end{center}
where the top arrow is formed using the adjunctions of the composites
\begin{multline*}
(\bF(S)\times\cC(S',S))^*\sprod_{\bF(S')}[g_{S'},h_{S'}]_{\bF(S')}\simeq e_{S'S}^*\bF(S')^*\sprod_{\bF(S')}[g_{S'},h_{S'}]_{\bF(S')}\\
\xrightarrow{p_{S'S}^*(\epsilon_{S'})} e_{S'S}^*[g_{S'},h_{S'}]_{\bF(S')}\simeq[e_{S'S}^*(g_{S'}),e_{S'S}^*(h_{S'})]_{\cC(S',S)\times\bF(S)}\\
\xrightarrow{[\alpha_{S'S},e_{S'S}^*(h_{S'})]_{\cC(S',S)\times\bF(S)}}
[p_{S'S}^*(g_S),e_{S'S}^*(h_{S'})]_{\cC(S',S)\times\bF(S)},
\end{multline*}
where $\epsilon_{S'}$ was the counit of the adjunction $\bF(S')^*\dashv\sprod_{\bF(S')}$, and the $\cE_{\ov{\cC(S',S)\times\bF(S)}}$-morphism $\alpha_{S'S}:p_{S'S}^*(g_S)\to e_{S'S}^*(g_{S'})$ is obtained using the presheaf structure of $\bG\xrightarrow{g}\bF$. The bottom arrow is obtained analogously, using the presheaf structure of $\bH\xrightarrow{h}\bF$.

We note that the right hand side term in the above equaliser is isomorphic to
$$
\prod_{S,S'\in\cC}\sprod_{\bF(S')}[\ssum_{e_{S'S}}p_{S'S}^*(g_S),h_{S'}]_{\bF(S')}
\simeq
\prod_{S,S'\in\cC}\sprod_{\bF(S)}[g_S,\sprod_{p_{S'S}}e_{S'S}^*(h_{S'})]_{\bF(S)},
$$
and that for $\bF=\h_I$, the above equaliser coincides with the definition given in \ref{enr-fun-cat}.

\begin{theorem}\label{th-groth}
We have an equivalence of $\cE$-enriched  
categories
$$
\widehat{\cC_{\Ov\bF}}\simeq \hC_{\ov\bF}.
$$
\end{theorem}

\begin{proof}
Let $H=(H_0,H_1)$ be an internal presheaf on $\F=\cC_{\Ov\bF}=(F_0,F_1)$ as above, determined by
$$
\gamma_0:H_0\to F_0, 
$$
and the action 
$$
e:H_1=F_1\times_{F_0}H_0\to H_0.
$$

We define $\bH\in\hC_{\ov\bF}$ as follows. For $S\in\cC$, let $\bH(S)$ be the pullback
 \begin{center}
 \begin{tikzpicture} 
\matrix(m)[matrix of math nodes, row sep=2em, column sep=2em, text height=1.5ex, text depth=0.25ex]
 {
 |(1)|{\bH(S)}		& |(2)|{H_0} 	\\
 |(l1)|{\bF(S)}		& |(l2)|{F_0} 	\\
 }; 
\path[->,font=\scriptsize,>=to, thin]
(1) edge node[above]{} (2) edge node[left]{}   (l1)
(2) edge node[right]{$\gamma_0$} 
(l2) 
(l1) edge  
(l2);
\end{tikzpicture}
\end{center}
where the bottom arrow is the natural morphism $\bF(S)\to F_0=\coprod_{T\in\cC}\bF(T)$.

Moreover, given $S,S'\in\cC$, using the fact that 
\begin{multline*}
\cC(S',S)\times\bF(S)\times_{F_1}H_1=\cC(S',S)\times\bF(S)\times_{F_1}\left(F_1\times_{F_0}H_0\right)\\
\simeq\cC(S',S)\times\bF(S)\times_{F_0}H_0
\simeq \cC(S',S)\times\bH(S),
\end{multline*}
we obtain the diagram
$$
 \begin{tikzpicture}
[cross line/.style={preaction={draw=white, -,
line width=4pt}}]
\matrix(m)[matrix of math nodes, row sep=.1em, column sep=1em,  
text height=1.2ex, text depth=0.25ex]
{
|(h)|{H_0} 		&		&[2em] |(a)|{H_1} 
&[-2.2em]		& |(b)|{H_0}	&			\\[1.2em]
			& |(H)|{\bH(S)}  	& 		&|(A)|{\cC(S',S)\times\bH(S)} 	&		& |(B)|{\bH(S')}\\[.6em]          
|(u)|{G_0}	&		&|(c)|{G_1}	&		& |(d)|{G_0} &			\\[1.2em]
			&|(U)|{\bF(S)}	&		&|(C)|{\cC(S',S)\times\bF(S)} &			& |(D)|{\bF(S')}   \\};
\path[->,font=\scriptsize,>=to, thin,inner sep=1pt]
(a) edge node[pos=0.5,above]{$\pi_2$}(h)
(c) edge node[pos=0.3,above]{$d_1$}(u)
(h) edge node[pos=0.5,left]{$\gamma_0$}(u) 
(U) edge  (u) 
(H) edge [cross line] (U) edge (h)
(C) edge (U) 
(a)edge node[pos=0.5,above]{$e$}(b)
(b)edge node[pos=0.2, left
]{$\gamma_0$}(d)
(a) edge node[pos=0.3,left]{$\pi_1$}(c)
(c)edge node[pos=0.6,above
]{$d_0$}(d)
(A)edge[cross line, dashed]  (B)
(B)edge[cross line]  (D)
(A)edge[cross line]  (C)
(C)edge[cross line]  (D)
(A)edge 
(a)
(C)edge (c)
(B)edge (b)
(A) edge [cross line] (H)
(D)edge (d);
\end{tikzpicture}
$$ 
in which all the vertical parallelograms are pullbacks, and the resulting dashed action yields the $\cV$-presheaf structure morphism
$$
\bH_{S'S}:\cC(S',S)\to[\bH(S),\bH(S')].
$$

Conversely, let $f\in\hC_{\ov\bF}$ be a $\cV$-natural transformation $f:\bH\to\bF$. We define an internal presheaf $H=(H_0,H_1)$ on $\F$ by setting
$$
H_0=\coprod_{S\in\cC}\bH(S), \text{ and } \gamma_0=\coprod_{S\in\cC}f_S: H_0\to F_0.
$$
The action 
$$
e:F_1\times_{F_0}H_0\to H_0
$$
is obtained, using the disjointness of coproducts and the fact that coproducts commute with pullbacks, as
\begin{multline*}
\coprod_{S,S'}\cC(S',S)\times\bF(S)\times_{\coprod_T\bF(T)}\coprod_T\bH(T)\simeq
\coprod_{S,S'}\left(\cC(S',S)\times\bF(S)\right)\times_{\bF(S)}\bH(S)\\
\to\coprod_{S,S'}\cC(S',S)\times\bH(S)\xrightarrow{\coprod_{S,S'}e_{S'S}}
\coprod_{S'}\bH(S').
\end{multline*}
We leave the full details of the verification that the assignments 
$$
H\mapsto (\bH\to \bF), \ \ \  \text{ and } \ \ \ (\bH\to\bF)\mapsto H
$$
define an equivalence of $\cE$-categories $\widehat{\F}\simeq \hC_{\ov\bF}$ to the reader, and we illustrate the proof by showing that, given objects $\bG\xrightarrow{g}\bF, \bG\xrightarrow{h}\bF\in\hC_{\ov\bF}$ and the corresponding $G,H\in\widehat{\F}$, we have
$$
\hC_{\ov\bF}(g,h)\simeq \widehat{\F}(G,H).
$$
By  \ref{enr-int-presh},  
$$
\begin{tikzcd}[column sep=1em]
\widehat{\F}(G,H)= {\Eq\left(\sprod_{F_0}[G_0,H_0]_{F_0}\right.}\ar[yshift=2pt]{r}{} \ar[yshift=-2pt]{r}[swap]{} &{\left.\sprod_{F_1}[{}^{\F}d_1^*G_0,{}^{\F}d_0^*H_0]_{F_1}\right)}.
\end{tikzcd}
$$
We claim that
$$
\sprod_{F_0}[G_0,H_0]_{F_0}\simeq\prod_{S\in\cC}\sprod_{\bF(S)}[g_S,h_S]_{\bF(S)}.
$$
Indeed, for an arbitrary $X\in\cE$, using the disjointness of coproducts and the fact that they commute with pullbacks, 
\begin{multline*}
\cE(X,\sprod_{F_0}[G_0,H_0]_{F_0})\simeq \cE_{\ov F_0}(F_0^*X, [G_0,H_0])\simeq
\cE_{\ov F_0}(F_0^*X\times_{F_0}G_0,H_0)\\ \simeq
\cE_{\ov\coprod_S\bF(S)}\left(\coprod_S\bF(S)^*X\times_{\coprod_S\bF(S)}\coprod_S\bG(S),\coprod_S\bH(S)\right)\\
\simeq \cE_{\ov\coprod_S\bF(S)}\left(\coprod_S\left(\bF(S)^*X\times_{\bF(S)}\bG(S)\right),\coprod_S\bH(S)\right)\\
\simeq \prod_S\cE_{\ov\bF(S)}\left(\bF(S)^*X\times_{\bF(S)}\bG(S),\bH(S)\right)\\
\simeq \prod_S\cE_{\ov\bF(S)}\left(\bF(S)^*X,[\bG(S),\bH(S)]_{\bF(S)}\right)\\
\simeq \prod_S\cE\left(X,\sprod_{\bF(S)}[\bG(S),\bH(S)]_{\bF(S)}\right)\\
\simeq \cE\left(X,\prod_S\sprod_{\bF(S)}[\bG(S),\bH(S)]_{\bF(S)}\right).
\end{multline*}
\end{proof}
Similarly, we obtain that
$$
\sprod_{F_1}[{}^{\F}d_1^*G_0,{}^{\F}d_0^*H_0]_{F_1})=\prod_{S,S'}\sprod_{\cC(S',S)\times\bF(S)}[p_{S'S}^*(g_S),e_{S'S}^*(h_S')]_{\cC(S',S)\times\bF(S)},
$$
and, upon a further verification that the morphisms in the equalisers are also compatible, we obtain the expression for $\hC_{\ov\bF}(g,h)$ given previously.

To draw an analogy with the considerations of \ref{ord-groth-constr}, let
$$
\ssum_\bF:\hC_{\ov\bF}\to \hC
$$
be the natural $\cE$-functor which maps $\bG\to\bF$ to $\bG$, and let 
$$p_\bF=\bF^*:\hC\to\hC_{\ov\bF}$$
be the base change functor.

Moreover, the canonical discrete fibration $i_\F:\F=\cC_{\Ov\bF}\to\C$ yields the functor
$$
i_\F^*:\hC\simeq\widehat{\C}\to\widehat{\cC_{\Ov\bF}}.
$$
For $\bG\in\hC$, we write $\bG_\F=i_\F^*\bG$.

\begin{corollary}
\begin{enumerate}
\item\label{adj-bc} There is an adjunction of $\cE$-enriched functors
$$
\ssum_\bF\dashv \bF^*. 
$$
\item\label{bc-comp} There is an isomorphism of $\cE$-functors 
$$\alpha\circ i_\F^*\simeq \bF^*,$$
where $\alpha:\widehat{\cC_{\Ov\bF}}\simeq \hC_{\ov\bF}$.
\end{enumerate}
\begin{proof}
For (\ref{adj-bc}), given $\bG\to\bF\in\hC_{\ov\bF}$ and $\bH\in\hC$,
\begin{multline*}
\hC\left(\ssum_\bF(\bG\to\bF),\bH\right)\stackrel{{\ref{th-groth}}}{\simeq}
\widehat{\C}(\ssum_F(G\to F),H)\\
\stackrel{{\ref{adj-int-slices}}}{\simeq}\widehat{\cC_{\Ov\bF}}(G,F\times H)
\stackrel{\ref{th-groth}}{\simeq}
\hC_{\ov\bF}(\bG\to\bF, \bF\times \bH\to \bF).
\end{multline*}

To see (\ref{bc-comp}), we note that, by the construction from the proof of \ref{th-groth}, $\alpha$ takes $\bG_\F\in\widehat{\F}$ to an object of $\hC_{\ov\bF}$ whose value at $S\in\cC$ is the pullback
$$
(F_0\times G_0)\times_{F_0}\bF(S)\simeq \bF(S)\times\bG(S).
$$ 
\end{proof}

\end{corollary}

\begin{corollary}\label{psh-int-hom-rel}
For $S\in\cC$, $\bF,\bG\in\hC$,
$$
\widehat{\cC_{\Ov S}}(\bF_S,\bG_S)\simeq \hC_{\ov\h_S}(\bF\times\h_S,\bG\times\h_S)\simeq\hC(\bF\times\h_S,\bG)\simeq[\bF,\bG](S).
$$
\end{corollary}

\subsection{Torsors and Diaconescu's Theorem}\label{tors-diacon}

Let $\C$ be an internal category in a topos $\cS$, and let $\cE\xrightarrow{p}\cS$ be a geometric morphism. 

A \emph{$\C$-torsor} in $\cE$ (or a \emph{flat presheaf}) is an object $F\in[\C^\op,\cE]$ such that the domain of the corresponding discrete fibration $\F\to p^*(\C)$ is filtered in the sense that
each of the morphisms
$$
F_0\to 1,\ \ \ (d_1p_1,d_1p_2):P_\F\to F_0\times F_0, \ \ \ (d_2t_1,d_2t_2):T_\F\to R_\F
$$
is a cover (regular epimorphism), where the objects $P_\F$ (of pairs of morphisms), $R_\F$ (of parallel pairs) and $T_\F$ (of coequaliser-type diagrams) are the pullbacks 
\begin{center}
\begin{tikzpicture} 
\matrix(m)[matrix of math nodes, row sep=2em, column sep=2em, text height=1.5ex, text depth=0.25ex]
 {
 |(1)|{P_\F}		& |(2)|{F_1} 	\\
 |(l1)|{F_1}		& |(l2)|{F_0} 	\\
 }; 
\path[->,font=\scriptsize,>=to, thin]
(1) edge node[above]{$p_1$} (2) edge node[left]{$p_2$}   (l1)
(2) edge node[right]{$d_0$} (l2) 
(l1) edge  node[above]{$d_0$} (l2);
\end{tikzpicture}
 \begin{tikzpicture} 
\matrix(m)[matrix of math nodes, row sep=2em, column sep=2.5em, text height=1.5ex, text depth=0.25ex]
 {
 |(1)|{R_\F}		& |(2)|{F_1} 	\\
 |(l1)|{F_1}		& |(l2)|{F_0\times F_0} 	\\
 }; 
\path[->,font=\scriptsize,>=to, thin]
(1) edge node[above]{$r_1$} (2) edge node[left]{$r_2$}   (l1)
(2) edge node[right]{$(d_0,d_1)$} (l2) 
(l1) edge  node[above]{$(d_0,d_1)$} (l2);
\end{tikzpicture}
 \begin{tikzpicture} 
\matrix(m)[matrix of math nodes, row sep=2em, column sep=2.5em, text height=1.5ex, text depth=0.25ex]
 {
 |(1)|{T_\F}		& |(2)|{F_2} 	\\
 |(l1)|{F_2}		& |(l2)|{F_0\times F_0} 	\\
 }; 
\path[->,font=\scriptsize,>=to, thin]
(1) edge node[above]{$t_1$} (2) edge node[left]{$t_2$}   (l1)
(2) edge node[right]{$(d_0,d_1)$} (l2) 
(l1) edge  node[above]{$(d_0,d_1)$} (l2);
\end{tikzpicture}
\end{center}
These diagrams are internal versions of the usual axioms for a filtered small category, i.e., 
\begin{enumerate}
\item there exists and object;
\item  for any pair of objects, there exists a pair of morphisms from them with common codomain;
\item for any parallel pair of morphisms, there exists a morphism coequalising them.
\end{enumerate}

We consider the full subcategory
$$
\Tors(\C,\cE)
$$
of $[\C^\op,\cE]$ whose objects are $\C$-torsors in $\cE$.

\begin{theorem}[Diaconescu]\label{diaconescu}
Let $\C$ be an internal category in a topos $\cS$, and $p:\cE\to\cS$ a geometric morphism. Then there is an equivalence of categories
$$
\mathfrak{Top}_{\ov \cS}(\cE,[\C,\cS])\simeq \Tors(\C,\cE),
$$
which is natural in $\cE$.
\end{theorem}

\section{Algebraic structures in enriched categories and topoi}

\subsection{Algebraic theories}

A \emph{finitary algebraic theory} $\bbT$ consists of a set of operations/function symbols of arbitrary arity and a set of finitary equations between them. It is \emph{finitely presented} if both sets are finite.  

If $\cC$ is a category with finite products, a \emph{model of $\bbT$ in $\cC$} is an object $M\in \cC$, equipped with a morphism $f:M^n\to M$ for each $n$-ary operation $f$ in $\bbT$, such that each equation in $\bbT$ yields a commutative diagram in $\cC$. We write
$$
\bbT(\cC)=\cC_\bbT
$$
for the category of models of $\bbT$ in $\cC$.

In other words, a category of \emph{algebraic structures} and their morphisms in $\cC$ is defined by a set of commutative diagrams expressing the existence of certain finite limits. Typical examples include monoids, groups, rings and modules. 

In categories of presheaves, limits (and colimits) are defined argument-wise, so we remark (as in \cite[I~3.2]{sga4.1}) that 
$$
\bbT(\hC)\simeq [\cC^\op,\bbT(\Set)].
$$ 
An enriched version of the statement is the following. Let $\cV$ be a complete cartesian closed category. We define the $\cV$-category
$$
\cV\da\bbT
$$
with the same objects as $\bbT(\cV)$. For $G,H\in\bbT(\cV)$, the internal hom object
$$
\cV\da\bbT[G,H]\in\cV
$$  
is the subobject of $[G,H]$ defined by the appropriate finite limit as specified by $\bbT$.

If $\cC$ is a small $\cV$-category, \ref{enrich-presh-duality} shows that the $\cV$-category $\hC=[\cC^\op,\cV]$ of $\cV$-presheaves is complete cartesian closed, so by the above, we defined the $\hC$-category
$$
\hC\da\bbT.
$$
On the other hand, $\hC\da\bbT$ is a $\cV$-category with internal hom objects
$$
\hC\da\bbT(\bG,\bH)=\hC\da\bbT[\bG,\bH](e).
$$
We have an equivalence of $\cV$-categories
$$
\hC\da\bbT\simeq[\cC^\op,\cV\da\bbT].
$$
Moreover, if we replace $\cV$ by a topos $\cE$, then, using \ref{psh-int-hom-rel},
$$
\hC\da\bbT[\bG,\bH](X)\simeq \hC_{\Ov X}\da\bbT(i_X^*\bG,i_X^*\bH),
$$
for $\bG,\bH\in \hC\da\bbT$ and $X\in\cC$. In order to avoid awkward notation for a general $\bbT$, we will provide a proof in the concrete case of groups in \ref{V-gp-prop}.

If $\cE$ is a topos with a natural number object, the forgetful functor $U:\bbT(\cE)\to\cE$ admits a left adjoint, called the \emph{free $\bbT$-model} functor
$$
F:\cE\to \bbT(\cE).
$$
Moreover, $\bbT(\cE)$ is monadic over $\cE$ and has finite colimits.

If $f:\cF\to\cE$ is a geometric morphism, since both $f^*$ and $f_*$ preserve finite limits, they define an adjoint pair of functors 
$$
f^*:\bbT(\cE)\to\bbT(\cF), \ \ \ f_*:\bbT(\cF)\to \bbT(\cE),
$$
and we have canonical isomorphisms
$$
F\circ f^*\simeq f^*\circ F, \ \ \ U\circ f_*\simeq f_*\circ U.
$$

We lay the details of particularly important algebraic structures out in the sequel.

\subsection{Monoid objects in monoidal categories}\label{monoid-obs}

Let $\cV$ be a symmetric monoidal category. 

A \emph{monoid object in $\cV$} is an object $G$, together with $\cV$-morphisms
$$
m_G:G\otimes G\to G, \ \ \ e_G:I\to G,
$$
making the diagrams
 \begin{center}
 \begin{tikzpicture} 
\matrix(m)[matrix of math nodes, row sep=2em, column sep=3em, text height=1.9ex, text depth=0.25ex]
 {
 |(1)|{G\otimes G\otimes G}		& |(2)|{G\otimes G} 	\\
 |(l1)|{G\otimes G}		& |(l2)|{G} 	\\
 }; 
\path[->,font=\scriptsize,>=to, thin]
(1) edge node[above]{$\id\otimes m_G$} (2) edge node[left]{$m_G\otimes\id$}   (l1)
(2) edge node[right]{$m_G$} (l2) 
(l1) edge  node[below]{$m_G$} (l2);
\end{tikzpicture}
 \begin{tikzpicture} 
\matrix(m)[matrix of math nodes, row sep=2em, column sep=3em, text height=1.9ex, text depth=0.25ex]
 {
 |(1)|{G}		& |(2)|{G\otimes G} 	\\
 |(l1)|{G\otimes G}		& |(l2)|{G} 	\\
 }; 
\path[->,font=\scriptsize,>=to, thin]
(1) edge node[above]{$(\id,\bar{e}_G)$} (2) edge node[left]{$(\bar{e}_G,\id)$}   (l1)
(2) edge node[right]{$m_G$} (l2) 
(l1) edge  node[below]{$m_G$} (l2)
(1) edge  node[pos=0.6,above]{$\id$} (l2);
\end{tikzpicture}
\end{center}
commutative, where we wrote $\bar{e}_G$ for the composite $G\to I\stackrel{e_G}{\to}G$. We say that $G$ is \emph{commutative,} if it satisfies an additional diagram expressing that $m_G$ is symmetric.

\subsection{Monoid actions and modules}\label{monoid-action}

Let $\cV$ be a symmetric monoidal category, let $G$ be a $\cV$-monoid, and let $X$ be an object of $\cV$. A \emph{$\cV$-action of $G$ on $X$} is a $\cV$-morphism
$$
\mu:G\otimes X\to X
$$
that makes the diagrams
 \begin{center}
 \begin{tikzpicture} 
\matrix(m)[matrix of math nodes, row sep=2em, column sep=3em, text height=1.9ex, text depth=0.25ex]
 {
 |(1)|{G\otimes G\otimes X}		& |(2)|{G\otimes X} 	\\
 |(l1)|{G\otimes X}		& |(l2)|{X} 	\\
 }; 
\path[->,font=\scriptsize,>=to, thin]
(1) edge node[above]{$\id\times \mu$} (2) edge node[left]{$m_G\times\id$}   (l1)
(2) edge node[right]{$\mu$} (l2) 
(l1) edge  node[below]{$\mu$} (l2);
\end{tikzpicture}
 \begin{tikzpicture} 
\matrix(m)[matrix of math nodes, row sep=2em, column sep=3em, text height=1.9ex, text depth=0.25ex]
 {
 |(1)|{I\otimes X}		&  |(2)|{X}	\\
 |(l1)|{G\otimes X}		& |(l2)|{X} 	\\
 }; 
\path[->,font=\scriptsize,>=to, thin]
(2) edge (1) edge node[right]{$\id$} (l2) 
(1)  edge node[left]{$e_G\times\id$}   (l1)
(l1) edge  node[below]{$\mu$} (l2)
;
\end{tikzpicture}
\end{center}
commutative. We also say that $X$ is a \emph{(left) $G$-module}. The category of $G$-modules is denoted
$$
G\Mod.
$$
The category 
of right $G$-modules is defined analogously. 

If $\cV$ is symmetric monoidal closed, an action $G\otimes X\to X$ gives rise to a $\cV$-monoid morphism
$$
G\to [X,X],
$$  
where $[X,X]$ is a $\cV$-monoid through the composition morphism $[X,X]\otimes[X,X]\to[X,X]$. Conversely, any such $\cV$-monoid morphism yields an action of $G$ on $X$.

If  $\cV$ is cartesian closed, $G$ can be considered an internal category in $\cV$ and right $G$-modules are identified with internal presheaves on $G^\op$, 
$$
G\Mod\simeq [G,\cV]=[(G^\op)^\op,\cV].
$$
Hence, if $\cV$ is a topos, then 
$$
\bB G=G\Mod
$$
is also a topos  called the \emph{classifying topos of $G$}.

By the symmetry of $\cV$, the category of right $G$-modules is isomorphic to the category of left $G$-modules. 

\subsection{Commutative monoid actions}

Let $\cV$ be symmetric monoidal closed, and let $A$ be a commutative monoid in $\cV$. 
It is known (\cite{florian-marty}) that the category $A\Mod$
is complete/cocomplete whenever $\cV$ is. If $\cV$ admits coequalisers, we obtain a symmetric monoidal category
$$
(A\Mod, \otimes_A)
$$
with the tensor product of $M,N\in A\Mod$ defined as
$$
M\otimes_AN=
\begin{tikzcd}[cramped, sep=small, ampersand replacement=\&]
{\coeq\left(M\otimes A\otimes N\right.}\ar[yshift=2pt]{r}{} \ar[yshift=-2pt]{r}[swap]{} \&{\left.M\otimes N\right)}
\end{tikzcd},
$$
where the top arrow is defined in an evident way using $\mu_M$, and the bottom one using $\mu_N$.

If $\cV$ is bicomplete, then $(A\Mod, \otimes_A)$ is closed with the internal homs
$$
[M,N]_A=
\begin{tikzcd}[cramped, sep=small, ampersand replacement=\&]
{\Eq\left([M,N]\right.}\ar[yshift=2pt]{r}{} \ar[yshift=-2pt]{r}[swap]{} \&{\left.[A\otimes M,N]\right)}
\end{tikzcd},
$$
where the top arrow is defined using the action $\mu_M$, and the bottom using $\mu_N$.

\subsection{Monoid actions and Day convolution}\label{module-ops-day-conv}

Let $\cV$ be symmetric monoidal closed, and let $A$ be a commutative $\cV$-monoid. 

Let $\cA$ be a monoidal $\cV$-category with a single object $I$ and $\cA(I,I)=A$, with composite morphism given by the monoid operation $m_A:A\otimes A\to A$, and with $I\otimes I=I$. Note that $\cA^\op\simeq\cA$.

There is an equivalence of $\cV$-categories (\cite[Example~4.2]{brian-day-on-closed-categories-of-functors})
$$
A\Mod\simeq [\cA,\cV],
$$
such that Day convolution and internal homs on $[\cA,\cV]$, as defined in \ref{day-conv}, coincide with the tensor product and internal homs of $A$-modules.


\subsection{Monoid actions in tensored categories}

Let $\cV$ be a symmetric monoidal closed category and let $\cC$ be a $\cV$-tensored category. Let $G$ be a monoid in $\cV$, and $X\in\cC$. An \emph{action of $G$ on $X$} is a $\cC_0$-morphism
$$
G\otimes X\to X
$$
satisfying the diagrammatic conditions analogous to those from \ref{monoid-action}. The tensored structure yields the relation $\cC(G\otimes X,X)\simeq [G,\cC(X,X)]$, whence we obtain a $\cV$-monoid morphism
$$
G\to \cC(X,X),
$$
where the monoid structure on $\cC(X,X)$ is obtained from the composition $\cV$-morphism
$c_{XXX}:\cC(X,X)\otimes\cC(X,X)\to\cC(X,X)$.

If the enrichment comes from a geometric morphism $\gamma:\cF\to\cE$ of topoi as in \ref{enrich-via-geom}, the tensored structure is given by
$$
E\otimes X=\gamma^*E\times X
$$
for $E\in\cE$ and $X\in\cF$. Hence, if $G$ is an $\cE$-groupoid and $X\in\cF$, an action of $G$ on $X$ in the above sense is in fact an action of $\gamma^*G$ on $X$ in the sense of \ref{monoid-action} (note that $\gamma^*G$ is a groupoid because $\gamma^*$ commutes with finite limits and the theory of groupoids is algebraic).

\subsection{Monoid actions in the opposite category}

Let $\cC$ be a tensored and cotensored category over a symmetric monoidal closed category $\cV$. By \ref{tens-cotens-op}, $\cC^\op$ is again tensored and cotensored over $\cV$.

A (left) action of a $\cV$-groupoid $G$ on $X\in\cC$ corresponds to a right action of $G$ on $X^\op\in\cC^\op$, i.e., to the (left) action of the opposite groupoid $G^\op$ on $X^\op$.

Indeed, a $\cC$-morphism $G\otimes X\to X$ corresponds to a $\cV$-groupoid morphism
$G\to \cC(X,X)$, which corresponds to a $\cV$-groupoid morphism $G^\op\to\cC^\op(X^\op,X^\op)$, which in turn corresponds to an action $G^\op\otimes^\op X^\op\to X^\op$. Note that, in view of the definition of $\otimes^\op$, the last morphism in fact comes from a $\cC$-morphism $X\to [G,X]$ obtained through $\cC(G\otimes X,X)\simeq\cC(X,[G,X])$.

\subsection{Group objects in cartesian closed categories}\label{V-groups}

Let $\cV_0$ be a category with products.

A group object in $\cV_0$ is a monoid object $G$ equipped with a $\cV_0$-morphism
$$
i_G:G\to G
$$
so that the diagram
 \begin{center}
 \begin{tikzpicture} 
\matrix(m)[matrix of math nodes, row sep=2em, column sep=2em, text height=1.5ex, text depth=0.25ex]
 {
|(0)|{G\times G}  &  |(1)|{G}		& |(2)|{G\times G} 	\\
|(l0)|{G\times G} & |(l1)|{G}		& |(l2)|{G\times G} 	\\
 }; 
\path[->,font=\scriptsize,>=to, thin]
(1) edge node[above]{$\Delta$} (0)
(l0) edge node[below]{$m_G$} (l1)
(0) edge  node[left]{$\id\times i_G$} (l0)
(1) edge node[above]{$\Delta$} (2) 

(1) edge node[left]{$\bar{e}_G$}   (l1)
(2) edge node[right]{$i_G\times\id$} (l2) 
(l2) edge node[below]{$m_G$} (l1) 
;
\end{tikzpicture}
\end{center}
commutes.

Let $\cV$ be a cartesian closed category. A \emph{$\cV$-monoid} is a monoid object in $\cV_0$, and a \emph{$\cV$-group} is a group object in $\cV_0$. 

If $G$ and $H$ are $\cV$-groups, the internal group homomorphism object
$$
\cV\daGp[G,H]
$$ 
is the equaliser of the diagram
\begin{center}
 \begin{tikzpicture} 
\matrix(m)[matrix of math nodes, row sep=2em, column sep=3em, text height=1.5ex, text depth=0.25ex]
 {
  		& |(2)|{[G\times G,H\times H]}  & 	\\
 |(l1)|{[G,H]}	& 	& |(l3)|{[G\times G,H]} 	\\
 }; 
\path[->,font=\scriptsize,>=to, thin]
(l1) edge node[above,pos=.4]{$\delta$} (2) edge node[below]{$m_G^{*}$}   (l3)
(2) edge node[above,pos=.7]{$m_{H*}$} (l3) 
;
\end{tikzpicture}
\end{center}
where $\delta$ is obtained by adjunction from the composite of
$$
G\times G\times[G,H]\xrightarrow{\id\times\id\times\Delta}
G\times G\times[G,H]\times[G,H]\simeq
G\times [G,H]\times G\times[G,H]\xrightarrow{\mathop{\rm ev}\times\mathop{\rm ev}}H\times H.
$$
Thus we have defined the $\cV$-category
$$
\cV\daGp
$$
of $\cV$-groups. 

We also write
$$
\uHom_{\cV\daGp}(G,H)=\cV\daGp[G,H].
$$
It is proved in \cite[Proof of Theorem~2.1]{borceux-clementino-montoli} that 
$$
\uEnd_{\cV\da\Gp}(G)=\cV\da\Gp[G,G]
$$ 
is a submonoid of $[G,G]$. By analogy with \ref{int-isom}, we can define objects
$$
\uIsom_{\cV\da\Gp}(G,H), \uAut_{\cV\da\Gp}(G)\in\cV.
$$

\subsection{Group actions in cartesian closed categories}\label{V-actions}
Let $\cV$ be a complete cartesian closed category. 

Let $G$ be a $\cV$-group, and let $X$ be an object of $\cV$. A \emph{$\cV$-action of $G$ on $X$} is a $\cV$-monoid action
$$
\mu:G\times X\to X.
$$

This yields a $\cV$-monoid morphism $\rho:G\to [X,X]$, when $[X,X]$ is equipped with monoid multiplication coming from composition  
$$c=c_{XXX}:[X,X]\times[X,X]\to[X,X].$$
It follows, using the fact that $G$ is a $\cV$-group, that the solid part of the diagram
\begin{center}
 \begin{tikzpicture} 
\matrix(m)[matrix of math nodes, row sep=1em, column sep=2em, text height=1.9ex, text depth=0.25ex]
 {
 |(0)|{G} & & \\
 & |(1)|{\uAut(X)}		& |(2)|{I} 	\\[1em]
 & |(l1)|{[X,X]\times[X,X]}		& |(l2)|{[X,X]\times[X,X]} 	\\
 }; 
\path[->,font=\scriptsize,>=to, thin]
(0) edge[dashed] (1)
(0) edge[bend right=15] node[left]{$(\rho,\rho\circ i_G)$} (l1)
(0) edge[bend left=10] (2)
(1) edge node[above]{} (2) edge node[left]{}   (l1)
(2) edge node[right]{$(u_X,u_X)$} (l2) 
(l1) edge  node[below]{$(c,c^{\mathop{\rm op}})$} (l2);
\end{tikzpicture}
\end{center}
commutes, whence we obtain the dashed homorphism of $\cV$-groups
$$
G\to\uAut(X).
$$
Conversely, such a homomorphism defines an action of $G$ on $X$.

\subsection{Group objects in a category tensored over a cartesian closed category}\label{C-V-groups}

Let $\cC$ be a $\cV$-category with products, which is tensored over a complete cartesian closed category $\cV$. 

A \emph{$\cC\da\cV$-group} is a group object in $\cC_0$. For  $\cC\da\cV$-groups $G$, $H$, the internal group homomorphism object 
$$
\cC\da\cV\daGp(G,H)\in\cV
$$
is the equaliser of the diagram
\begin{center}
 \begin{tikzpicture} 
\matrix(m)[matrix of math nodes, row sep=2em, column sep=3em, text height=1.5ex, text depth=0.25ex]
 {
  		& |(2)|{\cC(G\times G,H\times H)}  & 	\\
 |(l1)|{\cC(G,H)}	& 	& |(l3)|{\cC(G\times G,H)} 	\\
 }; 
\path[->,font=\scriptsize,>=to, thin]
(l1) edge node[above,pos=.4]{$\delta$} (2) edge node[below]{$m_G^{*}$}   (l3)
(2) edge node[above,pos=.7]{$m_{H*}$} (l3) 
;
\end{tikzpicture}
\end{center}
where $\delta$ is obtained by adjunction from the composite of
$$
G\times G\times\cC(G,H)\xrightarrow{\id\times\id\times\Delta}
G\times G\times\cC(G,H)\times\cC(G,H)\simeq
G\times \cC(G,H)\times G\times\cC(G,H)\xrightarrow{\mathop{\rm ev}\times\mathop{\rm ev}}H\times H.
$$
Thus we have defined a $\cV$-category
$$
\cC\da\cV\daGp
$$
of $\cC\da\cV$-groups.

\subsection{Module objects in cartesian closed categories}\label{V-modules}

Let $\cV$ be a cartesian closed category. 

A \emph{$\cV$-ring with identity} is an object $R\in\cV$, together with $\cV_0$-morphisms 
$$
a_R:R\times R\to R,\ \ m_R:R\times R\to R,\ \ z_R:I\to R,\ \ e_R:I\to R,
$$
satisfying natural commutative diagrams expressing the classical axioms relating addition, multiplication, zero and identity element. 

A $\cV$-abelian group $F$ is a \emph{(left) $R$-module,} if we have a $\cV_0$-morphism
$$
\mu:R\times F\to F,
$$
together with the natural commutative diagrams expressing the classical axioms for being a (left) module. Equivalently, passing through the motions of \ref{V-actions}, an $R$-module structure on $F$ can be defined through a $\cV$-ring homomorphism 
$$
R\to \uEnd_{\cV\da\Ab}(F).
$$
We also write
$$
(R,F)\in\cV\Mod.
$$
If $\cV$ is complete, we endow $\cV\Mod$ with the structure of a $\cV$-category as follows. Given $\cV$-modules $(R,F)$ and $(R',F')$, the internal module homomorphism object 
$$
\cV\Mod[(R,F),(R',F')]
$$
is the equaliser of the diagram 
\begin{center}
 \begin{tikzpicture} 
\matrix(m)[matrix of math nodes, row sep=2em, column sep=3em, text height=1.5ex, text depth=0.25ex]
 {
  		& |(2)|{[R\times F,R'\times F']}  & 	\\
 |(l1)|{{\cV\rng}[R,R']\times{\cV\da\Ab}[F,F']}	& |(l2)|{{\cV\da\Ab}[F,F']	} & |(l3)|{[R\times F,F']} 	\\
 }; 
\path[->,font=\scriptsize,>=to, thin]
(l1) edge node[above,pos=.4]{$\delta$} (2) edge (l2) 
(l2) edge node[below]{$\mu^{*}$}   (l3)
(2) edge node[above,pos=.7]{$\mu'_{*}$} (l3) 
;
\end{tikzpicture}
\end{center}
where $\delta$ is obtained by adjunction from the composite of
$$
R\times F\times\cV\rng[R,R']\times\cV\da\Ab[F,F']\simeq
R\times\cV\rng[R,R']\times F\times\cV\da\Ab[F,F']\xrightarrow{\mathop{\rm ev}\times\mathop{\rm ev}}R'\times F'.
$$

If $R$ is a $\cV$-ring, and $F, F'$ are left $R$-modules, we define the internal homomorphism object
$$
R\Mod[F,F']=\cV\da R\Mod[F,F']\in\cV
$$
as the pullback
\begin{center}
 \begin{tikzpicture} 
\matrix(m)[matrix of math nodes, row sep=2em, column sep=5em, text height=1.9ex, text depth=0.25ex]
 {
 |(1)|{R\Mod[F,F']}		& |(2)|{\cV\da\Ab[F,F']} 	\\
 |(l1)|{\cV\Mod[(R,F),(R,F')]}		& |(l2)|{\cV\rng[R,R]\times\cV\da\Ab[F,F']} 	\\
 }; 
\path[->,font=\scriptsize,>=to, thin]
(1) edge  (2) edge   (l1)
(2) edge  (l2) 
(l1) edge   (l2);
\end{tikzpicture}
\end{center}
where the right vertical arrow is the composite
$$
\cV\da\Ab[F,F']\simeq I\times\cV\da\Ab[F,F']\xrightarrow{u_R\times\id}\cV\rng[R,R]\times\cV\da\Ab[F,F'].
$$
Thus, we have defined the $\cV$-category
$$
R\Mod
$$
of $R$-modules.

Note that the additive structure on $F'$ induces a
the structure of a $\cV$-abelian group on the internal homomorphism object, i.e., that
$$
R\Mod[F,F']\in\cV\da\Ab.
$$

\subsection{Module objects in a category tensored over a cartesian closed category}\label{C-V-modules}

Let $\cC$ be a $\cV$-category with products, which is tensored over $\cV$. A $\cC\da\cV$-module is a module object in $\cC_0$. In the spirit of \ref{C-V-groups} and by analogy with \ref{V-modules}, given $\cC\da\cV$-modules $(R,F)$ and $(R',F')$, we can define the internal group homomorphism object 
$$
\cC\da\cV\Mod((R,F),(R',F'))\in\cV.
$$  
Hence we obtain a $\cV$-category
$$
\cC\da\cV\Mod
$$
of $\cC\da\cV$-modules. 

Given a $\cC\da\cV$-ring $R$, by analogy to \ref{V-modules}, we can also define the $\cV$-category
$$
\cC\da\cV\da R\Mod.
$$

\subsection{Group module objects}\label{group-modules}

Let $\cV$ be a cartesian closed category. Let $G$ be a $\cV$-group, $R$ a $\cV$-ring and $F$ an $R$-module. A structure of a $G\da R$-module on $F$ is a morphism of $\cV$-groups
$$
\rho:G\to \uAut_{R\Mod}(F).
$$ 
Let $F$ and $F'$ be $G\da R$-modules, and write $\mu:G\times F\to F$, $\mu':G\times F'\to F'$ for the associated actions via \ref{V-actions}. The internal $G\da R$-module homomorphism object
$$
G\da R\Mod[F,F']=\cV\da G\da R\Mod[F,F']\in\cV
$$
is the equaliser of the diagram
\begin{center}
 \begin{tikzpicture} 
\matrix(m)[matrix of math nodes, row sep=2em, column sep=3em, text height=1.5ex, text depth=0.25ex]
 {
  	&	& |(2)|{[G\times F,G\times F']}  & 	\\
 |(l1)|{R\Mod[F,F']} & |(l2)|{[F,F']}	& 	& |(l3)|{[G\times F,F']} 	\\
 }; 
\path[->,font=\scriptsize,>=to, thin]
(l1) edge (l2)
(l2) edge node[above,pos=.4]{$\delta$} (2) edge node[below]{$\mu^{*}$}   (l3)
(2) edge node[above,pos=.7]{$\mu'_{*}$} (l3) 
;
\end{tikzpicture}
\end{center}
where $\delta$ is obtained by adjunction from 
$$
G\times F\times[F,F']\xrightarrow{\id\times\mathop{\rm ev}} G\times F'.
$$


\subsection{Enriched group presheaves}

Let $\cV$ be a complete cartesian closed category, let $\cC$ be a small $\cV$-category, and let $\hC=[\cC^\circ,\cV]$ be the $\cV$-category of $\cV$-presheaves on $\cC$. By \ref{enrich-presh-duality}, $\hC$ is complete cartesian closed. 
Hence, \ref{V-groups} gives a meaning to the notion of $\hC$-groups.

More explicitly, a \emph{$\hC$-group} is a group object in $\hC_0$, i.e., a $\cV$-presheaf $\bG$ on $\cC$, together with $\cV$-natural transformations
$$
\bG\times\bG\to \bG, \ \ \  \bG\to\bG, \ \ \ \ \underline{e}\to \bG
$$
affording diagrams analogous to those from \ref{V-groups}.

Since $\hC$ is complete cartesian closed, the construction from \ref{V-groups} endows the above category with group homomorphism presheaves
$$
\hC\daGp[\bG,\bH]\in\hC,
$$
for two $\hC$-groups $\bG$, $\bH$.

Moreover, the class of $\hC$-groups is endowed with a $\cV$-category structure with internal homomorphism objects
$$
\hC\daGp(\bG,\bH)=\Gamma(\hC\daGp[\bG,\bH])\simeq\hC\da\cV\daGp(\bG,\bH),
$$
where the last object is obtained by the construction from \ref{C-V-groups} in view of the fact that  $\hC$ is (trivially) tensored over $\cV$.

Thus we have obtained the $\cV$-category of $\hC$-groups
$$
\hC\daGp.
$$ 
\begin{proposition}\label{V-gp-prop}
We have an equivalence of $\cV$-categories
$$
\hC\daGp\simeq[\cC^\circ,\cV\daGp].
$$
If $\cV$ is a topos, for $X\in\cC$ and $\bG,\bH\in \hC\daGp$, we have
$$
\hC\daGp[\bG,\bH](X)\simeq \hC_{\Ov X}\daGp(i_X^*\bG,i_X^*\bH).
$$ 
\end{proposition}
\begin{proof}
By definition,
\begin{multline*}
[\cC^\op,\cV\daGp](\bG,\bH) =\int_{X\in\cC}\cV\daGp[\bG X,\bH X]\\ 
\simeq
\int_{X\in\cC}
\begin{tikzcd}[cramped, sep=small, ampersand replacement=\&]
{\Eq\left([\bG X,\bH X]\right.}\ar[yshift=2pt]{r}{} \ar[yshift=-2pt]{r}[swap]{} \&{\left.[\bG X\times \bG X,\bH X]\right)}
\end{tikzcd}\\
\simeq
\begin{tikzcd}[cramped, sep=small, ampersand replacement=\&]
{\Eq\left(\int_X[\bG X,\bH X]\right.}\ar[yshift=2pt]{r}{} \ar[yshift=-2pt]{r}[swap]{} \&{\left.\int_X[\bG X\times \bG X,\bH X]\right)}
\end{tikzcd}\\
\simeq
\begin{tikzcd}[cramped, sep=small, ampersand replacement=\&]
{\Eq\left(\hC(\bG,\bH)\right.}\ar[yshift=2pt]{r}{} \ar[yshift=-2pt]{r}[swap]{} \&{\left.\hC(\bG\times\bG,\bH)\right)}
\end{tikzcd}
\simeq \hC\daGp(\bG,\bH).
\end{multline*}
For the second claim, using the fact that $[\h_X,\mathord{-}]$ is left-exact and \ref{psh-int-hom-rel},
\begin{multline*}
\hC\daGp[\bG,\bH](X)=[\h_X,\hC\daGp[\bG,\bH]]\simeq
[\h_X,\begin{tikzcd}[cramped, sep=small, ampersand replacement=\&]
{\Eq\left([\bG,\bH]\right.}\ar[yshift=2pt]{r}{} \ar[yshift=-2pt]{r}[swap]{} \&{\left.[\bG\times\bG,\bH]\right)}
\end{tikzcd}]\\
\simeq 
\begin{tikzcd}[cramped, sep=small, ampersand replacement=\&]
{\Eq\left([\bG,\bH](X)\right.}\ar[yshift=2pt]{r}{} \ar[yshift=-2pt]{r}[swap]{} \&{\left.[\bG\times\bG,\bH](X)\right)}
\end{tikzcd}\\
\simeq 
\begin{tikzcd}[cramped, sep=small, ampersand replacement=\&]
{\Eq\left(\hC_{\Ov X}(\bG_X,\bH_X)\right.}\ar[yshift=2pt]{r}{} \ar[yshift=-2pt]{r}[swap]{} \&{\left.\hC_{\Ov X}(\bG_X\times\bG_X,\bH_X)\right)}
\end{tikzcd}
\simeq \hC_{\Ov X}\daGp(\bG_X,\bH_X).
\end{multline*}
\end{proof}

Using the principle of \ref{int-isom}, we  define the objects
$$
\Isom_{\hC\daGp}(\bG,\bH), \Aut_{\hC\daGp}(\bG)\in\cV,
$$
as well as $\cV$-presheaves
$$
\uIsom_{\hC\daGp}(\bG,\bH), \uAut_{\hC\daGp}(\bG)\in\hC.
$$
 
We say that an object $G\in\cC$ is a $\cC$-group, if the associated $\cV$-presheaf $\h_G$ is a $\hC$-group. For $\cC$-groups $G$, $H$, we let
$$
\cC\da\cV\daGp(G,H)=\hC\da\cV\daGp(\h_G,\h_H)\in\cV.
$$
We have thus defined the $\cV$-category of $\cC$-groups
$$
\cC\daGp.
$$

\subsection{Enriched group presheaf actions}\label{enr-actions}

Let $\cV$ be a complete cartesian closed category, and let $\cC$ be a $\cV$-category.

Let $\bE\in \hC$ and $\bG\in\hC\daGp$. An \emph{action of $\bG$ on $\bE$} is a $\cV$-natural transformation between $\cV$-presheaves
$$
\mu:\bG\times\bE\to\bE,
$$
that fits the diagrams analogous to those from \ref{V-actions}.

In view of \ref{enrich-presh-duality}, this yields a morphism of $\hC$-monoids $\bG\to\uEnd(\bE)$, and, by the argument of \ref{V-actions} and using the fact that $\bG$ is a $\hC$-group, we see that this morphism lands in $\uAut(\bE)$. Thus, an action of $\bG$ on $\bE$ is equivalently given by a $\hC$-group homomorphism
$$
\rho:\bG\to\uAut(\bE).
$$

\subsection{Enriched module presheaves}

Let $\cV$ be a complete cartesian closed category, and let $\cC$ be a $\cV$-category.

Using the conventions from \ref{V-modules}, let $\bO$ be a $\hC$-ring and let $\bF$ be a $\hC$-abelian group. We say that $\bF$ is a \emph{$\bO$-module}, if the pair $(\bO,\bF)$ is a $\hC$-module, or, equivalently, if the pair defines a $\cV$-functor 
$$
(\bO,\bF):\cC^\circ\to \cV\Mod.
$$
Using the fact that $\hC$ is a complete cartesian closed $\cV$-category which is tensored over $\cV$ and the results of \ref{V-modules} and \ref{C-V-modules}, the category
$$
\hC\Mod
$$
is a $\cV$-category with internal homomorphism $\cV$-presheaves 
$$
\hC\Mod[(\bO,\bF),(\bO',\bF')]\in\hC.
$$
and with internal homomorphism objects
\begin{multline*}
\hC\Mod((\bO,\bF),(\bO',\bF'))=
\hC\Mod[(\bO,\bF),(\bO',\bF')](I)\\
\simeq\hC\da\cV\Mod((\bO,\bF),(\bO',\bF'))\in\cV,
\end{multline*}

Equivalently, it is the $\cV$-functor category
$$
\hC\Mod=[\cC^\circ,\cV\Mod].
$$

In particular, for a fixed $\hC$-ring $\bO$, the class 
$$\bO\Mod$$ 
of $\bO$-modules can be made into a $\cV$-category with internal homomorphism objects
$$
\bO\Mod(\bF,\bF')=\hC\da\bO\Mod(\bF,\bF')\in\cV\da\Ab,
$$
and with homomorphism $\cV$-presheaves
$$
\bO\Mod[\bF,\bF']\in\hC\da\Ab.
$$

\subsection{Enriched group module presheaves}\label{enr-gp-mod}

Let $\cV$ be a complete cartesian closed category, and let $\cC$ be a $\cV$-category.

Suppose $\bG$ is a $\hC$-group, $\bO$ is a $\hC$-ring, and $\bF$ is an $\bO$-module.
The structure of a $\bG\da\bO$-module on $\bF$ is given by a morphism of $\hC$-groups
$$
\rho:\bG\to \uAut_{\bO\Mod}(\bF).
$$
Using the fact that $\hC$ is a complete cartesian closed category, \ref{group-modules} yields the $\cV$-category 
$$
\bG\da\bO\Mod=\hC\da\bG\da\bO\Mod
$$
with internal group module homomorphism $\cV$-presheaves
$$
\bG\da\bO\Mod[\bF,\bF']=\hC\da\bG\da\bO\Mod[\bF,\bF']\in\hC\da\Ab
$$
and with internal group module homomorphism objects
$$
\bG\da\bO\Mod(\bF,\bF')=\bG\da\bO\Mod[\bF,\bF'](I)\in\cV\da\Ab.
$$

\subsection{Functorial algebraic geometry in a topos}

Let $\cE$ be a topos. Let $$\cE\rng$$ be the category of ring objects (commutative, with identity) in $\cE$, considered as an $\cE$-category with internal hom objects
$$
\cE\rng[R,R']\in\cE.
$$
Let $k\in\cE\rng$, and write
$$
k\Alg
$$
for the $\cE$-category of (commutative unitary) $k$-algebra objects in $\cE$. 

The category of affine schemes in $\cE$ is
$$
\cE\aff\simeq \cE\rng^\op,
$$
often identified (by $\cE$-Yoneda) with the subcategory of the category of $\cE$-functors $[\cE\rng,\cE]$ consisting of representable $\cE$-functors
$$
\h^A:\cE\rng\to \cE, \ \ \h^A(R)=\cE\rng[A,R],
$$
for $A\in \cE\rng$. If $\bX=\h^A$ is the affine scheme associated to the ring $A$, and $R\in\cE\rng$, then
$$
\bX(R)=\h^A(R)\in\cE
$$
is the object of \emph{$R$-rational points of $X$}.

The $\cE$-category of $\cE$-presheaves on $\cE\aff$ is a topos, and we have an enriched Yoneda embedding 
$$
\cE\aff\to \widehat{\cE\aff}=[\cE\aff^\op,\cE]\simeq [\cE\rng,\cE].
$$

Analogously, we define the $\cE$-category of $k$-affine schemes
$$
k\aff\simeq k\Alg^\op \to [k\Alg,\cE].
$$
For $A,R\in k\Alg$, if $\bX=\h^A$ is the affine $k$-scheme associated to $A$, then
$$
\bX[R]=k\Alg[A,R]
$$
is the object of \emph{$R$-rational points of the $k$-scheme $\bX$}. The \emph{set of $R$-rational points of $\bX$} is 
$$
\bX(R)=k\Alg(A,R)\simeq\Gamma(\bX[R]).
$$

\section{Topos cohomology}\label{top-coh}

\subsection{Ringed topoi}\label{ringed-top}

A \emph{ringed topos} is a pair $(\cE,A)$ consisting of a Grothendieck topos $\cE$ and an $\cE$-ring $A$. The category 
$$
\lsub{A}{\cE} =A\Mod 
$$  
of (unitary) left $A$-modules is an abelian category satisfying axioms (AB5) and (AB3) (the existence of small products), as shown in \cite[II~6.7]{sga4.1}, and therefore has enough injectives.

 For $M,N\in\lsub{A}{\cE}$, we write
$$
\lsub{A}{\cE}(M,N)=\Hom_A(M,N)
$$ 
for the abelian group of $A$-modules from $M$ to $N$. If $A$ is commutative, the group $\Hom_A(M,N)$ has a structure of a $\Gamma(A)$-module. 

Taking for $A$ the constant sheaf $\Z$, we obtain the category
$$
\lsub{\Z}{\cE}=\cE\da\Ab
$$
of abelian groups in $\cE$.

A morphism $\phi:A\to B$ of $\cE$-rings yields the \emph{restriction of scalars by $\phi$} functor
$$
\mathop{\rm Res}(\phi):\lsub{B}{\cE}\to \lsub{A}{\cE}.
$$

A \emph{morphism $u:(\cE,A)\to(\cE',A')$ of ringed topoi} is a pair $(m,\theta)$ consisting of a topos morphism $m:\cE\to\cE'$ and an $\cE$-ring morphism $\theta:m^*A'\to A$. Equivalently, a morphism can be given by $m$ as above and a $\cE$-ring morphism $\Theta:A'\to m_*A$.

A morphism $u$ of ringed topoi as above gives rise to the \emph{direct image functor for modules}
$$
u_*:\lsub{A}{\cE}\to\lsub{A'}{\cE'},
$$
as well as the \emph{inverse image functor for modules}
$$
u^*:\lsub{A'}{\cE'}\to \lsub{A}{\cE}
$$
as follows.
Given $M\in \lsub{A}{\cE}$, the object $m_*M$ is canonically an $m_*A$-module, so we define
$$
u_*(M)=\mathop{\rm Res}(\Theta)m_*A. 
$$
Conversely, for $N\in \lsub{A'}{\cE'}$, the object $m^*N$ is naturally a $m^*A'$-module, and we let
$$
u^*N=A\otimes_{m^*A'}m^*N,
$$
where $A$ is given the $m^*A'$-module structure via $\theta$.

While the `underlying sheaf of sets' of $u_*M$ is $m_*M$, the underlying set-valued sheaf of $u^*N$ is generally non-isomorphic to $m^*N$. Hence, we will write $u^{-1}$ for $m^*$ in the sequel to emphasise this distinction between the two operations.

These functors form an adjunction
$$
u^*\dashv u_*,
$$
%
whence we conclude that $u_*$ is left exact, and $u^*$ is right exact. While $u^{-1}$ is always exact, $u^*$ is exact when the canonical morphism $u^{-1}A\to A$ is flat.

In particular, if $X\in\cE$ and $i_X:\cE_{\ov X}\to E$ is the localisation morphism, then $i_X^*A=A|X$ is a ring in $\cE_{\ov X}$ and we obtain the morphism of ringed topoi
$$
(\cE_{\ov X},A|X)\to (\cE,A).
$$
In this case there is an \emph{extension by zero functor} ${i_X}_!:\lsub{A|X}{\cE_{\ov X}}\to\lsub{A}{\cE}$, which is exact, faithful, commutes with colimits, and
$$
{i_X}_!\dashv i_X^*\dashv {i_X}_*,
$$
so $i_X^*$ is exact.

The \emph{free $A$-module generated by $X\in\cE$} is the $A$-module
$$
A_X={i_X}_!(A|X).
$$
There is a canonical isomorphism 
$$
\cE(X,M)\simeq\lsub{A}{\cE}(A_X,M),
$$
functorial in $M\in\lsub{A}{\cE}$.

All the above results have analogues for the category
$$
\cE_A
$$
of right $A$-modules in $\cE$.

\subsection{Operations on modules in topoi}\label{mod-ops-top}

Let $(\cE,A)$ be a ringed topos. Let $M,N\in \lsub{A}{\cE}$. For an object $X\in\cE$, the object $[X,N]={i_X}_*i_X^* N$ is naturally an $A$-module, so we may consider the functor 
$$
\cE\to \cE\da\Ab, \ \ \
X\mapsto \Hom_A(M,[X,N]).
$$
It is representable by an object
$$
[M,N]_A=A\Mod[M,N]=\uHom_A(M,N)\in \cE\da\Ab,
$$
and we have a canonical isomorphism
$$
[M,N]_A(X)  \simeq \Hom_{A|X}(i_X^*M,i_X^*N).
$$

Moreover, for an $A|X$-module $P$, we obtain natural isomorphisms
\begin{align*}
 i_X^*[M,N]_A & 
 \simeq [i_X^*M,i_X^*N]_{A|X},\\
{i_X}_*[i_X^*M,P]_{A|X} & \simeq [M,{i_X}_*P]_A,\\
[{i_X}_!P,N]_A & 
\simeq {i_X}_*[P,i_X^*N]_{A|X}.
\end{align*}

If $A, B, C$ are rings in $\cE$, $M$ an $A\da B$-bimodule, $N$ an $A\da C$-bimodule, then $[M,N]_A$ has a structure of a $B\da C$-bimodule. If $A$ is commutative and $M$, $N$ are $A$-modules, then $[M,N]_A$ is naturally an $A$-module.

Let $M\in\cE_A$, $N\in\lsub{A}{\cE}$. The functor
$$
\cE\da\Ab\to\cE\da\Ab, \ \ \ P\mapsto \Hom_A(N,[M,P]_\Z)
$$
is representable by an object 
$$
M\otimes_AN\in\cE\da\Ab,
$$
i.e., 
$$
\Hom_\Z(M\otimes_AN,P)\simeq\Hom_A(N,[M,P]_\Z).
$$

For $X\in\cE$, we have a canonical isomorphism
$$
i_X^*(M\otimes_AN)\simeq i_X^*M\otimes_{A|X}i_X^*N.
$$
For $P\in \cE_{A|X}$ and $Q\in\lsub{A|X}{\cE}$, we have the projection formulae
\begin{align*}
{i_X}_!(P\otimes_{A|X}i_X^*N) & \simeq {i_X}_!P\otimes_AN,\\
{i_X}_!(i_X^*M\otimes_{A|X}Q) & \simeq M\otimes_A {i_X}_!Q.
\end{align*}
Writing $A_X$ for the free left $A$-module generated by $X$, there is a canonical isomorphism
$$
M\otimes_AA_X\simeq {i_X}_!i_X^*M.
$$

If $A,B,C$ are rings in $\cE$, $M$ a $B\da A$-bimodule and $N$ an $A\da C$-bimodule, then $M\otimes_AN$ is canonically a $B\da C$-bimodule. If $A$ is commutative and $M,N$ are $A$-modules, then $M\otimes_AN$ is an $A$-module.

\begin{proposition}
Let $A$ and $B$ be rings in a Grothendieck topos $\cE$, $M\in\cE_A$, $N$ an $A\da B$-bimodule and $P\in\cE_B$. There is a canonical isomorphism
$$
\Hom_B(M\otimes_AN,P)\simeq\Hom_A(M,[N,P]_B).
$$
\end{proposition}

\begin{corollary}\label{module-sh-mon-cl}
Let $(\cE,A)$ be a ringed topos with $A$ commutative. The category $\lsub{A}{\cE}$ of $A$-modules is monoidal closed.
\end{corollary}

Let $u:(\cE,A)\to(\cE',A')$ be a morphism of ringed topoi. We have an isomorphism
$$
u_*[u^*N,M]_A\simeq [N,u_*M]_{A'},
$$
natural in $M\in \lsub{A}{\cE}$ and $N\in\lsub{A'}{\cE'}$.

Moreover, there is an isomorphism
$$
u^*A'_X\simeq A_{u^{-1}(X)},
$$
functorial in $X\in\cE'$.

If $A'$ is commutative and the morphism $u^{-1}A'\to A$ is central, there is an isomorphism
$$
u^*M\otimes_A u^*N\simeq u^*(M\otimes_{A'}N),
$$
natural in $M\in\cE'_{A'}$ and $N\in\lsub{A'}{\cE'}$.

\subsection{Topos cohomology}\label{topos-coh}

Let $(\cE,A)$ be a ringed topos, and let $M,N\in \lsub{A}{\cE}$. Given \ref{ringed-top}, we may write
$$
\Ext^n_A(\cE; M,N)
$$
for the value at $N$ of the $n$-th right-derived functor of the functor $\Hom_A(M,\mathord{-})$. By \cite{groth-tohoku}, the functors $\Ext^n_A(\cE; M,N)$ form a $\delta$-functor in $N$, and a contravariant $\delta$-functor in $M$.

Let $X\in\cE$ and write $A_X$ for the free $A$-module generated by $X$. We define
$$
\tH^n(X,N)=\Ext^n_A(\cE; A_X,N).
$$
In other words, the functor $\tH^n(X,\mathord{-})$ is the $n$-th right-derived functor of the functor $\Gamma(X,\mathord{-})=\cE(X,\mathord{-})=\Hom_A(A_X,\mathord{-})$.

In particular, we write
$$
\tH^n(\cE,N)=\tH^n(e,N)=\Ext^n(\cE;A,N).
$$

If $i_X:\cE_{\ov X}\to\cE$ is a localisation morphism, we explained  in \ref{ringed-top} that the functor $i_X^*$ is exact and admits an exact left adjoint ${i_X}_!$, so we conclude that it takes injective modules to injective modules, and we obtain an isomorphism
$$
\Ext^n_{A|X}(\cE_{\ov X};M,i_X^*N)\simeq \Ext_A^n(\cE;{i_X}_!M,N),
$$ 
natural in $N\in\lsub{A}{\cE}$ and $M\in\lsub{A|X}{\cE_{\ov X}}$.

In particular, we obtain a canonical isomorphism
$$
\tH^n(\cE_{\ov X},i_X^*N)\simeq \tH^n(X,N).
$$
This motivates the definition
$$
\Ext^n_A(X;M,N)\simeq\Ext^n_{A|X}(\cE_{\ov X};i_X^*M, i_X^*N),
$$
for $M,N\in\lsub{A}{\cE}$. 

The functors $\Ext^n_A(X;M,\mathord{-})$ are the right-derived functors of the functor $N\mapsto\Hom_{A|X}(i_X^*M, i_X^*N)$, and the functors $(M,N)\mapsto\Ext^n_A(X;M,N)$ for $n\geq0$ form a $\delta$-functor in both variables.

In the special case when $A=Z$, the object of integers in $\cE$, we have that $\lsub{Z}{\cE}=\Ab(\cE)$ and we obtain definitions of
$$
\Ext^n(\cE,M,N) \text{ and } \tH^n(\cE,\N)
$$
as values at $N$ of the derived functors of $\Ab(\cE)(M,\mathord{-})$ and $\Gamma=\cE(1,\mathord{-})$, for $M,N\in\Ab(\cE)$. Equivalently,
$$
\tH^n(\cE,N)\simeq\Ext^n_{\Ab(\cE)}(\cE,Z,N).
$$
Given an object $X\in\cE$, 
$$
\tH^n(X,\mathord{-})
$$
is the $n$-th derived functor of $\Gamma(X,\mathord{-})=\cE(X,\mathord{-}):\Ab(\cE)\to\Ab$.

\subsection{Geometric morphisms and cohomology}\label{geom-m-coh}

Let $f:\cF\to\cE$ be a geometric morphism between Grothendieck topoi, let $A\in\Ab(\cE)$ and $B\in\Ab(\cF)$. 

For each $n$, there is a homomorphism
$$
f^*:\tH^n(\cE,A)\to H^n(\cF,f^*A),
$$
which is functorial in $f$ and natural in $A$.

Moreover, we have the \emph{Leray spectral sequence}
$$
\tH^p(\cE,R^qf_*(B))\Rightarrow \tH^{p+q}(\cF,B).
$$ 
It is obtained from the Grothedieck spectral sequence for the composite of direct image functors $\Gamma_\cF=\Gamma_\cE\circ f_*$.

If $\cE=\sh(\cC,J)$, we have an explicit description of $R^nf_*(B)$ as the $J$-sheaf associated to the presheaf 
$$
U\mapsto \tH^n(\cF,f^* l(U), B)
$$
on $\cC$, where $l:\cC\to \sh(\cC,J)$ is the canonical functor.

\subsection{Internal cohomology}\label{int-ext}

Let $(\cE,A)$ be a ringed topos, and let $M$ be an $A$-module. The functor
$$
N\mapsto [M,N]_A=\uHom_A(M,N)
$$
from $\lsub{A}{\cE}$ to $\cE\ab$ is left exact and we write
$$
\uExt_A^n(M,N)
$$
for its right-derived functors. 

By definition, we have
$$
\tH^0(X,\uExt_A^0(M,N))=\Ext^0_A(X;M,N)\simeq \Hom_{A|X}(M|X,N|X),
$$
which yields a spectral sequence
$$
E_2^{p,q}=\tH^p(X,\uExt_A^q(M,N))\Rightarrow \Ext_A^{p+q}(X;M,N),
$$
as well as a natural isomorphism
$$
\uExt_A^n(M,N)|X\simeq\uExt_{A|X}^n(M|X,N|X).
$$
Moreover, the sheaf $\uExt_A^n(M,N)$ is isomorphic to the sheaf associated to the presheaf
$$
X\mapsto \Ext_A^n(X;M,N).
$$
Using the fact that for each $X$, the functors $\Ext_A^n(X;M,N)$ form a $\delta$-functor in each variable, we obtain that the functors
$$
(M,N)\mapsto \uExt_A^n(M,N)
$$
form a $\delta$-functor in each variable.

\subsection{Group torsors}\label{group-torsors}

Let $G$ be an internal group in a topos $\cS$, and let $\bbG$ be the associated internal groupoid with object of objects 1, and let $p:\cE\to\cS$ be a geometric morphism.

An internal presheaf $T$ on $\bbG$ in $\cE$ is a right $G$-object, equipped with a unitary, associative action morphism
$$
\phi: T\times p^*G\to T.
$$
We say that $T$ is a \emph{$G$-torsor} if  it is a $\bbG$-torsor in the sense of \ref{tors-diacon}, i.e., if $T\to 1$ is epic, and 
$$
(\pi_1,\phi): T\times p^*G\to T\times T
$$
is an isomorphism.

The second condition is equivalent to the existence of a \emph{division morphism} 
$\delta:T\times T\to p^*G$ such that 
\begin{center}
 \begin{tikzpicture} 
\matrix(m)[matrix of math nodes, row sep=2em, column sep=2.5em, text height=1.5ex, text depth=0.25ex]
 {
 |(1)|{T\times T}		& |(2)|{p^*G\times T} 	\\
		& |(l2)|{T} 	\\
 }; 
\path[->,font=\scriptsize,>=to, thin]
(1) edge node[above]{$(\delta,\pi_2)$} (2) edge node[left]{$\pi_1$}   (l2)
(2) edge node[right]{$\phi$} (l2) 
;
\end{tikzpicture}
 \begin{tikzpicture} 
\matrix(m)[matrix of math nodes, row sep=2em, column sep=2.5em, text height=1.5ex, text depth=0.25ex]
 {
 |(1)|{p^*G\times T}		& |(2)|{T\times T} 	\\
		& |(l2)|{p^*G} 	\\
 }; 
\path[->,font=\scriptsize,>=to, thin]
(1) edge node[above]{$(\phi,\pi_2)$} (2) edge node[left]{$\pi_1$}   (l2)
(2) edge node[right]{$\delta$} (l2) 
;
\end{tikzpicture}
\end{center}
commute.

Equivalently, $T$ is a $G$-torsor if it is locally isomorphic to $p^*G$, i.e., there exists an epimorphism $U\to 1$ in $\cE$ and an isomorphism
$$
U^*(T)\simeq U^*(p^*G)
$$
of right $G$-objects in $\cE_{\ov U}$.

A $G$-torsor $T$ in $\cS$ is isomorphic to the \emph{trivial torsor}
$$
G\times G\xrightarrow{m} G
$$
if and only if it has a global element.
 
The category 
$$
\Tors(G,\cS)
$$
is a groupoid under the symmetric monoidal structure given by $\otimes_\bbG$, the unit being the trivial torsor. Its connected components form an abelian group
$$
\Tors^1(G,\cS),
$$
whose elements clearly correspond  to the isomorphism classes of torsors.

\begin{theorem}[{\cite[8.33]{johnstone}}]\label{H1-class-tors}
Let $G$ be an abelian group  in a Grothendieck topos $\cS$. Then
$$
H^1(\cS;G)\simeq \Tors^1(\cS,G).
$$
\end{theorem}

\section{Enriched homological algebra}\label{enr-homol-alg}

\subsection{Enriched abelian categories}\label{enr-ab-cat}

An \emph{abelian monoidal} category is an abelian category $\cA$ equipped with a symmetric monoidal structure where $\otimes$ is an additive bifunctor. An \emph{abelian monoidal closed} category is an abelian monoidal category which is also monoidal closed.

\begin{lemma}\label{ab-mon-cl}
Let $\cA$ be an abelian monoidal closed category, and $A\in\cA$. The functors
$$
A\otimes\mathord{-}:\cA\to \cA, \ \ \ \ \mathord{-}\otimes A:\cA\to \cA
$$
are right-exact, and the functors
$$
[A,\mathord{-}]:\cA\to\cA, \ \ \ \ \ [\mathord{-},A]:\cA^\circ\to\cA
$$
are left-exact.
\end{lemma}
\begin{proof}
By \ref{enr-limits}, the functors $A\otimes\mathord{-}$ and $\mathord{-}\otimes A$ preserve colimits, and therefore they preserve cokernels, so they are right-exact. Moreover, the functor $[A,\mathord{-}]$ preserves limits and thus kernels, so it is left-exact. Similarly, the contravariant functor $[\mathord{-},A]$ transforms colimits into limits and thus cokernels into kernels, so it is also left-exact.
\end{proof}

More generally, let $\cV$ be a cartesian closed category such that $\cV\da\Ab$ is abelian monoidal closed. A \emph{$\cV$-additive} category is a $\Vab$-category. 
A \emph{$\cV$-abelian} category is a tensored and cotensored $\cV\da\Ab$-category $\cA$ whose underlying category $\cA_0$ is abelian.

\begin{lemma}\label{lr-exactness}
Let $\cA$ be a $\cV$-abelian category, let $E\in\Vab$, and $A\in\cA$. The functors
$$
E\otimes\mathord{-}:\cA_0\to \cA_0, \ \ \ \ \mathord{-}\otimes A:\Vab\to \cA_0
$$
are right-exact, and the functors
\begin{align*}
& [E,\mathord{-}]:\cA_0\to\cA_0,   & [\mathord{-},A]:\Vab^\circ\to\cA_0, \\
& \cA(A,\mathord{-}):\cA_0\to\Vab,  & \cA(\mathord{-},A):\cA^\circ_0:\Vab
\end{align*}
are left-exact.
\end{lemma}
\begin{proof}
As in \ref{ab-mon-cl}, we use \ref{tens-pres-lim}, \ref{cotens-pres-lim}, \ref{tens-cotens-lim}.
\end{proof}

We proceed to define $\Vab$-functors 
$$
\ker, \coker: \cA^{\to}\to\cA^{\to}.
$$
On arrows $f\in \cA_0(A,B)$, we let
$$
\ker(A\xrightarrow{f}B)=(\Ker(f)\xrightarrow{\ker(f)}A),
$$
and
$$
\coker(A\xrightarrow{f}B)=(A\xrightarrow{\coker(f)}\Coker(f)).
$$
For $f:A\to B$ and $f':A'\to B'$, we define the $\Vab$-morphism
$$
\ker_{f,f'}:\cA^{\to}(f,f')\to\cA^{\to}(\ker{f},\ker{f'})
$$
as follows. Using \ref{lr-exactness}, we see that
$$
\cA(\Ker(f),\Ker(f'))=\Ker(f'_*:\cA(\Ker(f),A')\to\cA(\Ker(f),B')),
$$
so it suffices to show that the composite
$$
\cA^{\to}(f,f')\to \cA(\Ker(f),A')\xrightarrow{f'_*}\cA(\Ker(f),B'))
$$
is zero. From the commutativity of the diagram
 \begin{center}
 \begin{tikzpicture} 
\matrix(m)[matrix of math nodes, row sep=2em, column sep=3em, text height=1.5ex, text depth=0.25ex]
 {
 |(1)|{\cA^{\to}(f,f')}		& |(2)|{\cA(B,B')} & |(3)|{\cA(B,\Coker(f'))}	\\
 |(l1)|{\cA(A,A')}		& |(l2)|{\cA(A,B')} 	& |(l3)|{\cA(A,\Coker(f'))}\\
 |(d1)|{\cA(\Ker(f),A')} & |(d2)|{\cA(\Ker(f),B')} &\\
 }; 
\path[->,font=\scriptsize,>=to, thin]
(1) edge node[above]{} (2) edge node[left]{}   (l1)
(2) edge node[right]{$f^*$} (l2) 
(l1) edge  node[above]{$f'_*$} (l2)
(2) edge node[above]{$\coker(f')_*$} (3) 
(l2) edge node[above]{$\coker(f')_*$} (l3) 
(3) edge node[right]{$f^*$} (l3) 
(l2) edge node[right]{$\ker(f)^*$} (d2) 
(l1) edge node[left]{$\ker(f)^*$} (d1) 
(d1) edge  node[above]{$f'_*$} (d2)
;
\end{tikzpicture}
\end{center}
this is the same as the composite
$$
\cA^{\to}(f,f')\to\cA(B,B')\xrightarrow{\ker(f)^*\circ f^*}[\Ker(f),B'],
$$
which is zero since $\ker(f)^*\circ f^*=(f\circ \ker(f))^*=0$.

To define the $\cV$-morphism
$$
\coker_{f,f'}:\cA^{\to}(f,f')\to\cA^{\to}(\coker(f),\coker(f')),
$$
we observe, using \ref{lr-exactness}, that 
$$
\cA(\Coker(f),\Coker(f'))=\Ker(f^*:\cA(B,\Coker(f')),\cA(A,\Coker(f'))),
$$
so it suffices to show that the composite 
$$
\cA^{\to}(f,f')\to\cA(B,\Coker(f'))\xrightarrow{f^*}\cA(A,\Coker(f'))
$$
is zero, but the above diagram shows that this equals
$$
\cA^{\to}(f,f')\to\cA(A,A')\xrightarrow{\coker(f')_*\circ f'_*}\cA(A,\Coker(f')),
$$
which is zero since $\coker(f')_*\circ f'_*=(\coker(f')\circ f')_*=0$.

We also define 
$$
\coim=\coker\circ\ker, \ \ \ \ \im(f)=\ker\circ\coker.
$$
as $\Vab$-functors $\cA^{\to}\to\cA^{\to}$.

Since $\cA_0$ is assumed abelian, the natural morphism
$$
\hat{f}:\Coim(f)\to\Im(f)
$$
is an isomorphism, which yields the $\cA_0$-diagram
 \begin{center}
 \begin{tikzpicture} 
\matrix(m)[matrix of math nodes, row sep=.5em, column sep=.7em, text height=1.5ex, text depth=0.25ex]
 {
  & |(a)|{A} & & &    |(b)|{B} & 	\\
 |(k)|{\Ker f}& & |(ci)|{\Coim f} & |(i)|{\Im f} & & |(ck)|{\Coker f}	\\
 }; 
\path[->,font=\scriptsize,>=to, thin]
(a) edge node[above]{$f$} (b) edge [style=->>]  (ci)
(k) edge[style=>->]  (a) 
(i) edge[style=>->]  (b) 
(b) edge[style=->>]  (ck) 
(ci) edge  node[above]{$\hat{f}$} (i);
\end{tikzpicture}
\end{center}
called the \emph{analysis} of $f$. Note that the analysis can also be viewed as a $\Vab$-functor on $\cA^{\to}$.

\subsection{Complexes in enriched abelian categories}

Let us consider $\Z$ as a diagram category with arrows $d_n$ with source $n$ and target $n+1$. Let $\Sigma$ be the diagram scheme based on $\Z$ with commutativity relations $d_{n+1}d_n=0$. We can trivially consider $\Z$ as a $\cV\da\Ab$-category, and consider the 
$\Vab$-functor category
$$
\cA^\Z=[\Z,\cA].
$$
We also consider the full $\Vab$-subcategory 
$$
\ch(\cA)=\cA^\Sigma
$$
of \emph{complexes} with values in $\cA$. More precisely, the objects of $\cA^\Z$ are diagrams in $\cA_0$
$$
\cdots\xrightarrow{d^{n-1}} A^{n}\xrightarrow{d^n} A^{n+1} \xrightarrow{d^{n+1}}\cdots
$$
and internal hom objects are
$$
\cA^\Z(A,B)=\Eq(d_A^{*},d_{B,*}),
$$
where 
$$
d_A^{*}=\prod_i d_A^{i,*}:\prod_i\cA(A^i,B^{i})\to\prod_i\cA(A^i, B^{i+1}),
$$
and
$$
d_{B,*}=\prod_i d_{B,*}^{i}:\prod_i\cA(A^i,B^{i})\to\prod_i\cA(A^i, B^{i+1}).
$$
Objects of $\ch(\cA)$ are diagrams $A\in\cA^\Z$ with the property
$$
d_A^{n+1}\circ d_A^n=0, 
$$
and internal homs are
$$
\ch(\cA)(A,B)=\cA^\Z(A,B).
$$
Sequences of chain morphisms in $\ch(\cA)_0$ are defined to be exact if they are exact in each degree. We extend the $\Vab$-functors $\ker$, $\coker$, $\im$, $\coim$ in a natural way to $\Vab$-functors $\ch(\cA)^{\to}\to\ch(\cA)^{\to}$. With this structure, $\ch(\cA)$ is also a $\cV$-abelian category. 

The $\Vab$-category of \emph{short exact sequences} in $\cA$
$$
\Ex(\cA)
$$
is a full $\Vab$-subcategory of $\ch(\cA)$.

The \emph{mapping cone} functor is the $\Vab$-functor
$$
\cone:\ch(\cA)^{\to}\to \ch(\cA)
$$
which, for a chain morphism $f:A\to B$, returns the complex
$$
\cone(f)^n=A^{n+1}\oplus B^n,
$$ 
with differential
$$
d_f^n=\begin{bmatrix}-d_A^{n+1} & 0 \\f^{n+1} & d_B^n\end{bmatrix}.
$$
We obtain the \emph{translation} $\Vab$-functor
$$
\Sigma:\ch(\cA)\to\ch(\cA), \ \ \ \ \Sigma A=\cone(A\to 0).$$ 
More explicitly, 
$$
(\Sigma A)^n=A^{n+1}, \ \ \ \ d_{\Sigma A}^n=-d_A^{n+1}.
$$

The \emph{homotopy category} $\K(\cA)$ is the $\Vab$-category with the same objects as $\ch(\cA)$ and internal homs
$$
\K(\cA)(A,B)=\Coker(\chi:\prod_n\cA(A^n,B^{n-1})\to\ch(\cA)(A,B)),
$$
where $\chi$ is obtained from the composite 
\begin{multline*}
\prod_n\cA(A^n,B^{n-1})\xrightarrow{\Delta}\prod_n\cA(A^n,B^{n-1})\times\cA(A^{n+1},B^n)\\\xrightarrow{\prod_n d_{B,*}^{n-1}\times d_A^{n,*}}\prod\cA(A^n,B^n)\times\cA(A^n,B^n)\xrightarrow{\prod_n+_n}\prod_n\cA(A^n,B^n),
\end{multline*}
where $+_n$ denotes the addition in the $\cV$-abelian group $\cA(A^n,B^n)$.

The \emph{cohomology} $\Vab$-functor
$$
\coh:\ch(\cA)\to\ch(\cA)
$$ is defined as follows. For $A\in\ch(\cA)$, consider the complexes $K_A=\Ker(d_A)$ and $I_A=\Im(d_A)$. Since $A$ is a chain complex, we obtain chain morphisms
$$
K_A\to A\to I_A\to \Sigma K_A,
$$
so we can define 
$$
\coh(A)=\Coker(I_A\to \Sigma K_A),
$$
i.e., $\coh^n(A)=\Coker(\Im(d_A^{n-1})\to\Ker(d_A^n))$.

The cohomology $\Vab$-functor factors through the homotopy category, i.e., there is a dashed arrow making the diagram
 \begin{center}
 \begin{tikzpicture} 
\matrix(m)[matrix of math nodes, row sep=.7em, column sep=.7em, text height=1.5ex, text depth=0.25ex]
 {
  |(a)|{\ch(\cA)} & & |(b)|{\ch(\cA)}  	\\
 & |(k)|{\K(\cA)} &  	\\
 }; 
\path[->,font=\scriptsize,>=to, thin]
(a) edge node[above]{$\coh$} (b) edge (k)
(k) edge[style=dashed] (b) 
;
\end{tikzpicture}
\end{center}
commutative.

\subsection{Enriched injectives and projectives}

Let $\cA$ be a $\cV$-abelian category. 

\begin{definition}\label{enr-injectives}
An object $I\in\cA$ is called
\begin{enumerate}
\item\emph{$\cA$-injective,} if the functor
$$
\cA(\mathord{-},I):\cA_0^\op\to\Vab
$$
is exact;
\item \emph{$\cA_0$-injective,} if the functor 
$$
\cA_0(\mathord{-},I):\cA_0^\op\to\Ab
$$
is exact;
\item \emph{enriched injective,} if it is both $\cA$-injective and $\cA_0$-injective.  
\end{enumerate}
\end{definition}
\begin{definition}\label{enr-projectives}
An object $P\in\cA$ is called
\begin{enumerate}
\item\emph{$\cA$-projective,} if the functor
$$
\cA(P,\mathord{-}):\cA_0\to\Vab
$$
is exact; 
\item \emph{$\cA_0$-projective,} if the functor
$$
\cA_0(P,\mathord{-}):\cA_0\to\Ab
$$
is exact;
\item \emph{enriched projective,} if it is both $\cA$-projective and $\cA_0$-projective.
\end{enumerate}
\end{definition}

We say that $\cA$ has \emph{enough $\cA$-injectives}, if every object $A\in\cA$ admits a monomorphism $A\to I$ into a $\cA$-injective object $I$. We make analogous definitions for $\cA_0$-injectives and enriched injectives.

Dually, $\cA$ has \emph{enough $\cA$-projectives} if every object $A\in\cA$ admits an epimorphism $P\to A$ from a $\cA$-projective object $P$. We express that $\cA$ has enough $\cA_0$-projectives or enriched projectives analogously.

\begin{proposition}\cite[]{harting1983}\label{inj-is-int-inj}
Let $\cE$ be a topos with a natural number object (such as a Grothendieck topos). If $I\in \cE\da\Ab$ is injective, then $I$ is \emph{internally injective}, i.e., the functor $$[\mathord{-},I]:\cE\da\Ab^\op\to \cE\da\Ab$$ is exact. 
\end{proposition}

\begin{lemma}\label{V-inj-res}
If $\cA$ has enough $\cA$-injectives, then every object $A\in\cA$ has a \emph{$\cA$-injective resolution,} i.e., there is an exact sequence
$$
0\to A\to I^0\to I^1\to\cdots
$$
where $I^n$ are $\cA$-injective. We have the same statement for $\cA_0$-injective resolutions and enriched injective resolutions.
\end{lemma}
\begin{proof}
Since $\cA$ has enough $\cV$-injectives, we can find a monomorphism $A\to I^0$ with $I^0$ $\cV$-injective. Let $A^0=\Coker(A\to I^0)$, and let us find a monomorphism $A^0\to I^1$ with $I^1$ $\cV$-injective. We take $A^1=\Coker(A^1\to I^1)$ and proceed inductively. By construction, the sequence $A\to I^0\to I^1\to\cdots$ is exact.
\end{proof}

By duality, we formulate the notion of \emph{$\cA$-projective resolutions} and show that a $\cV$-abelian category with enough $\cA$-projectives has $\cA$-projective resolutions of all objects, and similarly for $\cA_0$-projective resolutions and enriched projective resolutions.

\subsection{Derived enriched functors}

Let $\cA$ be a $\cV$-abelian category.
\begin{proposition}\label{comparison-th}
Let $A,B\in\cA$, and suppose that $A\to I$, $B\to J$ are $\cA$-injective resolutions. There exists a $\Vab$-monomorphism
$$
\cA(A,B)\to \K(\cA)(A\to I,B\to J)\simeq \K(\cA)(I,J).
$$
\end{proposition} 

\begin{lemma}\label{horseshoe-lemma}
Suppose we have a solid horseshoe diagram
\begin{center}
 \begin{tikzpicture} 
\matrix(m)[matrix of math nodes, row sep=1.7em, column sep=1.7em, text height=1.5ex, text depth=0.25ex]
 {
  |(ad)|{A'} & |(i0d)|{I_0'} & |(i1d)|{I_1'} &  |(i3d)|{\cdots}	\\
 |(a)|{A} & |(i0)|{I_0} & |(i1)|{I_1} &  |(i3)|{\cdots}	\\
  |(add)|{A''} & |(i0dd)|{I_0''} & |(i1dd)|{I_1''} &  |(i3dd)|{\cdots}	\\
 }; 
\path[->,font=\scriptsize,>=to, thin]
(ad) edge[style=>->] (i0d) edge[style=>->]  (a) 
(a) edge[style=>->,dashed] (i0) edge[style=->>]  (add) 
(add) edge[style=>->]  (i0dd) 
(i0d) edge (i1d) edge[style=>->, dashed] (i0)
(i1d) edge (i3d) edge[style=>->, dashed] (i1)
(i0) edge[dashed] (i1) edge[style=->>, dashed] (i0dd)
(i1) edge[dashed] (i3) edge[style=->>, dashed] (i1dd)
(i0dd) edge (i1dd)
(i1dd) edge (i3dd)
;
\end{tikzpicture}
\end{center}
where the first column is short exact, the rows are enriched injective resolutions of $A'$ and $A''$, and $I_n=I_n'\times I_n''$. There exist dashed arrows that fill the horseshoe to produce a diagram with exact rows and columns, thus giving an enriched injective resolution of $A$.
\end{lemma}
\begin{proof}
The classical horseshoe lemma in $\cA_0$ produces the dashed arrows, and we merely observe that $A\to I$ is an enriched injective resolution since the product of enriched injective objects is an enriched injective object.
\end{proof}

\begin{definition}\label{enr-derived-fun}
Let $F:\cA\to\cB$ be $\Vab$-functor between two $\cV$-abelian categories, where $\cA$ has enough $\cA$-injectives. We naturally extend $F$ to a $\Vab$-functor $\ch(\cA)\to\ch(\cB)$ by applying it in each degree, and also to a $\Vab$-functor $\K(\cA)\to\K(\cB)$.

We define the $\Vab$-functor 
$$\der F:\cA\to \ch(\cB)$$
as follows. 
For an object $A\in\cA$, we take an $\cA$-injective resolution $A\to I$, and let 
$$
\der F(A)=\coh(F(I)).
$$
Note that the definition does not depend on the choice of $I$  using \ref{comparison-th}.

For $A,B\in\cA$, we construct the $\Vab$-morphism
$$
\cA(A,B)\to \ch(\cA)(\der F(A),\der F(B))
$$ 
as the composite 
$$
\cA(A,B)\to \K(I,J)\to\K(F(I),F(J))\to \ch(\cA)(\coh F(I),\coh F(J))=\ch(\cA)(\der F(A),\der F(B)).
$$

Thus we have simultaneously defined all the $n$-th \emph{derived $\Vab$-functors} of $F$,
$$
\der^nF(A)=\coh^n(F(I)).
$$

\end{definition}

\begin{definition}
An \emph{enriched $\delta$-functor} from a $\cV$-abelian category $\cA$ to a $\Vab$-category $\cB$ is a $\Vab$-functor
$$
T:\Ex(\cA)\to\ch(\cB)
$$
such that for every $0\to A\to B\to C\to 0\in \Ex(\cA)$, $T$ gives a chain 
$$
\cdots\to T^{n-1}(C)\xrightarrow{\delta} T^n(A)\to T^n(B)\to T^n(C)\xrightarrow{\delta} T^{n+1}(A)\to\cdots
$$
A $\delta$-functor is \emph{exact}, if it lands in the exact chains of $\cB$, i.e., if the above chain is exact for every choice of the short exact sequence.

A \emph{morphism of enriched $\delta$-functors} $T, T'$ is a $\Vab$-natural transformation $T\to T'$. We also have the internal hom object
$$
[\Ex(\cA),\ch(\cB)](T,T')\in\Vab.
$$
\end{definition}

In particular,   an enriched $\delta$-functor $T$ between $\cA$ and $\cB$ gives rise to a sequence of $\Vab$-functors 
$$
T^n:\cA\to\cB,
$$
and a morphism $T\to T'$ of enriched $\delta$-functors yields $\Vab$-natural transformations
$$
T^n\to {T'}^n.
$$
An enriched $\delta$-functor $T$ is \emph{universal,} if, for any other enriched $\delta$-functor $T'$, and any $\Vab$-natural transformation $f^0: T^0\to {T'}^0$, there exists a morphism $f:T\to T'$ of enriched $\delta$-functors that extends $f^0$. 

\begin{theorem}
Let $\cA$ be a $\cV$-abelian category with enough enriched injectives, and let $F:\cA\to \cB$ be a left-exact functor to a $\Vab$-category $\cB$. The functor $\der F:\cA\to\ch(\cB)$ 
extends to the universal (cohomological) enriched $\delta$-functor
$$
\der F:\Ex(\cA)\to\ch(\cB).
$$
\end{theorem}
\begin{proof}
Since $\cA$ has enriched injectives, we can compute $\der F:\cA\to\ch(\cB)$ using enriched injective resolutions. Since an enriched injective resolution is an $\cA$-injective resolution, this coincides with the calculation in \ref{enr-derived-fun}. 

On the other hand, since an enriched injective resolution is also $\cA_0$-injective, we see that, on objects, the calculation of $\der F$ agrees with the classical calculation of $\der F_0$ in $\cA_0$, and the latter is known to be a universal cohomological $\delta$-functor. 

Key results used in the classical construction of the connecting morphisms $\delta$ are the horseshoe lemma and the snake lemma. The enriched horseshoe lemma \ref{horseshoe-lemma} shows that the classical steps can be followed and we can simultaneously define the enriched functors on internal hom objects.
\end{proof}

\subsection{Relative and enriched topos cohomology}\label{rel-top-coh}

Let $u:(\cE,A)\to (\cS,R)$ be a bounded morphism of ringed (Grothendieck) topoi. As discussed in \ref{enrich-via-geom}, $\cE$ is $\cS$-enriched with internal  hom objects
$$
\cE_{\ov\cS}(X,Y)=u_*[X,Y]_\cE\in \cS.
$$
If $R$ is commutative, then \ref{module-sh-mon-cl} shows that $R\Mod=\lsub{R}{\cS}$ is monoidal closed, and $A\Mod=\lsub{A}{\cE}$ is $R\Mod$-enriched with internal hom objects
$$
A\Mod(M,N)=u_*[M,N]_A\in \lsub{R}{\cS}.
$$
Hence, 
in terminology of \ref{enr-ab-cat}, $A\Mod$ is an $\cS$-abelian category.

The isomorphism $u_*[u^*N,M]_A\simeq[N,u_*M]_R$ shows that the adjunction $u^*\dashv u_*$ is enriched, i.e.,  we have isomorphisms
$$
A\Mod(u^*F,M)\simeq R\Mod(F,u_*M),
$$
natural in $M\in\lsub{A}{\cE}$ and $F\in\lsub{R}{\cS}$.

\begin{remark}\label{exact-enr-inj}
If $u_*$ preserves epimorphisms, then $A\Mod$ has enough enriched injectives. Indeed, an injective object $I$ is also $A\Mod$-injective. By \ref{inj-is-int-inj}, $I$ is internally injective, so $[\mathord{-},I]$ takes monomorphisms to epimorphism, and hence $A\Mod(\mathord{-},I)=u_*\circ[\mathord{-},I]$ also takes monomorphisms to epimorphisms.
\end{remark}

Suppose $A\Mod$ has enough enriched injectives. For $M,N\in A\Mod$, we define 
$$
\Ext^n_{A\Mod}(M,N)=\Ext^n_A(\cE_{\ov\cS}; M,N)=\Ext^n_A(u; M,N)
$$
as the value at $N$ of the $n$-th right-derived $R\Mod$-functor of  the $R\Mod$-functor $A\Mod(M,\mathord{-}):A\Mod\to R\Mod$. The functors $\Ext^n_{A\Mod}(M,N)$ form an enriched $\delta$-functor in each variable.

As before, for $X\in\cE$, we let
$$
\tH^n(\cE_{\ov\cS}; X,N)=\tH^n(u; X,N)=\Ext^n_A(u; A_X,N),
$$
which is the value at $N$ of the $n$-th derived enriched functor of the $R\Mod$-functor
$\Gamma_u(X,\mathord{-})=\cE_{\ov\cS}(X,\mathord{-})=A\Mod(A_X,\mathord{-})$.

In particular, 
$$
\tH^n(\cE_{\ov\cS},N)=\tH^n(u,N)=\tH^n(u;e,N)=\Ext^n_A(u;A,N).
$$

If $A\Mod$ does not have enough enriched injectives, the above definitions still make sense for the underlying ordinary functors. We allow ourselves to use the same notation, given that enriched derived functors agree with corresponding ordinary derived functors  on objects. 

Suppose that $u$ fits in a commutative diagram
$$
\begin{tikzpicture} 
\matrix(m)[matrix of math nodes, row sep=2em, column sep=2em, text height=1.5ex, text depth=0.25ex]
 {
 |(1)|{(\cE,A)}		& |(2)|{(\cS,R)} 	\\
		& |(l2)|{(\Set,R_0)} 	\\
 }; 
\path[->,font=\scriptsize,>=to, thin]
(1) edge node[above]{$u$} (2) edge node[left]{$w$}   (l2)
(2) edge node[right]{$v$} (l2) 
;
\end{tikzpicture}
$$
where $v$ and $w$ are canonical geometric morphisms of Grothendieck topoi $\cS$ and $\cE$ to $\Set$. Then, for $M\in\lsub{A}{\cE}$, 
$$
\lsub{A}{\cE}(M,\mathord{-})=v_*\circ A\Mod(M,\mathord{-}),
$$
so Grothendieck's spectral sequence for the composite of two functors yields a spectral sequence
$$
E_2^{p,q}(N)=\der^p v_*(\Ext_A^q(u; M,N))\Rightarrow \Ext_A^{p+q}(\cE;M,N),
$$
which we can also write as $\tH^p(\cS,\Ext_{A\Mod}^q(M,N))\Rightarrow \Ext_A^{p+q}(M,N)$.

Moreover, by the definition of internal homs in $A\Mod$,
$$
\Ext^0_{A\Mod}(M,\mathord{-})=u_*\circ \uExt^0_A(M,\mathord{-}),
$$
so we obtain a spectral sequence
$$
E_2^{p,q}=\der^p u_*(\uExt^q_A(M,N))\Rightarrow\Ext^{p+q}_{A\Mod}(M,N).
$$
Combining these two spectral sequences gives the familiar spectral sequence relating $\uExt_A$ and $\Ext_A$ from \ref{int-ext}.

\subsection{Enriched and internal cohomology}

Let $(\cE,A)$ be a ringed topos with $A$ commutative. Then $A\Mod$ is monoidal closed and we consider it as enriched over itself. By \ref{inj-is-int-inj}, $A\Mod$ has enough internal injectives, i.e., enough enriched injectives for the self-enrichment, so the functors 
$$
\uExt^n_A(M,N)
$$
from \ref{int-ext} in fact form an enriched $\delta$-functor in each variable.

\begin{remark}
Let $u:(\cE,A)\to (\cS,R)$ be a bounded morphism of ringed topoi with $A$ and $R$ commutative, making $A\Mod$ into a $R\Mod$-category. If $u_*$ is exact, then $A\Mod$ has enough enriched injectives by \ref{exact-enr-inj} and the functors $\Ext_{A\Mod}^n$ form an enriched $\delta$-functor in each variable, and we know that the same holds of $\uExt_A^n$. On the other hand, the spectral sequence relating the two degenerates to the statement $\Ext_{A\Mod}^n(M,N)=u_*(\uExt_A^n(M,N))$, so in this case the enriched structure of $\Ext_{A\Mod}$ is obtained by base change via the cartesian functor $u_*$ from that of $\uExt_A$. 
\end{remark}

\section{Algebraic geometry over a base topos}\label{rel-AG}

\subsection{Topologies in algebraic geometry}\label{ag-sites}

Let $\Sch$ denote the category of affine schemes of finite type over $\spec(\Z)$, and let $\tau$ be a Grothendieck topology on $\Sch$. For an affine scheme $S$, let 
$$
(\Sch_{\ov S},\tau)
$$
denote the big $\tau$-site on $S$, and let
$$
(S_\tau,\tau)
$$
denote the relevant small $\tau$-site. 

In the sequel we will most commonly use the Zariski, \'etale and fppf topologies.

\subsection{Hakim's spectra}\label{external-spectra}

Let $(\cC,J)$ be a standard site for a Grothendieck topos $\cE$, let $A$ be a ring in $\cE$, and let $\tau$ be a topology on schemes as in \ref{ag-sites}. 

Hakim defines the \emph{$\tau$-spectrum of $(\cE,A)$} as the ringed topos
$$
\spec.\tau(\cE,A)=(\sh(\cC_\tau,J_\tau),\cO_\tau)\xrightarrow{\pi_\tau}(\cE,A),
$$
where
\begin{enumerate}
\item the category $\cC_\tau$ has objects $(U,P)$, with $U\in\cC$ and $P\to \spec(A(U))$ is a scheme morphism in $\spec(A(U))_\tau$, and arrows $(V,Q)\xrightarrow{(u,r)}(U,P)$ consist of $V\xrightarrow{u}U\in\cC$ and a scheme morphism $Q\xrightarrow{r} P$ such that the diagram
 $$
 \begin{tikzpicture} 
\matrix(m)[matrix of math nodes, row sep=3em, column sep=2em, text height=1.9ex, text depth=0.25ex]
 {
 |(1)|{Q}		& |(2)|{P} 	\\
 |(l1)|{\spec(A(V))}		& |(l2)|{\spec(A(U))} 	\\
 }; 
\path[->,font=\scriptsize,>=to, thin]
(1) edge node[above]{$r$} (2) 
(1) edge   (l1)
(2) edge  (l2) 
(l1) edge  node[above]{$\bar{u}$} (l2)
;
\end{tikzpicture}
$$
commutes, where $\bar{u}$ is associated to $A(u):A(U)\to A(V)$;

\item the coverage $J_\tau$ is generated by families
$$
\{(U_\lambda,\spec(A(U_\lambda))\to (U,\spec(A(U))) : \lambda\in\Lambda\},
$$
where $\{U_\lambda\to U:\lambda\in\Lambda\}$ is a $J$-covering, and
$$
\{(U,P_i)\to (U,P): i\in I\}, 
$$
where $\{P_i\to P:i\in I\}$ is a $\tau$-covering in $\spec(A(U))_\tau$;

\item the ring $\cO_\tau$ is the ring in $\sh(\cC_\tau,J_\tau)$ associated to the presheaf
$$
(U,P)\mapsto \cA(P),
$$
where $\cA(P)$ is the ring corresponding to the affine scheme $P$;

\item the topos morphism $\pi_\tau$ is associated with the morphism of sites $p:(\cC,J)\to (\cC_\tau,J_\tau)$ given by
$$
p(U)=(U,\spec(A(U))).
$$
\end{enumerate}

\subsection{Localisations of rings}\label{tau-local}

Let $(\cE,A)$ be a ringed topos, $U\in\cE$, and $P$ an affine scheme over $\spec(A(U))$. Let the object 
$$
U_P\xrightarrow{\varphi}U\in \cE_{\ov U}
$$
be defined by
$$
U_P(V)=\coprod_{\cE(V,U)}\Sch_{\ov\spec A(U)}(\spec A(V),P),
$$
and let $\varphi_V$ be the natural projection to $U(V)=\cE(V,U)$.

If $\tau$ is a topology on $\Sch$, we say (cf.~\cite[III.4.4]{hakim}) that $A$ is \emph{$\tau$-local}, if for every $U\in\cE$ and every $\tau$-covering $\{P_i\to \spec A(U):i\in I\}$, the family
$\{U_{P_i}\to U:i\in I\}$ is epimorphic.

Hakim proves that the structure ring $\cO_\tau$ of $\spec.\tau(\cE,A)$ is $\tau$-local, at least in the case of $\tau$ being the Zariski, \'etale and fppf topology. Moreover, for the case of the Zariski and \'etale spectra, the corresponding structure rings are universal in the appropriate sense to be discussed below.

\subsection{The Zariski spectrum of a ringed topos}\label{spec-zar}

Hakim  proved in \cite{hakim} that the inclusion 2-functor from the category of locally ringed topoi to the category of ringed topoi admits a 2-right adjoint 
$$
\speczar
$$ 
called the \emph{Zariski spectrum}. Moreover, it can be constructed by taking $\tau=\text{Zar}$, the Zariski topology, in \ref{external-spectra}.

More explicitly, given a ringed topos $(\cE,A)$, there exists a locally ringed topos 
$$\speczar(\cE,A)=(\tilde{\cE},\tilde{A})$$ and a morphism of ringed topoi $\pi:(\tilde{\cE},\tilde{A})\to (\cE,A)$ such that any morphism $\varphi:(\cF,B)\to (\cE,A)$ with $B$ local, factors uniquely through $(\tilde{\cE},\tilde{A})$ by a morphism $\bar{\varphi}$ of locally ringed topoi. The diagram
\begin{center}
\begin{tikzpicture} 
\matrix(m)[matrix of math nodes, row sep=2em, column sep=2em, text height=1.5ex, text depth=0.25ex]
 {
|(0)|{(\cF,B)}  &  	& |(2)|{(\tilde{\cE},\tilde{A})} 	\\
 & |(l1)|{(\cE,A)}		& 	\\
 }; 
\path[->,font=\scriptsize,>=to, thin]
(0) edge[style=dashed] node[above]{$\bar{\varphi}$} (2)
edge  node[below left]{$\varphi$} (l1)
(2) edge node[below right]{$\pi$} (l1) 
;
\end{tikzpicture}
\end{center}
illustrates the above universal property, as well as the fact that the morphism
$$
\pi^*A\to\tilde{A}
$$
should be viewed as the solution to the problem of finding a universal localisation of $A$, which is only possibly by changing the topos from $\cE$ to $\tilde{\cE}$.


\subsection{Tierney's construction of the Zariski spectrum}\label{tierney-spec-zar}

We provide an internal construction of $\speczar(\cE,A)$ of a ringed topos $(\cE,A)$, following the steps outlined in \cite{tierney-spec}, when $\cE$ has a natural number object.

The \emph{object of radical ideals} of $A$ is the subobject of realisations 
$$
R\mono PA
$$
of the formula
\begin{align*}
\phi(\A:PA) \equiv \ \ &  (\forall f:A\, \forall g:A\ f\in\A\land g\in\A\Rightarrow f+g\in\A) \\
\land\ \  & (\forall f:A\, \forall a:A\ f\in\A \Rightarrow af\in\A)\\
 \land\ \  & (\forall f:A\ f\in\A \Leftrightarrow f^2\in\A)\\
\land\ \  & 0\in \A. \\
\end{align*}
Moreover, $R$ is an internal frame (a complete Heyting algebra), and we have a retraction
$$
r:PA\to R, \ \ \ r(X)=\bigwedge\{ \A\in R\, |\, X\leq \A\}.
$$
The composite
$$
\rho:A\xrightarrow{\{\}}PA\xrightarrow{r}R,
$$
satisfies 
\begin{multline*}
\rho(1)=A\ \land\  \rho(0)=\mathop{\rm Nil}(A)\\
\land\ \forall f:A\ \forall g:A\ \  (\rho(fg)=\rho(f)\land\rho(g))\ \land\ (\rho(f+g)\leq\rho(f)\lor\rho(g)).
\end{multline*}

Using the above notation, we define a poset $$\bbA=(A_1,A),$$
 where 
$A_1\mono A\times A$ is described by the formula
$$
\{(f,g):A\times A \,|\, \rho(f)\leq \rho(g)\}.
$$

An internal coverage on $\bbA$, in the sense of \ref{int-cov-poset},
$$
T\xhookrightarrow{(b,c)}A\times \mathop{\rm Idl}(\bbA)
$$
is given by the formula
$$
\{(f,I):A\times\mathop{\rm Idl}(\bbA)\, | \, \rho(f)\leq r(I)\},
$$
and we call it the \emph{Zariski topology}.

Tierney now considers the canonical geometric morphism
$$
\tilde{\cE}=\sh_{\cS}(\bbA,T)\xrightarrow{\pi}\cS
$$
and notes that the map $\varphi:\pi^*A\to\Omega_{\sh(\bbA,T)}$ which corresponds by adjunction to $\rho:A\to R\simeq \pi_*\Omega_{\sh(\bbA,T)}$ (where the last identification holds since the topology is induced by $r$), satisfies
\begin{multline*}
\varphi(1)=\mathop{\rm true}\ \land\  \varphi(0)=\mathop{\rm false}\\
\land\ \forall f:\pi^*A\ \forall g:\pi^*A\ \  (\varphi(fg)=\varphi(f)\land\varphi(g))\ \land\ (\varphi(f+g)\leq\varphi(f)\lor\varphi(g)).
\end{multline*}
Hence, the object $S\mono \pi^*A$, classified by the map $\varphi$ is the `universal coprime' in $\pi^*A$, and the localisation
$$
\tilde{A}=S^{-1}(\pi^*A)
$$
is therefore a local ring in $\tilde{\cE}$.

Equivalently, working with the canonical geometric morphism $\gamma:[\bbA^\op,\cE]\to\cE$, $A_1^\op$ is the universal saturated multiplicative subobject in $\gamma^*A$, so 
$$
\bar{A}=(A_1^\op)^{-1}(\gamma^*A).
$$
is a local ring object in $[\bbA^\op,\cE]$, and 
$\tilde{A}$ can be obtained as the $T$-sheaf associated to $\bar{A}$.

Tierney did not prove in \cite{tierney-spec} that this particular construction works when $\cE$ is a topos with a natural number object because it was difficult to establish the required universal property. Instead, he found an alternative construction by the method of `forcing topologies'. 

We favour the above construction because we can implement it explicitly in the special case when $\cE$ is the topos of difference sets in \ref{speczar-diff}. We avoid the difficulty of establishing the universal property in the case of a Grothendieck topos  $\cE$ by showing that the externalisation of Tierney's construction gives the original Hakim's construction.

\subsection{Tierney's vs.\ Hakim's construction of the Zariski spectrum}\label{hakim-spec-zar}

\begin{proposition}\label{tierney-works}
If $\cE$ is a Grothendieck topos, the construction from \ref{tierney-spec-zar} yields the Zariski spectrum.
\end{proposition}

\begin{proof}

Let $\cE=\sh(\cC,J)$ be a Grothendieck topos with a ring object $A$. We constructed the Zariski spectrum $\speczar(\cE,A)$ as a locally ringed topos $(\tilde{\cE},\tilde{A})$, where
$$
\tilde{\cE}=\sh_\cE(\bbA,T),
$$
where $(\bbA,T)$ is an internal site in $\cE$. 

We externalise the above by the construction from \ref{semidir} and obtain another description of $\tilde{\cE}$ as
$$
\tilde{\cE}=\sh(\cC\rtimes\bbA, J\rtimes T).
$$
Explicitly, the category 
$$
\cC\rtimes\bbA
$$
consists of objects $(U,s)$, where $U\in\cC$, and $s\in A(U)$, while a morphism
$\varphi:(U',s')\to (U,s)$ consists of a $\cC$-morphism $\varphi:U'\to U$ and an $\bbA(U')$-morphism $s'\to \bbA(\varphi)(s)$. More explicitly, this means that $s'\leq s|\varphi$, where $s|\varphi$ denotes the image of $s$ in $A(U')$, i.e., 
$$
s'\in\sqrt{s|\varphi} \text{ in } A(U').
$$
The coverage
$$
J\rtimes T
$$
consists of sieves generated by families
$$
\{(U_\lambda,s_{\lambda,i})\xrightarrow{\varphi_\lambda}(U,s):i\in I_\lambda, \lambda\in\Lambda\},
$$
where $\{U_\lambda\xrightarrow{\varphi_\lambda}U:\lambda\in\Lambda\}$ is a $J$-covering family, and, for each $\lambda \in\Lambda$, 
$$
s|\varphi_\lambda\in\sqrt{\{s_{\lambda,i}:i\in I_\lambda\}}. 
$$

The local ring $\tilde{A}$ is the associated $T$-sheaf to the internal presheaf $\bar{A}=(A_1^\op)^{-1}(\gamma^*A)$, where $\gamma:[\bbA^\op,\cE]\to\cE$ is the canonical geometric morphism. The corresponding presheaf
$$
\bar{\bA}\in[(\cC\rtimes\bbA)^\op,\Set]
$$
is defined as
$$
\bar{\bA}(U,s)=\{x\in \bar{A}_0(U) : \pi_2(x)=s\}\simeq A(U)_s.
$$
Thus, the $J\rtimes K$-sheaf on $\cC\rtimes\bbD$ corresponding to $\tilde{A}$ is the sheaf $\tilde{\bA}$ associated to the presheaf $\bar{\bA}$.

The above description of $\speczar(\cE,A)$ as $(\sh(\cC\rtimes\bbA,J\rtimes T,\tilde{\bA})$ is identical to Hakim's in \cite[IV.1]{hakim} (or, with slight adaptation, to construction given in \ref{external-spectra}), so we have shown that her construction is an externalisation of Tierney's construction \ref{tierney-spec-zar}. Given that Hakim proves that her construction satisfies the universal property from \ref{spec-zar}, so does Tierney's.
\end{proof}


\subsection{The \'etale spectrum}\label{spec-et}

Hakim calls an \'etale-local ring in a topos \emph{strictly local}. She proved that the inclusion 2-functor of the full subcategory of strictly locally ringed topoi in the category of locally ringed topoi admits a 2-right adjoint
$$
\specet
$$
called the \emph{\'etale spectrum}. Moreover, it can be constructed by taking $\tau=\text{\'et}$, the \'etale topology, in \ref{external-spectra}.

More explicitly, given a locally ringed topos $(\cE,A)$, there exists a strictly locally ringed topos 
$$\specet(\cE,A)=(\tilde{\cE},\tilde{A})$$ and a morphism of locally ringed topoi $\pi:(\tilde{\cE},\tilde{A})\to (\cE,A)$ such that any local morphism $\varphi:(\cF,B)\to (\cE,A)$ with $B$ strictly local, factors uniquely through $(\tilde{\cE},\tilde{A})$ by a morphism $\bar{\varphi}$ of locally ringed topoi, i.e., the diagram
\begin{center}
\begin{tikzpicture} 
\matrix(m)[matrix of math nodes, row sep=2em, column sep=2em, text height=1.5ex, text depth=0.25ex]
 {
|(0)|{(\cF,B)}  &  	& |(2)|{(\tilde{\cE},\tilde{A})} 	\\
 & |(l1)|{(\cE,A)}		& 	\\
 }; 
\path[->,font=\scriptsize,>=to, thin]
(0) edge[style=dashed] node[above]{$\bar{\varphi}$} (2)
edge  node[below left]{$\varphi$} (l1)
(2) edge node[below right]{$\pi$} (l1) 
;
\end{tikzpicture}
\end{center}
can be completed by a unique dashed arrow.


\subsection{Spectra as classifying topoi}\label{cole}

Let $\bbS$ and $\bbT$ be finitely presented geometric theories in the same language such that $\bbT$ is a quotient theory of $\bbS$, i.e., the axioms of $\bbT$ include those of $\bbS$. Let $\bbA$ be a class of morphisms of $\bbT$-models which satisfies the following conditions:
\begin{enumerate}
\item the property of being in $\bbA$ is preserved by inverse image functors;
\item $\bbA$ contains all identity morphisms;
\item given $L\xrightarrow{f}M\xrightarrow{g}P$ with $g\in\bbA$, then $f\in\bbA$ if and only if $gf\in\bbA$;
\item every $\bbS$-model morphism $M\xrightarrow{f}L$ to a $\bbT$-model $L$ has a factorisation as $M\xrightarrow{q}M_f\xrightarrow{\hat{f}}L$, where $M_f$ is a $\bbT$-model and $\hat{f}\in\bbA$ which is universal in the sense that any for other factorisation $M\xrightarrow{r}P\xrightarrow{g}L$, there exists a unique $M_f\xrightarrow{h}P$ such that $gh=\hat{f}$ and $hq=r$;
\item the above factorisations are preserved by inverse image functors. 
\end{enumerate}

Let us write $\bbS\da\Top_N$ for the 2-category of $\bbS$-modelled topoi with a natural number object and $\bbA\da\Top_N$ for the sub-2-category of $\bbT\da\Top_N$ with those 1-arrows $(p,f)$ for which $f\in\bbA$. 

A theorem of Cole \cite{cole} states that the inclusion 2-functor $\bbA\da\Top_N\to \bbS\da\Top_N$ admits a right 2-adjoint 
$$
\spec: \bbS\da\Top_N\to \bbA\da\Top_N.
$$

The universal property of the Zariski spectrum from \ref{spec-zar} can be viewed as a special case of Cole's theorem for $\bbS$ the algebraic theory of rings, $\bbT$ the geometric theory of local rings, and $\bbA$ the class of local homomorphisms of local rings. In perticular, that entails that Zariski local rings are local rings.

The universal property of the Zariski spectrum of a ringed topos $(\cE,A)$ can therefore be restated as follows. Let $\bbT_{\text{Zar},A}$ be the geometric theory over $\cE$ of localisations of $A$. If $\speczar(\cE,A)=(\tilde{\cE},\tilde{A})$, then $\tilde{\cE}$ is the classifying topos of $\bbT_{\text{Zar},A}$, and $\tilde{A}$ is its generic model. Modulo a small abuse of notation, we can write
$$
\speczar(\cE,A)\simeq \cE[\bbT_{\text{Zar},A}].
$$
In other words, for any $\cE$-topos $\cF$, we have an equivalence
$$
\Top_\cE(\cF,\tilde{\cE})\simeq \text{Loc.Rng}_A(\cF),
$$
which takes $\cF\xrightarrow{f}\tilde{\cE}$ to $f^*\tilde{A}$.

Joyal and Wraith \cite{wraith-strict-local} have shown that strictly local rings are models of the geometric theory of strictly henselian rings, so the universal property of the \'etale spectrum from \ref{spec-et} is a special case of Cole's theorem for $\bbS$ the geometric theory of local rings, $\bbT$ the geometric theory of strictly henselian rings, and $\bbA$ the class of local homomorphisms of strictly henselian rings.

Hence, given a locally ringed topos $(\cE,A)$, we have a geometric theory $\bbT_{\text{\'et},A}$ over $\cE$ of strict henselisations of $A$, and, if $\specet(\cE,A)=(\hat{\cE},\hat{A})$, then $\hat{\cE}$ is the classifying topos of 
$\bbT_{\text{\'et},A}$ and $\hat{A}$ is its generic model, and we can write
$$
\specet(\cE,A)\simeq \cE[\bbT_{\text{\'et},A}].
$$
In other words, for any $\cE$-topos $\cF$, we have an equivalence
$$
\Top_\cE(\cF,\hat{\cE})\simeq \text{Str.Loc.Rng}_A(\cF),
$$
which takes $\cF\xrightarrow{f}\tilde{\cE}$ to $f^*\hat{A}$.

\subsection{Relative Affine schemes}\label{rel-affine}

The \emph{affine scheme} associated to a ringed topos $(\cS,\cO_\cS)$ is the locally ringed topos
$$
(X,\cO_X)=\speczar(\cS,\cO_\cS)\xrightarrow{\pi}(\cS,A). 
$$

If $\tau$ is a topology on schemes that refines the Zariski topology, we define the $\tau$-topos of $X$ as 
$$
(X_\tau,\cO_\tau)=\spec.\tau(X,\cO_X).
$$
In particular, we obtain the \'etale and fppf topoi of $X$,
$$
X_\text{\'et}, \ \ X_\text{fppf}.
$$

\subsection{Relative schemes}\label{rel-schemes}

In the first draft of this manuscript, we will mostly deal with affine schemes over a base topos, but we indicate Hakim's approach to treating general schemes and the subtle glueing constructions involved.

Let $(\cS,\cO_\cS)$ be a ringed topos, and let
$$
\Sigma:\cS\to\cat
$$
be the 2-functor given, for $U\in\cS$, by
$$
\Sigma(U)=\Sch_{\ov \spec\Gamma(U,\cO_S)},
$$
and, for $V\xrightarrow{\varphi}U\in \cS$, by
$$
\Sigma(\varphi): X\mapsto X\times_{\spec\Gamma(U,\cO_X)}\spec\Gamma(V,\cO_S).
$$
Let
$$
\{\Sch;(\cS,\cO_\cS)\}
$$
be the stackification of the fibered category associated with $\Sigma$, and we define the category of \emph{relative schemes over $(\cS,\cO_\cS)$} as the fibre
$$
\Sch_{(\cS,\cO_\cS)}
$$
of the stack $\{\Sch;(\cS,\cO_\cS)\}$ over the terminal object of $\cS$.

An $\cS$-quasi-scheme is a locally ringed topos $(X,\cO_X)$ over $\cS$ which is associated with a relative $\cS$-scheme as in \cite[V.7.1]{hakim}. 

For the first reading of this manuscript, it is enough to work with relative affine schemes $X$.

\section{Relative Galois theory}\label{rel-galois}

\subsection{Janelidze's categorical Galois theory}\label{janelidze}

We review Janelidze's pure Galois theory, as set out in \cite{janelidze} and \cite{borceux-janelidze}. 

We start with we an adjoint pair of functors
\begin{center}
 \begin{tikzpicture} 
 [cross line/.style={preaction={draw=white, -,
line width=3pt}}]
\matrix(m)[matrix of math nodes, minimum size=1.7em,
inner sep=0pt, 
row sep=3.3em, column sep=1em, text height=1.5ex, text depth=0.25ex]
 { 
  |(dc)|{\cA}	\\
 |(c)|{\cP} 	      \\ };
\path[->,font=\scriptsize,>=to, thin]
%
(dc) edge [bend right=30] node (ss) [left]{$S$} (c)
(c) edge [bend right=30] node (ps) [right]{$C$} (dc)
(ss) edge[draw=none] node{$\dashv$} (ps)
;
\end{tikzpicture}
\end{center}
with unit $\eta:\id\to CS$ and counit $\epsilon:SC\to \id$.

For any $X\in\cA$, we obtain an adjunction
\begin{center}
 \begin{tikzpicture} 
 [cross line/.style={preaction={draw=white, -,
line width=3pt}}]
\matrix(m)[matrix of math nodes, minimum size=1.7em,
inner sep=0pt, 
row sep=3.3em, column sep=1em, text height=1.5ex, text depth=0.25ex]
 { 
  |(dc)|{\cA_{\ov X}}	\\
 |(c)|{\cP_{\ov S(X)}} 	      \\ };
\path[->,font=\scriptsize,>=to, thin]
%
(dc) edge [bend right=30] node (ss) [left]{$S_X$} (c)
(c) edge [bend right=30] node (ps) [right]{$C_X$} (dc)
(ss) edge[draw=none] node{$\dashv$} (ps)
;
\end{tikzpicture}
\end{center}
where  $C_X(E\xrightarrow{e}S(X)$ is obtained by forming the pullback

\begin{center}
 \begin{tikzpicture} 
\matrix(m)[matrix of math nodes, row sep=2em, column sep=2em, text height=1.9ex, text depth=0.25ex]
 {
 |(1)|{C_X(e)}		& |(2)|{C(E)} 	\\
 |(l1)|{X}		& |(l2)|{CS(X)} 	\\
 }; 
\path[->,font=\scriptsize,>=to, thin]
(1) edge  (2) edge   (l1)
(2) edge node[right]{$C(e)$} (l2) 
(l1) edge node[above]{$\eta_X$}  (l2);
\end{tikzpicture}
\end{center}
and
$$
S_X(A\xrightarrow{a}X)=S(A)\xrightarrow{S(a)}S(X).
$$

An object $A\xrightarrow{a}Y\in\cA_{\ov Y}$ is \emph{split} by a morphism $X\xrightarrow{f}Y\in\cA$ when the unit $\eta^X:\id\to C_XS_X$ of adjunction $S_X\dashv C_X$ gives an isomorphism
$$
\eta^X_{f^*a}:f^*a\to C_XS_X(f^*a),
$$
or, equivalently, if there exists an object $E\xrightarrow{e}S(X)$ such that 
$$
f^*a\simeq C_X(e).
$$
We write
$$
\Split_Y(f)
$$
for the full subcategory of $\cA_{\ov Y}$ of objects split by $f$.

We say that a morphism $X\xrightarrow{f}Y$ in $\cA$ is an \emph{effective descent morphism} if the functor
$$
f^*:\cA_{\ov Y}\to\cA_{\ov X}
$$
is monadic. Moreover, $f$ is of \emph{relative Galois descent} if
\begin{enumerate}
\item $f$ is an effective descent morphism;
\item the counit $\epsilon^X:S_XC_C\to \id$ of adjunction $S_X\dashv C_X$ is an isomorphism;
\item for every $E\xrightarrow{e}S(X)$ in $\cP_{\ov S(X)}$, the object $\Sigma_f\circ C_X(e)\in\cA_{\ov Y}$ is split by $f$.
\end{enumerate} 

If $X\xrightarrow{f}Y$ is of relative Galois descent, we define 
$$
\Gal[f]
$$
as the internal groupoid in $\cP$ given by the data

\begin{center}
 \begin{tikzpicture} 
\matrix(m)[matrix of math nodes, row sep=0em, column sep=3em, text height=1.5ex, text depth=0.25ex]
 {
|(0)|{S(X\times_YX)\times_{S(X)}S(X\times_YX)}  &[3em]  |(1)|{S(X\times_YX)}		&[1em] |(2)|{S(X)} \\
 }; 
\path[->,font=\scriptsize,>=to, thin]
(0) edge node[above]{$(S(\pi_1),S(\pi_4))$} (1)
([yshift=1em]1.east) edge node[above=-2pt]{$S(\pi_1)$} ([yshift=1em]2.west) 
(2)  edge node[above=-2pt]{$S(\Delta)$} (1) 
([yshift=-1em]1.east) edge node[above=-2pt]{$S(\pi_2)$} ([yshift=-1em]2.west) 
 (1) edge [loop below] node {$S(\tau)$} (1)
;
\end{tikzpicture}
\end{center}
and now Janelidze's \emph{Galois theorem} gives us an equivalence of categories
$$
\Split_Y(f)\simeq [\Gal[f],\cP].
$$


\subsection{Bunge's localic fundamental group}\label{bunge-pi1}

Let $\gamma:\cE\to \cS$ be a locally connected geometric morphism, i.e., the adjunction 
$$
\gamma_!\dashv \gamma^*
$$
is $\cS$-indexed. The guiding philosophy is that in such a case, $\gamma_!$ can be thought of as the `connected components' functor.

An object $A\in\cE$ is \emph{split} by a cover $U$ in $\cE$ (i.e., an epimorphism $U\to 1$), if any of the following equivalent conditions is satisfied
\begin{enumerate}
\item there is a morphism $S\to \gamma_!U$ in $\cS$ and an isomorphism
$$
\gamma^*S\times_{\gamma^*\gamma_!U}U\to A\times U
$$
over $U$, where $\eta_U:U\to \gamma^*\gamma_!U$ is the unit of adjunction $\gamma_!\dashv\gamma^*$ evaluated at $U$;
\item there is a morphism $\alpha:S\to\gamma_!U$ and a morphism $\zeta:A\times U\to \gamma^*S$ so that the square
\begin{center}
 \begin{tikzpicture} 
\matrix(m)[matrix of math nodes, row sep=2em, column sep=2em, text height=1.9ex, text depth=0.25ex]
 {
 |(1)|{A\times U}		& |(2)|{U} 	\\
 |(l1)|{\gamma^*S}		& |(l2)|{\gamma^*\gamma_!U} 	\\
 }; 
\path[->,font=\scriptsize,>=to, thin]
(1) edge node[above]{$\pi_2$}  (2) edge node[left]{$\zeta$}   (l1)
(2) edge node[right]{$\eta$} (l2) 
(l1) edge node[above]{$\gamma^*\alpha$}  (l2);
\end{tikzpicture}
\end{center}
is a pullback;
\item there is a morphism $J\to I$ in $\cS$, a morphism $U\to \gamma^*I$ in $\cE$ and an isomorphism
$$
\gamma^*J\times_{\gamma^*I}U\to A\times U
$$
over $U$.
\end{enumerate}
Note that this agrees with the notion of $A$ being split by $U$ for the adjoint pair $\gamma_!\dashv\gamma^*$ in the sense of Janelidze \ref{janelidze}.

Let 
$$
\Split(U)
$$
denote the full subcategory of $\cE$ of objects split by the cover $U$, and let
$$
\Split(\cE)
$$
be the full subcategory of \emph{locally constant} objects, which are split by some cover.

For a cover $U$ in $\cE$, Bunge \cite{bunge-04} forms the \emph{fundamental pushout} topos $\cG_U$ as the pushout in $\Top_\cS$

\begin{center}
 \begin{tikzpicture} 
\matrix(m)[matrix of math nodes, row sep=2em, column sep=2em, text height=1.9ex, text depth=0.25ex]
 {
 |(1)|{\cE_{\ov U}}		& |(2)|{\cE} 	\\
 |(l1)|{\cS_{\ov\gamma_!U}}		& |(l2)|{\cG_U} 	\\
 }; 
\path[->,font=\scriptsize,>=to, thin]
(1) edge node[above]{$\varphi_U$}  (2) edge node[left]{$\rho_U$}   (l1)
(2) edge node[right]{$\sigma_U$} (l2) 
(l1) edge node[above]{$p_U$}  (l2);
\end{tikzpicture}
\end{center}
where $\varphi_U$ is the canonical local homeomorphism, $\rho_U$ is the connected locally connected part in the unique factorisation of the (locally connected) composite $\gamma\varphi_U:\cE_{\ov U}\to\cS$ into a connected locally connected morphism followed by a surjective local homeomorphism. 

As a category, $\cG_U$ is equivalent to $\Split(U)$.

By the properties of pushout, we obtain that $\sigma_U$ is connected and $p_U$ is a surjective local homeomorphism, so $p_U$ is of effective descent in the sense of Joyal-Tierney \cite{joyal-tierney}. Hence, 
$$
\cG_U=B G_U,
$$
the classifying topos of the localic groupoid $G_U$ in $\cS$ of automorphisms of $p_U$. 

We proceed to give an even more explicit description of $G_U$. Consider the diagram
\begin{center}
 \begin{tikzpicture} 
\matrix(m)[matrix of math nodes, row sep=2em, column sep=2em, text height=1.9ex, text depth=0.25ex]
 {
 |(2)|{\cE} 	 & \\
 |(l2)|{\cG_U} & |(l3)|{\cS} 	\\
 }; 
\path[->,font=\scriptsize,>=to, thin]
(2) edge node[above right]{$\gamma$}  (l3)
(2) edge node[left]{$\sigma_U$} (l2) 
(l2) edge node[above,pos=0.3]{$g$}  (l3);
\end{tikzpicture}
\end{center}
where $g$ is the structure geometric morphism. Since $\sigma$ is connected and locally connected and $\gamma$ is locally connected and $g\sigma= \gamma$, it follows that $g$ is locally connected. 

Let
$$
\zeta=u_{\sigma_!U}:\sigma_!U\to g^*g_!\sigma_!U\simeq g^*\gamma_!U,
$$
where $u$ is the unit of adjunction $g_!\dashv g^*$. Then $p_U^*$ is represented by $\zeta$, 
$$
\sigma^*\zeta:\sigma^*\sigma_!U\to \sigma^*g^*\gamma_!U\simeq \gamma^*\gamma_!U
$$
is a Galois family that generates $\cG_U$, and $\cG_U$ is the classifying topos of the discrete localic groupoid $G_U\simeq \Aut(\sigma^*\zeta)$,
$$
\cG_U\simeq B \Aut(\sigma^*\zeta).
$$

The \emph{coverings fundamental topos} of $\cE$, denoted  
$$\Pi_1(\cE),$$ is the limit in $\Top_\cS$ of the filtered system of toposes $\cG_U$, indexed by a cofinal generating poset of covers $U$ in $\cE$. As a category, it is equivalent to $\Split(\cE)$. 

We define the \emph{coverings fundamental groupoid} as the prodiscrete localic groupoid
$$
\pi_1(\cE)=\lim G_U,
$$
obtained as the limit of discrete localic groupoids $G_U$ in the category of localic groupoids. 

If $\cS$ is a Grothendieck topos, we have that $\Pi_1(\cE)$ is the classifying topos of $\pi_1(\cE)$, i.e., 
$$
\Pi_1(\cE)=B\pi_1(\cE).
$$
Moreover, $\pi_1(\cE)$ represents first-degree cohomology of $\cE$ with coefficients in discrete groups, i.e., for a group $K$ in $\cS$, we have
$$
\tH^1(\cE,\gamma^*K)\simeq \pi_0(\Hom(\pi_1(\cE),K)),
$$
where $\pi_0$ stands for taking connected components.

\subsection{Comparing theories of Bunge and Janelidze}

Janelidze's theory is more general, because we do not have to work with topoi at all, but it requires us to find normal objects/morphisms of relative descent. When applied to the adjunction
$$
\gamma_!\dashv\gamma^*
$$
associated with a locally constant geometric morphism $\gamma:\cE\to\cS$, with notation from \ref{bunge-pi1}, we have that $$\sigma^*\zeta$$
is a normal object, $$\cG_U\simeq\Split(\sigma^*\zeta),$$
and
$$
\Gal[\sigma^*\zeta]\simeq \Aut(\sigma^*\zeta)
$$
as groupoids in $\cS$, so $\cG_U$ is the classifying topos for either of the above groupoids.

\subsection{Relative \'etale fundamental group}\label{rel-pi1}

Let $(\cS,A)$ be a ringed Grothendieck topos, $(X,\cO_X)=\speczar(\cS,A)$ and $(X_\text{\'et},\cO_\text{\'et})=\spec.\text{\'et}(X,\cO_X)$. Consider the structure geometric morphism
$$
\gamma:X_\text{\'et}\to X\to\cS.
$$
If $\gamma$ is locally connected (e.g., when $\spec(\forg{A})$ is connected), we define the {relative \'etale fundamental groupoid}
$$
\pi^\text{\'et}_1(X)
$$
as the prodiscrete localic fundamental groupoid of $X_\text{\'et}\xrightarrow{\gamma}\cS$ following \ref{bunge-pi1}.

Using the fact that $X_\text{\'et}$ is the classifying space for the geometric theory of strict henselisations of localisations of $A$ from \ref{cole}, a point
$$
x:\cS\to X_\text{\'et}
$$
corresponds to a model $\Omega_x$ of that theory in $\cS$, and we can calculate the pro-localic group $\pi^\text{\'et}_1(X)$ as the pro-localic group of automorphisms of $\Omega$ in $\cS$.

\section{Cohomology in relative algebraic geometry}\label{rel-coh}

\subsection{Cohomology}\label{coh-rel-sch}

Let $(\cS,A)$ be a ringed Grothendieck topos, let $(X,\cO_X)\xrightarrow{\pi}(\cS,A)$ be a relative quasi-scheme, let $\tau$ be a scheme topology finer than Zariski topology, and let $(X_\tau,\cO_\tau)$ be the $\tau$-topos associated with $X$, giving rise to a diagram
$$
\begin{tikzpicture} 
\matrix(m)[matrix of math nodes, row sep=2em, column sep=2em, text height=1.5ex, text depth=0.25ex]
 {
 |(1)|{(X_\tau,\cO_\tau)}		& |(2)|{(\cS,A)} 	\\
		& |(l2)|{(\Set,\Z)} 	\\
 }; 
\path[->,font=\scriptsize,>=to, thin]
(1) edge node[above]{$\pi_\tau$} (2) edge node[left]{$$}   (l2)
(2) edge node[right]{$$} (l2) 
;
\end{tikzpicture}
$$
of geometric morphisms. 

For abelian groups $M, N$ in $X_\tau$, we consider the abelian groups
$$
\Ext^n_\tau(M,N)=\Ext^n(X_\tau, M,N)\ \ \text{ and } \ \ \tH^n(X_\tau,N)
$$
defined through classical topos cohomology \ref{topos-coh}, as well as the \emph{relative $\Ext$ and cohomology} abelian groups in $\cS$,
$$
\Ext^n(X_\tau/\cS, M,N)\ \ \text{ and }\ \ \tH^n(X_\tau/\cS,N)=R^n\pi_{\tau*}(N),
$$
as expounded in \ref{rel-top-coh}, where we also discuss the Leray/Grothendieck spectral sequences relating the two. 

Our guiding philosophy stipulates that, whilst we should appreciate the importance of the absolute cohomology groups, it is natural to expect that in relative algebraic geometry, the relative cohomology groups play an important role, and we should strive to identify it.

\subsection{Cohomology of relative quasi-coherent sheaves}\label{coh-quasicoh}

Quasi-coherent modules can be defined in any ringed topos,  as in \cite[03DL]{stacks-project}. Given  an $(\cS,\cO_S)$-quasi-scheme $(X,\cO_X)$, Hakim defines in \cite[VI.1.4]{hakim} the category of \emph{$\cS$-quasi-coherent $\cO_X$-modules}.

In perfect analogy with vanishing theorems for cohomology of classical affine schemes, she proves the following.

Let $(\cS,\cO_\cS)$ be a ringed topos, let $\varphi:A\to B$ be a homomorphism of $\cO_\cS$-algebras inducing a morphism $f:\speczar(\cS,B)\to\speczar(\cS,A)$, and let $\cF$ be an $\cS$-quasi-coherent $\cO_{\speczar(\cS,B)}$-module. Then, for all $i>0$, we have
$$
R^if_*\cF=0.
$$
Moreover, if $\pi:\speczar(\cS,\cO_\cS)\to (\cS,\cO_\cS)$ is the structure morphism, and 
$\cG$ is an $\cS$-quasi-coherent $\cO_{\speczar(\cS,\cO_\cS)}$-module, then, for $i>0$
$$
R^i\pi_*\cG=0.
$$

\subsection{Picard group of a ringed topos}\label{picard}

Let $(\cE,\cO)$ be a ringed topos. We say that an $\cO$-module $\cM$ is \emph{invertible}, if there exists an $\cO$-module $\cN$ such that
$$
\cM\otimes\cN\simeq\cO.
$$
Following \cite[040C]{stacks-project}, we define the \emph{Picard group} of $(\cE,\cO)$ as
the abelian group
$$
\Pic(\cO)
$$
of isomorphism classes of invertible $\cO$-modules, where addition corresponds to tensor product.

If $(\cE,\cO)$ is a locally ringed topos, by \cite[040E]{stacks-project}, there is a canonical isomorphism
$$
\tH^1(\cE,\cO^\times)\simeq \Pic(\cO).
$$

\subsection{Hilbert's Theorem 90}\label{hilbert90}

Let $(X,\cO_X)$ be an $(\cS,\cO_S)$-quasi-scheme. We write
$$
\bbG_m
$$
for the object of $X_\et$ represented by the object $(1,\spec(A(1)[t,t^{-1}])$, using the notation from \ref{external-spectra}. For an object of the \'etale site, it assigns
$$
\bbG_m(U,P\to\spec A(U))= \cA(P)^{\times},
$$ 
the set of invertible elements in the $A(U)$-algebra $\cA(P)$. We can also write
$$
\bbG_m=\cO_\et^\times.
$$
According to \ref{picard}, 
$$
\tH^1(X_\et,\bbG_m)\simeq \Pic(\cO_\et).
$$
On the other hand, we have a generalised version of Hilbert's Theorem 90, stating that,
if we consider the composite $\epsilon$ of the natural morphisms
$$
X_\text{fppf}\to X_\et\to X,
$$
then 
$$
R^1\epsilon_*\bbG_m=0,
$$
which entails that the canonical homomorphisms
$$
\tH^1(X,\cO_X^\times)\to \tH^1(X_\et,\bbG_m)\to \tH(X_\text{fppf},\bbG_m) 
$$
are isomorphisms, and we conclude that
$$
\tH^1(X_\et,\bbG_m)\simeq\Pic(\cO_X)=\Pic(X),
$$
the usual (Zariski) Picard group.

\subsection{Kummer Theory}\label{rel-kummer}

With notation from \ref{hilbert90}, if $n$ is invertible in $X$, we have a short exact sequence in $X_\et$
$$
1\to \mu_n\to \bbG_n \xrightarrow{(\,)^n}\bbG_n\to 1,
$$
where $\mu_n$ is the sheaf of $n$-th roots of unity, defined as the kernel of the $n$-th power map. 

Indeed, in order to show that the $n$-th power map is surjective, take an arbitrary $u\in \bbG_m(U,P)=\cA(P)^\times$, where $P\to \spec A(U)$ is \'etale. Since $n$ is invertible over $P$, the polynomial
$$
T^n-u
$$
is separable over $P$, i.e., $P'=\spec(\cA(P)[T]/(T^n-u))$ is \'etale over $P$, so we have found a cover $(U,P')\to (U,P)$ with an $n$-th root of $u$ in $\bbG_m(U,P')$.

A reader familiar with categorical logic will note that the statement is obvious in view of Wraith's description of the \'etale spectrum from \cite{wraith-strict-local}.

The exact cohomology sequence associated to the above is
$$
0\to\Gamma(\mu_n)\to\Gamma(\cO_X)^\times\xrightarrow{(\,)^n}\Gamma(\cO_X)^\times\to \tH^1(X_\et,\mu_n)\to \Pic(X)\xrightarrow{n\cdot}\Pic(X),
$$
whence
$$
0\to \Gamma(\cO_X)^\times/(\Gamma(\cO_X)^\times)^n\to \tH^1(X_\et,\mu_n)\to {}_n\Pic(X)\to0,
$$
where we wrote ${}_n\Pic(X)=\Ker(\Pic(X)\xrightarrow{n\cdot}\Pic(X))$.

The long exact sequence sequence for higher direct images of $\pi=\pi_\et:X_\et\to \cS$ yields an exact sequence of relative cohomology groups
$$
0\to \pi_*\mu_n\to \pi_*\bbG_m\to\pi_*\bbG_m\to \tH^1(X_\et/\cS,\mu_n)\to \uPic(X)\to\uPic(X),
$$
where by $\uPic(X)$ we mean the abelian group $H^1(X/\cS,\cO_X^\times)$ in $\cS$, given the interpretation of the `enriched Picard group'.

\subsection{Artin-Schreier Theory}\label{rel-AS}

 If $X$ is a relative scheme of characteristic $p>0$, we have an exact sequence in $X_\et$
 $$
 0\to \Z/p\Z\to \cO_\et\xrightarrow{F-\id}\cO_\et\to 0,
 $$
 where $F-\id$ is associated to a homomorphism of additive groups $f\mapsto f^p-f$. 
 
To see that it is an epimorphism, using the notation from \ref{external-spectra}, let $a\in\cO_\et(U,P)=\cA(P)$, where $P\to\spec A(U)$ is \'etale. The polynomial $T^p-T-a$ is separable, so it defines an \'etale extension $P'\to P$, hence a covering $(U,P')\to(U,P)$, so that the equation $T^p-T-a=0$ has a solution in $\cO_\et(U,P')$.

The corresponding long exact cohomology sequence is
$$
\Gamma(\cO_X)\xrightarrow{F-\id}\Gamma(\cO_X)\to \tH^1(X_\et,\Z/p\Z)\to\tH^1(X_\et,\cO_\et)\xrightarrow{F-\id}\tH^1(X_\et,\cO_\et),
$$
whence we obtain a short exact sequence
$$
0\to \Gamma(\cO_X)/(F-\id)(\Gamma(\cO_X))\to H^1(X_\et,\Z/p\Z)\to \tH^1(X_\et,\cO_\et)^F\to 0.
$$

The long exact sequence for relative cohomology is
$$
\pi_*\cO_\et \xrightarrow{F-\id} \pi_*\cO_\et
\to \tH^1(X_\et/\cS,\Z/p\Z)\to\tH^1(X_\et/\cS,\cO_\et)\xrightarrow{F-\id}\tH^1(X_\et/\cS,\cO_\et).
$$

\section{Group cohomology}\label{s:rel-gp-coh}

\subsection{Group cohomology in a presheaf category}

This subsection is a special case of \ref{gp-coh-top}, and is only included for the sake of the reader who prefers to revise the material in the more concrete setting of an ordinary presheaf category as in Demazure's  \cite{sga3.1}. 

Let $\cC$ be a category, and write $\hC=[\cC^\op,\Set]$ for the category of presheaves on $\cC$.


Let $\bG\in\Ob(\hC)$. We say that $\bG$ is a \emph{$\hC$-group} if there exists a multiplicative law on $\bG$, i.e., a $\hC$-morphism 
$$
\pi_{\bG}:\bG\times\bG\to\bG
$$
such that for every $S\in\Ob(\cC)$, the operation $\pi_{\bG}(S)$ makes $\bG(S)$ a group.

A $\hC$-morphism $f:\bG\to\bH$ is a \emph{morphism} of $\hC$-groups if the diagram
\begin{center}
 \begin{tikzpicture} 
\matrix(m)[matrix of math nodes, row sep=2em, column sep=3em, text height=1.5ex, text depth=0.25ex]
 {
 |(1)|{\bG\times\bG}		& |(2)|{\bG} 	\\
 |(l1)|{\bH\times\bH}		& |(l2)|{\bH} 	\\
 }; 
\path[->,font=\scriptsize,>=to, thin]
(1) edge node[above]{${\pi_{\bG}}$} (2) edge node[left]{$f\times f$}   (l1)
(2) edge node[right]{$f$} (l2) 
(l1) edge  node[above]{$\pi_{\bH}$} (l2);
\end{tikzpicture}
\end{center}
commutes. The resulting category of $\hC$-groups is denoted
$$
\hCGp.
$$

Equivalently, $\bG$ is a $\hC$-group if, for every $S\in\Ob(\cC)$, the set $\bG(S)$ is equipped with a group structure in a functorial way, i.e., for any morphism $f:S'\to S''$ in $\cC$, the map $\bG(f):\bG(S'')\to\bG(S')$ is a group homomorphism. Moreover, a $\cC$-morphism $f:\bG\to\bH$ is a morphism in $\hCGp$ if, for every $S\in\Ob(\cC)$, the map $f(S):\bG(S)\to\bH(S)$ is a group homomorphism.

The terminal object $\be$ of $\hC$ has a trivial $\cC$-group structure, hence it acts as as a terminal object of $\hCGp$ as well.

The \emph{identity section} of a $\hC$-group $\bG$ is the element $e_{\bG}\in\Gamma(\bG)=\Hom(\be,\bG)$ such that, for every $S\in\Ob(\cC)$, $e_{\bG}(S)$ is the identity of the group $\bG(S)$. Naturally, $e_{\bG}$ is a $\hC$-group morphism $\be\to\bG$.

A $\hC$-group $\bH$ is a sub-$\hC$-group of $\bG$ if $\bH$ is a sub-object of $\bG$ such that, for every $S\in\Ob(\cC)$, $\bH(S)$ is a subgroup of $\bG(S)$, i.e., such that the monomorphism $\bH\to\bG$ is a morphism of $\hC$-groups.

The \emph{product} of $\hC$-groups $\bG$ and $\bH$ in the category of $\hC$-groups is the object $\bG\times\bH$ equipped with the obvious $\hC$-group structure: for every $S\in\Ob(\cC)$, $\bG(S)\times\bH(S)$ is the direct product of groups $\bG(S)$, $\bH(S)$. 

If $\bG$ is a $\hC$-group, for each $S\in\Ob(\cC)$, $\bG_S$ is a $\widehat{\cC_{\ov S}}$-group. Thus, if $\bH$ is another $\hC$-group, we can define an object $\uHom_\hCGps(\bG,\bH)$ of $\hC$ by
$$
\uHom_\hCGps(\bG,\bH)(S)=\Hom_{\widehat{\cC_{\ov S}}\text{-}\Gp}(\bG_S,\bH_S).
$$
We define $\uIsom_\hCGps(\bG,\bH)$, $\uEnd_\hCGps(\bG)$ and $\uAut_\hCGps(\bG)$ analogously.

We say that an object $G$ of $\cC$ is a \emph{$\cC$-group}, provided we have a $\hC$-group structure on $\h_G\in\Ob(\hC)$. A \emph{morphism} between $\cC$-groups $G$ and $H$ is an element
$u\in\Hom(G,H)$ which defines a morphism of $\hC$-groups between $\h_G$ and $\h_H$.

We denote the category of $\cC$-groups by
$$
\CGp.
$$
Note that all the constructions from the discussion of $\cC$-groups apply to $\cC$-groups, as long as all the objects involved are representable.

Example: If $\bE\in\Ob(\hC)$, then $\uAut(\bE)$ is a $\hC$-group with the obvious structure.



Let $\bE\in\Ob(\hC)$ and $\bG\in\Ob\hCGp$. We say that $\bG$ \emph{acts} on $\bE$ if we are given a $\hC$-morphism 
$$
\mu:\bG\times\bE\to\bE
$$
such that, for every $S\in\Ob(\cC)$, the map $\mu(S)$ gives an action of the group $\bG(S)$ on the set $\bE(S)$.

Equivalently, using the bijection $\Hom(\bG\times\bE,\bE)\simeq\Hom(\bG,\uEnd(\bE))$, a group action is defined by a morphism of $\hC$-groups
$$
\rho:\bG\to\uAut(\bE).
$$

Suppose that a $\hC$-group $\bG$ acts on an object $\bE$ in $\cC$. We write $\bE^{\bG}$ for a sub-object of $\bE$ defined by
$$
\bE^{\bG}(S)=\{x\in\bE(S): x_{S'}\text{ is invariant under }\bG(S')\text{ for all }S'\to S\},
$$
where $x_{S'}$ is the image of $x$ via the map $\bE(S)\to\bE(S')$.

\begin{definition}
Let $\bF$ be a sub-object of $\bE$ in $\hC$. 
\end{definition}

\subsection{Categories of modules}

\begin{definition}
Let $\bO$ and $\bF$ be objects in $\hC$. We say that $\bF$ is a \emph{$\hC$-module over a $\hC$-ring} $\bO$ (or an \emph{$\bO$-module} for short), if for all $S\in\Ob(\cC)$, $\bF(S)$ has a structure of a module over a ring $\bO(S)$, so that for every morphism $S'\to S''$ in $\cC$, $\bO(S'')\to\bO(S')$ is a ring homomorphism, and $\bF(S'')\to\bF(S')$ is a homomorphism of abelian groups compatible with the above ring homomorphism.
\end{definition}

If $\bF$, $\bF'$ are $\bO$-modules, we have an abelian group of $\bO$-morphisms
$$
\Hom_\bO(\bF,\bF').
$$
The category of $\bO$-modules is denoted $\Omod$.

\begin{lemma}\label{omod-is-ab}
The category $\Omod$ is abelian. In fact, is is an \emph{(AB 5)} category, in the sense that it has arbitrary coproducts (and hence colimits), and filtered colimits of exact sequences are exact.
\end{lemma}
The structure of the abelian category is defined `argument-wise'. More precisely, if $f:\bF\to\bF'$ is a morphism of $\bO$-modules, we define the $\bO$-modules $\Ker f$, $\Im f$, $\Coker f$ via
$(\Ker f)(S)=\Ker f(S)$, $(\Im f)(S)=\Im f(S)$, $(\Coker f)(S)=\Coker f(S)$. It is immediate that $\Ker f$ is the kernel of $f$, $\Coker f$ is the cokernel of $f$, and there is an isomorphism of $\bO$-modules $\bF/\Ker f\simeq \Im f$ whence $\Omod$ is abelian. 

Coproducts (direct sums) are also defined argument-wise so they clearly exist. If $\bM$ is an $\bO$-module, $\bN$ a submodule, and $(\bF_i)_{i\in I}$ an increasing directed family of submodules of $\bM$, then 
$$
\textstyle\bigcup_{i\in I}(\bF_i\cap\bN)=\left(\bigcup_{i\in I}\bF_i\right)\cap\bN.
$$
Indeed, if $S\in\Ob(\cC)$ and $x\in \bN(S)\cap\bigcup_i\bF_i(S)$, there exists an $i\in I$ such that $x\in\bN(S)\cap \bF_i(S)$.

\begin{proposition}\label{omod-has-inj}
If $\cC$ is small, then $\bO$ is a generator of $\Omod$ and $\Omod$ has enough injectives.
\end{proposition}

Let $\bO_0$ be the $\hC$-ring defined by $\bO_0(S)=\Z$ for all $S\in\Ob(\cC)$. The category of $\bO_0$-modules is isomorphic to the category of abelian $\cC$-groups.

\begin{definition}
Let $\bF$ be an $\bO$-module. For every $S\in\Ob(\cC)$, $\bF_S$ is an $\bO_S$-module, so we can define an abelian $\cC$-group $\uHom_\bO(\bF,\bF')$ by
$$
\uHom_\bO(\bF,\bF')(S)=\Hom_{\bO_S}(\bF_S,\bF'_S).
$$
By analogy we define objects
$$
\uIsom_\bO(\bF,\bF'),\ \ \ \ \uEnd_\bO(\bF),\ \ \ \ \uAut_\bO(\bF),
$$
where the last object has the structure of a $\hC$-group induced by composition of morphisms.
\end{definition}

\begin{definition}
Let $\bO$ be a $\hC$-ring, $\bF$ an $\bO$-module and $\bG$ a $\hC$-group. The structure of a $\bG$-$\bO$-module on $\bF$ is given by a morphism of $\hC$-groups
$$
\rho:\bG\to\uAut_\bO(\bF).
$$
We define the abelian group
$$
\Hom_{\bG\text{-}\bO}(\bF,\bF')
$$
in a natural way. The resulting additive category of $\bG$-$\bO$-modules is denoted $\GOmod$.
\end{definition}

\begin{remark}
The category $\GOmod$ can also be defined as the category of $\bO[\bG]$-modules, where $\bO[\bG]$ is the $\hC$-algebra of the $\hC$-group $\bG$ over the $\hC$-ring $\bO$. Using \ref{omod-is-ab} and \ref{omod-has-inj}, we conclude that $\GOmod$ is an (AB 5) abelian category and, when $\cC$ is small, $\GOmod$ has enough injectives.
\end{remark}

\subsection{Group cohomology}\label{group-coh}

Let $\cC$ be a category, $\bG$ a $\hC$-group, $\bO$ a $\hC$-ring and $\bF$ a $\bG$-$\bO$-module. For $n\geq 0$, we let
$$
\Ch^n(\bG,\bF)=\Hom(\bG^n,\bF), \ \ \ \uC^n(\bG,\bF)=\uHom(\bG^n,\bF),
$$
with convention that $\bG^0$ is the final object $\be$. 

The objects $\uC^n(\bG,\bF)$ (resp.\ $\Ch^n(\bG,\bF)$) are naturally $\bO$-modules (resp.\ $\Gamma(\bO)$-modules), and we have that
$$
\Ch^n(\bG,\bF)\simeq\Gamma(\uC^n(\bG,\bF)),\ \ \ \text{ and }\ \ \ \uC^n(\bG,\bF)(S)=\Ch^n(\bG_S,\bF_S). 
$$
Thus, an element of $\Ch^n(\bG,\bF)$ is given by a family of $n$-cochains of $\bG(S)$ in $\bF(S)$, functorially in $S\in\Ob(\cC)$. The coboundary operator
$$
\partial:\Ch^n(\bG(S),\bF(S))\to \Ch^{n+1}(\bG(S),\bF(S))
$$ 
is given by the formula
\begin{align*}
\partial f(g_1,\ldots,g_{n+1})=g_1f(g_2,\ldots,g_{n+1})&+\sum_{i=1}^n (-1)^i f(g_1,\ldots,g_ig_{i+1},\ldots,g_{n+1})\\ & +(-1)^{n+1}f(g_1,\ldots,g_n),
\end{align*}
and, since it is functorial in $S$, it defines a homomorphism
$$
\partial:\Ch^n(\bG,\bF)\to\Ch^{n+1}(\bG,\bF)
$$
such that 
$$
\partial\circ\partial=0.
$$
We have thus obtained a complex of abelian groups (and even of $\Gamma(\bO)$-modules), denoted
$\Ch^{*}(\bG,\bF)$ and called the \hoch (or standard) complex.

Similarly we define a complex of $\bO$-modules $\uC^{*}(\bG,\bF)$, so that
$$
\Ch^{*}(\bG,\bF)=\Gamma(\uC^{*}(\bG,\bF).
$$

The \emph{\hoch cohomology groups} 
$$
\HH^n(\bG,\bF)
$$
are defined as cohomology groups of the complex $\Ch^{*}(\bG,\bF)$, while the $\hC$-groups
$$
\uH^n(\bG,\bF)
$$
are the cohomology groups of the complex $\uC^{*}(\bG,\bF)$.

In particular,
$$
\uH^0(\bG,\bF)=\bF^{\bG}\ \ \ \text{ and }\ \ \ \ \HH^0(\bG,\bF)=\Gamma(\bF^{\bG}).
$$

\begin{proposition}\label{der-of-fix}
Let $\bG$ be representable, and suppose either that $\cC$ is small with finite products, or that $\GOmod$ has enough injectives. Then the functors $\HH^n(\bG,\mathord{-})$ (resp.\ $\uH^n(\bG,\mathord{-})$) are the derived functors of the left exact functor $\HH^0(\bG,\mathord{-})$ (resp.\ $\uH^0(\bG,\mathord{-})$) on the category $\GOmod$.
\end{proposition}

\subsection{Group cohomology in a topos}\label{gp-coh-top}

Let $(\cE,O)$ be a ringed (Grothendieck) topos, let $G$ be a group in $\cE$, and let $\bB G$ be the classifying topos of $G$, considered in \ref{monoid-action}. Then $\bB G$ is also a Grothendieck topos.

The ring $O$ is an object of $\bB G$ when endowed with the trivial action of $G$. Moreover, $O$ is trivially a ring in $\bB G$. The group ring
$$
O[G]
$$
of $G$ with respect to $O$ is the $O$-algebra whose underlying module is the free $O$-module $O_G$ generated by $G$, together with the multiplication $O_G\otimes_O O_G\to O_G$ induced by the group multiplication $G\times G\to G$. 

Unravelling the definitions, we obtain equivalences of categories
$$
G\da O\Mod\simeq O[G]\Mod\simeq O\Mod(\bB G).
$$
They all have enough injectives since the last item is the category of modules in a ringed topos $(\bB G,O)$ so the general principles of \ref{ringed-top} apply. Hence, for $G\da O$-modules $M$ and $N$, it makes sense to  define 
$$
\Ext^n_{G\da O}(M,N)
$$
as the value at $N$ of the $n$-th derived functor of $\Hom_{G\da O}(M,\mathord{-}):G\da O\Mod\to\Ab$, and the above equivalence shows that
$$
\Ext^n_{G\da O}(M,N)\simeq \Ext^n_O(\bB G;M,N).
$$
Using the fact that the free $O$-module $O_e\simeq O$ in $O\Mod(\bB G)$, for $M\in\bB G$, 
$$
\Gamma(M)={\textstyle\varprojlim_G}(M)\simeq M^G\simeq \bB G(e, M)\simeq O\Mod(\bB G)(O,M)\simeq G\da O\Mod(O,M),
$$
so we conclude that group cohomology is an instance of topos cohomology,
$$
\HH^n(G,M)=\Ext^n_{G\da O}(O,M)\simeq \tH^n(\bB G,M).
$$
It also makes sense to define the internal group cohomology, and it is again an instance of internal topos cohomology, i.e., we obtain that
$$
\uH^n(G,M)\simeq \uH^n(\bB G,M),
$$
and
$$
\uExt^n_{G\da O}(M,N)\simeq \uExt^n_O(\bB G; M,N).
$$

\subsection{Group cohomology over a base topos}\label{rel-gp-coh}

Let $\gamma:(\cE,O)\to(\cS,R)$ be a (bounded) geometric morphism of ringed topoi with $R$ commutative, and let $G$ be a group in $\cE$. We can consider the category
$$
G\da O\Mod
$$
an $R\Mod$-category with internal hom objects
$$
G\da O\Mod(M,N)=\gamma_*[M,N]_{G\da O}\in R\Mod,
$$
and we define
$$
\Ext_{G\da O\Mod}^n(M,\mathord{-})
$$
as the $n$-th derived functor of $G\da O\Mod(M,\mathord{-})$. If $G\da O\Mod$ has enough enriched injectives, we can consider its $n$-th  derived enriched functor.

Moreover, we consider the functor $\Gamma_{\cE_{\ov\cC}}:G\da O\Mod\to R\Mod$ given by
$$
\Gamma_{\cE_{\ov \cS}}(M)=\cE_{\ov\cS}(e,M)=\gamma_*[e,M].
$$ 
and we let 
$$
\HH^n_{\cE_{\ov\cS}}(G,\mathord{-})
$$
be its $n$-th derived functor. Again, if $G\da O\Mod$ has enough enriched injectives, we can take its $n$-th  derived enriched functor.

Since $\HH^0_{\cE_{\ov\cS}}(G,\mathord{-})\simeq\gamma_*\circ \uH^0(G,\mathord{-})$, we have a spectral sequence 
$$
E_2^{p,q}=\der^p\gamma_*(\uH^q(G,M))\Rightarrow \HH^{p+q}_{\cE_{\ov\cS}}(G,M).
$$
Similarly, we obtain a spectral sequence
$$
E_2^{p,q}=\der^p\gamma_*(\uExt^q(M,N))\Rightarrow \Ext_{G\da O\Mod}^{p+q}(G,M).
$$
It is now straightforward to formulate spectral sequences relating $\HH^n(G,\mathord{-})$, $\HH^n_{\cE_{\ov\cS}}(G,\mathord{-})$ and $\uH^n(G,\mathord{-})$, as well as
$\Ext^n_{G\da O}$, $\Ext^n_{G\da O\Mod}$ and $\uExt^n_{G\da O}$, in the spirit of \ref{rel-top-coh}.

\subsection{Induced group modules}\label{ind-gp-mod}

Let $\cV$ be a complete cartesian closed category, let $\cC$ be a $\cV$-category, and let $\hC=[\cC^\circ,\cV]$ be the $\cV$-category of $\cV$-presheaves on $\cC$.
 
\begin{definition}
Let $\bO$ be a $\hC$-ring, let $\bG$ be a $\hC$-group, and let $\bP$ be an $\bO$-module. We write 
$$
E(\bP)=[\bG,\bP],
$$
endowed with the structure of a $\bG\da\bO$-module as follows. Writing $\bG_1$ and $\bG_2$ for two copies of $\bG$, we have an isomorphism
\begin{multline*}
\hC(\bG_1\times E(\bP),E(\bP))\simeq\hC(E(\bP),[\bG_1,[\bG_2,\bP]])\simeq\hC(E(\bP),[\bG_2\times\bG_1,\bP])\\
\simeq \hC([\bG,\bP],[\bG_2\times\bG_1,\bP]).
\end{multline*}
The required action 
$$
\bG_1\times E(\bP)\to E(\bP)
$$ 
corresponds via the above isomorphism to the morphism
$$
\mu_G^*:[\bG,\bP]\to [\bG_2\times\bG_1,\bP],
$$
coming from the multiplication $\mu_G:\bG_2\times\bG_1\to \bG$.

Using the identity section $e_\bG:\h_I\to \bG$, we obtain a morphism
$$
\varepsilon_P=e_\bG^*:E(\bP)=[\bG,\bP]\to [\h_I,\bP]\simeq\bP.
$$
We thus obtain a $\cV$-functor
$$
E:\bO\Mod\to \bG\da\bO\Mod
$$
and a $\cV$-natural transformation
$$
\varepsilon:E\to \id.
$$
\end{definition}

\begin{lemma}\label{E-radj-forg}
The $\cV$-functor $E$ is an enriched right adjoint to the forgetful functor $\bG\da\bO\Mod\to\bO\Mod$. More explicitly, the $\cV$-natural transformation $\varepsilon:E\to\id$ induces an isomorphism
$$
\bG\da\bO\Mod(\bM,E(\bP))\simeq \bO\Mod(\bM,\bP),
$$
natural in $\bM\in \bG\da\bO\Mod$ and $\bP\in\bO\Mod$.
\end{lemma}
\begin{proof}
The counit of the adjunction is (modulo a slight imprecision of omitting a symbol for the forgetful functor $F$) is given by $\epsilon$. The unit $$\eta:\id\to E\circ F$$ is obtained by taking, given $\bM\in\bG\da\bO\Mod$, the map $\eta_\bM:\bM\to E(F(\bM))=[\bG,\bM]$, which corresponds to the action $\bG\times\bM\to\bM$ via the isomorphism
$$
\hC(\bM,[\bG,\bM])\simeq\hC(\bG\times\bM,\bM).
$$
\end{proof}

\subsection{Enriched group module presheaves have enough injectives}

Let $\cC$ be enriched over a complete cartesian closed category $\cV$, let $\hC=[\cC^\op,\cV]$ be the $\cV$-category of $\cV$-presheaves on $\cC$, let $\bO$ be a $\hC$-ring, and let $\bG$ be a $\hC$-group.

\begin{proposition}\label{enough-G-enr-inj}
If $\bI$ is an enriched injective object in $\bO\Mod$, then $E(\bI)$ is an enriched injective object in $\bG\da\bO\Mod$. Consequently, if $\bO\Mod$ has enough enriched injectives, then $\bG\da\bO\Mod$ has enough enriched injectives.
\end{proposition}
\begin{proof}
Suppose $\bI$ is an $\bO\Mod$-injective object. We shall verify that $E(\bI)$ is $\bG\da\bO\Mod$-injective. Indeed, let $0\to\bA\to\bB\to\bC\to0$ be an exact sequence in $\bG\da\bO\Mod$. Then it is also exact as a sequence in $\bO\Mod$, so, by assumption,
the sequence
$$
0\to\bO\Mod(\bC,\bI)\to\bO\Mod(\bB,\bI)\to\bO\Mod(\bA,\bI)\to 0
$$
is exact. By adjunction, this sequence corresponds to
$$
0\to\bG\da\bO\Mod(\bC,E(\bI))\to\bG\da\bO\Mod(\bB,E(\bI))\to\bG\da\bO\Mod(\bA,E(\bI))\to 0,
$$
as required. An analogous proof, using the underlying adjunction of \ref{E-radj-forg}, shows that, if $\bI$ is $(\bO\Mod)_0$-injective, then it is also $(\bG\da\bO\Mod)_0$-injective.

Let $\bM\in\bG\da\bO\Mod$ be arbitrary. The forgetful functor gives its underlying $\bO$-module $F(\bM)$. If $\bO\Mod$ has enough enriched injectives, there exists an enriched injective $\bI\in\bO\Mod$ and a monomorphism $F(\bM)\to\bI$. Then the composite
$$
\bM\xrightarrow{\eta_M} E(F(\bM))\to E(\bI)
$$
is a monomorphism into an enriched injective $E(\bI)$ in $\bG\da\bO\Mod$.
\end{proof}

\begin{proposition}\label{tens-groth-enr-inj}
If $\cV$ is a Grothendieck topos and $\cC$ is a small tensored $\cV$-category with a terminal object, then $\bO\Mod$ has enough enriched injectives.
\end{proposition}
\begin{proof}
By \ref{Gamma-exact}, $\Gamma:\hC\to\cV$ is exact, so \ref{exact-enr-inj} entails that $\bO\Mod$ has enough enriched injectives.
\end{proof}

\begin{corollary}\label{groth-enough-enr-inj}
With the hypothesis of \ref{tens-groth-enr-inj}, the $\cV$-abelian category $\bG\da\bO\Mod$ has enough enriched (and internal) injectives.
\end{corollary}

\subsection{Enriched group cohomology}\label{enr-gp-coh}

Let $\cV$ be a complete cartesian closed category, let $\cC$ be a $\cV$-category, and let $\hC$ be the $\cV$-category of $\cV$-presheaves on $\cC$. We write $\Gamma=\hC(\be,\mathord{-})$ for the global section 
$\cV$-functor $\hC\to \cV$ considered in \ref{enrich-global-sect}.

Let $\bG$ be a $\hC$-group, $\bO$ a $\hC$-ring and $\bF$ a $\bG\da\bO$-module. For $n\geq 0$, we let
$$
\Ch^n(\bG,\bF)=\hC(\bG^n,\bF)\in\cV, \ \ \ \uC^n(\bG,\bF)=[\bG^n,\bF]\in\hC,
$$
with convention that $\bG^0$ is the final object $\be$. 

The objects $\uC^n(\bG,\bF)$ (resp.\ $\Ch^n(\bG,\bF)$) are naturally $\bO$-modules (resp.\ $\Gamma(\bO)$-modules), and we have that
$$
\Ch^n(\bG,\bF)\simeq\Gamma(\uC^n(\bG,\bF)).
$$
The coboundary operator
$$
\partial:\uC^n(\bG,\bF)\to\uC^{n+1}(\bG,\bF)
$$
is defined as
$$
\partial=\mu_n+\sum_{i=1}^{n}(-1)^i m_i^*+(-1)^{n+1}p_n^*,
$$
where
$$
p_n:G^{n+1}=G^n\times G\to G^n
$$
is the projection,
$$
m_i=\id_{G^{i-1}}\times m\times\id_{G^{n-i}}:G^{n+1}\to G^n,
$$
and $\mu_n$ is obtained by adjunction from the composite
$$
\bG^{n+1}\times[\bG^n,\bF]\simeq \bG\times \bG^n\times[\bG^n,\bF]\xrightarrow{\id\times\text{ev}}\bG\times\bF\xrightarrow{\mu}\bF.
$$

We have that
$$
\partial\circ\partial=0,
$$
so we have obtained a complex  of $\bO$-modules
$$\uC^{*}(\bG,\bF),$$
called the enriched \hoch (or standard) complex.

Similarly we define a complex of $\cV$-abelian groups (and even $\Gamma(\bO)$-modules), 
$$
\Ch^{*}(\bG,\bF)=\Gamma(\uC^{*}(\bG,\bF)).
$$
We can view the standard complex as a $\Vab$-functor
$$
\uC^{*}(\bG,\mathord{-}):\bG\da\bO\Mod\to \ch(\bO\Mod), 
$$
and 
$$
\Ch^{*}(\bG,\mathord{-})=\Gamma\circ\uC^{*}(\bG,\mathord{-}):\bG\da\bO\Mod\to \ch(\Gamma(\bO)\Mod).
$$

The \emph{enriched \hoch cohomology groups} 
$$
\HH^n(\bG,\bF)\in\cV\da\Ab
$$
are defined as cohomology groups of the complex $\Ch^{*}(\bG,\bF)$, while the $\hC$-groups
$$
\uH^n(\bG,\bF)
$$
are the cohomology groups of the complex $\uC^{*}(\bG,\bF)$. We can view cohomology groups as $\Vab$-functors
\begin{align*}
\uH^n(\bG,\mathord{-})&=\coh^n\circ\uC^{*}(\bG,\mathord{-}):\bG\da\bO\Mod\to \bO\Mod\\
\HH^n(\bG,\mathord{-})&=\coh^n\circ\Ch^{*}(\bG,\mathord{-}):\bG\da\bO\Mod\to \Gamma(\bO)\Mod.
\end{align*}

In particular,
\begin{align*}
\uH^0(\bG,\bF)& =\bF^{\bG}=\Ker(\partial^0:\bF\to[\bG,\bF])\simeq \bG\da\bO\Mod[\bO,\bF],  \text{ and }\\
 \HH^0(\bG,\bF) &=\Gamma(\bF^{\bG})=\Ker(\partial^0:\bF\to\hC(\bG,\bF))\simeq \bG\da\bO\Mod(\bO,\bF).
\end{align*}

\begin{proposition}\label{enr-der-of-fix}
Suppose that $\cV$ is bicomplete, $\cC$ is a small $\cV$-tensored category with finite products, and that $\bO\Mod$ has enough enriched (resp.~internal) injectives. 
Let $\bG$ be  a representable $\hC$-group. 
Then the functors $\HH^n(\bG,\mathord{-})$ (resp.\ $\uH^n(\bG,\mathord{-})$) are the derived enriched functors of the left exact functor $\HH^0(\bG,\mathord{-})$ (resp.\ $\uH^0(\bG,\mathord{-})$) on the category $\GOmod$, i.e., 
$$
\HH^n(\bG,\mathord{-})\simeq\Ext^n_{\bG\da\bO\Mod}(\bO,\mathord{-})\ \ \ (\text{resp. }
\uH^n(\bG,\mathord{-})\simeq\uExt^n_{\bG\da\bO}(\bO,\mathord{-})).
$$
\end{proposition}

\begin{proof}
Note that \ref{enough-G-enr-inj} implies that $\GOmod$ has enough enriched injectives, so
it suffices to prove that the functors $\HH^n(\bG,\mathord{-})$ (resp.\ $\uH^n(\bG,\mathord{-})$) are enriched cohomological functors effaceable in degrees $n>0$.

The assumptions that $\bG=\h_G$ and that $\cC$ has products entail that the $\Vab$-functor 
$$\uC^*(\bG,\mathord{-}):\bG\da\bO\Mod\to\ch(\bO\Mod)$$
is exact. Indeed, if
$$
0\to\bF'\to\bF\to\bF''\to0
$$
is an exact sequence of $\bG\da\bO$-modules, then the sequence of $\bO$-modules
$$
0\to\uC^n(\h_G,\bF')\to\uC^n(\h_G,\bF)\to\uC^n(\h_G,\bF'')\to 0,
$$
evaluated at $S\in\cC$ is
$$
0\to\bF'(G^n\times S)\to\bF(G^n\times S)\to\bF''(G^n\times S)\to 0,
$$
which is exact. This shows that $\uH^*(\bG,\mathord{-})$ is an enriched cohomological functor. Given that $\Ch^*(G,\mathord{-})$ is obtained from $\uC^*(G,\mathord{-})$ by applying the functor $\Gamma$, and that operation is exact by \ref{Gamma-exact}, we obtain the same result for $\HH^*(\bG,\mathord{-})$.

It remains to verify that the cohomology of induced modules vanishes, i.e. for $\bP\in\bO\Mod$ and $n>0$,
$$
\HH^n(\bG,E(\bP))=0 \ \ \ \text{and}\ \ \ \  \uH^n(\bG,E(\bP))=0,
$$
and we will do this by exhibiting that the chains $\Ch^*(\bG,E(\bP))$ and $\uC^*(\bG,E(\bP))$ are null-homotopic.  For $n\geq 0$, we let
$$
\gamma:\Ch^{n+1}(\bG,E(\bP))\to \Ch^n(\bG,E(\bP))
$$
be the morphism obtained through the isomorphism
\begin{multline*}
\hC([\bG_0\times\bG^n,[\bG_1,\bP]],[\bG^n,[\bG_0,\bP]])\simeq
\hC(\bG^n\times[\bG_0\times\bG^n,[\bG_1,\bP]],[\bG_0,\bP])\\
\simeq
\hC(\bG_0\times\bG^n\times[\bG_0\times\bG^n,[\bG_1,\bP]],\bP),
\end{multline*}
applied to the composite 
\begin{multline*}
\bG_0\times\bG^n\times[\bG_0\times\bG^n,[\bG_1,\bP]]\simeq 
\bG_0\times\bG^n\times[\bG_0\times\bG^n\times\bG_1,\bP]\\
\xrightarrow{\id\times e_{G_1}}
\bG_0\times\bG^n\times\bG_1\times[\bG_0\times\bG^n\times\bG_1,\bP]\xrightarrow{\text{ev}}\bP,
\end{multline*}
and taking $\bG_0=\bG_1=\bG$. A substantial but routine verification shows that
$$
\partial\gamma+\gamma\partial=\id,
$$
as required.
\end{proof}
The following is immediate from \ref{tens-groth-enr-inj} (or \ref{groth-enough-enr-inj}) and \ref{enr-der-of-fix}.
\begin{corollary}\label{enr-coh-is-top-coh}
Suppose that $\cV$ is a Grothendieck topos and $\cC$ is small $\cV$-tensored with finite products. 
Let $\bG$ be  a representable $\hC$-group. 
Then the functors $\HH^n(\bG,\mathord{-})$ (resp.\ $\uH^n(\bG,\mathord{-})$) are the derived enriched functors of the left exact functor $\HH^0(\bG,\mathord{-})$ (resp.\ $\uH^0(\bG,\mathord{-})$) on the category $\GOmod$, and they coincide with the relative and internal group (topos) cohomology from \ref{rel-gp-coh} and \ref{gp-coh-top},
\begin{align*}
\HH^n(\bG,\mathord{-})&\simeq \tH^n_{\hC_{\ov\cV}}(\bG,\mathord{-})\simeq\Ext^n_{\bG\da\bO\Mod}(\bO,\mathord{-})\\
\uH^n(\bG,\mathord{-})&\simeq\uExt^n_{\bG\da\bO}(\bO,\mathord{-}).
\end{align*}
\end{corollary}

\subsection{Group cohomology of the associated presheaves}\label{coh-assoc}

With the notation of the previous section, let $\widehat{\cC_0}$ be the category of presheaves on the underlying category $\cC_0$ of $\cC$. Let
\begin{enumerate}
\item $\assoc{\bG}$ be the $\widehat{\cC_0}$-group associated with the $\Vab$-group $\bG$;
\item $\assoc{\bO}$ be a $\widehat{\cC_0}$-ring associated with $\bO$;
\item $\assoc{\bF}$ be the $\assoc{\bG}\da\assoc{\bO}$-module associated with the $\bG\da\bO$-module $\bF$.
\end{enumerate}
In \ref{group-coh}, we defined the (ordinary) group presheaf cohomology groups
$$
\HH^n(\assoc{\bG},\assoc{\bF})\ \ \ \text{ and }\ \ \ \ \uH^n(\assoc{\bG},\assoc{\bF})
$$
as the cohomology groups of standard complexes $\Ch^*(\assoc{\bG},\assoc{\bF})$ and $\uC^*(\assoc{\bG},\assoc{\bF})$. 

Suppose that $\bG$ is representable. Using \ref{inthom-assoc-psh}, we see that
$$
\Ch^*(\assoc{\bG},\assoc{\bF})=V(\Ch^*(\bG,\bF))\ \ \ \text{ and }\ \ \ \ \uC^*(\assoc{\bG},\assoc{\bF})=(\uC^*(\bG,\bF))\assoc{\vphantom{)}}.
$$
Using the fact that the functors $V:\Gamma(\bO)\Mod\to V(\Gamma(\bO))\Mod$ and
$\assoc{V}=\assoc{(\mathord{-})}:\bO\Mod\to \assoc{\bO}\Mod$ are left exact, we see that
\begin{multline*}
\HH^0(\assoc{\bG},\assoc{\bF})=\Ker(\partial^0:\Ch^0(\assoc{\bG},\assoc{\bF})\to\Ch^1(\assoc{\bG},\assoc{\bF}))
=\Ker(V(\partial^0:\Ch^0(\bG,\bF)\to\Ch^1(\bG,\bF)))\\=
V(\Ker(\partial^0:\Ch^0(\bG,\bF)\to\Ch^1(\bG,\bF)))=V(\HH^0(\bG,\bF)),
\end{multline*}
and similarly
$$
\uH^0(\assoc{\bG},\assoc{\bF})=(\uH^0(\bG,\bF))\assoc{\vphantom{(}}=\assoc{V}(\uH^0(\bG,\bF)).
$$
Thus,
$$
\HH^0(\assoc{\bG},\mathord{-})=V\circ\HH^0(\bG,\mathord{-})\ \ \ \text{and}\ \ \ \uH^0(\assoc{\bG},\mathord{-})=\assoc{V}\circ\uH^0(\bG,\mathord{-}),
$$
so, in the context of \ref{der-of-fix} and \ref{enr-der-of-fix}, by the Grothendieck spectral sequence for the composite of two functors, we obtain spectral sequences
$$
E_2^{p,q}(\bF)=\der^pV (\HH^q(\bG,\bF))\Rightarrow \HH^{p+q}(\assoc{\bG},\assoc{\bF}).
$$
and
$$
\underline{E}_2^{p,q}(\bF)=\der^p\assoc{V} (\uH^q(\bG,\bF))\Rightarrow \uH^{p+q}(\assoc{\bG},\assoc{\bF}).
$$

\begin{lemma}
Suppose that $\der^pV=0$ and $\der^p\assoc{V}=0$ for $p>1$. Then the spectral sequences degenerate (\cite[Exercise~5.2.1]{weibel}) and we obtain exact sequences
$$
0\to \der^1V(\HH^{n-1}(\bG,\bF))\to \HH^n(\assoc{\bG},\assoc{\bF})\to V(\HH^n(\bG,\bF))\to 0,
$$
and
$$
0\to \der^1\assoc{V}(\uH^{n-1}(\bG,\bF))\to \uH^n(\assoc{\bG},\assoc{\bF})\to \assoc{V}(\uH^n(\bG,\bF))\to 0.
$$
\end{lemma}

\part[{$\sigma$GA}]{Difference algebraic geometry}

\section{Difference categories}\label{diff-cats}

\subsection{Monoids of operators}\label{mon-of-ops}

If $G$ is a small monoid, considered a category, there are unique functors 
$\begin{tikzcd}[cramped, sep=small, ampersand replacement=\&] 
{\mathbf{1}}\arrow[r,yshift=2pt,"i"]  \&{G} \arrow[l,yshift=-2pt,"r"]
\end{tikzcd}$. For a complete category $\cC$, they induce functors 
$$
\begin{tikzcd}[cramped, column sep=normal, ampersand replacement=\&] 
{[G,\cC]}\arrow[r,yshift=2pt,"i^*"]  \&{\cC} \arrow[l,yshift=-2pt,"r^*"]
\end{tikzcd}
$$
satisfying $$i^*\circ r^*=\id_\cC.$$
As in \cite[V.2]{lawvere-thesis} (or using our discussion from \ref{funct-int-presh}), we obtain adjoints
$$
i_!\dashv i^* \dashv i_*, \ \ \ \ r_!\dashv r^*\dashv r_*,
$$
satisfying
$$
r_!\circ i_!=\id_\cC, \ \ \ 
r_*\circ i_*=\id_\cC, \ \ \ 
r_!\circ r^*=\id_\cC, \ \ \
r_*\circ r^*=\id_\cC.
$$
Moreover, $r_!=\varinjlim_G$, and, for $X\in[G,\cC]$, 
$$
r_!(X)=X/M
$$
is the `orbit object', equipped with a regular epimorphism $X\to r^*(X/M)$, while
$r_*=\varprojlim_G$ and
$$
r_*(X)=M\backslash X=X^M
$$
is the `fixed object', equipped with a regular monomorphism $r^*(M\backslash X)\to X$.

\subsection{Topos of monoid actions}\label{monoid-topos}

Let $G$ be a monoid, considered as a category with a single object $o$. The classifying topos
$$\bB G\simeq[G,\Set]=[(G^\op)^\op,\Set]$$ 
of left $G$-actions (discussed in \ref{monoid-action} and \ref{module-ops-day-conv}) is  a Grothendieck topos, as a category of presheaves on $G^\op$. 

The pair of adjoint triples established in \ref{mon-of-ops} witness the fact that the global section geometric morphism
$$
\bB G\to \Set
$$
admits a right inverse.   

Limits in $\bB G$ are computed component-wise (for underlying sets), and the terminal object $e$ is a point with a trivial $G$-action. A morphism is a monomorphism (resp.\ epimorphism), if it underlying function is injective (resp.\ surjective). 

By general principles (\cite[I.4]{maclane-moerdijk}), the subobject classifier $\Omega$ is isomorphic to the set of sieves on the object $o$, which corresponds to the set of left ideals of $G$, with the action 
$$G\times\Omega\to\Omega, \ \ \ (g,R)\mapsto \{h\in G: hg\in R\}.$$  

The `truth' map 
$$
t:e\to\Omega
$$ 
sends the only point of $e$ to the maximal left ideal in $G$.

The internal hom/exponential of $Y,Z\in \bB G$ is given as
$$
[Y,Z]=\bB G(G\times Y,Z),
$$
where $G$ is considered a left $G$-module through its monoid operation. 

Indeed, a functor $X\in[G,\Set]$ is identified with the set $X(o)\in \bB G$, so it suffices to evaluate
$$
[Y,Z](o)=[G,\Set](\h_o\times Y,Z),
$$
and the representable functor $\h_o$ identifies with the left $G$-module $G$.

Note, if $G$ is commutative, using the by the closedness of $(\bB G,\times_G)$, we also get
$$
\bB G(X,[Y,Z])=\bB G(X,\bB G(G\times Y,Z))\simeq \bB G(X\times_G(G\times Y),Z)\simeq \bB G(X\times Y,Z).
$$

We will study the special case of $\bB\N$ in more detail in \ref{topos-diff-sets}.

\subsection{Ordinary difference categories}\label{dif-cats}

Let $\bsi$ be a category with a single object $o$ and $\Hom_{\bsi}(o,o)$ consists of composites of a single endomorphism $\sigma:{o}\to{o}$. In other words, as a monoid,
$$
\End_{\bsi}(o)\simeq \N. 
$$

Let $\cC$ be a category. Its \emph{difference category} is defined as the functor category
$$
\diff\cC=\Diff(\cC)=[\bsi,\cC].
$$

More explicitly, $\diff\cC$ has objects of form $(X,\sigma)$, where $X\in\Ob(\cC)$, and $\sigma\in\cC(X,X)$. A morphism $f:(X,\sigma_X)\to (Y,\sigma_Y)$ is a commutative diagram
 \begin{center}
 \begin{tikzpicture} 
\matrix(m)[matrix of math nodes, row sep=2em, column sep=2em, text height=1.5ex, text depth=0.25ex]
 {
 |(1)|{X}		& |(2)|{Y} 	\\
 |(l1)|{X}		& |(l2)|{Y} 	\\
 }; 
\path[->,font=\scriptsize,>=to, thin]
(1) edge node[above]{$f$} (2) edge node[left]{$\sigma_X$}   (l1)
(2) edge node[right]{$\sigma_Y$} (l2) 
(l1) edge  node[above]{$f$} (l2);
\end{tikzpicture}
\end{center}
in $\cC$, i.e., $f\in\Hom_\cC(X,Y)$ such that $f\circ\sigma_X=\sigma_Y\circ f$.

If $e$ is a terminal object of $\cC$, then $(e,\id)$ is a terminal object of $\diff\cC$. 

If $\cC$ has  products (resp.\ fibre products), so does $\diff\cC$. Indeed, if $(X,\sigma_X)$ and $(Y,\sigma_Y)$ are objects of $\diff\cC$ (resp.\ of $\diff\cC_{\ov (S,\sigma)})$, then 
$$
(X\times Y,\sigma_X\times\sigma_Y)\ \ \ \text{resp. }\ \ \ (X\times_S Y,\sigma_X\times\sigma_Y)
$$ is their product in $\diff\cC$ (resp.\ in $\diff\cC_{\ov (S,\sigma)}$).

We have a natural embedding of categories
$$
\I:\cC\to\diff\cC,\ \ \ \ \ \I(X)=(X,\id),
$$
as well as the forgetful functor $\forg{\, }:\diff\cC\to\cC$,
$$
\forg{(X,\sigma)}=X\in\Ob(\cC),
$$
satisfying 
$$
\forg{\,}\circ\I=\id_\cC.
$$

\subsection{Difference category of a monoidal closed category}\label{diff-mon-closed}

Let $(\cV,\otimes,I)$ be a complete symmetric monoidal closed category. If $\cC$ is a small category, it is a classical fact that the functor category
$$
[\cC,\cV]
$$
is monoidal closed with the argument-wise tensor product.

\begin{corollary}
The difference category $$\diff\cV=[\bsi,\cV]$$ is symmetric monoidal closed.
\end{corollary}

Although the statement  is a direct consequence of the general principle stated above, we give an explicit construction of internal homs. Let $X,Y,Z\in\diff\cV$, and denote $X_0=\forg{X}$, $Y_0=\forg{Y}$, $Z_0=\forg{Z}$. Since $\cV$ is monoidal closed, we have an isomorphism
$$
\alpha:\cV(X_0\otimes Y_0,Z_0)\to\cV(X_0,[Y_0,Z_0])
$$
natural in all three variables. In particular, for $\cV$-morphisms $\sigma_X:X_0\to X_0$, $\sigma_Y:Y_0\to Y_0$, $\sigma_Z:Z_0\to Z_0$, we have commuting endomorphisms
$$
\sigma_X^*, \sigma_Y^*, \sigma_{Z*}:\cV(X_0\otimes Y_0,Z_0)\to\cV(X_0\otimes Y_0,Z_0)
$$
given by
\begin{align*}
\sigma_X^*(u) & =u\circ(\sigma_X\otimes\id_Y),\\
\sigma_Y^*(u) & =u\circ(\id_X\otimes\sigma_Y),\\
\sigma_{Z*}(u) & =\sigma_Z\circ u.
\end{align*}
Similarly, we have commuting endomorphisms 
$$
\bar{\sigma}_X^*,\bar{\sigma}_{Y*},\bar{\sigma}_{Z*}:\cV(X_0,[Y_0,Z_0])\to\cV(X_0,[Y_0,Z_0]),
$$
defined by 
\begin{align*}
\bar{\sigma}_X^*(v) & =v\circ\sigma_X,\\
\bar{\sigma}_{Y*}(v) & =[\sigma_Y,\id_Z]\circ v,\\
\bar{\sigma}_{Z*}(v) & =[\id_Y,\sigma_Z]\circ v.
\end{align*}
Naturality states that 
$$
\bar{\sigma}_X^{*i}\bar{\sigma}_{Y*}^j\bar{\sigma}_{Z*}^k\alpha(u)=
\alpha(\sigma_X^{*i}\sigma_Y^{*j}\sigma_{Z*}^k u)
$$
for all $i,j,k\geq0$.
Let $$[Y_0,Z_0]^\N=\prod_{i\in\N}[Y_i,Z_i],$$ where $Y_i$ (resp.\ $Z_i$) are isomorphic copies of $Y_0$ (resp.\ $Z_0$). Consider the endomorphisms of $[Y_0,Z_0]^\N$ defined by
\begin{align*}
\tilde{\sigma}_Y & =\prod_{i\in\N}[\sigma_Y,\id_Z],\\
\tilde{\sigma}_Z & =\prod_{i\in\N}[\id_Y,\sigma_Z],
\end{align*}
as well as the shift $s$ given as the composite
$$
\prod_{i\geq0}[Y_i,Z_i]\xrightarrow{\text{proj}}\prod_{i\geq 1}[Y_i,Z_i]\xrightarrow{\prod s_i} \prod_{i\geq0}[Y_i,Z_i],
$$
where $s_i:[Y_{i+1},Z_{i+1}]\to[Y_i,Z_i]$ is the identity.

We define 
$$
[Y,Z]=(\text{eq}(s\circ\tilde{\sigma}_Y,\tilde{\sigma}_Z),s).
$$
Note that $s$ restricts to $[Y,Z]$ because it commutes with $\tilde{\sigma}_Y$ and $\tilde{\sigma}_Z$.

We claim that 
$$
\diff\cV(X\otimes Y,Z)\simeq\diff\cV(X,[Y,Z]).
$$
Suppose $\phi\in\diff\cV(X\otimes Y,Z)$, and let 
$$\theta_i=\alpha(\sigma_X^{*i}\phi)=\bar{\sigma}_X^{*i}\alpha(\phi):X\to[Y_i,Z_i],$$
and
$$
\theta_\phi=\theta=\prod_{i\in\N}\theta_i:X\to[Y_0,Z_0]^\N.
$$
By definition, $\bar{\sigma}_X^*\theta_i=\theta_i\circ\sigma_X=\theta_{i+1}$,
so $\theta$ satisfies
$$
\theta\circ\sigma_X=s\circ\theta.
$$
By assumption, $\phi\circ(\sigma_X\otimes\sigma_Y)=\sigma_Z\circ\phi$, i.e., 
$\sigma_X^*\sigma_Y^*\phi=\sigma_{Z*}\phi$, so
$$
\bar{\sigma}_Y^*\theta_{i+1}=\bar{\sigma}_Y^*\alpha(\sigma_X^{*(i+1)}\phi)=
\alpha(\sigma_Y^*\sigma_X^*\sigma_X^{*i}\phi)=\alpha(\sigma_{Z*}\sigma_X^{*i}\phi)=\bar{\sigma}_{Z*}\theta_i.
$$
In other words, 
$$
s\circ\tilde{\sigma}_Y\circ\theta=\tilde{\sigma}_Z\circ\theta,
$$
so $\theta$ lands in $[Y,Z]$.

Conversely, if $\theta\in\diff\cV(X,[Y,Z])$, define 
$$\phi=\phi_\theta=\alpha^{-1}(\theta_0).$$ The assignments $\phi\mapsto\theta_\phi$ and $\theta\mapsto\phi_\theta$ are mutually inverse.

\subsection{Difference category of a topos}\label{diff-topos}

Let $\cE$ be an elementary topos with a natural number object $N$. Then $\bsi$ is an $\cE$-internal category with the object of objects $I$ and the object of morphisms $N$, so
$$
\diff\cE=[\bsi,\cE]
$$
can be identified with a category of internal presheaves, and, by the results of \ref{comon-coalg}, it is again a topos. Moreover, if $\cE$ is a Grothendieck topos, so is $\diff\cE$, and most considerations of \ref{topos-diff-sets} generalise to this context.

\subsection{Difference category of a presheaf category}

The situation of \ref{diff-topos} becomes explicit for the case $\cE=\hC$, where $\cC$ is a small category.
We observe that
\begin{align*}
\diff\hC & \simeq[\bsi,[\cC^\op,\Set]]\simeq[\bsi\times\cC^\op,\Set]  \\
& \simeq [\cC^\op\times\bsi,\Set]
\simeq[\cC^\op,[\bsi,\Set]] \simeq[\cC^\op,\diff\Set].
\end{align*}

Indeed, an object $(\bF,\sigma)\in\diff\hC$ consists of a presheaf $\bF\in\hC$ and a natural transformation $\sigma:\bF\to\bF$. To each $X\in\cC$ we can assign an object $(\bF(X),\sigma_X)\in\diff\Set$, and to each morphism 
$f\in\cC(X,Y)$ we can assign a morphism
$$
\bF(f)\in{\diff\Set}((\bF(Y),\sigma_Y),(\bF(X),\sigma_X)).
$$
Thus, starting from an object in $\diff\hC$, we can define a functor $\cC^{\circ}\to \diff\Set$ and vice versa, yielding an equivalence of categories
$\diff\hC$ and $\mathbf{Hom}(\cC^\circ,\diff\Set)$.


\subsection{Forgetting and recovering the difference structure}\label{forget-and-adj}

In \ref{mon-of-ops}, we considered adjoints of the forgetful functor $[G,\cC]\to \cC$ for an arbitrary monoid $G$ and a suitable category $\cC$. Here, we work out the details for the case $G=\N$ and a category $\cC$ admitting (countable) coproducts, i.e., we discuss adjoints of the forgetful functor
$$
\forg{\,}:\diff\cC\to\cC.
$$
In particular, this applies to the forgetful functor $\bB\N\simeq\diff\Set\to\Set$.

We define a functor $\ssig{\,}:\cC\to\diff\cC$
as follows. For $X\in\Ob(\cC)$, let 
$$
\ssig{X}=\coprod_{i\in\N}X_i,
$$
where each $X_i$ is an isomorphic copy of $X$, together with
$\sigma:\ssig{X}\to \ssig{X}$ which takes $X_i$ identically to $X_{i+1}$. 
There is a natural morphism $\iota_X: X\to \forg{\ssig{X}}$ mapping $X$ to $X_0$. For any $Z\in\diff\cC$, the natural morphism
$$
{\diff\cC}(\ssig{X},Z)\to{\cC}(X,\forg{Z})
$$
is bijective, i.e., $\ssig{\,}$ is left adjoint to $\forg{\,}$.

If $f\in {\diff\cC}(\ssig{X},Z)$, then clearly $f_0=\forg{f}\circ\iota_X\in {\cC}(X,\forg{Z})$. Conversely, if $f_0\in{\cC}(X,\forg{Z})$, there exists a unique $f\in {\diff\cC}(\ssig{X},Z)$ given on $X_i$ by $\sigma_Z^i\circ f_0$. The above assignments are mutually inverse.

Let $S\in\Ob(\diff\cC)$ and consider a forgetful functor $\forg{\,}_S=\forg{\, }:\diff\cC_{\ov S}\to\cC_{\ov\forg{S}}$.
Assume that $\cC$ admits (countable) fibre products. 

Let $f_0:X\to\forg{S}$ be an object of $\cC_{\ov\forg{S}}$. Write $f_i:X_i\to\forg{S}$ for the base change 
of $f_0$ along $\forg{\varsigma^i}:\forg{S}\to\forg{S}$, and let $\sigma_i:X_{i+1}\to X_i$ be the induced morphism satisfying
$$
f_i\circ\sigma_i=\varsigma\circ f_{i+1}.
$$ 
Consider the fibre product, i.e., the product in ${\cC_{\ov\forg{S}}}$
$$
\psig{X}_{\ov S}=\prod_{i\in\N}f_i=\prod_{i\in\N}X_i/S,
$$
together with a morphism
$\sigma:\psig{X}_{\ov S}\to\psig{X}_{\ov S}$ defined as the composite 
$$
\prod_{i\geq0}X_i/S\xrightarrow{\text{proj}}\prod_{i\geq1}X_i/S\xrightarrow{\prod\sigma_i}\prod_{i\geq0}X_i/S.
$$
This construction gives a functor
$$
\psig{\,}_{\ov S}:\cC_{\ov \forg{S}}\to\diff\cC_{\ov S}
$$
which is right adjoint to $\forg{\,}_S$. Indeed, the $\cC_{\ov \forg{S}}$-projection $\pi_X:\forg{\psig{X}}\to X$ induces a natural bijection
$$
{\diff\cC_{\ov S}}(Z,\psig{X}_{\ov S})\to {\cC_{\ov\forg{S}}}(\forg{Z},X).
$$
For $f\in{\diff\cC_{\ov S}}(Z,\psig{X}_{\ov S})$, the morphism $\pi_X\circ\forg{f}$ is in ${\cC_{\ov\forg{S}}}(\forg{Z},X)$.

Conversely, let $f_0\in{\cC_{\ov\forg{S}}}(\forg{Z},X)$ and let 
$f_i=\forg{\bar{\sigma}_Z^i}\circ f_{0,\forg{\varsigma^i}} \in {\cC_{\ov\forg{S}}}(\forg{Z},X_i)$. We have that
$$
\sigma_{i}\circ f_{i+1}=f_i\circ\sigma_Z,
$$
whence $f=\prod_i f_i\in{\diff\cC_{\ov S}}(Z,\psig{X}_{\ov S})$. The above constructions are mutually inverse. 

Analogously, when $\cC$ has (countable) products, we define a functor
$$
\psig{\,}:\cC\to\diff\cC,
$$
which is right adjoint to $\forg{\,}$. For $X\in\Ob(\cC)$, we let 
$$
\psig{X}=\prod_{i\in\N}X_i,
$$
where each $X_i$ is an isomorphic copy of $X$, together with the `shift' $\sigma:\psig{X}\to\psig{X}$ defined as the composite
$$
\prod_{i\geq0}X_i\xrightarrow{\text{proj}}\prod_{i\geq1}X_i\xrightarrow{\prod\sigma_i}\prod_{i\geq0}X_i,
$$
where $\sigma_i$ is the identity $X_{i+1}\to X_i$.

Clearly, if $\diff\cC$ has a terminal object $e$, then $\psig{X}=\psig{X}_{\ov e}$.


\subsection{Fixed points and quotients}\label{fix-quo}

Let $\cC$ be a category. 
The adjoints of the `constant object' functor $\cC\to [G,\cC]$ were discussed in \ref{mon-of-ops}, and we provide their explicit constructions in the case of $\diff\cC$ below.

The \emph{fixed point object} of $X\in\Ob(\diff\cC)$, is the equaliser $X_0$ 
\begin{center}
 \begin{tikzpicture} 
\matrix(m)[matrix of math nodes, row sep=0em, column sep=1.7em, text height=1.5ex, text depth=0.25ex]
 {
|(0)|{X_0} & |(1)|{\forg{X}}		& |(2)|{\forg{X}} 	\\
 }; 
\path[->,font=\scriptsize,>=to, thin,yshift=12pt]
(0) edge node[above]{} (1)
([yshift=2pt]1.east) edge node[above]{$\id$} ([yshift=2pt]2.west) 
([yshift=-2pt]1.east)edge node[below]{$\sigma$}   ([yshift=-2pt]2.west) 
;
\end{tikzpicture}
\end{center}
of the pair $(\id,\forg{\sigma})$ of morphisms in $\cC$, if it exists. Equivalently, using the difference language, it is
an object $X_0\in\Ob(\cC)$, equipped with a morphism
$$
i\in{\diff\cC}(\I(X_0),X),
$$ 
such that, for any $Z\in\Ob(\cC)$ and $f\in {\diff\cC}(\I(Z),X)$, there is a unique morphism $h\in\cC(Z,X_0)$ satisfying
$$
f=i\circ\I(h).
$$
If a fixed point object for $X$ exists, it is unique up to an (unique) isomorphism and we denote it by
$$
\Fix(X).
$$
Dually, a \emph{quotient object} for $X\in\Ob(\diff\cC)$ is the coequaliser $\bar{X}$ of the pair $(\id,\forg{\sigma})$ in $\cC$. Equivalently, it is an object $\bar{X}\in\Ob(\cC)$, equipped with a morphism
$$
q\in{\diff\cC}(X,\I(\bar{X})),
$$ 
such that, for any $Z\in\Ob(\cC)$ and $f\in {\diff\cC}(X,\I(Z))$, there is a unique morphism $h\in\cC(\bar{X},Z)$ satisfying
$$
f=\I(h)\circ q.
$$
If a quotient object for $X$ exists, it is unique up to an (unique) isomorphism and we denote it by
$$
\Quo(X).
$$
If all the fixed point objects in $\diff\cC$ exist (for instance, if equalisers in $\cC$ exist), the above universal property shows that the functor $\Fix:\diff\cC\to\cC$ satisfies
$$
{\diff\cC}(\I(Z),X)\simeq\cC(Z,\Fix(X)),
$$
i.e., $\Fix$ is right adjoint to $\I$.

Dually, if all the quotient objects in $\diff\cC$ exist (for instance, if $\cC$ has coequalisers), the functor $\Quo:\diff\cC\to\cC$ satisfies
$$
{\diff\cC}(X,\I(Z))\simeq\cC(\Quo(X),Z),
$$
i.e., $\Quo$ is left adjoint to $\I$.

\subsection{Moving between a category and its difference category}\label{cat--diff-cat}

To summarise \ref{forget-and-adj}, the diagram of categories and functors
\begin{center}
 \begin{tikzpicture} 
 [cross line/.style={preaction={draw=white, -,
line width=3pt}}]
\matrix(m)[matrix of math nodes, minimum size=1.7em,
inner sep=0pt, 
row sep=3.3em, column sep=1em, text height=1.5ex, text depth=0.25ex]
 { 
  |(dc)|{\diff\cC}	\\
 |(c)|{\cC} 	      \\ };
\path[->,font=\scriptsize,>=to, thin]
(dc) edge[draw=none] node (mid) {} (c)
(dc) edge node (fo) [pos=0.25,right=-3.5pt]{$\forg{\kern1pt}$}  (c)
(c) edge [bend left=45] node (ss) [left]{$\ssig{\kern1pt}$} (dc) edge [bend right=45] node (ps) [right]{$\psig{\kern1pt}$} (dc)
(mid) edge[draw=none] node{$\dashv$} (ss)
(mid) edge[draw=none] node{$\dashv$} (ps)
;
\end{tikzpicture}
\end{center}
shows the adjunctions found there, when $\cC$ has enough products and coproducts.

If $\cE$ is a Grothendieck topos, the above triple, together with the considerations of \ref{diff-topos} constitutes an essential geometric surjection
$$
\cE\to\diff\cE.
$$
%
%

The diagram
\begin{center}
 \begin{tikzpicture} 
 [cross line/.style={preaction={draw=white, -,
line width=3pt}}]
\matrix(m)[matrix of math nodes, minimum size=1.7em,
inner sep=0pt, 
row sep=3.3em, column sep=1em, text height=1.5ex, text depth=0.25ex]
 { 
  |(dc)|{\diff\cC}	\\
 |(c)|{\cC} 	      \\ };
\path[->,font=\scriptsize,>=to, thin]
(c) edge[draw=none] node (mid) {} (dc)
(c) edge node (fo) [pos=0.65,right=-2pt]{$\I$}  (dc)
(dc) edge [bend right=45] node (ss) [left]{$\Quo$} (c) edge [bend left=45] node (ps) [right]{$\Fix$} (c)
(mid) edge[draw=none] node{$\dashv$} (ss)
(mid) edge[draw=none] node{$\dashv$} (ps)
;
\end{tikzpicture}
\end{center}
shows the adjunctions found in \ref{fix-quo}, if $\cC$ has equalisers and coequalisers.

If $\cE$ is a Grothendieck topos, in view of \ref{diff-topos}, the resulting geometric morphism
$$
\diff\cE\to\cE
$$
is the canonical global section morphism, and it is an essential surjection.

As in \ref{mon-of-ops}, these functors satisfy
$$
\Quo\circ\ssig{\,}\simeq\id_\cC, \ \ \Fix\circ\psig{\,}\simeq\id_\cC, \ \ \Quo\circ\I\simeq\id_\cC, \ \ \Fix\circ\I\simeq\id_\cC.
$$

\subsection{Difference categories with pullbacks}\label{diff-cat-pb}

We shall write $$\diffst\cC$$ for the $\diff\Set$-category with the same objects as $\diff\cC$, in which, for objects $X$, $Y$,
$$
\diffst\cC(X,Y)=(\diff\cC(X,Y),\sigma^*_X)\in\diff\Set,
$$
where $\sigma^*_X:\diff\cC(X,Y)\to\diff\cC(X,Y)$ is given by 
$$\sigma^*_X(u)=u\circ\sigma_X.$$

%

Note that $\diff\cC(X,Y)$, together with the operator 
$\sigma_{Y*}=\sigma_Y\circ\mathord{-}$, would yield the same difference set.

If $\cV$ is symmetric monoidal closed with coequalisers and a natural number object, applying the discussion of \ref{module-ops-day-conv} to the monoid $\bsi\simeq\N$,  Day convolution on $\diff\cV=[\bsi,\cV]$, or equivalently, tensor product of $\bsi$-modules in $\cV$ specialises to the operation
$$
X*Y=\begin{tikzcd}[cramped, column sep=normal, ampersand replacement=\&]
{\coeq\left(X\otimes Y\right.}\ar[yshift=2pt]{r}{\sigma_X\otimes\id} \ar[yshift=-2pt]{r}[swap]{\id\otimes\sigma_Y} \&{\left.X\otimes Y\right)}
\end{tikzcd},
$$ 
with $\sigma_{X*Y}$ induced by $\sigma_X\otimes\id$.

If $\cV$ is bicomplete, then $(\diffst\cV, *)$ is closed with internal homs 
$$
\diffst\cV(X,Y).
$$
Indeed, the general construction of $\N$-modules internal hom gives
$$
\begin{tikzcd}[cramped, column sep=normal, ampersand replacement=\&]
{\Eq\left([X,Y]\right.}\ar[yshift=2pt]{r}{\sigma_X^*} \ar[yshift=-2pt]{r}[swap]{{\sigma_Y}_*} \&{\left.[X,Y]\right)}\simeq\diff\cV(X,Y)
\end{tikzcd},
$$
with the difference operator $\sigma_X^*$.

\subsection{Difference twists}

Suppose that $\cC$ is a category with fibre products, and let $(S,\varsigma)\in\Ob(\diff\cC)$.
Note that $\varsigma\in\End_{\diff\cC}(S)$, so we can consider the base change along $\varsigma$ functor 
$$
p_\varsigma:\diff\cC_{\ov S}\to \diff\cC_{\ov S},\ \ \ \ p_\varsigma(X)=X_\varsigma.
$$
By the universal property of fibre products, for $X\in\Ob(\diff\cC_{\ov S})$ the diagram
$$
 \begin{tikzpicture} 
 [cross line/.style={preaction={draw=white, -,
line width=3pt}}]
\matrix(m)[matrix of math nodes, minimum size=1.7em,
inner sep=0pt, 
row sep=1.5em, column sep=1em, text height=1.5ex, text depth=0.25ex]
 { 
 	    			&[2em]	& |(3)|{X}		\\
 |(2)|{X} 	& |(P)| {X_{\varsigma}} &         \\[1em]
 |(h)|{S}             & |(1)|{S}  &\\};
\path[->,font=\scriptsize,>=to, thin]
(P) edge node[left=-3pt]{$p_\varsigma(X)$}  (1) edge (2)
(1) edge node[below]{$\varsigma$}  (h)
(2) edge (h)
(3) edge [bend left=10] (1) edge [bend right=10] node[above]{$\sigma_X$} (2) edge[dashed] node[pos=0.6,above left=-4pt]{$\bar{\sigma}_X$} (P)
;
\end{tikzpicture}
$$
yields a $\diff\cC_{\ov S}$-morphism $\bar{\sigma}_X:X\to X_\varsigma$. Since this construction is functorial in $X$, we obtain a natural transformation $$\bar{\sigma}:\id_{\diff\cC_{\ov S}}\to p_\varsigma.$$
The above morphism induces a $\widehat{\diff\cC_{\ov S}}$-morphism $h_X\to h_{X_\varsigma}$ as follows. For $T\in \Ob(\diff\cC_{\ov S})$, it maps $x\in X(T)$ to 
$$
\bar{\sigma}_X\circ x=x_\varsigma\circ\bar{\sigma}_T\in X_\varsigma(T).
$$
Note that for $T=S$, we have $\bar{\sigma}_T=\bar{\varsigma}=\id$ and so
$$
\bar{\sigma}_X\circ x=x_\varsigma
$$
for $x\in X(S)$.

By \ref{bc-abstr}, we know that $p_\varsigma$ is right adjoint to the functor $i_\varsigma$. In other words, given $f:Y\to S\in\Ob(\diff\cC_{\ov S})$, we have a natural bijection 
$$
\Hom_{\diff\cC_{\ov S}}(Y,X_\varsigma)\simeq\Hom_{\diff\cC_{\ov S}}(i_\varsigma(Y),X),
$$
where $i_\varsigma(f)=\varsigma\circ f$.  

On the other hand, the restriction of scalars via $\varsigma$
$$
\prod_\varsigma:\widehat{\diff\cC_{\ov S}}\to \widehat{\diff\cC_{\ov S}}
$$
is right adjoint to $p_\varsigma$ as discussed in \ref{weilres-alg}; for $\bY,\bZ\in\widehat{\diff\cC_{\ov S}}$, we have
$$
\Hom_S(\bZ,\prod_\varsigma\bY)\simeq\Hom_S(\bZ_\varsigma,\bY).
$$
The unit of adjunction $\eta_{\bZ}:\bZ\to\prod_\varsigma(\bZ_\varsigma)$ is given, for $S'$ over $S$, by
$$
\eta_\bZ(S'):\bZ(S')\to \bZ_\varsigma(S'_\varsigma),\ \ \ \ u\mapsto u_\varsigma.
$$

\subsection{Twist-equivariance}

Let $(S,\varsigma)\in\Ob(\diff\cC)$. A  \emph{$\varsigma$-equivariant} (or \emph{$S$-equivariant}) object is a pair $(X,\varphi_X)$, where $X\in\Ob(\diff\cC_{\ov S})$, and $\varphi:X_\varsigma\to X$ is a morphism in $\diff\cC_{\ov S}$.
A \emph{$\varsigma$-equivariant} morphism between $\varsigma$-equivariant objects $(X,\varphi_X)$ and $(Y,\varphi_Y)$ is a morphism $f:X\to Y$ in $\diff\cC_{\ov S}$ satisfying 
$$
f\circ\varphi_X=\varphi_Y\circ f_\varsigma.
$$
Let $\cD^\varsigma=\diff\cC_{\ov\eq{S}}$ denote the category of $\varsigma$-equivariant objects and morphisms. The family $\varphi_X$, $X\in\Ob(\diff\cC_{\ov\eq{S}})$ defines a natural transformation
$$
\varphi:p_\varsigma\to \id_{\cD^\varsigma}.
$$
Composing it with $\bar{\sigma}:\id_{\cD^\varsigma}\to p_\varsigma$, we obtain a morphism
$$
\varphi\circ\bar{\sigma}:\id_{\cD^\varsigma}\to \id_{\cD^\varsigma}.
$$

\section{The topos of difference sets}\label{topos-diff-sets}\label{diff-set-top}

\subsection{Elementary topos structure of difference sets}

The category
$$
\diff\Set=[\bsi,\Set]
$$
is a Grothendieck topos (as a preheaf category on $\bsi^\op\simeq\bsi$). We explicate the key elements of its structure as an elementary topos.


The terminal object is a point/singleton with the identity difference operator,
$$
I=(\{*\},\id).
$$
Monomorphisms  are injective difference maps, and epimorphisms are surjective.

The subobject classifier is the set 
$$
\Omega=\N\cup\{\infty\},
$$
with the difference operator
$$
\sigma_\Omega:   0\mapsto 0,\  \infty\mapsto\infty, \text{ and } i+1\mapsto i\text{ for }i\in\N.
$$
and the `truth' map
$$
t:I\to\Omega, \ *\mapsto 0.
$$
It is, of course, in bijection with the set of sieves on the object $o\in\bsi$, which consists of $R_{\infty}=\emptyset$ and $R_n=\{\sigma^i: i\geq n\}$ for $n\in\N$.

For a monomorphism $Y\xrightarrow{u}X$, the classifying map is
$$
\chi_u:X\to\Omega, \ \ \chi_u(x)=\min\{n:\sigma_X^n(x)\in Y\}.
$$
Clearly $Y=\chi_u^{-1}(\{0\})$, and, more generally, $\sigma_X^{-n}Y=\chi_u^{-1}(\{0,1,\ldots,n\})$.

The representable presheaf $\h_o$ corresponds to the difference set
$$
\Nsucc=(\N,i\mapsto i+1),
$$
which serves as a \emph{generator} for $\diff\Set$, since
$$
\diff\Set(\Nsucc,\mathord{-})\simeq \forg{\,}:\diff\Set\to \Set
$$
is conservative (i.e., it is faithful and reflects isomorphisms).


As a presheaf category, $\diff\Set=[\bsi,\Set]\simeq[\bsi^\op,\Set]$ is cartesian closed, and the \emph{internal hom} objects 
are given as
\begin{align*}
[B,C] &=\Hom(\h_o\times B,C)
\simeq\diff\Set(\Nsucc\times B,C)\\
&\simeq\{ f:\N\times\forg{B}\to\forg{C}\,|\, f(i+1,\sigma_B(x))=\sigma_Cf(i,x), \text{ for }i\in\N, x\in \forg{B}\}\\
&\simeq \{(f_i)\in \Set(\forg{B},\forg{C})^\N: f_{i+1}\,\sigma_B=\sigma_C\, f_i\},
\end{align*}
together with the shift
$$
s:[B,C]\to [B,C], \ \ s(f_0,f_1,\ldots)=(f_1, f_2,\ldots).
$$
An internal morphism $f=(f_i)\in [B,C]$ is conveniently pictured as a commutative diagram
 \begin{center}
 \begin{tikzpicture} 
\matrix(m)[matrix of math nodes, row sep=2em, column sep=2em, text height=1.5ex, text depth=0.25ex]
 {
 |(u1)|{\forg{B}}		& |(u2)|{\forg{C}} 	\\
 |(1)|{\forg{B}}		& |(2)|{\forg{C}} 	\\
 |(l1)|{\forg{B}}		& |(l2)|{\forg{C}} 	\\[-.5em]
 |(b1)|{}		& |(b2)|{}\\[0em]
 }; 
\path[->,font=\scriptsize,>=to, thin]
(u1) edge node[above]{$f_0$} (u2) edge node[left]{$\sigma_B$}   (1)
(u2) edge node[right]{$\sigma_C$} (2) 
(1) edge node[above]{$f_1$} (2) edge node[left]{$\sigma_B$}   (l1)
(2) edge node[right]{$\sigma_C$} (l2) 
(l1) edge  node[above]{$f_2$} (l2)
(l1) edge[dashed] (b1)
(l2) edge[dashed] (b2);
\end{tikzpicture}
\end{center}
shaped as an `infinitely descending ladder'.

We construct functorial isomorphisms
$$
\diff\Set(A\times B,C)\simeq \diff\Set(A,[B,C]),
$$
for $A,B,C\in\Ob(\diff\Set)$ explicitly as follows. For a $\diff\Set$-morphism $\phi:A\times B\to C$ we define
$\theta=\theta_\phi:A\to[B,C]$ by
$$
\theta_\phi(a)_i(b)=\phi(\sigma_A^i(a),b).
$$
Clearly $\theta\circ\sigma_A=s\circ\theta$, as well as $\theta(a)_{i+1}\circ\sigma_B=\sigma_C\circ\theta_i(a)$, for $a\in A$ and $i\in\N$.

Conversely, for a $\diff\Set$-morphism $\theta:A\to[B,C]$, let $\phi=\phi_\theta:A\times B\to C$ be defined by
$$
\phi_\theta(a,b)=\theta(a)_0(b).
$$
It verifies the relation $\phi(\sigma_A(a),\sigma_B(b))=\sigma_C\phi(a,b)$, as required. 
The assignments $\phi\mapsto\theta_\phi$ and $\theta\mapsto\phi_\theta$ are mutually inverse.


Given that $\diff\Set$ is a Grothendieck topos, the natural number object $N$ is the coproduct
$$
N=\coprod_{i\in\N}I\simeq(\N,\id).
$$

\subsection{Locally cartesian closed structure of difference sets}\label{loc-cart-diff-set}

In a locally cartesian closed category $\cC$, the dependent product and local internal homs are inter-definable. If we have the functor $\sprod_p$ for every object 
 $X\xrightarrow{p} S\in\cC_{\ov S}$, then the internal hom in the slice category $\cC_{\ov S}$ is given by
$$
[p,\mathord{-}]_S=\sprod_p\circ p^*.
$$
Conversely, if we are given the internal hom on $\cC_{\ov S}$, then, for any morphisms
$X\xrightarrow{p} S$ and $Y\xrightarrow{f}X$, the dependent product/Weil restriction $\sprod_p f$ is obtained as the pullback
 \begin{center}
 \begin{tikzpicture} 
\matrix(m)[matrix of math nodes, row sep=2em, column sep=2em, text height=1.5ex, text depth=0.25ex]
 {
 |(1)|{\sprod_p f}		& |(2)|{[p,p\circ f]_S} 	\\
 |(l1)|{S}		& |(l2)|{[p,p]_S} 	\\
 }; 
\path[->,font=\scriptsize,>=to, thin]
(1) edge node[above]{} (2) edge node[left]{}   (l1)
(2) edge node[right]{} (l2) 
(l1) edge  node[above]{} (l2);
\end{tikzpicture}
\end{center}
where the bottow arrow is obtained by applying the adjunction 
$\cC_{\ov S}(\id_S\times_S p,p)\simeq \cC_{\ov S}(\id_S,[p,p]_S)$ to the isomorphism $\id_S\times_S p\simeq p$.


Given an object $S\in\diff\Set$, let us fix some notation needed to study the category 
$$\diff\Set_{\ov S}.$$
By the methods of \ref{discr-fibs}, we consider $S$ as a discrete fibration $\bbS\to \bsi$, i.e., a category with the set of objects $S_0=\forg{S}$, the set of morphisms $S_1\simeq \N\times S_0$,  the source map $\N\times S_0\to S_0$, $(s,n)\mapsto \sigma_S^n(s)$, and the second projection as the target map.  

An object $Y\xrightarrow{g}S$ of $\diff\Set_{\ov S}$ can be viewed as an object of $[\bbS^\op,\Set]$ via 
$$Y(s)=Y_s=g^{-1}(\{s\}),$$ for $s\in S_0$, and, for $(n,s)\in S_1$,
$$
Y(\sigma^n_S(s)\xrightarrow{(n,s)}s)=Y(s)\xrightarrow{\sigma^n_Y}Y(\sigma^n_S(s)).
$$

For $s\in S_0$, let $H_s\to S$ be the difference subset  of 
$\Nsucc\times S\xrightarrow{\pi_2} S$ given by
$$
H_s=\{(n,\sigma_S^ns):n\in\N\}.
$$
Its fibres are
$$
H_{s,u}\simeq \h_s(u)=\bbS(u,s)\simeq\{n\in\N: \sigma_S^n s=u\},
$$
for $u\in S$, and note that
$$
\diff\Set_{\ov S}(H_s, Y)\simeq Y_s.
$$



The internal hom of $Y\xrightarrow{g}S$, $Z\xrightarrow{h}S$ in $\diff\Set_{\ov S}$ is calculated through the identification
$$
[g,h]_S\simeq[\bbS^\op,\Set][Y,Z].
$$ 
As an object of $\diff\Set_{\ov S}$, it is determined by its fibres over all $s\in S_0$, and they are evaluated as
$$
[g,h]_{S,s}\simeq [\bbS^\op,\Set][Y,Z](s)\simeq [\bbS^\op,\Set](\h_s\times Y,Z)\simeq \diff\Set_{\ov S}(H_s\times_S Y,Z),
$$


Given $\varphi\in \diff\Set_{\ov S}(H_s\times_S Y,Z)$, writing
$$
\varphi_{n}=\varphi(n,\sigma^ns,\mathord{-}):Y_{\sigma^n s}\to Z_{\sigma^n s},
$$
we obtain that 
$$
[g,h]_{S,s}=\{(\varphi_{n})_n\in\prod_{n\in\N}\Set(Y_{\sigma^n s},Z_{\sigma^n s}) : \varphi_{n+1}\circ\sigma_X=\sigma_Y\circ \varphi_{n}\}.
$$

The difference structure on $[g,h]_S$ is given by shifts
$$
\sigma:[g,h]_{S,s}\to [g,h]_{S,\sigma_S(s)},
\ \ \ \ (\sigma\varphi)_{n}=\varphi_{n+1}.
$$

Given an object $X\xrightarrow{p}S$ in $\diff\Set_{\ov S}$, the canonical isomorphism
$$
\diff\Set_{\ov S}(p\times_S g,h)\simeq \diff\Set_{\ov S}(p,[g,h]_S)
$$
assigns to a $\diff\Set_{\ov S}$ morphism $\phi: X\times_S Y\to Z$ the $\diff\Set_{\ov S}$-morphism $\theta:X\to [g,h]_S$ by letting, for each $s\in S$, $\theta_s:X_s\to [g,h]_{S,s}$ be defined as
$$
(\theta_{s}(x))_n(y)=\phi_{\sigma_S^n s}(\sigma^n_Xx,y),
$$
for $y\in Y_{\sigma_S^n s}$.

Conversely, given a morphism $\theta:X\to [g,h]_S$ over $S$, we define the morphism
$\phi: X\times_S Y\to Z$ over $S$ fibre-wise by the rule
$$
\phi_s:X_s\to [g,h]_{S,s}, \ \ \phi_s(x,y)=(\theta_s(x))_0(y)\in Z_s.
$$
The evaluation map
$$
[Y,Z]_S\times_S Y\to Z
$$
sends a pair  $(\varphi,y)\in [Y,Z]_{S,s}\times Y_s$ to $\varphi_{0}(y)\in Z_s$.

We can now calculate the dependent product
$$\sprod_p(Y\xrightarrow{f}X)\to S$$
using the above pullback diagram. Its fibres are
\begin{multline*}
(\sprod_p f)_s=\{h\in\diff\Set_{\ov S}(H_s\times_SX,Y): f(h(n,\sigma^ns,x))=x, 
\text{for }n\in\N, x\in X_{\sigma^n s}\}\\
\simeq \{(h_n)\in [X,Y]_{S,s}: h_n(x)\in Y_x, \text{for }x\in X_{\sigma^n s}\}.
\end{multline*}
The reader can verify that we obtain the same result from the identification
$$
(\sprod_p f)_s=\diff\Set_{\ov S}(H_s,\sprod_p f)\simeq \diff\Set_{\ov X}(p^*H_s, Y).
$$

An explicit formula for $\sprod_X$ follows by choosing $p:X\to e$, and we compute that
\begin{multline*}
\sprod_X(Y\xrightarrow{f}X)=\{h\in [X,Y]\, | \, f\circ h=\id_X\in [X,X]\}\\
=\{h:\N\times X\to Y \, | \, h(n,x)\in Y_x \text{ for all }x\in X\}.
\end{multline*}
If we set $h_n(\mathord{-})=h(n,\mathord{-})$, we obtain that for all $n\in\N$ and $x\in X$, $h_n(x)\in Y_x$, 
$$
h_{n+1}(\sigma_X x)=\sigma_Y h_n(x),
$$
and $\sprod_X g$ is equipped with the shift $s(h_0,h_1,\ldots)=(h_1,h_2,\ldots)$.

Clearly, $$\sprod_X(Y\xrightarrow{f}X) \simeq \prod_i \sprod_{X_i} (f^{-1}(X_i)\to X_i),$$
where $X=\coprod_i X_i$ is the partition of $X$ into maximal $\sigma$-orbits. Hence, we may assume that $X=O(x)=\{x,\sigma x,\sigma^2 x,\ldots\}$, and, in that case, an element $h$ is determined by values $y_i=h_0(\sigma^i x)\in Y_{\sigma^n x}$ and $y_{j}=h_{-j}(x)\in Y_x$ for $i\geq 0$, $j<0$. 

If $x$ is not preperiodic, i.e., all the elements of $(\sigma^ix:i\in\N)$ are distinct, we obtain
$$
\sprod_X f\simeq \prod_{j<0} Y_x\times \prod_{i\geq 0} Y_{\sigma^i x},
$$
with the shift
\begin{alignat*}{5}
s(\ldots, & y_{-1},& y_0,&& y_1, &&y_2, &\ldots)\\
=(\ldots, & y_{-2},& y_{-1},&&\sigma y_0, &&\sigma y_1, &\ldots).\\
\end{alignat*}

If $x$ is periodic with the least period of length $n$, i.e., $\sigma^n x=x$, $n>0$, then
$$
\sprod_X f\simeq  \prod_{i=0}^{n-1} Y_{\sigma^i x},
$$
with the shift
$$
s(y_0,y_1,\ldots, y_{n-1})=(\sigma y_{n-1},\sigma y_0,\ldots, \sigma y_{n-2}).
$$

If $x$ is preperiodic with the least pair $0<m<n$ such that $\sigma^mx=\sigma^n x$, then
$$
\sprod_X g\subseteq \prod_{j<0} Y_x\times \prod_{i=0}^{n -1} Y_{\sigma^i x},
$$
consists of the tuples $y$ satisfying the relations
\begin{alignat*}{3}
\sigma y_{m-1} & =\sigma y_{n-1}, && \ldots,  \sigma^m y_0 &&=\sigma^m y_{n-m}\\
\sigma^m y_{-1}& =\sigma^{m+1}y_{n-m-1},  && \ldots,  \sigma^m y_{m-n+1}&&=\sigma^{n-1} y_1\\
\sigma^m y_{m-n-i}& =\sigma^n y_{-i},\ \ i\geq 0.  && &&
\end{alignat*}

\begin{example}
The dependent product in the difference case is not necessarily exact.  
Let $X=\{0,1,2\}$, with $\sigma(0,1,2)=(1,2,1)$, let $Y=\{a_i,b_i:i\in X\}$ with $\sigma(a_0,b_0, a_1, b_1, a_2,b_2)=(a_1,a_1, a_2, b_2, a_1,a_1)$, and let $f:Y\to X$ be the projection $a_i\mapsto i$, $b_i\mapsto i$. Then $\sprod_X f=\emptyset$ since there do not exist $y_0\in Y_0$ and $y_2\in Y_2$ with $\sigma y_0=\sigma y_2$.  
\end{example}

\subsection{Logic of difference sets}

Logical operations on $\Omega$ are the following. The map `true' was defined above, while the `false' map is
$$
f:I\to\Omega, \ \ \ *\mapsto\infty.
$$
Conjunction is the classifying map of $I\xrightarrow{(t,t)}\Omega\times\Omega$, which yields
$$
\land:\Omega\to\Omega, \ \ \ \land(i,j)=\max\{i,j\}.
$$
Disjunction is the classifying map of the union of subobjects $\Omega\times I\xrightarrow{\id\times t}\Omega\times\Omega$ and $I\times\Omega\xrightarrow{t\times\id}\Omega\times\Omega$, and we compute
$$
\lor(i,j)=\min\{i,j\}.
$$
Negation $\lnot:\Omega\to\Omega$ is the classifying map of $I\xrightarrow{f}\Omega$, which is
$$
\lnot(i)=\begin{cases}
0, & i=\infty;\\
\infty, & i\in\N.
\end{cases}
$$
Note that $\lnot\lnot\neq\id_\Omega$, so $\diff\Set$ is not a Boolean topos. On the other hand, we can directly verify the relation
$$
\lnot(i\land j)=\lnot i\lor \lnot j,
$$
so $\diff\Set$ is a \emph{De Morgan} topos.

The order relation $\Omega_1=\Eq(\land,\pi_1)\subseteq\Omega\times\Omega$ becomes
$$
(i,j)\in\Omega_1, \ \text{ if }i\geq j,
$$
and the implication is the classifying map of $\Omega_1\to \Omega\times\Omega$, i.e., 
$$
{\Rightarrow}(i,j)=\begin{cases}
0, & i\geq j;\\
j, & i<j.
\end{cases}
$$

The above structure makes the poset $\Omega$ into an internal Heyting algebra in $\diff\Set$. 

\subsection{Heyting algebra of difference subsets}\label{heyting-sub-diff}

Given an object $X\in\diff\Set$, arguing through characteristic maps of subobjects,  the Heyting algebra structure of $\Omega$ makes
$$
\Sub(X)
$$ 
a Heyting algebra with the following operations.

For subobjects $U$ and $V$ of $X$, clearly
$$
U\land V=U\cap V, \ \ \ \ U\lor V=U\cup V.
$$
Moreover,
$$
U\Rightarrow V=\{x\in X: \text{for all }n\in\N, \sigma_X^nx\in U\text{ implies }\sigma_X^nx\in V\}, 
$$
and
$$
\lnot U=\{x\in X: \text{for all }n\in\N, \sigma_X^n x\notin U\}.
$$

Given a morphism $f:X\to Y$ in $\diff\Set$,  we can describe the functors
\begin{center}
 \begin{tikzpicture} 
 [cross line/.style={preaction={draw=white, -,
line width=3pt}}]
\matrix(m)[matrix of math nodes, minimum size=1.7em,
inner sep=0pt, 
row sep=0em, column sep=4em, text height=1.5ex, text depth=0.25ex]
 { 
  |(dc)|{\Sub(X)}	 &
 |(c)|{\Sub(Y)} 	      \\ };
\path[->,font=\scriptsize,>=to, thin]
(c) edge[draw=none] node (mid) {} (dc)
(c) edge node (fo) [pos=0.2,above=-1pt]{$f^{-1}$}  (dc)
(dc) edge [bend right=35] node (ss) [below]{$\forall_f$} (c) edge [bend left=35] node (ps) [above]{$\exists_f$} (c)
(mid) edge[draw=none] node[sloped]{$\vdash$} (ss)
(mid) edge[draw=none] node[sloped]{$\dashv$} (ps)
;
\end{tikzpicture}
\end{center}
as follows. 

The functor $f^{-1}$ is defined as the restriction of the pullback functor $f^*:\diff\Set_{\ov Y}\to\diff\Set_{\ov X}$ to $\Sub(Y)$, whence
$$
f^{-1}(V\mono Y)=f^{-1}(V)\mono X
$$
is the usual set-theoretic preimage.

By definition,
$$
\exists_f(U\mono X)=\im(U\mono X\xrightarrow{f}Y)=f(U)\mono Y
$$
is the usual set-theoretic image of $U$ along $f$.

On the other hand, the functor $\forall_f$ is the restriction of $\sprod_f:\diff\Set_{\ov X}\to\diff\Set_{\ov Y}$ to $\Sub(X)$, so \ref{loc-cart-diff-set} yields that
$$
\forall_f(U\mono X)=\{y\in Y: \text{for all }n\in\N, U_{\sigma^n y}=X_{\sigma^n y}\}\mono Y.
$$

\subsection{Difference power objects}\label{diff-powerset}

Given an object $X\in\diff\Set$, its power object is constructed as
$$
PX=[X,\Omega]\simeq\diff\Set(\Nsucc\times X,\Omega)\simeq\Sub(\Nsucc\times X).
$$
It follows that $PX$ is the set of commutative diagrams of the form
 \begin{center}
 \begin{tikzpicture} 
\matrix(m)[matrix of math nodes, row sep=2em, column sep=2em, text height=1.5ex, text depth=0.25ex]
 {
 |(u1)|{Y_0}		& |(u2)|{\forg{X}} 	\\
 |(1)|{Y_1}		& |(2)|{\forg{X}} 	\\
 |(l1)|{Y_2}		& |(l2)|{\forg{X}} 	\\[-.5em]
 |(b1)|{}		& |(b2)|{}\\[0em]
 }; 
\path[->,font=\scriptsize,>=to, thin]
(u1) edge[style=right hook->] node[above]{} (u2) edge node[left]{}   (1)
(u2) edge node[right]{$\sigma_X$} (2) 
(1) edge[style=right hook->]  node[above]{} (2) edge node[left]{}   (l1)
(2) edge node[right]{$\sigma_X$} (l2) 
(l1) edge[style=right hook->]   node[above]{} (l2)
(l1) edge[dashed] (b1)
(l2) edge[dashed] (b2);
\end{tikzpicture}
\end{center}
which are thought of as `generalised difference subsets'. 

Indeed, if $Y\in PX$ is viewed as a subobject of $\Nsucc\times X$, 
writing $Y_i$ for the fibre of $Y$ over $i\in \N$, we have that
$$
\sigma_X(Y_i)\subseteq Y_{i+1}.
$$
Conversely, any system $(Y_i)_{i\in\N}$ of subsets of $X$ satisfying the above property gives rise to an an element of $PX$.

An element $y=(y_i)\in [X,\Omega]$ corresponds to the system of subsets $Y=(Y_i)$  via $Y_i=y_i^{-1}(0)$.

The membership relation $\text{\large$\in$}^X\subseteq PX\times X$ is defined by
$$
((Y_i),x)\in \text{\large$\in$}^X \ \text{ if } \ x\in Y_0,
$$
or, equivalently, 
$((y_i),x)\in \text{\large$\in$}^X$ if $y_0(x)=0$.

A monomorphism $R\to X\times Y$ induces the difference map 
$$
r:Y\to PX=[X,\Omega], \ \ r(y)_i(x)=\min\{n:(\sigma^{n+i}y, \sigma^nx)\in R\},
$$
which realises the universal property of power objects in this case.

By applying \ref{heyting-sub-diff} to $PX\simeq\Sub(\Nsucc\times X)$, we see that
$PX$ is an internal Heyting algebra in $\diff\Set$  with the following operations. 
For $Y=(Y_n), Z=(Z_n)\in PX$,
$$
(Y\land Z)_n=Y_n\land Z_n, \ \ \ (Y\lor Z)_n=Y_n\lor Z_n,
$$ 
inducing the partial order
$$
Y\leq Z \text{ whenever }Y_n\subseteq Z_n\text{ for all }n.
$$
Moreover, 
$$
(Y\Rightarrow Z)_n=\{x\in \forg{X}: \forall i\in\N\ \sigma^i x\in Y_{n+i}\Rightarrow \sigma^i x\in Z_{n+i}\},
$$
and
$$
(\lnot Y)_n=\{ x\in \forg{X}:\forall i\in\N\  \sigma^ix\notin X_{n+i}\}.
$$

\subsection{A standard site for difference sets}\label{std-site-diff-set}

Let $\cD$ be the finite completion of the full subcategory of finitely presented difference sets in $\diff\Set$. 

Then $\cD$ is dense  in $\diff\Set$ for the canonical topology (cf.~\cite[C2.2.1]{elephant2}), since every $U\in\diff\Set$ can be written as a union $U=\bigcup_{i\in I}U_i$ of (maximal) $\sigma_U$-orbits $U_i\in\cD$, and the sieve generated by the family $\{U_i\mono U: i\in I\}$ belongs to the canonical coverage.

Moreover, the restriction $J_\cD$ of the canonical coverage of $\diff\Set$ to $\cD$ is subcanonical, since for $U\in\cD$, presheaves $\cD(\mathord{-},U)$ are $J_\cD$-sheaves.

The Comparison Lemma \cite[C2.2.3]{elephant2} gives that
$$
\diff\Set\simeq\sh(\diff\Set,J_\text{canonical})\simeq \sh(\cD,J_\cD),
$$
whence we conclude that $(\cD,J_\cD)$ is the \emph{standard site} for $\diff\Set$. 

More explicitly, for each $L\in\sh(\cD,J_\cD)$, there exists an object $L_0\in\diff\Set$ such that $L$ is the restriction of $\h_{L_0}$ to $\cD$, i.e., for $U\in\cD$,
$$
L(U)=\diff\Set(U,L_0).
$$ 
Since $\Nsucc\in\cD$, the representing object can be calculated as 
$$
L_0=(L(\Nsucc), L(\sigma_{\Nsucc})).
$$

\subsection{Locales in difference sets}

Let $L$ be a poset in $\diff\Set$. For variables $a{:}L$ and $X{:}PL$, the upper bound and supremum predicates 
\begin{align*}
\text{ub}(a,X)  & \equiv \forall x{:}L\  \ x\in X\Rightarrow x\leq a, \\
\sup(a,X)  & \equiv \text{ub}(a,X) \land \forall b{:}L\  \ \text{ub}(b,X)\Rightarrow a\leq b
\end{align*}
are interpreted as
$$
\real{\text{ub}}=\{(a,X)\in \forg{L}\times \forg{PL}: \forall n\in\N\ \forall x\in X_n\ \ x\leq \sigma^na\}
$$
and
\begin{align*}
\real{\sup}& =
\begin{multlined}[t]\{(a,X)\in\forg{L}\times \forg{PL}: (a,X)\in\real{\text{ub}} \\
 \land \forall m\in\N\ \forall b\in\forg{L}\ \ (b,\sigma^mX)\in\real{\text{ub}} \Rightarrow \sigma^m a\leq b\}
\end{multlined}\\
&=\begin{multlined}[t]
 \{(a,X): \left(\forall n\ \forall x\in X_n\ x\leq \sigma^n a\right)\\
\land\left(\forall m\ \forall b\in\forg{L} (\forall n\ \forall x\in X_{n+m}\ x\leq \sigma^n b)\Rightarrow \sigma^m a\leq b\right)\}.
\end{multlined}\\
\end{align*}

The poset $L$ is complete if
$$
\models \forall X{:}PX\ \exists a{:}L\ \sup(a,X),
$$
and then there exists (\cite[6.11.6]{borceux-3}) a morphism
$$
\bigvee:PL\to L
$$
such that $\models \sup(\bigvee X,X)$.

\begin{example}\label{PX-int-frame-diff}
By the general principle, 
for every $X\in\diff\Set$, $PX$ is an internal frame, i.e., an internal Heyting algebra which is complete as an internal poset. Using the above description of suprema, we can directly verify that the join of $Y\in P(PX)$ is given by
$$
(\bigvee Y)_n=\{x\in\forg{X}: \exists Y\in X_n\ x\in Y_0\}.
$$ 
\end{example}

\begin{proposition}\label{locales-in-diff-set}
The category 
$$
\text{\rm Loc}(\diff\Set)
$$
of locales in $\diff\Set$ is equivalent to the category of locales equipped with an open endomorphism.
\end{proposition}
\begin{proof}
Let $X$ be a locale in $\diff\Set$, with associated internal frame $L=\cO(X)\in\text{\rm Loc}(\diff\Set)$. We apply \cite[C1.6.10]{elephant2} to $L$, viewed as a sheaf on the standard site $(\cD,J_\cD)$ for $\diff\Set$ discussed in \ref{std-site-diff-set}, to obtain a frame $L_0=L(\Nsucc)$, equipped with an endomorphism $\sigma^*=L(\sigma_{\Nsucc})$ of complete lattices admitting a left adjoint $\sigma_!$ and satisfying the Frobenius reciprocity
$$
\sigma_!(\sigma^*u\land v)=u\land \sigma_!v,
$$
for $u,v\in L_0$. By definition, this means that the corresponding locale map $\sigma:X\to X$ is open. 

Conversely, suppose that we are given a frame $L_0$ and an adjoint pair $\sigma_!\dashv\sigma^*$ satisfying Frobenius reciprocity, as above. Let us show that $L=(L_0,\sigma^*)$ is an internal frame in $\diff\Set$. 

To verify internal completeness, let $F=(F_n)\in PL$, so that $F_n\subseteq L_0$ with $\sigma^*F_n\subseteq F_{n+1}$, and let us show that
$$
a=\bigvee_{n\in\N}\bigvee\exists(\sigma_!^n)F_n=\bigvee_{n\in\N}\bigvee\{\sigma_!^n x:x\in F_n\}
$$ 
is the supremum of $F$ in the sense discussed above. Clearly, for every $n$ and every $x\in F_n$, $\sigma_!^nx\leq a$, so, by adjunction, $x\leq {\sigma^*}^na$. Moreover, suppose $b\in L_0$ and $i\in\N$ are such that for every $n$ and every $x\in F_{n+i}$, $x\leq {\sigma^*}^nb$. By adjunction, for every $n$, $\bigvee\{\sigma_!^n x: x\in F_{n+i}\}\leq b$, and, applying $\sigma_!^i$ to both sides, we obtain that 
$$
\bigvee_{n\in\N}\bigvee\{\sigma_!^{n+i}x:x\in F_{n+i}\}=\bigvee_{n\in\N}\bigvee\exists(\sigma_!^{n+i})F_{n+i}\leq \sigma_!^i b.
$$
On the other hand, using the fact that $\sigma^*(F_n)\subseteq F_{n+1}$, we have that 
$\bigvee\exists(\sigma_!^n)F_n\leq \bigvee\exists(\sigma_!^{n+1})F_{n+1}$ for all $n$, so the left hand side in the above inequality equals $a$, whence ${\sigma^*}^i a\leq b$, as required.

This complete lattice is a frame since, using Frobenius reciprocity,
\begin{multline*}
b\land\bigvee F=b\land\bigvee_{n\in\N}\bigvee\{\sigma_!^n x: x\in F_n\}=
\bigvee_{n\in\N}\bigvee\{b\land\sigma_!^n x: x\in F_n\}\\
=\bigvee_{n\in\N}\bigvee\{\sigma_!^n({\sigma^*}^n b\land x): x\in F_n\}
=\bigvee(b\land F),
\end{multline*}
where $(b\land F)_n=\{{\sigma^*}^n b\land x: x\in F_n\}$.
\end{proof}

\begin{example}
Let $L$ be a frame with an endomorphism $\sigma$. Then $(L,\sigma)$ is a lattice in $\diff\Set$, and we can complete it to a frame $\tilde{L}$ in $\diff\Set$ as follows. Let 
$$
\tilde{L}=\{(u_i)_{i\in\N}: u_i\in L, \sigma^*u_i\leq u_{i+1}\}.
$$
We define maps $s, r:\tilde{L}\to\tilde{L}$ by
$$
s(u_0,u_1,\ldots)=(u_1,u_2,\ldots), \ \ \text{ and }\ \ r(u_0,u_1,\ldots)=(\bot_L,u_0,u_1,\ldots).
$$
Then $\tilde{L}$ is a frame with coordinate-wise operations, endowed with an endomorphism $s$ admitting a left adjoint $r$ satisfying Frobenius reciprocity, so  \ref{locales-in-diff-set} shows that it is an internal frame in $\diff\Set$. 

It can be verified by a combination of techniques used in \ref{PX-int-frame-diff} and \ref{locales-in-diff-set} that the join map
$
\bigvee:P\tilde{L}\to\tilde{L}
$
is given, for $F=(F_n)\in  P\tilde{L}$, by the formula
$$
\left(\bigvee F\right)_n=\bigvee\{u_0: u\in F_n\}.
$$
\end{example}

\subsection{Points of the topos of difference sets}\label{points-diff-set}

Points of $\diff\Set$ correspond to flat functors $\bsi\to\Set$ by Diaconescu's theorem \ref{tors-diacon}, 
$$
[\Set,\diff\Set]=[\Set,[\bsi^\op,\Set]]\simeq \text{\rm Flat}(\bsi,\Set).
$$
A functor $E:\bsi\to\Set$ is flat if and only if 
it corresponds to a free and transitive $\N$-action in the sense that $\forg{E}\neq\emptyset$,  for all $y\in\forg{E}$, $\sigma^n y=\sigma^m y$ implies $m=n$ and for all $y,z\in\forg{E}$ there exists an $n$ with
$z=\sigma^n y$ or $y=\sigma^n z$. Hence, the only flat functors correspond to difference sets
$$
\Nsucc\ \ \text{ and }\ \ \Zsucc=(\Z,n\mapsto n+1). 
$$

The point 
$$
\tau(\Nsucc):\Set\to \diff\Set=[\bsi^\op,\Set]
$$
associated with $\Nsucc$ is the pair of adjoint functors
$$
\tau(\Nsucc)^*=\forg{\,}:\diff\Set\to\Set, \ \ \ \tau(\Nsucc)_*=\psig{\,}:\Set\to \diff\Set
$$
encountered in \ref{forget-and-adj}. Indeed, 
$\tau(\Nsucc)^*=\mathord{-}\otimes_\bsi\Nsucc$, where, for $X\in\diff\Set$, $X\otimes_\bsi\Nsucc$ is the set $\{x\otimes n: x\in\forg{X},n\in\N\}$ modulo the relation $\sigma x\otimes n=x\otimes(n+1)$, which clearly identifies with $\forg{X}$. On the other hand, $\tau(\Nsucc)_*=\uHom(\Nsucc,\mathord{-})$, and we see that, for $S\in\Set$, 
$\uHom(\Nsucc,S)=\psig{S}$, i.e., the set $\uHom(\Nsucc,S)(o)=\Hom(\Nsucc(o),S)=\Hom(\N,S)=S^\N$ equipped with the left shift.

We note that $\diff\Set$ has enough points since $\tau(\Nsucc)^*=\forg{\,}$ is conservative.

The point 
$$
\tau(\Zsucc):\Set\to \diff\Set=[\bsi^\op,\Set]
$$
associated with $\Zsucc$ is described as follows. Its inverse image is the functor
$$
\tau(\Zsucc)^*:\diff\Set\to\Set, \ \  X\mapsto \varinjlim(\cdots\xrightarrow{\sigma}\forg{X}\xrightarrow{\sigma} \forg{X}\xrightarrow{\sigma}\cdots).
$$
Indeed, by definition, $\tau(\Zsucc)^*(X)=X\otimes_\bsi\Zsucc$ is the set $\{x\otimes n: x\in\forg{X},n\in\Z\}$ modulo the relation $\sigma x\otimes n=x\otimes(n+1)$, which identifies with the colimit of the diagram consisting of copies of $\forg{X}$ indexed by $\Z$ and connected by maps $\sigma_X$. 

On the other hand, its direct image functor is given by
$$\tau(\Zsucc)_*(S)=\uHom(\Zsucc,S)=S^\Z,$$ 
 equipped with the (invertible) left shift.

\subsection{Categories internal in difference sets}\label{int-cat-diff-set}


There is an equivalence of 2-categories
$$
\cat(\diff\Set) 
\simeq [\bsi,\cat]=\diff\cat.
$$
In order to verify this, let $\bbD=(D_0,D_1)\in \cat(\diff\Set)$. Let $\cD$ be a category with objects $\forg{D_0}$ and arrows $\forg{D_1}$. The action of $\sigma_{D_0}$ on objects, together with the action of $\sigma_{D_1}$ on arrows give rise to an endofunctor 
$
\sigma_\cD:\cD\to \cD
$,
so we obtain an object
$$
(\cD,\sigma_\cD)\in\diff\cat.
$$
An internal functor $f:\C\to \bbD$ yields a commutative diagram of functors
 \begin{center}
 \begin{tikzpicture} 
\matrix(m)[matrix of math nodes, row sep=2em, column sep=2em, text height=1.5ex, text depth=0.25ex]
 {
 |(1)|{\cC}		& |(2)|{\cD} 	\\
 |(l1)|{\cC}		& |(l2)|{\cD} 	\\
 }; 
\path[->,font=\scriptsize,>=to, thin]
(1) edge node[above]{$f$} (2) edge node[left]{$\sigma_\cC$}   (l1)
(2) edge node[right]{$\sigma_\cD$} (l2) 
(l1) edge  node[above]{$f$} (l2);
\end{tikzpicture}
\end{center}
while an internal natural transformation $\alpha$ between two internal functors $f,g:\C\to\bbD$ yields a natural transformation $\alpha$ between functors $f,g:\cC\to\cD$ satisfying
$$
\alpha_{\sigma_\cC(X)}=\sigma_\cD(\alpha_X),
$$
for every $X\in\cC$.

\subsection{Internal presheaves in difference sets}\label{internal-diff-presh}

Let $\cS=\diff\Set=[\bsi,\Set]$, let $\C\in \cat(\cS)$.
We can  describe the category $[\C^\op,\cS]$ of internal $\cS$-valued presheaves on $\cD$ in terms of the difference category $(\cC,\sigma_\cC)\in\diff\cat$ associated to $\cC$ as in \ref{int-cat-diff-set} as follows.

Unravelling the definition of internal presheaves from \ref{int-presh}, we see that an $F\in[\C^\op,\cS]$ corresponds to an \emph{$\sigma_\cC$-equivariant} presheaf on $\cC$, i.e., an object $\bF\in [\cC^\op,\Set]$ equipped with a natural transformation
$$
\sigma_\bF:\bF\to \sigma_\cC^*\bF=\bF\circ \sigma_\cC.
$$
We have thus established an equivalence
$$
[\C^\op,\cS]\simeq [\cC^\op,\Set]^{\sigma_\cC},
$$
where the latter denotes the category of $\sigma_\cC$-equivariant presheaves on $\cC$.

Note that, for $X\in \cC$ (i.e., $X\in\forg{C_0}$) the value of the presheaf $\bF$ corresponding to the internal presheaf $F$ on $X$ is the fibre
$$
\bF(X)=F_{0,X}=\gamma_0^{-1}(\{X\}),
$$ 
and, for a morphism $X'\xrightarrow{f}X$ in $\cC$ (i.e., $f\in \forg{C_1}$), we have
$$
\bF(f)=e_F(f,\mathord{-}):\bF(X)\to\bF(X').
$$

\subsection{Topos of internal presheaves in difference sets}\label{top-int-psh-diff}

With the notation of \ref{internal-diff-presh}, we describe the elementary topos structure of the Grothendieck topos
$$
[\C^\op,\cS].
$$
For $F,G\in [\C^\op,\cS]$, the internal hom
$$
[F,G]\in [\C^\op,\cS]
$$
can be worked out from first principles as in \ref{comon-coalg}, but we proceed as follows.

The underlying object $U[F,G]\in \cS_{\ov C_0}$ can be determined through its fibres over $c\in\forg{C_0}$ by the relation
$$
(U[F,G])_c\simeq \cS_{\ov C_0}(H_c,U[F,G])\simeq \Hom(R(H_c),[F,G])\simeq \Hom(R(H_c)\times F,G),
$$
where $H_c\mono C_0^*\Nsucc$ is given by $H_c=\{(n,\sigma^nc):n\in \N\}$. 

By definition, 
$R(H_c)=C_1\times_{C_0}H_c\xrightarrow{d_0\pi_1}C_0$, hence
$$
R(H_c)=\{(f,n,\sigma^nc): d_1f=\sigma^nc\},
$$
with evaluation map 
$$
e_{R(H_c)}:C_1\times_{C_0}R(H_c)\to R(H_c), \ \ \ (g,f,n,\sigma^nc)\mapsto (f\circ g,n,\sigma^nc).
$$
Let 
$$
P=R(H_c)\times_{C_0}F_0=\{(f,n,\sigma^nc,x): d_1f=\sigma^nc, d_0f=\gamma_F(x)\},
$$
with evaluation
$$
e_P:C_1\times_{C_0}P\to P, \ \ \ (g,f,n,\sigma^nc,x)\mapsto (f\circ g, n,\sigma^nc,e_F(g,x)).
$$
Thus, an element $\varphi\in \Hom(R(H_c)\times F,G)$ is a difference map $P\to Y$ in $\cS_{\ov C_0}$ such that
$$
e_G\circ d_1^*\varphi=\varphi\circ e_P.
$$
More explicitly, for $(g,f,n,\sigma^n c,x)\in C_1\times_{C_0}P$,
$$
e_G(g,\varphi(f,n,\sigma^nc,x)=\varphi(f\circ g,n,\sigma^nc,e_F(g,x)),
$$
i.e., 
$$
\bG(g)(\varphi(f,n,\sigma^nc,x))=\varphi(f\circ g,n,\sigma^nc,\bF(g)(x)).
$$
Consider the maps in $\cS_{\ov C_0}$ defined by
$$
\varphi_n: C_{1,\sigma^nc}\times_{C_0}X\to Y, \ \ \ \varphi_n(f,x)=\varphi(f,n,\sigma^nc,x),
$$
where $C_{1,\sigma^nc}=\{f\in C_1: d_1f=\sigma^nc\}\xrightarrow{d_0}C_0$. The compatibility of $\varphi$ with evaluation maps yields that
$$
\bG(g)\varphi_n(f,x)=\varphi_n(f\circ g,\bF(g)(x)),
$$
so $\varphi_n$ can be viewed as a natural transformation
$$
\varphi_n: \h_{\sigma_\cC^n c}\times \bF\to \bG
$$
of presheaves on $\cC$. Moreover, the fact that $\varphi$ is a difference map yields that
$$
\varphi_{n+1}(\sigma_{C_1} f,\sigma_{F_0} x)=\sigma_{G_0}\varphi_n(f,x),
$$
so we obtain a commutative diagram
 \begin{center}
 \begin{tikzpicture} 
\matrix(m)[matrix of math nodes, row sep=2em, column sep=3em, text height=1.5ex, text depth=0.25ex]
 {
 |(1)|{\h_{\sigma_\cC^nc}\times\bF}		& |(2)|{\bG} 	\\
 |(l1)|{\sigma_\cC^*(\h_{\sigma_\cC^{n+1}c})\times\sigma_\cC^*(\bF)}		& |(l2)|{\sigma_{\cC}^*\bG} 	\\
 }; 
\path[->,font=\scriptsize,>=to, thin]
(1) edge node[above]{$\varphi_n$} (2) edge node[left]{$\sigma_n\times\sigma_\bF$}   (l1)
(2) edge node[right]{$\sigma_\bG$} (l2) 
(l1) edge  node[above]{$\sigma_\cC^*\varphi_{n+1}$} (l2);
\end{tikzpicture}
\end{center}
where $\sigma_{n,u}:\h_{\sigma^nc}(u)\to \h_{\sigma^{n+1}c}(\sigma_\cC u)$ is the restriction of $\sigma_{C_1}$ to $\h_{\sigma^nc}(u)=\cC(u,\sigma^n c)$, which maps to $\cC(\sigma u,\sigma^{n+1}c)=\h_{\sigma^{n+1}c}(\sigma_\cC u)$.

If we write $\bar{\varphi}_n$ for the morphisms obtained from $\varphi_n$ by adjunction, the above is equivalent to the commutativity of the diagram
 \begin{center}
 \begin{tikzpicture} 
\matrix(m)[matrix of math nodes, row sep=2em, column sep=3em, text height=1.5ex, text depth=0.25ex]
 {
 |(1)|{\h_{\sigma_\cC^nc}}		& |(2)|{[\bF,\bG]} 	\\
 &  |(m2)|{[\bF,\sigma^*\bG]} \\
 |(l1)|{\sigma^*(\h_{\sigma_\cC^{n+1}c})}& |(l2)|{[\sigma^*\bF,\sigma^*\bG]}  	\\
 }; 
\path[->,font=\scriptsize,>=to, thin]
(1) edge node[above]{$\bar{\varphi}_n$} (2) edge node[left]{$\sigma_n$}   (l1)
(2) edge node[right]{$[\bF,\sigma_\bG]$} (m2) 
(l2) edge node[right]{$[\sigma_\bF,\sigma^*\bG]$} (m2) 
(l1) edge  node[above]{$\overline{\sigma_\cC^*\varphi_{n+1}}$} (l2);
\end{tikzpicture}
\end{center}
where the bottom arrow decomposes as
$$
\sigma^*(\h_{\sigma_\cC^{n+1}c})\xrightarrow{\sigma^*\bar{\varphi}_{n+1}}\sigma^*[\bF,\bG]\xrightarrow{\tilde{\sigma}^*}[\sigma^*\bF,\sigma^*\bG],
$$
and $\tilde{\sigma}^*$ is the canonical natural transformation defined as follows.
 For $u\in \cC$, 
 $$
 \tilde{\sigma}^*_u:\Hom(\bF_{\sigma u},\bG_{\sigma u})\to \Hom((\sigma^*\bF)_u,(\sigma^*\bG)_u)
$$
is the map taking the natural transformation $\psi$ to the natural transformation $\tilde{\sigma}^*(\psi)$ whose component at $u'\xrightarrow{f} u$ is given by
$$
\tilde{\sigma}^*(\psi)_f=\psi_{\sigma f}:\bF(\sigma u')\to \bG(\sigma u').
$$

To summarise, for $c\in \forg{C_0}$, i.e., an object of $\cC$, modulo the identification 
$[\cC^\op,\Set](\h_{\sigma_\cC^n c}\times\bF,\bG) \simeq [\bF,\bG](\sigma_\cC^nc)$, the internal hom is given by
\begin{multline*}
[F,G]_c=[(\bF,\sigma_\bF),(\bG,\sigma_\bG)](c)\\
\simeq \left\{(\varphi_n)_n\in 
\prod_{n\in\N}[\bF,\bG](\sigma_\cC^nc) : 
 \sigma_G\circ \varphi_n=\sigma_\cC^*\varphi_{n+1}\circ (\sigma_n\times\sigma_\bF)\text{ for }n\in\N\right\}\\
\simeq \left\{(\bar{\varphi}_n)_n\in  \prod_{n\in\N}[\cC^\op,\Set](\h_{\sigma_\cC^n c},[\bF,\bG]):
[\bF,\sigma_\bG]\circ\bar{\varphi}_n=[\sigma_\bF,\sigma^*G]\circ\tilde{\sigma}^*\circ\sigma^*\bar{\varphi}_{n+1}\circ \sigma_n\right\}.
\end{multline*}

In order to determine the subobject classifier 
$$
\Omega=\Omega_{[\C^\op,\cS]},
$$
we observe that its fibre over an object $c\in\forg{C_0}$ can be calculated from
$$
(U\Omega)_c\simeq \cS_{\ov C_0}(H_c,U(\Omega))\simeq\Hom(R(H_c),\Omega)\simeq\Sub(R(H_c)),
$$
so it suffices to identify the subpresheaves of $R(H_c)$. Using the above description of $R(H_c)$, let $S\mono R(H_c)$ be a subobject, and let 
$$
S_n=\{f\in C_1: (f,n,\sigma^n c)\in S\}.
$$
This is a set of morphisms of $\cC$ with target $\sigma^n c$, and the requirement that $S$ must be closed under the action of $R(H_c)$ implies that $S_n$ is a sieve on $\sigma^nc$ in $\cC$.

Moreover, the difference structure of $S$ yields that
$$
\sigma_\cC(S_n)\subseteq S_{n+1}.
$$

Hence, the subobject classifier is the internal presheaf $\Omega$ on $\C$, or equivalently, 
the $\sigma_\cC$-equivariant presheaf $(\mathbf{\Omega},\sigma_\mathbf{\Omega})$,
determined by 
$$
\Omega_c=\mathbf{\Omega}(c)\simeq\left\{ (S_n)_n: S_n \text{ is a sieve on }\sigma_\cC^nc \text{ in }\cC, \text{ and }\sigma_\cC(S_n)\subseteq S_{n+1}\right\},
$$
with
$$
\sigma_{\mathbf{\Omega}}:\mathbf{\Omega}\to \sigma_\cC^*\mathbf{\Omega}
$$
acting on the sequences $(S_n)$ as the shift. 

We invite the reader to obtain the same result from the general description given in \ref{comon-coalg}, and we give another calculation in \ref{int-site-diff-set} using the description of $\Omega$ as the object of sieves in the internal logic of $\cS$.

\subsection{Externalising internal presheaves in difference sets}

In \ref{semidir} we constructed the category
$$
\bsi\rtimes \C\simeq \bsi\rtimes\cC
$$
with the same objects as $\cC$, and morphisms $X\to Y$ are pairs $(n,f)$, consisting of an $n\in\N$, and a $\cC$-morphism
$$
f:X\to \sigma_\cC^n Y,
$$
and the composite of $X\xrightarrow{(n,f)}Y\xrightarrow{(m,g)}Z$ is the pair $(m+n,\sigma_\cC^n(g)\circ f)$.

We obtain another explicit description on internal $\cS$-valued presheaves on $\C$ as ordinary presheaves on $\bsi\rtimes\cC$,
$$
[\C^\op,\diff\Set]\simeq [(\bsi\rtimes\cC)^\op,\Set].
$$

\subsection{Presheaves and internal presheaves in difference sets}\label{psh-vs-int-diff-psh}

With notation from \ref{internal-diff-presh}, suppose that there exist functors $i_\C$ and $r_\C$ that fit in the diagram
$$
\begin{tikzcd}[column sep=normal, ampersand replacement=\&] 
{\cC}\arrow[r,yshift=2pt,"i_\C"] \arrow[d] \&{\bsi\rtimes\C} \arrow[l,yshift=-2pt,"r_\C"]\arrow[d]\\
{1}\arrow[r,yshift=2pt,"i"]  \&{\bsi} \arrow[l,yshift=-2pt,"r"]
\end{tikzcd}
$$
and satisfy $r_\C\circ i_\C=\id_\cC$. By \ref{funct-int-presh} and \ref{internal-diff-presh}, we obtain a diagram of essential geometric morphisms
 $$
\begin{tikzcd}[column sep=normal, ampersand replacement=\&] 
{\hC}\arrow[r,yshift=2pt,"i_\C"] \arrow[d] \&{\widehat{\bsi\rtimes\C}} \arrow[l,yshift=-2pt,"r_\C"]\arrow[d]
\arrow[r,"\sim"]\& {[\C^\op,\diff\Set]} \arrow[dl] \\
{\Set}\arrow[r,yshift=2pt,"i"]  \&{\diff\Set} \arrow[l,yshift=-2pt,"r"]
\end{tikzcd}
$$
with $i_\C^*\circ r_\C^*\simeq \id$, $r_{\C!}\circ i_{\C!}\simeq\id$, $r_{\C*}\circ i_{\C*}\simeq\id$, while $i$ and $r$ conform to the general principles from \ref{mon-of-ops}, made more explicit for difference sets in \ref{cat--diff-cat}.

\subsection{Internal difference limits and colimits}\label{diff-int-lim-colim}

Let $(\cC,\sigma_\cC)$ be the category with an endofunctor associated to an internal category $\C$ in $\diff\Set$ via \ref{int-cat-diff-set}. We discuss internal limits and colimits over $\C$, recalling the general principles from \ref{int-lim-colim}.

The constant internal presheaf functor 
$$
\C^*:\diff\Set\to [\C^\op,\diff\Set]
$$
is described as follows. Given an object $X\in\diff\Set$,
the $\cC$-equivariant presheaf $(\bX,\sigma_\bX)$ corresponding to $\C^*X$ via \ref{internal-diff-presh} is the constant presheaf on $\cC$,
$$
\bX=\cC^*\forg{X}, \ \ \  \bX(c)=\forg{X}
$$
for $c\in\cC$, together with the natural transformation $\sigma_\bX:\bX\to \sigma_\cC^*\bX=\bX$ given by 
$$
\sigma_{\bX, c}=\sigma_X:\forg{X}\to\forg{X}.
$$

The internal colimit functor 
$$\textstyle{\varinjlim_\C}:[\C^\op,\diff\Set]\to\diff\Set$$ is left adjoint to $\C^*$. 
For an internal presheaf $F\in [\C^\op,\diff\Set]$, we have
$$
\textstyle{\varinjlim_\C F}=
\coeq(\begin{tikzcd}[cramped,sep=small]
F_1 \ar[yshift=2pt]{r}{e} \ar[yshift=-2pt]{r}[swap]{\pi_2} & F_0
\end{tikzcd}).
$$
Writing $(\bF,\sigma_\bF)$ for the $\sigma_\cC$-equivariant presheaf corresponding to $F$ via \ref{internal-diff-presh}, and taking into account that the above coequaliser in $\diff\Set$ is calculated as $\coeq(\begin{tikzcd}[cramped,sep=small]
\forg{F_1} \ar[yshift=2pt]{r}{\forg{e}} \ar[yshift=-2pt]{r}[swap]{\forg{\pi_2}} & \forg{F_0}
\end{tikzcd})$, with a naturally induced difference operator, we see that
$$
\textstyle{\varinjlim_\C F}=\varinjlim_\cC\bF,
$$
together with a difference operator induced by $\sigma_\cC$ and $\sigma_\bF$ by naturality of colimits.

Using the fact that the internal limit functor
$$\textstyle{\varprojlim_\C}:[\C^\op,\diff\Set]\to\diff\Set$$ is right adjoint to $\C^*$, we compute
$$
\forg{\textstyle{\varprojlim_\C F}}\simeq \diff\Set(\Nsucc,\textstyle{\varprojlim_\C F})\simeq[\C^\op,\diff\Set](\C^*\Nsucc,F),
$$
and the latter set consists of morphisms
$$
u\in\diff\Set_{\ov C_0}(C_0\times\Nsucc,F_0), \text{ such that } e\circ d_1^* u=u\circ e.
$$
More explicitly, the functions $u_n=u(\mathord{-},n)\in\Hom_{\ov\forg{C_0}}(\forg{C_0},\forg{F_0})$ satisfy $u_n(c)\in\bF(c)$, $u_n(d_0f)=e(f,u_n(d_1f))=\bF(f)(u_n(d_1f))$ and
$u_{n+1}(\sigma_\cC c)=\sigma_\bF(u_n(c))$. Thus, we can identify them with elements $u_n\in \hC(\be,\bF)\simeq\Gamma(\bF)$, where $\be=\cC^*e$ is the terminal presheaf on $\cC$ and $\Gamma:\hC\to\Set$ denotes the classical global sections functor.
Hence
$$
\textstyle{\varprojlim_\C F}=\{ (u_n)\in \Gamma(\bF)^\N\simeq \hC(\be,\bF)^\N: u_{n+1}\circ\sigma_\cC=\sigma_\bF\circ u_n\},
$$
with the difference structure provided by the usual shift. 

If $\cC$ has a terminal object $e$, then $\Gamma(\bF)\simeq\bF(e)$, and the above expression simplifies to
$$
\textstyle{\varprojlim_\C F}=\{ (u_n)\in \bF(e)^\N: \bF(\sigma_\cC e\to e)(u_{n+1})=\sigma_\bF(u_n)\}.
$$

\subsection{Internal sites in difference sets}\label{int-site-diff-set}

Let $(\cC,\sigma_\cC)$ be a category with an endofunctor associated with a category  $\C$ internal in $\diff\Set$ as in \ref{int-cat-diff-set}.

Interpreting the formula from \ref{int-sites}, we find that the fibre of the object of sieves of $\C$ over an object $b\in \forg{C_0}$ is
\begin{align*}
\mathop{\rm Sv}(\C)_b & = \{S\in \forg{P(d_1)_b}\, | \, \forall n, \forall f,f'\in\forg{C_1}, \ f\in \sigma^nS\land d_1(f')=d_0(f)\Rightarrow f\circ f'\in \sigma^n S\}\\
&=\{ (S_n)\in \forg{P(d_1)_b}\, | \, \forall n, \forall f,f'\in\forg{C_1}, \ f\in S_n\land d_1(f')=d_0(f)\Rightarrow f\circ f'\in S_n\}\\
&= \{(S_n)\in \forg{PC_1}\, | \, \text{for all }n, S_n \text{ is a sieve on }\sigma^n b\text{ in }\cC\}.
\end{align*} 

An internal coverage is a span of $\diff\Set$-morphisms
$$
C_0 \xleftarrow{b} T \xrightarrow{c} \mathop{\rm Sv}(\C)\mono PC_1
$$
such that, for all $t\in\forg{T}$, $c(t)_0$ is a sieve on $b(t)$, satisfying the following difference analogue of axiom (C). If we form the object
$$
Q=\{(t',f,t)\in T\times_{C_0}C_1\times_{C_0}T\, | \, \forall n, \forall f'\in\forg{C_1}, f'\in c(t')_n \Rightarrow \sigma^n f\circ f'\in c(t)_n\},
$$
then the difference morphism
$$
Q\xrightarrow{\pi_{23}} C_1\times_{C_0}T
$$
is an epimorphism. Rewritten in a more familiar form, we require that whenever $f\in\forg{C_1}$ and $t\in \forg{T}$ satisfy $d_1(f)=b(t)$, if we consider
$f^*c(t)\in\forg{PC_1}$ given by
$$
(f^*c(t))_n=\{f'\in\forg{C_1}: \sigma^n f\circ f'\in c(t)_n\},
$$
then there exists $t'\in\forg{T}$ such that $c(t')\leq f^* c(t)$, i.e., for every $n$ and every $f'\in c(t')_n$, we have that $\sigma^n f\circ f'\in c(t)_n$.

\subsection{Internal sheaves in difference sets}\label{int-sh-diff-set}

Let $(\C,T)$ be an internal site in $\diff\Set$, as described in \ref{int-site-diff-set}, and let $(\cF,\sigma_F)$ be the data associated with an internal presheaf $F\in[\C^\op,\diff\Set]$ as in 
\ref{internal-diff-presh}.

Let $s=(s_n)\in [d_0,\gamma_0]$ be a system of sections of $F$, and let $S=(S_n)\in \mathop{\rm Sv}(\C)$ be a sieve. By interpreting the formula $\mathop{\rm comp}_F$ from \ref{int-sh} in difference sets, we see that $s$ is a system of sections of $F$ compatible with $S$, provided
$$
\forall n, \forall (f,g)\in\forg{C_2}, \ f\in S_n\Rightarrow s_n(f\circ g)=e(g,s_n(f))=\cF(g)(s_n(f)),
$$
i.e., for each $n$, $s_n$ is a system of sections of $\cF$ compatible with $S_n$, with $$s_{n+1}\circ\sigma_\cC=\sigma_\cF\circ s_n.$$

By interpreting the internal sheaf property from \ref{int-sh}, we see that $F$ is a $T$-sheaf if
for all $t\in\forg{T}$ and $s\in[d_0,\gamma_0]$, 
\begin{multline*}
(\forall n, \forall (f,g)\in\forg{C_2}, f\in c(\sigma^nt) \Rightarrow s_n(f\circ g)=e(g,s_n(f)))
\\ \Rightarrow 
\left(
(\exists \tilde{s}\in\forg{F_0}, \gamma_0(\tilde{s})=b(t) \land \forall n, \forall f\in c(\sigma^nt)=c(t)_n, \ e(f,\sigma^n\tilde{s})=s_n(f))\right.\\
\land
(\forall i, \forall \tilde{s}',\tilde{s}''\in\forg{F_0}, \gamma_0(\tilde{s}')=\gamma_0(\tilde{s}'')=b(\sigma^i t) \\
\left.\Rightarrow (\forall n, \forall f\in c(t)_{n+i}, e(f,\sigma^n \tilde{s}')=e(f,\sigma^n \tilde{s}''))\Rightarrow \tilde{s}'=\tilde{s}'')
\right).
\end{multline*} 

Equivalently, the data $(\cF,\sigma_\cF)$ gives rise to a $T$-sheaf if, for every $t\in \forg{T}$ and every system $s=(s_n)$ of sections of $F$ compatible with $c(t)$ in the sense described above, there exists a section $\tilde{s}\in\cF(b(t))$ such that for every $n$, and every $f\in c(t)_n$, $$\cF(f)(\sigma^n\tilde{s})=s_n(f),$$ i.e., $\sigma^n\tilde{s}$ is obtained by glueing the system $s_n$ of sections of $\cF$ with respect to the sieve $c(t)_n$. Moreover, this glueing is unique in the sense that whenever $i$ and $\tilde{s}'\in \cF(\sigma^i b(t))$ are such that for every $n$ and every $f\in c(t)_{n+i}=c(\sigma^{n+i}t)$, $\cF(f)(\sigma^n\tilde{s}')=s_{n+i}(f)$, then $\tilde{s}'=\sigma^i\tilde{s}$.

\subsection{Internal sheaves and equivariant sheaves in difference sets}\label{int-sh-equiv-sh}

Let $(\C,T)$ be an internal site in $\cS=\diff\Set$, let $(\cC,\sigma_\cC)$ correspond to $\C$ by \ref{int-cat-diff-set}, and suppose that $T\mono C_0\times\mathop{\rm Sv}(\C)$ is such that there exists a coverage $\cT$ on $\cC$ such that
\begin{enumerate}
\item\label{nwise}
$(b,S)\in T \text{ if and only if, for all }n, (\sigma^nb, S_n)\in\cT$;
\item\label{ninv}
$(b,S)\in\cT \text{ implies } (\sigma b,\sigma(S))\in \cT$.
\end{enumerate}

Then we have an equivalence of categories
$$
\sh_\cS(\C,T)\simeq \sh(\cC,\cT)^{\sigma_\cC},
$$
i.e., the equivalence of categories of presheaves \ref{internal-diff-presh} restricts to the respective full subcategories of sheaves. 

Indeed, if $F\in\sh_\cS(\C,T)$, let $(\bF,\sigma_\bF)\in [\cC^\op,\Set]^{\sigma_\cC}$ be the corresponding $\sigma_\cC$-equivariant presheaf. In order to verify that $\bF$ is a $\cT$-sheaf, let $s_0$ be a system of sections of $\bF$ compatible with a $\cT$-covering sieve $S_0$ of $b\in\cC$. 

Define $S_n=\sigma^n(S_0)$, and let $s_n:S_n\to F_0$ be a morphism given by $s_n(\sigma_\cC^n h)=\sigma_\bF^n s_0(h)$. Using assumptions (\ref{nwise}) and (\ref{ninv}), $s=(s_n)$ is a system of sections of $F$ compatible with a $T$-covering sieve $S=(S_n)$, so the internal unique glueing property \ref{int-sh-diff-set} for $F$ yields a section $\tilde{s}\in \bF(b)$ that also realises the unique glueing property for $\bF$.

 Conversely, let $\bF\in \sh(\cC,\cT)^{\sigma_\cC}$, and let $F\in[\C^\op,\cS]$ be the corresponding internal presheaf. Let $s=(s_n)$ be a system of sections of $F$ compatible with a $T$-covering sieve $S=(S_n)$ in the sense of  \ref{int-sh-diff-set}. By assumption~(\ref{nwise}), for every $n\in\N$, $s_n$ is a system of sections of $\bF$ compatible with a $\cT$-covering sieve $S_n$, so the glueing property of $\bF$ yields a section $\tilde{s}_n$. 
 
 Assumption (\ref{ninv}) gives that, for each $n$, $\sigma(S_n)$ is a subcover of $S_{n+1}$, and the glueing data is compatible since $s_{n+1}\sigma_\cC=\sigma_\bF s_n$, so by the uniqueness of gluing for $\bF$ we deduce that 
$$
\tilde{s}_{n+1}=\sigma_\bF \tilde{s}_n. 
$$
This shows that the internal glueing condition from \ref{int-sh-diff-set} for $F$ is satisfied by $\tilde{s}_0$, and the internal uniqueness property  is verified in a similar way.

\section{Generalised difference categories}

\subsection{Generalised difference categories}\label{gen-cats}

Let $\cC$ be a category and let $E\in\Ob(\diff\Set)$. We let the category 
$$\difd{E}\cC$$ have the same objects as $\diff\cC$, and, for $X,Y\in\Ob(\difd{E}\cC)=\Ob(\diff\cC)$, 
$$
\difd{E}\cC(X,Y)=\{(f_e)\in\cC(X,Y)^E: f_{\sigma_E(e)}\circ\sigma_X=\sigma_Y\circ f_e,\text{ for }e\in E\}.
$$

Similarly, the category
$$
\difu{E}\cC
$$
has the same objects as $\diff\cC$ and morphism sets
$$
\difu{E}\cC(X,Y)=\{(f_e)\in\cC(X,Y)^E: f_e\circ\sigma_X=\sigma_Y\circ f_{\sigma_E(e)},\text{ for }e\in E\}.
$$

For a natural number $n$, let $\dif{n}\cC$ be the category with the same objects as $\diff\cC$, but with the following more flexible notion of morphism. A morphism $(X,\sigma)\to (Y,\sigma)$ is a sequence $(f_0,f_1,\ldots,f_n)$ of morphisms $f_i\in\Hom_\cC(X,Y)$ such that for $i<n$,
$$
f_{i+1}\circ\sigma_X=\sigma_Y\circ f_i.
$$
%
We abbreviate $\Hom_{\dif{n}(\cC)}$ as $\Hom_n$.  Given $X,Y\in\Ob(\diff\cC)$, we have maps
$
\pi_n, s_n:\Hom_{n+1}(X,Y)\to\Hom_n(X,Y),
$
where
$$
\pi_n(f_0,\ldots,f_{n+1})=(f_0,\ldots,f_n),\ \ \ \ s_n(f_0,\ldots,f_{n+1})=(f_1,\ldots,f_n).
$$

In the special case of $E=(\N,i\mapsto i+1)$, we obtain categories
$$
\diffinf\cC=\diffN\cC\ \ \ \text{ and }\ \ \ \difuinf\cC=\difu{\N}\cC.
$$
More explicity, the set of morphisms between objects $X$ and $Y$ in $\diffinf\cC$ is
$$
\Hom_\infty(X,Y)=\{(f_i)\in\Hom_\cC(X,Y)^\N : f_{i+1}\circ\sigma_X=\sigma_Y\circ f_i\text{ for all }i\in\N\}. 
$$
Note that $\Hom_\infty(X,Y)$ has a structure of a difference set
 with the shift map
$$
s(f_0,f_1,\ldots)=(f_1,f_2,\ldots).
$$
Equivalently, $(\Hom_\infty(X,Y),s)\in\Ob(\diff\Set)$ is the  limit of the system $(\Hom_n(X,Y),s_n)$ with respect to connecting maps $\pi_n$. 

There is an embedding of categories $$\I_\infty:\diff\cC\to\diffinf\cC$$ given as identity on objects, while for $X,Y\in\Ob(\diffinf\cC)=\Ob(\diff\cC)$,  the map $\I_\infty:\Hom_{\diff\cC}(X,Y)\to\Hom_\infty(X,Y)$ is defined by
$$
f_0\mapsto(f_0,f_0,\ldots).
$$
More precisely,
\begin{align*}
\Hom_{\diff\cC}(X,Y) &\simeq \Fix(\Hom_\infty(X,Y))=(\Hom_\infty(X,Y))^s\\ 
&=\{f\in\Hom_\infty(X,Y):s(f)=f\}\\
   &    \simeq \{f\in\Hom_1(X,Y): s_0(f)=\pi_0(f)\}.
\end{align*}
For $f\in\Hom_\infty(X,Y)$, identifying $\sigma_X$ and $\sigma_Y$ with $\I_\infty(\sigma_X)$ and $\I_\infty(\sigma_Y)$, we can write
$$
s(f)\circ\sigma_X=\sigma_Y\circ f.
$$
It is practical to write
$$
\jj_\infty=\I_\infty\circ\I: \cC\to \diffinf\cC.
$$

We leave a completely analogous discussion of sets
$$
\Hom^\infty(X,Y)=\{(f_i)\in\Hom_\cC(X,Y)^\N : f_i\circ\sigma_X=\sigma_Y\circ f_{i+1}\text{ for all }i\in\N\}, 
$$
and the embeddings 
$$
\I^\infty:\diff\cC\to\difuinf\cC\ \ \ \text{ and }\ \ \ \jj^\infty=\I^\infty\circ\I:\cC\to\difuinf\cC
$$
to the reader.

\begin{lemma} For any difference set $E$, we have isomorphisms
$$
\diff\Set(E,\difd{\infty}\cC(X,Y))\simeq\difd{E}\cC(X,Y),$$ and
$$\diff\Set(E,\difuinf\cC(X,Y))\simeq\difu{E}\cC(X,Y).$$
\end{lemma}

In the case $E=(\Z,i\mapsto i+1)$, we obtain categories $\difd{\Z}\cC$ and $\difu{\Z}\cC$. Explicitly,
$$
\Hom_\Z(X,Y)=\{(f_i)\in\Hom_\cC(X,Y)^\Z : f_{i+1}\circ\sigma_X=\sigma_Y\circ f_i\text{ for all }i\in\Z\}, 
$$
together with the (invertible) shift
$$
s(f_i)=(f_{i+1}).
$$
There is a natural embedding of categories
$$
\I_\Z:\diff\cC\to\diffZ(\cC).
$$
For $i\in\N$, we consider the forgetful functor 
$$
\forg{\,}_i:\diffinf\cC\to\cC,
$$
given on objects as $\forg{(X,\sigma)}=X$, and on morphisms as $\forg{(f_i)}=f_i$.\footnote{Does this have any adjoints?}

\subsection{Generalised difference categories are enriched over difference sets}\label{gen-cats-enr}

Let $\cC$ be a category. Given $X,Y\in\Ob(\diffinf\cC)$,  the discussion in \ref{gen-cats} yields an object
$$\diffinf\cC(X,Y)=(\Hom_\infty(X,Y),s)\in\diff\Set,$$ 
so we can consider $\diffinf\cC$ as a category enriched over $\diff\Set$, with the usual composition of morphisms and the usual identity morphisms, i.e., 
$$
\diffinf\cC\in(\diff\Set)\da\cat.
$$

The section functor $\Gamma$ on $\diff\Set$ is usually referred to as the \emph{underlying set} functor
$()_0:\diff\Set\to\Set$,
$$
S_0={\diff\Set}(e,S)=\Fix(S).
$$
It follows that the \emph{underlying category}  of a $\diff\Set$-category $\diffinf\cC$ is the ordinary difference category,
$$(\diffinf\cC)_0=\diff\cC.$$
By definition, the objects of $(\diffinf\cC)_0$ are $\Ob(\diffinf\cC)=\Ob(\diff\cC)$, and, for $X,Y\in\Ob(\diff\cC)$,
$$
{(\diffinf\cC)_0}(X,Y)=(\diffinf\cC(X,Y))_0=\Fix(\diffinf\cC(X,Y))\simeq{\diff\cC}(X,Y).
$$ 
Moreover, if $T:\cA\to\cB$ is $\diff\Set$-functor, its \emph{underlying functor} $T_0:\cA_0\to\cB_0$
is simply $T$ on objects, and $T_{0,XY}:\cA_0(X,Y)\to\cB_0(TX,TY)$ is 
$$
\Fix T_{XY}:\Fix\cA(X,Y)\to \Fix\cB(TA,TB).
$$


The operations $\sigma_X^{*}=\mathord{-}\circ\sigma_X$ and $\sigma_Y\circ{-}$ provide two additional ways of considering $\forg{\Hom_\infty(X,Y)}$ as an object of $\diff\Set$.

Similarly, we verify that 
$$
\difuinf\cC\in(\diff\Set)\da\cat,
$$
and its underlying category is again the ordinary difference category,
$$
(\difuinf\cC)_0=\diff\cC.
$$

The two constructions are related by duality of $\diff\Set$-categories via
$$
(\diffinf\cC)^\op\simeq \difuinf\cC^\op.
$$

By direct verification, $\diffZ\cC$ is also enriched over $\diff\Set$.

\subsection{Generalised difference categories are tensored and cotensored over difference sets}\label{gen-diff-is-tensored-cotensored}


\begin{proposition}\label{diffinf-tens-cotens}
\begin{enumerate}
\item\label{tens-1} If $\cC$ admits small coproducts, then $\diffinf\cC$ is a \emph{tensored} $\diff\Set$-category. 
\item\label{cotens-2} If $\cC$ is complete, 
then $\diffinf\cC$ is a \emph{cotensored} $\diff\Set$-category.
\end{enumerate}
\end{proposition}

\begin{proof}
(\ref{tens-1}) For tensored structure, given $E\in\diff\Set$ and $X\in\diffinf\cC$,
we define an object $E\otimes X$ of $\diffinf\cC$ as
$$
E\otimes X=\coprod_{e\in E}X_e,
$$ 
where $X_e$ are isomorphic copies of $\forg{X}$, together with $\sigma:E\otimes X\to E\otimes X$, given on $X_e$ as $\sigma_X:X_e\to X_{\sigma_E(s)}$.

Let us construct an isomorphism
$$
\diffinf\cC(E\otimes X,Y)\simeq[E,\diffinf\cC(X,Y)],
$$
functorial in $X$ and $Y$.

For $\phi=(\phi_i)\in\diffinf\cC(E\otimes X,Y)$, each $\phi_i$ belongs to $\cC(E\otimes X,Y)\simeq\prod_{e\in E}\cC(X_e,Y)$, so we can write it as $\phi_i=(\phi_{i,e})_{e\in E}$. By assumption,
$$
\phi_{i+1,\sigma_E(e)}\circ\sigma_X=\sigma_Y\circ\phi_{i,e}.
$$
We define a morphism $\theta=\theta_\phi=(\theta_i)\in[E,\diffinf\cC(X,Y)]$  by
$$
\theta_i(e)_j=\phi_{i+j,\sigma_E^j(e)}.
$$
The stated properties of $\theta$ are verified through relations
$$
\theta_{i+1}\circ\sigma_E=s\circ\theta_i,\ \ \ \text{i.e.}\ \ \ \theta_{i+1}(\sigma_E(e))_j=\theta_i(e)_{j+1},
$$
and, for every $e\in E$,
$$
(\theta_i(e))\in\diffinf\cC(X,Y),\ \ \ \text{i.e.}\ \ \ \theta_i(e)_{j+1}\circ\sigma_X=\sigma_Y\circ\theta_i(e)_j.
$$
Conversely, given $\theta=(\theta_i)\in[E,\diffinf\cC(X,Y)]$, we define $\phi=\phi_\theta=(\phi_i)\in\diffinf\cC(E\otimes X,Y)$ by setting
$$
\phi_{i,e}=\theta_i(e)_0,
$$
and we verify that the assignments $\phi\mapsto\theta_\phi$ and $\theta\mapsto\phi_\theta$ are mutually inverse.

(\ref{cotens-2}) For cotensored structure,
let $E\in\Ob(\diff\Set)$ and $X\in\diffinf\cC$. 
We write $X_e$ for isomorphic copies of $\forg{X}$, $e\in E$, and consider the 
endomorphisms $s,\tilde{\sigma}_E^*,\tilde{\sigma}_{X*}$  of the object 
$$
\prod_{i\in\N}(\prod_{e\in E}X_e)  
$$
as follows:
\begin{enumerate}
\item $s$ is the usual shift operator;
\item $\tilde{\sigma}_E^*=\prod_{i\in\N}\sigma_E^*$, where $\sigma_E^*$ is the permutation of the factors of $\prod_{e\in E}X_e$ taking $X_e$ identically onto $X_{\sigma_E(e)}$;
\item $\tilde{\sigma}_{X*}=\prod_{i\in\N}\sigma_{X*}$, where $\sigma_{X*}=\prod_{e\in E}\sigma_{X_e}$.
\end{enumerate}
We define the object $[E,X]$ in $\diffinf\cC$ as the equaliser 
$$
[E,X]=(\text{eq}(s\circ\tilde{\sigma}_E^*,\tilde{\sigma}_{X*}),s),
$$
noting that $s$ restricts to $[E,X]$ because it commutes with $\tilde{\sigma}_E^*$ and 
$\tilde{\sigma}_{X*}$.
%
%
Let us construct natural isomorphisms
$$
\diffinf\cC(X,[E,Y])\simeq[E,\diffinf\cC(X,Y)].
$$
Let $\phi=(\phi_i)\in\diffinf\cC(X,[E,Y])$. Thus, 
$$
\phi_{i+1}\circ\sigma_X=s\circ\phi_i.
$$
Moreover, each $\phi_i\in\cC(X,[E,Y])$, so, by the definition of $[E,Y]$ as an equaliser, 
$$s\circ\tilde{\sigma}_E^*\circ\phi_i=\tilde{\sigma}_{X*}\circ\phi_i.$$
Thus, writing the components of 
each $\phi_i$ as $\phi_{i,j,e}$, 
$$
\phi_{i,j+1,\sigma_E(e)}=\sigma_Y\circ\phi_{i,j,e}.
$$
We construct $\theta=\theta_\phi=(\theta_i)\in[E,\diffinf\cC(X,Y)]$ by setting
$$
\theta_i(e)_j=\phi_{i+j,0,\sigma_E^j(e)}.
$$
The stated properties of $\theta$ are verified through relations
$$
\theta_{i+1}\circ\sigma_E=s\circ\theta_i
$$
and
$$
\theta_i(e)_{j+1}\circ\sigma_X=\sigma_Y\circ\theta_i(e)_j,
$$
which are straightforward consequences of the above assumptions on $\phi$.

Conversely, starting with a $\theta\in[E,\diffinf\cC(X,Y)]$, we define $\phi=\phi_\theta$ by
$$
\phi_{i,j,e}=\theta_{i+j}(e)_0\circ\sigma_X^j,
$$
and we verify that $\phi\in \diffinf\cC(X,[E,Y])$, as well as the fact that the assignments $\phi\mapsto\theta_\phi$ and $\theta\mapsto\phi_\theta$ are mutually inverse.
\end{proof}


\begin{proposition}\label{difuinf-cotens}
\begin{enumerate}
\item\label{cotens-3} If $\cC$ has small products, then $\difuinf\cC$ is a cotensored $\diff\Set$-category.
\item\label{tens-4} If $\cC$ is cocomplete, 
then $\difuinf\cC$ is a tensored $\diff\Set$-category.
\end{enumerate}
\end{proposition}
\begin{proof}
Both claims follow directly from \ref{diffinf-tens-cotens}, the fact that 
$\difuinf\cC\simeq(\diffinf\cC^\op)^\op$ and the general principle from \ref{tens-cotens-op}.

We provide the explicit construction of cotensors for (\ref{cotens-3}) below. Item (\ref{tens-4})  and the explicit construction of tensors, dual to the construction of cotensors in \ref{diffinf-tens-cotens}, is left  to the reader.

Let $X\in \difuinf\cC$, and $E\in\diff\Set$. Using the notation from the proof of \ref{diffinf-tens-cotens}(\ref{cotens-2}), we define the object $\llbracket E,X\rrbracket=X^E$ of $\difuinf\cC$ as
$$
\llbracket E,X\rrbracket=X^E=\prod_{e\in E}X_e,
$$
together with $\sigma:X^E\to X^E$ given by
$$
\sigma=\sigma_{X*}\circ\sigma_E^*.
$$

For $X,Y\in\difuinf\cC$, and $E\in\diff\Set$, we construct a natural isomorphism
$$
[E,\difuinf\cC(X,Y)]\simeq\difuinf\cC(X,\llbracket E,Y\rrbracket).
$$

Let $\phi=(\phi_i)\in [E,\difuinf\cC(X,Y)]$. Thus 
$$
\phi_{i+1}\circ\sigma_E=s\circ\phi_i, \ \ \ \text{i.e.,}\ \ \ \ \phi_i(\sigma_E(e))_j=\phi_{i+1}(e)_{j+1},
$$ and, for each $e\in E$,
$$
\phi_i(e)\in\difuinf\cC(X,Y),  \ \ \ \text{i.e.,}\ \ \ \ \phi_i(e)_j\circ\sigma_X=\sigma_Y\circ\phi_i(e)_{j+1}.
$$
 We define a morphism 
$\theta=\theta_\phi\in\difuinf\cC(X,\llbracket E,Y\rrbracket)$ as follows. By definition,
$\theta=(\theta_i)$, where each $\theta_i$ should be an element of $\cC(X,\llbracket E,Y\rrbracket)\simeq\prod_{e\in E}\cC(X,Y_e)$, so we can write components of $\theta_i$ as $\theta_{i,e}$, $e\in E$. With this notation,  let
$$
\theta_{i,e}=\phi_i(e)_0,
$$
whence we immediately verify that
$$
\theta_{i}\circ\sigma_X=\sigma\circ\theta_{i+1}\ \ \ \text{i.e.,}\ \ \ \theta_{i,e}\circ\sigma_X=\sigma_Y\circ\theta_{i+1,\sigma_E(e)}.
$$
Conversely, given $\theta\in \difuinf\cC(X,\llbracket E,Y\rrbracket)$, we define $\phi=\phi_\theta$ by setting
$$
\phi_i(e)_j=\theta_{i+j,\sigma_E^j(e)},
$$
and it is straightforward to check that $\phi\in[E,\difuinf\cC(X,Y)]$.

The assignments $\phi\mapsto\theta_\phi$ and $\theta\mapsto\phi_\theta$ are mutually inverse, and $\theta_{s\phi}=s\theta_\phi$,
as required.
\end{proof}

Note, when $E$ is inversive in the sense that $\phi_E$ is invertible and
$\cC$ has equalisers in addition to arbitrary products, writing $E^{-1}=(E,\sigma_E^{-1})$, we have
$$
\llbracket E,X\rrbracket\simeq[E^{-1},X].
$$
Indeed, for $\phi=(\phi_i)\in[E^{-1},X]$, we have that $\phi_{i+1}\circ\sigma_E^{-1}=\sigma_X\circ\phi_i$, so $\phi_{i+1}=\sigma_X\circ\phi_i\circ\sigma_E$ and the entire sequence $(\phi_i)$ is determined by $\phi_0\in\llbracket E,X\rrbracket$. Thus, the assignment
$$
(f_i)\mapsto f_0
$$
realises the required isomorphism.

In this case, the familiar isomorphism 
$$
\diffinf\cC(E^{-1}\otimes X,Y)\simeq[E^{-1},\diffinf\cC(X,Y)]\simeq\diffinf\cC(X,[E^{-1},Y])$$
yields an isomorphism
$$
\diffinf\cC(E^{-1}\otimes X,Y)\simeq \llbracket E,\diffinf\cC(X,Y)\rrbracket\simeq\difd{\infty}\cC(X,\llbracket E,Y\rrbracket).
$$


\subsection{Constant difference objects}

Suppose $\cC$ is a category with coproducts and fibre products and that coproducts commute with base change.

Let $E\in\Ob(\diff\Set)$ and let $(S,\varsigma)\in\Ob(\diffinf\cC)$. The constant difference object $E_S\in\diff\cC$ is defined using the tensored structure as
$$
E_S=E\otimes S\in\diffinf\cC.
$$
By definition, it has the property
$$
\diffinf\cC(E_S,T)\simeq[E,\diffinf\cC(S,T)],
$$
for any $T\in\diffinf\cC$. In particular, applying the functor $\Fix$ to the above relation, 
$$
\diff\cC(E_S,T)\simeq\diff\Set(E,\diffinf\cC(S,T))\simeq\difd{E}\cC(S,T).
$$

%
%

The object $E_S$ comes equipped with a natural projection to $S$, and thus the assignment 
$$
E\mapsto E_S
$$
is a functor $\diff\Set\to\diff\cC_{\ov S}$. 

This functor commutes with products, for difference sets $E$, $F$, 
$$
(E\times F)_S\simeq E_S\times_SF_S.
$$
Consequently, if $E$ is a group (resp. ring, or other reasonable algebraic structure), then $E_S$ is an $S$-group (resp. $S$-ring, or another $S$-structure).

We offer another characterisation of constant objects using $\Hom_1$ in place of $\Hom_\infty$.
Given $X,Y\in\Ob(\diff\cC)$, even though the set $\Hom_1(X,Y)$ defined in \ref{dif-cats} is not quite a difference set, for $E\in\Ob(\diff\Set)$, we introduce notation
\begin{equation*}
\begin{split}
\Hom(E,\Hom_1(X,Y))=\{& \phi\in\Hom_\Set(\forg{E},\Hom_\cC(\forg{X},\forg{Y})):\\
& \phi(\sigma_E(i))\circ\sigma_X=\sigma_Y\circ\phi(i)\text{ for all }i\in E\}. 
\end{split}
\end{equation*}
Now, for any $T\in\diff\cC$,
$$
\Hom_{\diff\cC}(E_S,T)\simeq\Hom(E,\Hom_1(S,T)).
$$
Indeed, for $\phi\in\Hom(E,\Hom_1(S,T))$, we construct $\theta_\phi:(E_S,\sigma)\to(T,\sigma)$ by stipulating that $\theta_\phi$ on $S_i$ is $\phi(i):S\to T$.

Conversely, for $\theta\in \Hom_{\diff\cC}(E_S,T)$, we define $\phi_\theta:E\to\Hom_1(S,T)$ as the composite
$$
S_i\to E_S\stackrel{\theta}{\to}T,
$$
and we see that $(\phi_\theta(i),\phi_\theta(\sigma_E(i))\in\Hom_1(S,T)$. The assignments $\phi\mapsto\theta_\phi$ and $\theta\mapsto\phi_\theta$ are mutually inverse.

\subsection{Ordinary vs.\ generalised difference structures}\label{ord-vs-gen}


Let $\cC$ be a category with countable coproducts. For $X\in\diffinf\cC$, we define an object $\copinf X\in\diff\cC$ as
$$
\copinf X=\coprod_{i\in\N}X_i,
$$
where $X_i\in\Ob(\cC)$ are all isomorphic copies of $\forg{X}$, and 
$$
\sigma:\copinf X\to\copinf X
$$
is given on $X_i$ as $\sigma_X:X_i\to X_{i+1}$.
Given $(f_i)\in\Hom_\infty(X,Y)$, we consider $f_i$ as a function $f_i:X_i\to Y_i$ and we let
$$
\copinf(f_i)=\coprod_{i\in\N}f_i:\copinf X\to\copinf Y.
$$
For $Y\in\diff\cC$, there is a natural projection
$$
\copinf\left(\I_\infty Y\right)\to Y.
$$
The resulting functor
$$
\copinf:\forg{\diffinf\cC}\to\diff\cC
$$
is left adjoint to $\I_\infty$, i.e., for $X\in\Ob(\diffinf\cC)$ and $Y\in\Ob(\diff\cC)$ we have a natural bijection
$$
\forg{\Hom_\infty(X,\I_\infty Y)}\simeq\Hom_{\diff\cC}(\copinf X,Y).
$$
Indeed, the assignment 
$$
\prod_i\Hom(X_i,\forg{Y})\ni(f_i)\mapsto\coprod_i f_i\in\Hom_\cC(\coprod_iX_i,\forg{Y})
$$
is bijective by the definition of coproducts, and $\coprod_i f_i$ is a difference morphism if and only if $f_{i+1}\circ\sigma_X=\sigma_Y\circ f_i$, i.e., $(f_i)\in\Hom_\infty(X,Y)$.

\begin{lemma}\label{cotens-with-N}
For $X\in\diffinf\cC$,
$$
[\N,X]\simeq X^\Z,
$$
where we write $X^\Z$ for the difference object $(\prod_{i\in\Z}\forg{X},s)$ with
$$
s=\left(\prod_{i<1}\id_X\times\prod_{i\geq 1}\sigma_X\right)\circ r,
$$
where $r$ is the right shift on $\prod_{i\in\Z}\forg{X}$.
\end{lemma}
\begin{proof}
We give the proof for a concrete category $\cC$ since it is more intuitive, and we leave the general proof through commutativity of relevant diagrams to the reader. In the concrete case, the map $s$ is given by
$$
s(h)(i)=\begin{cases}
\sigma_X h(i-1),& i\geq 1\\
h(i-1),& i<1,
\end{cases}
$$
for $h:\Z\to\forg{X}$.
Given an $f=(f_i)\in [\N,X]$, the functions $f_i:\N\to \forg{X}$ satisfy 
$$
f_{i+1}(n+1)=\sigma_X f_i(n).
$$
Hence, we can assign to it the function $h=h_f\in X^\Z$ given by
$$
h(i)=\begin{cases}
f_0(i), & i\geq 0\\
f_{-i}(0), & i<0.
\end{cases}
$$
It is straightforward to verify that this assignment is an isomorphism.
\end{proof}

If $\cC$ has countable products, given $X\in\diffinf\cC$, we define
$$
\prodinf X=[\N,X]=X^\Z,
$$
as in the lemma above.

For $f=(f_i)\in\diffinf\cC(X,Y)$, we define
$$
\prodinf f=\prod_{i<1}f_0\times\prod_{i\geq 1} f_i:X^\Z\to X^\Z,
$$
and we verify that $\prodinf f\in \diff\cC(\prodinf X,\prodinf Y)$.

The resulting functor
$$
\prodinf:\forg{\diffinf\cC}\to\diff\cC
$$
is right adjoint to $I_\infty$, i.e., for $X\in\diff\cC$ and $Y\in\diffinf\cC$, we have a natural bijection
$$
\forg{\diffinf\cC(I_\infty X, Y)}\simeq \diff\cC(X,\prodinf Y).
$$
Indeed, the assignment
$$
\prod_{i\in\Z}\cC(\forg{X},Y_i)\ni(f_i)\mapsto\prod_{i\in\Z}f_i\in\cC(\forg{X},\prod_{i\in\Z}Y_i)
$$
is bijective by definition of products, and we have that
$(f_i)\in\diffinf\cC(X,Y)$ if and only if $\prod_if_i\in\diff\cC(X,Y)$.

By an analogous procedure, $\I_\Z:\diff\cC\to\diffZ\cC$ has a left adjoint $\copz$. 

On the other hand, given $X\in\diffZ\cC$, we let
$$
\prodz X=\prod_{i\in\Z}X_i,
$$
where $X_i$ are isomorphic copies of $\forg{X}$, together with 
$$
\sigma=\prod_{i\in\Z}\sigma_i:\prod_{i\in\Z}X_i\to\prod_{i\in\Z}X_i,
$$
where $\sigma_i:X_i\to X_{i+1}$ is defined as $\sigma_X$.

For $(f_i)\in\Hom_\Z(X,Y)$, we consider $f_i:X_i\to Y_i$ and let
$$
\prodz(f_i)=\prod_{i\in\Z}f_i:\prodz X\to\prodz Y.
$$
This yields a functor $\prodz:\diffZ\cC\to\diff\cC$ which is right adjoint to $\I_\Z$.

The diagram
\begin{center}
 \begin{tikzpicture} 
 [cross line/.style={preaction={draw=white, -,
line width=3pt}}]
\matrix(m)[matrix of math nodes, minimum size=1.7em,
inner sep=0pt, 
row sep=3.3em, column sep=1em, text height=1.5ex, text depth=0.25ex]
 { 
  |(dc)|{\forg{\diffinf\cC}}	\\
 |(c)|{\diff\cC} 	      \\ };
%
\path[->,font=\scriptsize,>=to, thin]
(c) edge[draw=none] node (mid) {} (dc)
(c) edge node (fo) [pos=0.7,right=-2pt]{$\I_\infty$}  (dc)
(dc) edge [bend right=45] node (ss) [left]{$\copinf$} (c) edge [bend left=45] node (ps) [right]{$\prodinf$} (c)
(mid) edge[draw=none] node{$\dashv$} (ss)
(mid) edge[draw=none] node{$\dashv$} (ps)
;
\end{tikzpicture}
\hspace{3em}
 \begin{tikzpicture} 
 [cross line/.style={preaction={draw=white, -,
line width=3pt}}]
\matrix(m)[matrix of math nodes, minimum size=1.7em,
inner sep=0pt, 
row sep=3.3em, column sep=1em, text height=1.5ex, text depth=0.25ex]
 { 
  |(dc)|{\forg{\diffZ\cC}}	\\
 |(c)|{\diff\cC} 	      \\ };
\path[->,font=\scriptsize,>=to, thin]
(c) edge[draw=none] node (mid) {} (dc)
(c) edge node (fo) [pos=0.65,right=-2pt]{$\I_\Z$}  (dc)
(dc) edge [bend right=45] node (ss) [left]{$\copz$} (c) edge [bend left=45] node (ps) [right]{$\prodz$} (c)
(mid) edge[draw=none] node{$\dashv$} (ss)
(mid) edge[draw=none] node{$\dashv$} (ps)
;
\end{tikzpicture}
\end{center}
depicts the above adjunctions. 

Note that the functor $\ssig{\,}:\cC\to\diff\cC$ from \ref{forget-and-adj} is the composite
$$
\ssig{\,}=\copinf\circ\I_\infty\circ\I.
$$ 
Indeed, for $X\in\Ob(\cC)$ and $Y\in\Ob(\diff\cC)$ we have natural bijections
\begin{align*}
\Hom_{\diff\cC}(\ssig{X},Y) & \simeq \Hom_{\diff\cC}(\copinf\circ\I_\infty\circ\I(X),Y) \\  &
\simeq\Hom_\infty(\I_\infty(X,\id),\I_\infty Y) \simeq\Hom_\cC(X,\forg{Y}).
\end{align*}
so the above composite is left adjoint to the forgetful functor $\forg{\,}:\diff\cC\to\cC$.

\subsection{Ordinary vs.\ generalised difference structures, enriched view}\label{diff-vs-gendif-enr}

In this section, we show how the ordinary functors $\copinf$ and $\prodinf$ can be recovered from their enriched counterparts.

We consider $\N$ as a difference set equipped with the difference operator $i\mapsto i+1$. 
\begin{lemma}\label{fromN}
For any difference set $E$,
$$
\diff\Set(\N,E)\simeq \forg{E}.
$$
Moreover, there is a natural map
$$
\tau_E:\forg{[\N,E]}\to \diff\Set(\N,E)\simeq\forg{E}.
$$
\end{lemma}
\begin{proof}
The first bijection is given by  
$$
h\mapsto h(0),
$$
for $h\in\diff\Set(\N,E)$. 

For the second, starting with $f=(f_i)\in[\N,E]$, we define
$$
\tau(f)(i)=f_i(i),
$$
and we verify that $\tau(f)\in\diff\Set(\N,E)$, and it eventually maps to $f_0(0)\in E$.
\end{proof}

Using the tensored and cotensored structure of the $\diff\Set$-category $\diffinf\cC$, we define the $\diff\Set$-functors
$$
\copinf:\diffinf\cC\to\diffinf\cC, \ \ \ \ X\mapsto \N\otimes X,
$$ 
and
$$
\prodinf:\diffinf\cC\to\diffinf\cC, \ \ \ \ Y\mapsto [\N,Y].
$$
We also consider 
$$
\produinf:\difuinf\cC\to\difuinf\cC, \ \ \ \ Y\mapsto Y^\N=\llbracket \N,Y\rrbracket.
$$
By general facts on tensors and cotensors, $\copinf$ is left $\diff\Set$-adjoint to $\prodinf$, i.e., there is a natural isomorphism
$$
\diffinf\cC(\copinf X,Y)\simeq \diffinf\cC(X,\prodinf Y).
$$

Let us discuss how the ordinary functors (denoted by same symbols)
$$
\copinf, \prodinf: \forg{\diffinf\cC}\to \diff\cC
$$
can be recovered from the above.

On objects, $\copinf$ is defined as before. On morphisms, it acts as the composite
\begin{multline*}
\forg{\diffinf\cC(X,Y)}\to\forg{\diffinf\cC(\N\otimes X,\N\otimes Y)}\simeq\forg{[\N,\diffinf\cC(X,\N\otimes Y)}\\
\xrightarrow{\tau (\ref{fromN})}\diffinf\cC(X,\N\otimes Y)\to\prod_{i\in\N}\cC(X_i,\N\otimes Y)\simeq\cC(\coprod_{i\in\N}X_i,\N\otimes Y),
\end{multline*}
which actually lands in $\diff\cC(\N\otimes X,\N\otimes Y)$. Hence,
$$
\diffinf\cC(\copinf X,Y)=\diffinf\cC(\N\otimes X,Y)\simeq[\N,\diffinf\cC(X,Y)],
$$
so, applying $\Fix$ and using \ref{fromN}, we obtain
$$
\diff\cC(\copinf X,Y)\simeq\diff\Set(\N,\diffinf\cC(X,Y))=\forg{\diffinf\cC(X,Y)}
$$
so $\copinf:\forg{\diffinf\cC}\to \diff\cC$ is left adjoint to $\I_\infty:\diff\cC\to\forg{\diffinf\cC}$.

Moreover, using \ref{CvsdiffinfC}
$$
\diff\cC(\copinf\jj_\infty(X_0),Y)\simeq\forg{\diffinf\cC(\jj_\infty(X_0),Y)}\simeq\cC(X_0,\forg{Y}),
$$
so we conclude that
$$
\copinf\circ\jj_\infty\simeq\ssig{\,}:\cC\to\diff\cC.
$$

The ordinary functor $\prodinf$ is defined analogously.  Starting from 
$$
\diffinf\cC(X,\prodinf Y)=\diffinf\cC(X,[\N,Y])\simeq[\N,\diffinf\cC(X,Y)]
$$
and applying $\Fix$, we see that
$$
\diff\cC(X,\prodinf Y)\simeq\forg{\diffinf\cC(X,Y)},
$$
so $\prodinf:\forg{\diffinf\cC}\to\diff\cC$ is right adjoint to $\I_\infty$.

On the other hand, 
$$
[\N,\difuinf\cC(X,Y)]\simeq \difuinf\cC(X,\produinf Y),
$$
and, applying $\Fix$, we deduce that
$$
\forg{\difuinf(X,Y)}\simeq\diff\cC(X,\produinf Y),
$$
so $\I^\infty:\diff\cC\to\forg{\difuinf\cC}$ is left adjoint to $\produinf:\forg{\difuinf\cC}\to\diff\cC$.

Moreover, for $Y_0\in\cC$, by \ref{CvsdiffinfC}
$$
\diff\cC(X,\produinf\jj^\infty(Y_0))\simeq\forg{\difuinf\cC(X,\jj^\infty(Y_0))}\simeq\cC(\forg{X},Y_0),
$$
and we deduce that $$\produinf\circ\jj^\infty\simeq\psig{\,}:\cC\to\diff\cC.$$

\subsection{Structures vs.\ generalised difference structures}\label{CvsdiffinfC}

The functor $\jj^\infty:\cC\to\difuinf\cC$ is right adjoint to the forgetful functor $\forg{\,}^0$ in a somewhat enriched sense, we have natural isomorphisms
$$
(\cC(\forg{X}^0,Y),\sigma_X^{*})\simeq\difuinf\cC(X,\jj^\infty(Y)),
$$
for $X\in\difuinf\cC$, $Y\in\cC$. Indeed, $(f_i)\in\difuinf\cC(X,\jj^\infty(Y))$ if and only if $f_i=f_0\circ\sigma_X^i$, so the whole sequence is determined by $f_0=\forg{(f_i)}^0$ and the shift has the same effect as $\sigma_X^*$.

Dually, we have
$$
\diffinf\cC(\jj_\infty(X),Y)\simeq(\cC(X,\forg{Y}_0),\sigma_{Y,{*}}).
$$

\section{Enriched difference presheaves}

Given a category $\cC$, we constructed the category
$$
\diff\cC,
$$
as well as enriched categories
$$
\diffst\cC\in(\diffst\Set)\da\cat, 
$$
and
$$
\diffinf\cC, \difuinf\cC\in (\diff\Set)\da\cat.
$$
Hence, we can consider the category of ordinary presheaves
$$
\widehat{\diff\cC}=[(\diff\cC)^\op,\Set]
$$
the $\diffst\Set$-category of $\diffst\Set$-presheaves
$$
\widehat{\diffst\cC} = [(\diffst\cC)^\op,\diffst\Set],
$$
as well as $\diff\Set$-categories of $\diff\Set$-presheaves
\begin{align*}
\widehat{\diffinf\cC} &= [(\diffinf\cC)^\op, \diff\Set], \\ 
\widehat{\difuinf\cC} &= [(\difuinf\cC)^\op, \diff\Set].
\end{align*}
We occasionally consider other $\diff\Set$-functor categories ($\diff\Set$-categories of copresheaves) such as $[\diffinf\cC, \diff\Set]$.

Through the internalisation construction of \ref{enrich-groth-constr}, the categories 
 $\diffinf\cC$ and $\difuinf\cC$ can be considered internal in $\diff\Set$, and the above enriched presheaf categories can be viewed as categories of internal presheaves as well.

We proceed to study the specific details of each of the above cases.

\subsection{Presheaves on difference categories}\label{func-diff-cat}

Let $\cC$ be a category (with countable products, coproducts and finite equalisers, coequalisers when needed). Using the adjunctions
$$
\Quo\dashv\I\dashv\Fix, \ \ \ \ \ \ \ \ssig{\,}\dashv\forg{\,}\dashv\psig{\,}
$$
and general principles from \ref{general-representab}, we obtain a number of ways to pass between categories $\hC$ and $\widehat{\diff\cC}$. In particular, the right adjoints 
$$
\I, \Fix, \forg{\,}, \psig{\,}
$$
can be naturally extended to presheaf categories as
$$
\widehat{\Quo}, \widehat{\I}, \widehat{\ssig{\,}}, \widehat{\forg{\,}},
$$
respectively, and any adjunctions between the original functors are preserved between the corresponding extensions.
The diagrams
 \begin{center}
  \begin{tikzpicture} 
\matrix(m)[matrix of math nodes, row sep=3em, column sep=3em, text height=1.9ex, text depth=0.25ex]
 {
 |(1)|{\cC}		& |(2)|{\hC} 	\\
 |(l1)|{\diff\cC}		& |(l2)|{\widehat{\diff\cC}} 	\\
 }; 
\path[->,font=\scriptsize,>=to, thin]
(1) edge node[above]{$\h$} (2) 
(l1) edge [bend left=20] node(i) [left]{$\forg{}$}   (1)
(1) edge [bend left=20] node(fi) [right]{$\psig{}$}   (l1)
(l2) edge [bend left=20] node(qo) [left]{$\widehat{\ssig{}}$} (2) 
(2) edge [bend left=20] node(hi) [right]{$\widehat{\forg{}}$}   (l2)
(l1) edge  node[above]{$\h$} (l2)
(i) edge[draw=none] node{$\dashv$} (fi)
(qo) edge[draw=none] node{$\dashv$} (hi)
;
\end{tikzpicture}
\hspace{3em}
 \begin{tikzpicture} 
\matrix(m)[matrix of math nodes, row sep=3em, column sep=3em, text height=1.9ex, text depth=0.25ex]
 {
 |(1)|{\cC}		& |(2)|{\hC} 	\\
 |(l1)|{\diff\cC}		& |(l2)|{\widehat{\diff\cC}} 	\\
 }; 
\path[->,font=\scriptsize,>=to, thin]
(1) edge node[above]{$\h$} (2) 
(1) edge [bend right=20] node(i) [left]{$\I$}   (l1)
(l1) edge [bend right=20] node(fi) [right]{$\Fix$}   (1)
(2) edge [bend right=20] node(qo) [left]{$\widehat{\Quo}$} (l2) 
(l2) edge [bend right=20] node(hi) [right]{$\widehat{\I}$}   (2)
(l1) edge  node[above]{$\h$} (l2)
(i) edge[draw=none] node{$\dashv$} (fi)
(qo) edge[draw=none] node{$\dashv$} (hi)
;
\end{tikzpicture}
\end{center}
summarise the above situation, when the curved parallelograms are interpreted as commutative.

Let $\cC$ be a category with (countable) fibre products and let $(S,\varsigma)\in\diff\cC$.
For an object $\bF$ in $\widehat{\cC_{\ov\forg{S}}}$, we define a functor
$$
\psig{\bF}_{\ov S}=\bF\circ\forg{\,}_S:\diff\cC_{\ov S}^\circ\to\Set.
$$
By adjointness of $\forg{\,}_S$ and $\psig{\,}_{\ov S}$ from \ref{forget-and-adj}, we see that for $X\in\Ob(\cC_{\ov S})$,
$$
\psig{\h_X}_{\ov S}=\h_{\psig{X}_{\ov S}},
$$
and we infer an analogous diagram to the above in the relative setting.

We obtain another relative version of $\psig{\,}$ when we work over an object $S_0\in\cC$. For a presheaf  $\bF$ in $\widehat{\cC_{\ov S_0}}$, we define a functor 
$$
\psig{\bF}\in\widehat{\diff\cC_{\ov\psig{S_0}}}
$$
as follows. For $S'\to\psig{S_0}$ in $\diff\cC_{\ov\psig{S_0}}$, naturally $\forg{S'}\to S_0$ is in $\cC_{\ov S_0}$, so we can set
$$
\psig{\bF}(S')=\bF(\forg{S'}).
$$
Suppose $X_0\to S_0$ is in $\cC_{\ov S}$. Then
$$
\psig{\h_{X_0/S_0}}(S')=\h_{X_0/S_0}(\forg{S'})=\cC_{\ov S_0}(\forg{S'},X_0)\simeq
\diff\cC_{\ov\psig{S_0}}(S',\psig{X_0})\simeq\h_{\psig{X_0}/\psig{S_0}}(S'),
$$ 
so $$\psig{\h_{X_0/S_0}}\simeq\h_{\psig{X_0}/\psig{S_0}}$$ and the notation is justified.

%
%

\subsection{Enriched presheaves on difference categories with pullbacks}

We define the category of presheaves on difference categories enriched with pullbacks (cf.\ \ref{diff-cat-pb}) as the category of $\diff\Set$-functors
$$
\diffst\psh(\diffst\cC)=\diff\Set[(\diffst\cC)^\circ,\diffst\Set].
$$
Such enriched presheaves $\bF$ come equipped with $\diff\Set$-morphisms
$$
\bF_{XY}:\diffst\cC(X,Y)\to\diffst\Set(\bF(Y),\bF(X)),
$$
i.e., $\bF_{XY}\circ\sigma^*_X=\sigma_{\bF(Y)}^*\circ \bF_{XY}$.

\subsection{Extending presheaves to difference presheaves}\label{psh-to-diffpsh}

For $\bF\in\hC$, we define an object $\psig{\bF}^{*}\in\diffst\psh(\diffst\cC)$ by
$$
\psig{\bF}^{*}(X,\sigma)=(\bF(\forg{X}),\bF(\forg{\sigma})).
$$
For objects $X$, $Y$, we define
$$
\psig{\bF}^{*}_{XY}:\diffst\cC(X,Y)\to\diffst\Set(\psig{\bF}^{*}(Y),\psig{\bF}^{*}(X))
$$
by simply setting $$\psig{\bF}^{*}_{XY}(f)=\bF(\forg{f}).$$
 Since $\bF(f\circ\sigma_X)=\bF(\sigma_Y\circ f)=\bF(f)\circ \bF(\sigma_Y)$, we conclude that
$\psig{\bF}^{*}_{XY}\circ\sigma_X^*=\sigma_{\bF(\sigma_Y)}^*\circ\psig{\bF}^{*}_{XY}$, so $\psig{\bF}^{*}_{XY}$ is a $\diff\Set$-morphism.

This construction specialises to the one from \ref{func-diff-cat} via the relation
$$
\psig{\bF}=\forg{\psig{\bF}^{*}}.
$$
The natural bijection giving the adjunction of $\forg{\,}$ and $\psig{\,}$ respects the difference structure defined by pullbacks, i.e.,
$$
(\Hom_{\diffst(\cC)}(Y,\psig{X}),\sigma_Y^*)\simeq (\Hom_\cC(\forg{Y},X),\forg{\sigma_Y}^*),
$$
which entails that $\psig{\h_X}^\Diff$ is represented by $\psig{X}$ in $\diffst(\cC)$. Indeed,
\begin{align*}
\psig{\h_X}^{*}(S)&=(\h_X(\forg{S}),\h_X(\forg{\sigma_S})\simeq(\cC(\forg{S},X),\sigma_{\forg{S}}^*)\\
& \simeq(\diff\cC(S,\psig{X}),\sigma_S^*)\simeq\diffst\cC(S,\psig{X})\simeq\h_{\psig{X}}(S).
\end{align*}
We obtain a diagram of categories
 \begin{center}
 \begin{tikzpicture} 
\matrix(m)[matrix of math nodes, row sep=2em, column sep=3em, text height=1.5ex, text depth=0.25ex]
 {
 |(1)|{\cC}		& |(2)|{\hC} 	\\
 |(l1)|{\diffst\cC}		& |(l2)|{\diffst\psh(\diffst\cC)} 	\\
 }; 
\path[->,font=\scriptsize,>=to, thin]
(1) edge node[above]{$\h$} (2) edge node[left]{$\psig{\,}$}   (l1)
(2) edge node[right]{$\psig{\,}^{*}$} (l2) 
(l1) edge  node[above]{$\h$} (l2);
\end{tikzpicture}
\end{center}
refining the picture from \ref{func-diff-cat}. 

Let $(S,\varsigma)\in\diff\cC$. For a functor $\bF\in\widehat{\cC_{\ov\forg{S}}}$, we define a functor
$$
\psig{\bF}^{*}:\diff{\cC}_{\ov\eq{S}}\to \diff\Set
$$
as follows. For a $\varsigma$-equivariant object $(X,\varphi_X)$, we let
$$
\psig{\bF}^{*}(X)=\left(\bF(\forg{X}),\bF(\forg{\varphi_X\circ\bar{\sigma}_X})\right).
$$

\subsection{Presheaves on generalised difference categories}\label{gendif-psh}

Using the adjunctions 
$$
\copinf\dashv\I_\infty, \ \ \ \ \ \ \copz\dashv\I_\Z\dashv\prodz
$$
from \ref{ord-vs-gen} and general principles of \ref{func-diff-cat}, we construct functors
$$
\I_\infty:\widehat{\diff\cC}\to \widehat{\forg{\diffinf\cC}}
$$
and
$$
\I_\Z:\widehat{\diff\cC}\to\widehat{\forg{\diffZ\cC}},
$$
along with its right adjoint $\prodz$. The diagrams
\begin{center}
 \begin{tikzpicture} 
\matrix(m)[matrix of math nodes, row sep=3em, column sep=3em, text height=1.9ex, text depth=0.25ex]
 {
 |(1)|{\diff\cC}		& |(2)|{\widehat{\diff\cC}} 	\\
 |(l1)|{\forg{\diffinf\cC}}		& |(l2)|{\widehat{\forg{\diffinf\cC}}} 	\\
 }; 
\path[->,font=\scriptsize,>=to, thin]
(1) edge node[above]{$\h$} (2) 
(1) edge node(i) [left]{$\I_\infty$}   (l1)
(2) edge node(qo) [left]{$\widehat{\copinf}$} (l2) 
(l1) edge  node[above]{$\h$} (l2)
;
\end{tikzpicture}
\hspace{3em}
\begin{tikzpicture} 
\matrix(m)[matrix of math nodes, row sep=3em, column sep=3em, text height=1.9ex, text depth=0.25ex]
 {
 |(1)|{\diff\cC}		& |(2)|{\widehat{\diff\cC}} 	\\
 |(l1)|{\forg{\diffZ\cC}}		& |(l2)|{\widehat{\forg{\diffZ\cC}}} 	\\
 }; 
\path[->,font=\scriptsize,>=to, thin]
(1) edge node[above]{$\h$} (2) 
(1) edge [bend right=20] node(i) [left]{$\I_\Z$}   (l1)
(l1) edge [bend right=20] node(fi) [right]{$\prodz$}   (1)
(2) edge [bend right=20] node(qo) [left]{$\widehat{\copz}$} (l2) 
(l2) edge [bend right=20] node(hi) [right]{$\widehat{\I_\Z}$}   (2)
(l1) edge  node[above]{$\h$} (l2)
(i) edge[draw=none] node{$\dashv$} (fi)
(qo) edge[draw=none] node{$\dashv$} (hi)
;
\end{tikzpicture}
\end{center}
show the relevant relationships between the above functors.

\subsection{Enriched presheaves on generalised difference categories}\label{repr-gendiff}


Let $\cC$ be a category. We have seen that $\diffinf\cC$ has the structure of of a $\diff\Set$-category. The category of difference presheaves on $\diffinf\cC$ is the $\diff\Set$-functor category
$$
\diffpsh(\diffinf\cC)=\diff\Set[(\diffinf\cC)^\circ,\diff\Set].
$$
In particular, a difference presheaf $F\in \diffpsh(\diffinf\cC)$ associates, to $X,Y\in\diffinf\cC$, a $\diff\Set$-morphism 
$$
\diffinf\cC(X,Y)\to[F(Y),F(X)].
$$
For $F,G\in\diffpsh(\diffinf\cC)$, the object $\Nat(F,G)\in\diff\Set$ is the equaliser
\begin{center}
\begin{tikzpicture} 
\matrix(m)[matrix of math nodes, row sep=0em, column sep=1.7em, text height=1.5ex, text depth=0.25ex]
 {
|(0)|{\cV\nat(F,G)} & |(1)|{\displaystyle\prod_{X\in\diffinf\cC}[FX,GX]}		& |(2)|{\displaystyle\prod_{X,X'\in\diffinf\cC}[\diffinf\cC(X,X'),[FX,GX']].} 	\\
 }; 
\path[->,font=\scriptsize,>=to, thin,yshift=12pt]
(0) edge node[above]{} (1)
([yshift=2pt]1.east) edge node[above]{} ([yshift=2pt]2.west) 
([yshift=-2pt]1.east)edge node[below]{}   ([yshift=-2pt]2.west) 
;
\end{tikzpicture}
\end{center}
Unravelling the definitions, we find that $\Nat(F,G)$ is the set of collections of morphisms $\varphi_X\in[F(X),G(X)]$, for $X\in\diffinf\cC$, such that, for every $f\in\diffinf\cC(X,X')$, for every $i\in\N$,
$$
G(f)\circ s^i\varphi_{X'}=s^i\varphi_{X}\circ F(f),
$$
with the natural difference structure.

In particular, a $\diff\Set$-natural transformation $\varphi:F\Rightarrow G$ is an element of $\Gamma(\Nat(F,G))=\Fix(\Nat(F,G))$, i.e., it is a collection of $\varphi_X\in\diff\Set(F(X),G(X))$ such that for every $f\in\diffinf\cC(X,X')$,
$$
G(f)\circ\varphi_{X'}=\varphi_{X}\circ F(G).
$$

For any $X\in\Ob(\diffinf\cC)$, we have a functor
$$
\h_X:\diffinf\cC^\circ\to\diff\Set, \ \ \ \ S\mapsto \diffinf\cC(S,X).
$$
By enriched Yoneda lemma, given a difference presheaf $F$,
there is a natural $\diff\Set$-isomorphism
$$
\Nat(\h_X,F)\simeq F(X),
$$
and the functor
$$
\h:\diffinf\cC\to\diffpsh(\diffinf\cC).
$$
is an embedding of $\diff\Set$-categories.

%

Generally speaking, if $\bF$ is an object in $\hdifC$, then functors $\Fix\circ\bF$, $\Quo\circ\bF$ and $\forg{\,}\circ\bF$ are objects of $\widehat{\diff\cC}$. 


Note that, for $X\in\Ob(\diff\cC)$, 
$$
\Fix\circ\h_{\I_\infty(X)}\circ\I_\infty\simeq\h_X,
$$
since, given $S\in\diff\cC$, we have
$$
\Fix\circ\h_{\I_\infty(X)}(\I_\infty(S))=\Fix(\Hom_\infty(\I_\infty(S),\I_\infty(X)))=\Hom_{\diff{\cC}}(S,X).
$$

Similarly, for $X\in\Ob(\diffZ(\cC))$, 
$$
\forg{\,}\circ\h_X\circ\I_\Z\simeq\h_{\prodz X},
$$
because
$$
\forg{\,}\circ\h_X\circ\I_\Z(S)=\forg{\Hom_{\diffZ(\cC)}(\I_\Z(S),X)}\simeq\Hom_{\diff\cC}(S,\prodz X).
$$

\subsection{Enriched difference presheaves and ordinary presheaves}

Let $\cC$ be a category. There exist functors $i^\infty$ and $r^\infty$ that make the diagram
$$
\begin{tikzcd}[column sep=normal, ampersand replacement=\&] 
{\forg{\difuinf\cC}}\arrow[r,yshift=2pt,"i^\infty"] \arrow[d] \&{\bsi\rtimes\difuinf\cC} \arrow[l,yshift=-2pt,"r^\infty"]\arrow[d]\\
{1}\arrow[r,yshift=2pt,"i"]  \&{\bsi} \arrow[l,yshift=-2pt,"r"]
\end{tikzcd}
$$
commutative, and such that $r^\infty\circ i^\infty=\id$.

Indeed, a morphism $X\xrightarrow{(n,f)}Y$ in $\bsi\rtimes\difuinf\cC$ consists of a natural number $n$ and an element $f\in\forg{\difuinf\cC(X,Y)}$, and we define
$$
r^\infty(n,f)=\sigma_Y^n\circ f\in \forg{\difuinf\cC}.
$$
This is a functor because
\begin{align*}
r^\infty((m,g)\circ(n,f)) &=r^\infty(m+n,s^n(g)\circ f)=\sigma_Z^{m+n}\circ s^n(g)\circ f=\sigma_Z^m \circ g\circ \sigma_Y^n \circ f\\
& = r^\infty(m,g)\circ r^\infty(n,f).
\end{align*}
Conversely, $i^\infty$ is defined by
$$
i^\infty(f)=(0,f).
$$

The principle \ref{psh-vs-int-diff-psh} yields a diagram 
 $$
\begin{tikzcd}[column sep=normal, ampersand replacement=\&] 
{[\forg{\difuinf\cC}^\op,\Set]}\arrow[r,yshift=2pt,"i^\infty"] \arrow[d] \&{[(\difuinf\cC)^\op,\diff\Set]} \arrow[l,yshift=-2pt,"r^\infty"]\arrow[d]\\
{\Set}\arrow[r,yshift=2pt,"i"]  \&{\diff\Set} \arrow[l,yshift=-2pt,"r"]
\end{tikzcd}
$$
of essential geometric morphisms.

Using the duality $(\diffinf\cC)^\op\simeq \difuinf(\cC^\op)$, we obtain a diagram 
$$
\begin{tikzcd}[column sep=normal, ampersand replacement=\&] 
{[\forg{\diffinf\cC},\Set]}\arrow[r,yshift=2pt] \arrow[d] \&{[\diffinf\cC,\diff\Set]} \arrow[l,yshift=-2pt]\arrow[d]\\
{\Set}\arrow[r,yshift=2pt]  \&{\diff\Set} \arrow[l,yshift=-2pt]
\end{tikzcd}
$$
of essential geometric morphisms.

\subsection{Creating generalised difference presheaves}\label{create-gen-dif-psh}

For $\bF\in\hC$, we define a $\diff\Set$-functor $$\psig{\bF}^{\infty}:(\difuinf\cC)^\circ\to\diff\Set$$ as follows. For $X\in\Ob(\diffinf\cC)$, 
$$
\psig{\bF}^{\infty}(X)=(\bF(\forg{X}),\bF(\forg{\sigma_X}))\in\Ob(\diff\Set).
$$  
For a morphism $(f_i)\in\difuinf\cC(X,Y)$, let
$$
\psig{\bF}^{\infty}(f_i)=(\bF(f_i))\in[\bF(Y),\bF(X)].
$$
This construction respects representability. If $Y\in\cC$, then, using \ref{CvsdiffinfC},
$$
\h_{\jj^\infty(Y)}(X)=\difuinf\cC(X,\jj^\infty(Y))\simeq(\cC(\forg{X},Y),\sigma_X^{*})=(\h_Y(\forg{X},\h_Y(\forg{\sigma_X}),
$$
whence
$$
\psig{\h_Y}^\infty\simeq\h_{\jj^\infty(Y)}.
$$

\subsection{Solving difference equations in difference presheaves}

Let $E\in\diff\Set$, $e$ a terminal object in $\diffinf\cC$ and let $\bG\in\diffpsh(\diffinf\cC)$. We define an object $\bG^E\in\diffpsh(\diffinf\cC)$ as
$$
\bG^E=\uHom_{(\diff\Set)\text{-}\mathbf{Cat}}(E_{e},\bG).
$$
By the universal property of constant difference objects,  for $S\in\Ob(\diffinf\cC)$,
\begin{align*}
\bG^E(S)& =\uHom(E_{e},\bG)(S) 
\simeq\Nat(E_S,\bG) \\ &\simeq\Nat(E\otimes S,\bG)\simeq[E,\Nat(S,G)]
 \simeq [E,\bG(S)].
\end{align*}

The underlying functor $\bG_0^E\in\Ob(\widehat{\diff\cC})$ is therefore given by the rule
$$
\bG_0^E(S)=\Hom_{\diff\Set}(E,\bG_0(S)).
$$

Hence, if $\bH\in\diffpsh(\diff\cC)$, it is natural to define $\bH^E\in\widehat{\diff\cC}$ as
$$
\bH^E(S)=\diff\Set(E,\bH(S)).
$$ 

%

\section{Difference algebra}

\subsection{Difference rings}

\subsection{Difference modules}

Let $R\in\diff\Rng$. A difference abelian group $M$ is an $R$-module, if $\forg{M}$ is an $\forg{R}$-module, and, for all $r\in R$, $m\in M$,
$$
\sigma_M(r.m)=\sigma_R(r).\sigma_M(m).
$$
A morphism of $R$-modules is a morphism of the underlying $\forg{R}$-modules that commutes with $\sigma$. We thus obtain the category of difference $R$-modules denoted
$$
R\Mod.
$$
%
%
%
%
%
It is a monoidal category, with the tensor product structure given by
$$
M\otimes_RN=(\forg{M}\otimes_{\forg{R}}\forg{N},\sigma_M\otimes\sigma N).
$$
This is an $R$-module equipped with a difference bilinear map $\otimes:M\times N\to M\otimes_RN$ such that any difference bilinear map $M\times N\to P$  to an $R$-module $P$ factors uniquely through an $R$-module map $M\otimes_RN\to P$, i.e., the diagram
\begin{center}
 \begin{tikzpicture} 
\matrix(m)[matrix of math nodes, row sep=1em, column sep=2em, text height=1.9ex, text depth=0.25ex]
 {
 			& |(2)|{M\otimes_RN} 	\\
 |(1)|{M\times N}		& 	\\
			& |(l2)|{P} 	\\
 }; 
\path[->,font=\scriptsize,>=to, thin]
(1) edge node[above]{$\otimes$} (2) edge    (l2)
(2) edge[dashed]  (l2);
\end{tikzpicture}
\end{center}
can be uniquely completed by a dashed arrow. Equivalently,
$$
\text{Bil}_R(M,N;P)\simeq R\Mod(M\otimes_RN,P).
$$
We provide  $R\Mod$ with the structure of symmetric monoidal closed category as follows. Given $M,N\in R\Mod$, we let
$$
[M,N]=\{(f_i): f_i\in\forg{R}\Mod(\forg{M},\forg{N}), \ \ f_{i+1}\circ\sigma_M=\sigma_N\circ f_i\}, 
$$
together with the shift $s(f_0,f_1,\ldots)=(f_1,f_2,\ldots)$. Moreover, $[M,N]$ is an $R$-module with component-wise addition and the the scalar multiplication given by
$$
r.(f_i)=(\sigma^i_R(r).f_i).
$$
Following the template from \ref{diff-mon-closed} and \ref{topos-diff-sets}, we verify that
$$
\Hom_R(M\otimes_RN,P)\simeq\Hom_R(M,[N,P]).
$$
Applying general principles on monoidal closed categories, this generalises to 
$$
[M\otimes_RN,P]\simeq[M,[N,P]].
$$
We will write
$$
R_\infty\Mod
$$
for the monoidal closed category $R\Mod$ considered as enriched over itself, i.e., we may write
$$
R_\infty(M,N)=[M,N]\in R\Mod,
$$
and we refer to it as the category of \emph{enriched difference modules}.

\subsection{Skew polynomial rings}\label{skew-poly-prop}

Let $R$ be a difference ring, and consider the skew polynomial ring $\Do(R)=R[T;\sigma_R]$. Equivalently, $\Do(R)$ can be described as the constant object $\N_R=\N\otimes R$, or as a difference semigroup ring $R[\N]$. 

The ring $\Do(R)$ acts naturally on any $R$-module $M$ by $\Fix(R)$-module endomorphisms,
$$
(\sum_ir_iT^i).m=\sum_i r_i\sigma_M^i(m),
$$
where $r_i\in R$, $m\in M$. 

Note, since $\sigma_{\Do(R)}(f)=Tf$ for $f\in\Do(R)$, we deduce the following.
\begin{remark}\label{skew-diff-equiv}
The category of difference $R$-modules is equivalent to the category of left $\Do(R)$-modules,
$$
R\Mod\simeq \Do(R)\Mod.
$$
Through this equivalence, finitely $\sigma$-generated $R$-modules
correspond to finite left $\Do(R)$-modules.
\end{remark}

Consequently, for any $R$-module $M$, and $f\in\Do(R)$, 
$$
\Hom_{R\Mod}(\Do(R)/\Do(R)f,M)\simeq \{a\in M: f.a=0\}
$$
is the set of solutions to a linear difference equation $f=0$ in $M$.

\subsection{Difference twists and ring changes}

The classical base change isomorphism is compatible with the difference structure. Indeed, given a homomorphism $\varphi:R\to R'$ of difference rings, we have functors
$$
(\mathord{-})_\varphi :R\Mod\to R'\Mod, \ \ \ \ M_\varphi=M\otimes_RR',
$$
and
$$
i_\varphi:R'\Mod\to R\Mod, 
$$
where, given an $R'$-module $M'$, $i_\varphi(M')$ is the abelian group $M'$ given the $R$-module structure through the homomorphism $\varphi$.

The difference ring change functor $(\mathord)_\varphi$ is left adjoint to the functor $i_\varphi$, i.e., 
for every $R$-module $M$ and $R'$-module $M'$, we have a natural bijection
$$
R'\Mod(M_\varphi,M')\simeq R\Mod(M,i_\varphi M').
$$
%

In particular, for $\varsigma=\sigma_R:R\to R$, we obtain the difference twist functor
$$
(\mathord{-})_\varsigma:R\Mod\to R\Mod
$$
and its right adjoint 
$$
i_\varsigma:R\Mod\to R\Mod.
$$

\subsection{Semilinear maps}

Let $(R,\varsigma)\in\diff\Rng$, let $M,M'\in R\Mod$. A map $\varphi:M\to M'$ is called \emph{$\varsigma$-linear} (or \emph{semilinear} when $\varsigma$ is clear from the context), if, for all $r\in R$ and $m\in M$,
$$
\varphi(r\cdot m)=\varsigma(r)\cdot\varphi(m).
$$
In other words, a semilinear map $\varphi:M\to M'$ can be viewed as an element 
$$
\varphi\in R\Mod(M,i_\varsigma M').
$$
Using the adjunction $(\mathord{-})_\varsigma \dashv i_\varsigma$, we obtain the 
\emph{associated linear map}
$$
\bar{\varphi}\in R\Mod(M_\varsigma,M).
$$
More explicitly, we have a natural semilinear map $\iota:M\to M_\varsigma$ given by
$$
\iota(m)=1\otimes m,
$$ 
and $\bar{\varphi}:M_\varsigma\to M'$ is given by
$$
\bar{\varphi}(r\otimes m)=r\cdot\varphi(m),
$$
so that
$$
\varphi=\bar{\varphi}\circ\iota.
$$

Suppose that $M$ and $M'$ are free $R$-modules with bases $e=(e_i)_{i\in I}$ and $f=(f_j)_{j\in J}$, and that we are given a semilinear map $\varphi:M\to M'$. Let $B=(b_{ji})$ be the $R$-matrix such that
$$
\varphi(e_i)=\sum_j b_{ji}f_j.
$$
For an element $x=\sum x_i e_i\in M$, we denote its coordinate tuple by
$$
[x]_e=(x_i)\in \oplus_{i\in I}\forg{R}.
$$
Then, writing $\varsigma$ for the component-wise action of $\varsigma$ on $\oplus_{i\in I}\forg{R}$, we obtain that
$$
[\varphi(x)]_f=B\varsigma([x]_e).
$$ 
We call $B$ \emph{the matrix of $\varphi$} with respect to the pair of bases $e,f$. 

Let $A$ be the transition matrix from the basis $e$  to a basis $e'$, and let $C$ be the transition matrix from $f$ to a basis $f'$. Let $B'$ be the matrix of $\varphi$ with respect to the pair $e',f'$. Then we compute that
$$
B'=CB\varsigma(A)^{-1},
$$
noting that $\varsigma(A)$ is invertible because $A$ is invertible.

\subsection{Modules vs.~difference modules}

 If $R$ is a difference ring with  $\sigma_R=\varsigma\neq\id$, the category $R\Mod$ is not the difference category associated to a category of modules so we cannot readily resort to techniques from \ref{forget-and-adj}. On the other hand, we can develop completely analogous techniques, as long as we treat twisting by $\varsigma$ carefully.
 
 We consider the natural forgetful functor
 $$
 \forg{\,}=\forg{\,}_R:R\Mod\to\forg{R}\Mod,
 $$
 and construct its right and left adjoints.
 
 Given a module $A_0\in\forg{R}\Mod$, write $A_j=i_{\varsigma^j}A_0$ and let
 $$
 \psig{A_0}=\psig{A_0}_R=\prod_{j\in\N}A_j,
 $$
 together with the shift map 
 $$
 \sigma:\psig{A_0}\to\psig{A_0}, \ \ \ \ \sigma(a_0,a_1,\ldots)=(a_1,a_2,\ldots).
 $$
 We directly verify that $(\psig{A_0}_k,\sigma)\in R\Mod$, i.e., that $\sigma(\lambda.a)=\varsigma(\lambda).a$, for $\lambda\in R$ and $a\in \psig{A_0}$.
 
 For an $\forg{R}$-module map $f_0:A_0\to B_0$, we let $f_j=i_{\varsigma^j}f_0$ and
 $$
 \psig{f_0}=\prod_j f_j:\psig{A_0}\to\psig{B_0},
 $$
 and we verify directly that $\psig{f_0}$ is a morphism of $R$-modules.
 
 Thus we have defined the functor
 $$
 \psig{\,}=\psig{\,}_R:\forg{R}\Mod\to R\Mod,
 $$
 which is right adjoint to $\forg{\,}_R$, 
 $$
 \forg{\,}_R\dashv\psig{\,}_R.
 $$
 
 Indeed, for $A_0\in\forg{R}\Mod$ and $B\in R\Mod$, we have a natural bijection
 $$
 \forg{R}\Mod(\forg{B},A_0)\simeq R\Mod(B,\psig{A_0}).
 $$ 
 
 The counit $\epsilon:\forg{\psig{\,}}\to\id$ is defined by stipulating that
 $$
 \epsilon_{A_0}:\forg{\psig{A_0}}=\prod_j A_j\to A_0
 $$
 be the projection onto the first factor.
 
 The unit $\eta:\id\to\psig{\forg{\,}}$ is defined by
 $$
 \eta_B:B\to\psig{\forg{B}}=\prod_j i_{\varsigma^j}\forg{B}, \ \ \ \ b\mapsto(b,\sigma b,\sigma^2 b,\ldots).
 $$
 
 Dually, given $A_0\in\forg{R}\Mod$, write $A_i=(A_0)_{\varsigma^i}$, so that $A_{i+1}\simeq (A_i)_\varsigma=A_i\otimes_RR$, and we have semilinear maps $\sigma_{i}:A_i\to A_{i+1}$, $\sigma_{i}(a)=a\otimes 1$.  Let 
 $$\ssig{A_0}=\oplus_{i\in\N}A_i,$$
 together with 
 $$
 \sigma:\ssig{A_0}\to\ssig{A_0}, \ \ \ \ \sigma(\sum_{i\in\N}a_i)=\sum_{i\in\N}\sigma_{i}(a_i),
 $$
 where $a_i\in A_i$.
 
 Given an $\forg{R}$-module map $f_0:A_0\to B_0$, we let $f_i=(f_0)_{\varsigma^i}:A_i\to B_i$ and
 $$
 \ssig{f_0}=\oplus_i f_i:\ssig{A_0}\to\ssig{B_0},
 $$
 and we readily verify that $\ssig{f_0}$ is a morphism of $R$-modules.
 
 Thus we have defined the functor
 $$
 \ssig{\,}=\ssig{\,}_R:\forg{R}\Mod\to R\Mod,
 $$
 which is left adjoint to $\forg{\,}_R$, 
 $$
 \ssig{\,}_R\dashv\forg{\,}_R.
 $$
 In other words, for $A_0\in\forg{R}\Mod$ and $B\in R\Mod$, there is a natural bijection
 $$
 R\Mod(\ssig{A_0}_R,B)\simeq\forg{R}\Mod(A_0,\forg{B}).
 $$
 The counit $\epsilon:\ssig{\forg{\,}}\to\id$ is defined in terms of linear maps $\bar{\sigma}_B^i:\forg{B}_{\varsigma^i}\to \forg{B}$ associated to $\sigma_B^i$ via
 $$
 \epsilon_B:\ssig{\forg{B}}\to B, \ \ \ \ \epsilon_B(\sum_i b_i)=\sum_i\bar{\sigma}_B^i(b_i),
 $$ 
 where $b_i\in \forg{B}_{\varsigma^i}$.

The unit $\eta:\id\to\forg{\ssig{\,}}$ is defined by taking canonical maps
$$
\eta_{A_0}:A_0\to\forg{\ssig{A_0}}\simeq\oplus_i A_i
$$
into the first summand.

\subsection{Difference modules and enriched difference modules}

The constructions of tensors and cotensors from \ref{gen-diff-is-tensored-cotensored}, applied to $R_\infty\Mod$, bear a natural difference module structure, so we conclude that
$R_\infty\Mod$ is tensored and cotensored over $\diff\Set$. 

Hence, the general constructions from \ref{ord-vs-gen} and \ref{diff-vs-gendif-enr} carry over virtually unchanged to the context of generalised difference modules.

Hence, we obtain 
additive functors 
$$
\copinf, \prodinf:\forg{R_\infty\Mod}\to R\Mod
$$
which are adjoints of $I_\infty:R\Mod\to \forg{R_\infty\Mod}$, i.e.,  
$$
\copinf \dashv I_\infty \dashv \prodinf.
$$
Moreover, by the same argument as in \ref{co-prodinf-exact}, they are exact when considered as functors $R\Mod\to R\Mod$.

Their enriched counterparts are the
$R\Mod$-functors
$$
\copinf:R_\infty\Mod\to R_\infty\Mod, \ \ \ \ M\mapsto\N\otimes M,
$$
and 
$$
\prodinf:R_\infty\Mod\to R_\infty\Mod, \ \ \ \ M\mapsto[\N,M],
$$
so that $\copinf$ is left $R\Mod$-adjoint to $\prodinf$.

\subsection{\'Etale difference modules}

\begin{definition}
Let $(R,\varsigma)\in\diff\Rng$, let $M\in R\Mod$, and let $\bar{\sigma}_M:M_\varsigma\to M$ be the module map associated with the $\varsigma$-linear map $\sigma_M$. We say that $M$ is an \emph{\'etale $R$-module} if $\bar{\sigma}_M$ is an isomorphism.
\end{definition}


\begin{lemma}
Let $M$ be an $R$-module such that $\forg{M}$ is finite free over $\forg{R}$. The following conditions are equivalent.
\begin{enumerate}
\item $M$ is \'etale;
\item for every basis $(e_i)$ of $M$, the sequence $(1\otimes e_i)$ is a basis for $M_\varsigma$;
\item for every basis $(e_i)$ of $M$, the sequence $(\sigma(e_i))$ is a basis for $M$;
\item for some basis $(e_i)$ of $M$, the sequence $(\sigma(e_i))$ is a basis for $M$;
\item the matrix of $\sigma_M$ is invertible.
\end{enumerate}
\end{lemma}
\begin{proof}
Suppose $\bar{\sigma}:M_\varsigma\to M$ has inverse $u:M\to M_\varsigma$. Let $(e_i)$ be a basis for $M$, and hence $(u(e_i))$ is a basis for $M_\varsigma$.  Write
$$
u(e_i)=\sum_j u_{ji}\otimes e_j,
$$ 
and let $A=(a_{kj})$ be the matrix for $\sigma$, i.e., for $x\in M$,
$$
[\sigma(x)]_e=A\varsigma([x]_e).
$$
Then
$$
e_i=\bar{\sigma}u(e_i)=\sum_j u_{ji}\sigma(e_j)=\sum_j u_{ji}\sum_k a_{kj}e_k=
\sum_k\left(\sum_j a_{kj}u_{ji}\right)e_k,
$$
whence we conclude that $AU=I$ so $U$ and A are invertible and it follows that $(1\otimes e_i)$ and $(\sigma(e_i))$ are bases.

Moreover, we see that $A$ is the matrix of the module map $\bar{\sigma}$ in bases $(1\otimes e_i)$ and $(e_i)$. 

Note, if $A'$ is the matrix of $\sigma$ in a different basis, and $B$ is the transition matrix between the bases,
$$
A'=B A \varsigma(B)^{-1},
$$
so that the condition of invertibility of the matrix of $\sigma$ does not depend on the choice of a basis.
\end{proof}


\begin{lemma}\label{inthom-etale}
Let $R\in\diff\Rng$ and let $M,M'\in R\Mod$ with $M$ \'etale. Then
$$
[M,M']\simeq\forg{R}\Mod(\forg{M},\forg{M'}),
$$
together with the `shift' 
$$
\varphi_0\mapsto \bar{\sigma}_{M'}\circ (\varphi_0)_\varsigma\circ \bar{\sigma}_M^{-1}.
$$
\end{lemma}
\begin{proof}
Given a morphism $\varphi=(\varphi_i)\in[M,M']$, we assign the component $\varphi_0$ to it.

Conversely, let $\varphi_0\in \forg{R}\Mod(\forg{M},\forg{M'})$. For $i\in\N$, define
$$
\varphi_{i+1}= \bar{\sigma}_{M'}\circ (\varphi_i)_\varsigma\circ \bar{\sigma}_M^{-1}.
$$
We have
$$
\varphi_{i+1}\circ\sigma_M=\bar{\sigma}_{M'}\circ (\varphi_i)_\varsigma\circ \bar{\sigma}_M^{-1}\circ\sigma_M=\bar{\sigma}_{M'}\circ (\varphi_i)_\varsigma\circ\iota_M=\bar{\sigma}_{M'}\circ\iota_{M'}\circ\varphi_i=\sigma_{M'}\circ\varphi_i,
$$
so that $(\varphi_i)\in[M,M']$. 

These assignments are mutually inverse.
\end{proof}

\subsection{Duals}

Given a difference $R$-module $M$, its (unrestricted) dual is defined as
$$
M^\vee=[M,R].
$$
Taking the adjoint of the evaluation map
$$
[M,R]\otimes M\to R, \ \ \ f\otimes m\mapsto f_0(m)
$$
yields the canonical map
$$
\iota:M\to M^{\vee\vee},\ \ \ \iota(m)_i(f)=f_0(\sigma_M^i(m)). 
$$

\begin{example}
Let $(R,\varsigma)\in\diff\Rng$. The $R$-modules $\copinf R$ and $\ssig{\forg{R}}$ are isomorphic, given that for $M\in R\Mod$,
$$
R\Mod(\copinf R,M)\simeq\forg{[R,M]}\simeq\forg{M}\simeq\forg{R}\Mod(\forg{R},\forg{M})\simeq R\Mod(\ssig{\forg{R}},M).
$$ 
More explicitly, they are both isomorphic to $R_\infty=\oplus_{j\in\N}R\simeq \oplus_{j\in\N}Re_j$, with the endomorphism $\sigma(\sum_j x_je_j)=\sum_j\varsigma(x_j)e_{j+1}$.

We claim that
$$
R_\infty^\vee=[R_\infty,R]\simeq R^\Z,
$$
together with the shift $s:R^\Z\to R^\Z$ defined, for $c=(c_i)\in R^\Z$, by the rule
$$
s(c)_j=\begin{cases}
c_{j-1}, & \text{for }j\leq 0\\
\varsigma c_{j-1}, & \for{for }j>0.
\end{cases} 
$$
In particular, $R_\infty^\vee$ is naturally equipped with an $R$-algebra structure.

Indeed, suppose $f=(f_i)\in R_\infty^\vee$. Then each $f_i$ belongs to
$$
\forg{R_\infty}^\vee=\left(\oplus_{j\in\N}\forg{R}\right)^\vee\simeq\prod_{j\in\N}\forg{R}^\vee\simeq\prod_{j\in\N}\forg{R}\simeq\forg{R}^\N.
$$
Hence, each $f_i$ can be identified with a tuple $b_i=(b_i^j)\in \forg{R}^\N$ so that, for $x=\sum_j x_j e_j\in R_\infty$, $f_i(x)=\sum_j b_i^j x_j$. The condition $f_{i+1}\circ\sigma=\varsigma\circ f_i$ implies that
$$
b_{i+1}^{j+1}=\varsigma(b_i^j).
$$
Thus, the whole array $b_i^j$ is determined by $b_0^j$, $j\in\N$ and $b_i^0$, $i\in\N$. We combine the two sequences into a sequence $c\in\forg{R}^\Z$ through
$$
c_j=\begin{cases}
b_0^{j}, & \text{for }j\geq 0\\
b^0_{-j}, & \for{for }j<0.
\end{cases} 
$$
It is straightforward to verify that the usual shift of the sequence $(b_i)$ has the claimed effect on the sequence $c$.
\end{example}

\begin{proposition}
Let $R\in\diff\Rng$ and let $M,N\in R\Mod$.
There is a canonical map 
$$
\alpha:M^\vee\otimes_RN\to[M,N], 
$$
given by $\alpha(f\otimes n)_i(m)=f_i(m).\sigma_N^i(n)$.
\begin{enumerate}
\item If $M$ is finite \'etale, then $\alpha$ is bijective.
\item If $N$ is \'etale (resp.\ finite \'etale), then $\alpha$ is injective (resp.\ bijective).
\end{enumerate}
\end{proposition}

\begin{proof}
For (1), we start by noting that, if $M$ is finite \'etale, then $$[M,N]\simeq (\forg{k}\Mod(\forg{M},\forg{N}),s),$$ where the shift $s$ is defined as follows. We identify $M\simeq\oplus_{j=1}^n ke_j\simeq k^{\oplus n}$ so that $\sigma_M(x)=A\cdot\sigma(x)$ for some $A\in \text{GL}_n(R)$. For a $\forg{R}$-module map $f:\forg{M}\to\forg{N}$, let $f(e_j)=b_i\in N$, so we can formally write
$$
f(x)=B\cdot x,
$$ 
for $x\in R^{\oplus n}$, where $B=[b_1,\ldots,b_n]$, which stands for the relation $f(\sum_j x_j e_j)=\sum_j x_jb_j$.
With this notation, we define $s(f)$ by
$$
s(f)(x)=\sigma_N(B)\cdot A^{-1}\cdot x.
$$
We can now directly verify that
$$
s(f)\circ\sigma_M=\sigma_N\circ f.
$$
Indeed,
$$
s(f)(\sigma_M(x))=s(f)(A\cdot\varphi(x))=\sigma_N(B)\cdot A^{-1}\cdot A\cdot\varsigma(x)=
\sigma_N(B)\cdot\varsigma(x)=\sigma_N(B\cdot x)=\sigma_N (f(x)).
$$

In particular, 
$$
M^\vee=(\forg{M}^\vee,s). 
$$
On the other hand, classically (\cite[II,\S4.2, Proposition~2(ii)]{bourbaki-alg}) we know that
$$
\forg{M}^\vee\otimes\forg{N}\simeq\forg{R}\Mod(\forg{M},\forg{N})\simeq\forg{[M,N]},
$$
so the canonical map is bijective.

The proof of (2) proceeds in the spirit of \cite[II,\S4.2, Proposition~2(i)]{bourbaki-alg}.
Suppose that $N$ is \'etale, so it has a basis $(e_j:j\in J)$ so that for every $i$, 
$(\sigma_N^i(e_j):j\in J)$ is again a basis. Every element of $M^\vee\otimes_R N$ has a unique representation as a finite sum
$$
\sum_j f^j\otimes e_j.
$$
It is mapped to 0 in $[M,N]$ if, for every $i$ and $m\in M$,
$\sum_j  f^j_i(m)\sigma^ie_j=0$, which implies that all $f^j_i=0$, i.e., $f^j=0$ for all $j$, so we conclude that $\alpha$ is injective.

Assume now that $N$ is finite \'etale, i.e., that $J$ is finite. To show that $\alpha$ is surjective, let $f=(f_i)\in [M,N]$, and write
$$
f_i(x)=\sum_j f_i^j(x)\sigma_N^i e_j
$$
for unique $f_i^j$, $j\in J$. Let $f^j=(f_i^j)$, and we claim that for all $j\in J$, $f^j\in M^\vee=[M,R]$, i.e., that
$$
f_{i+1}^j\circ\sigma_M=\varsigma\circ f_i^j.
$$
Since $f_{i+1}\circ\sigma_M=\sigma_N\circ f_i$, we get that
$$
\sum_j f_{i+1}^j(\sigma_Mx)\sigma_N^{i+1}e_j=\sigma_N\left(\sum_j f_i^j(x)\sigma_N^ie_j\right)=\sum_j\varsigma(f_i^j(x))\sigma_N^{i+1}e_j,
$$
whence we obtain the required relation by comparing terms. It is routine to verify that
$$
\alpha\left(\sum_j f^j\otimes e_j\right)=f.
$$
\end{proof}

\begin{corollary}\label{etale-is-reflexive}
If $M$ is a finite \'etale module over a difference ring $R$, then $M$ is reflexive,
$$
M\simeq M^{\vee\vee}.
$$ 
\end{corollary}

\subsection{Enriched difference module maps with finite-dimensional support}

Define 
$$
[M,N]_c
$$
as the $R$-module consisting of those $(f_i)\in [M,N]$ such that $f_0$ has finite-dimensional support, 
%
%
and there exists an $i\in\N$ such that for every $j\geq i$,  $f_j\restriction_{M\setminus R\sigma M}=0$.
The $R$-module
$$
[M,N]_{+}
$$
consists of those $(f_i)\in[M,N]_c$ such that, for $i>0$, $f_i\restriction_{M\setminus R\sigma M}=0$.

An $R$-module $L$ is free finitely $\sigma$-generated, if there exist elements $e_1,\ldots,e_n\in L$ so that $L$ is a free module on $e_1,\ldots,e_n, \sigma e_1,\ldots,\sigma e_n, \sigma^2 e_1,\ldots, \sigma^2 e_n,\ldots$. Equivalently, $L$ is free finitely $\sigma$-generated by $n$ elements if it is of the form
$$
L\simeq \ssig{\forg{R}^n}_R.
$$

In addition to $L^\vee=[L,R]$, we consider the $R$-modules
$$
L^\vee_c=[L,R]_c\ \ \ \text{ and }\ \ \ L^\vee_{+}=[L,R]_{+}.
$$
The operators
$$
\ldexp{i}{e^\vee}{j}:\forg{L}\to\forg{R}, \ \ \  \ldexp{i}{e^\vee}{j}(\sigma^ke_l)=\delta_{ik}\delta_{jl}
$$
fit together to form elements of the dual basis
$$
e_j^\vee=(\ldexp{i}{e^\vee}{j})_i\in L^\vee_{+},
$$
for $j=1,\ldots,n$.
Although the shift $s$ on $L^\vee$ is not invertible, we write
$$
s^{-r}(e_j^\vee)=(\underbrace{0,\ldots,0}_{r\text{ times}})^\frown e_j^\vee\in L^\vee_c.
$$
It follows that
$L^\vee_{+}$ is freely $s$-generated by $e_1^\vee,\ldots,e_n^\vee$. More explicitly, for $f\in L^\vee_{+}$,
$$
f=\sum_{r\geq 0}\sum_{j=1}^n f_0(\sigma_L^re_j).s^r e_j^\vee=\sum_{r\geq 0}\sum_{j=1}^n \sigma_R^rf_r(e_j).s^r e_j^\vee,
$$
noting that the sum on $r$ is actually finite. The expression can be verified by evaluating both sides on the generators $\sigma^ie_j$. 

Moreover, if $R$ is inversive, then $L^\vee_c$ is freely $\{s,s^{-1}\}$-generated by $e_1^\vee,\ldots,e_n^\vee$. Explicitly, for $f\in L^\vee_c$,
$$
f=\sum_{r\in\Z}\sum_{j=1}^n \sigma_R^rf_r(e_j).
$$
For a free finitely $\sigma$-generated $R$-module $L$ and any $R$-module $M$  we have a canonical isomorphism
$$
L^\vee_{+}\otimes_RM\simeq[L,M]_{+}.
$$
We show that the canonical map $\alpha: L^\vee_{+}\otimes_RM\to[L,M]_{+}$, defined in the general case as 
$$
\alpha(f\otimes m)_i(x)=f_i(x).\sigma_M^i(m)
$$
realises this isomorphism. Let $h=(h_i)=\alpha(f\otimes m)$, for $f=(f_i)\in L^\vee_{+}$ and $m\in M$. In order to show that $h\in[L,M]_{+}$, we verify
\begin{enumerate}
\item $h_{i+1}\sigma_L=\sigma_M\circ h_i$;
\item for every $x\in L$, there exists an $i$ so that for every $j\geq i$, $h_0(\sigma^jx)=0$;
\item for $i>0$, $h_i(e_j)$.
\end{enumerate}
These are all direct verifications, so let us check the first condition only. For $x\in L$,
\begin{align*}
h_{i+1}(\sigma_Lx)& =f_{i+1}(\sigma_Lx).\sigma_M^{i+1}m=\sigma_R(f_i(x)).\sigma_M^{i+1}(m)\\ & =\sigma_M(f_i(x).\sigma_M^i(m))=\sigma_M h_i(x).
\end{align*}
Conversely, for $h\in[L,M]_{+}$, we consider
$$
u=\sum_{i\geq 0}\sum_{j=1}^n s^ie_j^\vee\otimes h_0(\sigma_L^i e_j).
$$
It is readily verified that $\alpha(u)=h$, that $\alpha$ is invertible and a difference map.

When $L$ is a free finitely $\sigma$-generated $R$-module,  we have a canonical isomorphism 
$$
L\simeq (L^\vee_{+})^\vee_{+}.
$$
We show that the previously defined canonical map $\iota:L\to L^{\vee\vee}$ given by 
$$\iota(x)_i(f)=f_0(\sigma_L^i(x))$$
realises this isomorphism. Indeed, a nonzero $x\in L$ must have a non-zero coordinate with respect to the $\sigma$-basis $e_1,\ldots,e_n$ of $L$, so there must be an operator $\ldexp{i}{e^\vee}{j}$ with $\ldexp{i}{e^\vee}{j}(x)\neq 0$. Thus, for $e^\vee_j\in L^\vee_{+}$, $\iota(x)(e^\vee_j)\neq 0$, so $\iota$ is injective. On the other hand, if 
$h=(h_i)\in (L^\vee_{+})^\vee_{+}$, we define
$$
x=\sum_{i\geq 0}\sum_{j=1}^{n}h_0(s^i e^\vee_j)\sigma^i e_j,
$$
and we verify that $\iota(x)=h$ by evaluating both on each element of the dual basis $s^i e^\vee_j$.

More generally, let $L$ be a finite \'etale $R$-module, and let $\tilde{L}=\copinf L$. Then 
$$
(\tilde{L}^{\vee}_{+})^{\vee}_{+}\simeq\tilde{L}.
$$
Indeed, for any $R$-module $M$,
$$
[\tilde{L},M]_{+}=[\copinf L,M]_{+}\simeq\copinf[L,M]_{+}\simeq \copinf[L,M].
$$
Thus,
$$
(\copinf L)^\vee_{+}=[\copinf L,R]_{+}\simeq \copinf [L,R]\simeq\copinf L^\vee.
$$
and 
$$
((\copinf L)^\vee_{+})^\vee_{+}\simeq\copinf L^{\vee\vee}\simeq\copinf L.
$$

\subsection{Difference symmetric tensor algebra}\label{diff-sym-alg}

Let $R$ be a difference ring. We have seen that the category $R\Mod$ is a (cocomplete) symmetric monoidal closed category. The category
$$
R\Alg
$$
of commutative difference $R$-algebras is the category of commutative algebra objects (monoids) in $R\Mod$. By general principles of \cite[Lemma~4.4.5]{branden}, the forgetful functor $R\Alg\to R\Mod$ has a left adjoint 
$$\sym:R\Mod\to R\Alg.$$
More explicitly, given an $R$-module $M$, the $R$-algebra $\sym(M)$ is simply the symmetric algebra of $\forg{M}$, endowed with a natural difference structure induced by $\sigma_R$ and $\sigma_M$. The natural morphism
$$
M\to \sym(M)
$$
is a morphism in $R\Mod$, such that every $R$-module map $M\to B$ to an $R$-algebra $B$ factors through it by an $R$-algebra map $\sym(M)\to B$, i.e.,
$$
\Hom_{R\Mod}(M,B)\simeq\Hom_{R\Alg}(\sym(M),B).
$$
Moreover, this generalises to
$$
[M,B]_{R\Mod}\simeq [\sym(M),B]_{R\Alg},
$$
where the right-hand side stands for $R_\infty\Alg(\sym(M),B)$.

\section{Difference homological algebra}\label{diff-homol}

\subsection{Difference category of an abelian category}

\begin{remark}
If $\cA$ is an abelian category, then the category $\diff\cA$ is again abelian since
 it is the functor category $[\bsi,\cA]$.

Moreover, in view of the fact that $(\diff\Set)\da\Ab=\diff\Ab$, the category $\diffinf\cA$ is 
$\diff\Set$-abelian in the sense of \ref{enr-ab-cat}.
\end{remark}

\begin{lemma}\label{co-prodinf-exact}
The functors $\copinf$ and $\prodinf$, considered as functors $\diff\cA\to\diff\cC$ are exact. 
\end{lemma}
\begin{proof}
Using the explicit constructions of these functors from \ref{ord-vs-gen}, we see that $\copinf$ transforms a short exact sequence into a sum of its copies, and $\prodinf$ transforms a short exact sequence into a product of its copies, and both of these operations are exact. 
\end{proof}

\begin{proposition}\label{diffcat-enough-inj}
If the abelian category $\cA$ has enough injectives, then the $\diff\Set$-abelian category $\diffinf\cA$ has enough enriched injectives. 
\end{proposition}
\begin{proof}
Let us show that $\diff\cA$ has enough injectives. For $F\in \diff\cA$, let $I_0$ be an injective object in $\cA$  such that 
$$
0\to\forg{F}\to I_0
$$
is exact. The functor $\psig{\,}$ is left-exact being a right-adjoint, so we obtain the exact sequence
$$
0\to\psig{\forg{F}}\to\psig{I_0}.
$$
Precomposing with the unit $\eta_F$ of the adjunction $\forg{\,}\dashv\psig{\,}$ which is known to be a monomorphism, we obtain a monomorphism
$$
F\stackrel{\eta_F}{\longrightarrow} \psig{\forg{F}}\to\psig{I_0},
$$
so it is enough to verify that $\psig{I_0}$ is an enriched injective object in $\diff\cA$, i.e., that the functors
$\diff\cA(\mathord{-}, \psig{I_0})$ and $\diffinf\cA(\mathord{-},\psig{I_0})$ are exact. 

To see this, let
$
0\to A\to B\to C\to 0
$
be an exact sequence in $\diff\cA$. It follows that $0\to\forg{A}\to\forg{B}\to\forg{C}\to 0$ is exact in $\cA$, so, by injectivity of $I_0$, we obtain the exact  sequence
$$
0\to\cA(\forg{C},I_0)\to\cA(\forg{B},I_0)\to\cA(\forg{A},I_0)\to0,
$$
which, by adjunction, yields the exact sequence
$$
0\to \diff\cA(C,\psig{I_0})\to \diff\cA(B,\psig{I_0})\to \diff\cA(A,\psig{I_0})\to0,
$$ 
and thus $\psig{I_0}$ is $\diff\cA$-injective. 

Moreover, using \ref{co-prodinf-exact}, we obtain the exact sequence
$$
0\to\forg{\copinf A}\to\forg{\copinf B}\to\forg{\copinf C}\to 0.
$$
By injectivity of $I_0$ again, we obtain that
$$
0\to\cA(\forg{\copinf C}, I_0)\to\cA(\forg{\copinf B},I_0)\to\cA(\forg{\copinf A},I_0)\to0
$$
is exact and hence, by adjointness, 
$$
0\to\forg{\diffinf\cA(C,\psig{I_0})}\to \forg{\diffinf\cA(B,\psig{I_0})}\to\forg{\diffinf\cA(A,\psig{I_0})}\to0
$$
is exact in $\Ab$. Since the underlying morphisms in the above sequence are difference morphisms, it follows that the sequence without the $\forg{\,}$ is also exact in $\diff\Ab$, so we conclude that $\psig{I_0}$ is also $\diffinf\cA$-injective.
\end{proof}

\begin{proposition}\label{diffcat-enough-proj}
If the abelian category $\cA$ has enough projectives, then the $\diff\Set$-abelian category $\diffinf\cA$ has enough enriched projectives.
\end{proposition}
\begin{proof}
For the existence of enough enriched projectives, we argue dually to \ref{diffcat-enough-inj}. Given an object $F\in\diff\cA$, let $P_0$ be a projective object in $\cA$ so that 
$$
P_0\to \forg{F}\to 0
$$
is exact. Using the fact that $\ssig{\,}$ is right-exact, being a left adjoint, as well as the fact that the counit $\epsilon_F$ of the adjunction $\ssig{\,}\dashv\forg{\,}$ is an epimorphism, we obtain the epimorphism
$$
\ssig{P_0}\to \ssig{\forg{F}}\stackrel{\epsilon_F}{\longrightarrow}F,
$$
and it suffices to show that $\ssig{P_0}$ is an enriched projective object in $\diff\cA$, i.e., that the functors
$\diff\cA(\ssig{P_0},\mathord{-})$ and $\diffinf\cA(\ssig{P_0},\mathord{-})$ are exact. 

Assume that $0\to A\to B\to C\to0$ is exact in $\diff\cA$, so trivially $0\to\forg{A}\to\forg{B}\to\forg{C}\to0$ is exact in $\cA$. Using the projectivity of $P_0$, we obtain the exact sequence
$$
0\to\cA(P_0,\forg{A})\to\cA(P_0,\forg{B})\to\cA(P_0,\forg{C})\to0,
$$
which, by adjunction, yields the exact sequence
$$
0\to \diff\cA(\ssig{P_0},A)\to \diff\cA(\ssig{P_0},B)\to \diff\cA(\ssig{P_0},C)\to 0,
$$
whence $\ssig{P_0}$ is $\diff\cA$-projective.

Moreover, using \ref{co-prodinf-exact}, we obtain the exact sequence
$$
0\to\forg{\prodinf A}\to\forg{\prodinf B}\to\forg{\prodinf C}\to0.
$$
By projectivity of $P_0$, the sequence
$$
0\to\cA(P_0,\forg{\prodinf A})\to\cA(P_0,\forg{\prodinf B})\to\cA(P_0,\forg{\prodinf C})\to0
$$
is exact in $\Ab$. By the adjunction from \ref{ord-vs-gen}, this coincides with the sequence
$$
0\to\forg{\diffinf\cA(\ssig{P_0},A)}\to\forg{\diffinf\cA(\ssig{P_0},B)}\to\forg{\diffinf\cA(\ssig{P_0},C)}\to0.
$$ 
Since the underlying morphisms are difference morphisms, the same sequence without the $\forg{\,}$ is exact in $\diff\Ab$, so $\ssig{P_0}$ is also $\diffinf\cA$-projective.
\end{proof}

\subsection{Injective and projective difference modules}

\begin{proposition}\label{diff-mod-enough-enr-inj-proj}
The category of enriched difference modules over a difference ring $R$ is abelian, with enough enriched injectives and enough enriched projectives (in the sense of \ref{enr-ab-cat}, \ref{enr-injectives}, \ref{enr-projectives}).
\end{proposition}

The existence of enough ordinary injective and projective objects is a straightforward consequence of \ref{skew-diff-equiv} and the fact that the category of left $\Do(R)$-modules has enough injectives and projectives, but we offer a proof more congenial to difference algebra techniques, which also leads to the existence of enough enriched injectives and projectives.

We introduce a shorthand notation for enriched injectives and projective specific to the context of enriched difference modules as follows.
\begin{definition}
Let $F$ be a difference module over a difference ring $R$. We say that
\begin{enumerate}
\item $F$ is \emph{$\diffinf$projective,} if $[F,\mathord{-}]: R\Mod\to R\Mod$ is an exact functor;
\item $F$ is \emph{$\diffinf$injective,} if $[\mathord{-},F]: R\Mod\to R\Mod$ is an exact functor.
\end{enumerate}
\end{definition}

\begin{proof}
It is straightforward to verify that $R\Mod$ is an abelian category. 

Let us show that it has enough injectives. For $F\in R\Mod$, let $I_0$ be an injective $\forg{R}$-module such that 
$$
0\to\forg{F}\to I_0
$$
is exact. The functor $\psig{\,}$ is left-exact being a right-adjoint, so we obtain the exact sequence
$$
0\to\psig{\forg{F}}\to\psig{I_0}.
$$
Precomposing with the unit $\eta_F$ of the adjunction $\forg{\,}\dashv\psig{\,}$ which is known to be injective, we obtain an injective map
$$
F\stackrel{\eta_F}{\longrightarrow} \psig{\forg{F}}\to\psig{I_0},
$$
so it is enough to verify that $\psig{I_0}$ is an injective object in $R\Mod$, i.e., that the functor
$R\Mod(\mathord{-}, \psig{I_0})$ is exact. To see this, let
$
0\to A\to B\to C\to 0
$
be an exact sequence in $R\Mod$. It follows that $0\to\forg{A}\to\forg{B}\to\forg{C}\to 0$ is exact in $\forg{R}\Mod$, so, by injectivity of $I_0$, we obtain the exact  sequence
$$
0\to\forg{R}\Mod(\forg{C},I_0)\to\forg{R}\Mod(\forg{B},I_0)\to\forg{R}\Mod(\forg{A},I_0)\to0,
$$
which, by adjunction, yields the exact sequence
$$
0\to R\Mod(C,\psig{I_0})\to R\Mod(B,\psig{I_0})\to R\Mod(A,\psig{I_0})\to0,
$$ 
as required.

For the existence of enough projectives, we argue dually. Given an $R$-module $F$, let $P_0$ be a projective $\forg{R}$-module so that 
$$
P_0\to \forg{F}\to 0
$$
is exact. Using the fact that $\ssig{\,}$ is right-exact, being a left adjoint, as well as the fact that the counit $\epsilon_F$ of the adjunction $\ssig{\,}\dashv\forg{\,}$ is surjective, we obtain the surjective map
$$
\ssig{P_0}\to \ssig{\forg{F}}\stackrel{\epsilon_F}{\longrightarrow}F,
$$
and it suffices to show that $\ssig{P_0}$ is a projective object in $R\Mod$, i.e., that the functor
$R\Mod(\ssig{P_0},\mathord{-})$ is exact. Assume that $0\to A\to B\to C\to0$ is exact in $R\Mod$, so trivially $0\forg{A}\to\forg{B}\to\forg{C}\to0$ is exact in $\forg{R}\Mod$. Using the projectivity of $P_0$, we obtain the exact sequence
$$
0\to\forg{R}\Mod(P_0,\forg{A})\to\forg{R}\Mod(P_0,\forg{B})\to\forg{R}\Mod(P_0,\forg{C})\to0,
$$
which, by adjunction, yields the exact sequencs
$$
0\to R\Mod(\ssig{P_0},A)\to R\Mod(\ssig{P_0},B)\to R\Mod(\ssig{P_0},C)\to 0,
$$
as required.

For the existence of enough enriched injectives and projectives, we invite the reader to mimic the proofs of \ref{diffcat-enough-inj} and \ref{diffcat-enough-proj}.
\end{proof}

\begin{lemma}\label{etale-is-diffinf-proj}
Let $F$ be an \'etale $R$-module with $\forg{F}$ a projective $\forg{R}$-module. Then $F$ is $\diffinf$projective.
\end{lemma}
\begin{proof}
Let $0\to A\to B\to C\to0$ be an exact sequence in $R\Mod$. In particular, the sequence 
$0\to\forg{A}\to\forg{B}\to\forg{C}\to0$ is exact in $\forg{R}\Mod$. Using the fact that $\forg{F}$ is projective, we obtain the exact sequence
$$
0\to\forg{R}\Mod(\forg{F},\forg{A})\to\forg{R}\Mod(\forg{F},\forg{B})\to\forg{R}\Mod(\forg{F},\forg{C})\to0.
$$
On the other hand, using \ref{inthom-etale}, the above entails that
$$
0\to[F,A]\to[F,B]\to[F,C]\to 0
$$
is also exact, as required.
\end{proof}

\begin{lemma}
Let $k$ be a difference field and let $I$ be an inversive difference vector space over $k$. Then $I$ is $\diffinf$injective.
\end{lemma}

\begin{proposition}\label{proj-res-diffinf-proj}
If $F$ is a $\diffinf$projective $R$-module, then
$$
0\to\copinf F\stackrel{\iota-\id}{\longrightarrow}\copinf F\stackrel{\tau}{\longrightarrow}F\to 0
$$
is a projective resolution of $F$ in $R\Mod$, where
$$
\iota(f_0,f_1,\ldots)=(0,f_0,f_1,\ldots)\, \, \text{ and }\, \, \tau(f_0,f_1,\ldots)=f_0+f_1+\cdots.
$$
\end{proposition}
\begin{proof}
Recall that the difference operator on $\copinf F$ is given by the formula
$$
\sigma(f_0,f_1,\ldots)=(0,\sigma_F(f_0),\sigma_F(f_1),\ldots)
$$
and that the multiplication by elements of $R$ is component-wise, so $\iota$ and $\tau$ are clearly $R$-module maps. 

Let us verify the exactness of the above sequence. It is immediate that $\Ker(\iota-\id)=0$ and that $\tau$ is onto, so it remains to check that $\Ker(\tau)=\Im(\iota-\id)$. The right to left inclusion holds since
$$
\tau\circ(\iota-\id)=0.
$$
On the other hand, suppose $a=(a_0,a_1,a_2,\ldots)\in\Ker(\tau)$, i.e., $a_0+a_1+a_2+\cdots=0$. Then
\begin{align*}
a & =a-0=(a_0,a_1,a_2,\ldots)-(\sum_{i\geq 0} a_i,0,\ldots)=\sum_{i\geq 1}(\iota^i-\id)(a_i,0,\ldots)= \\
& = \sum_{i\geq 1}\sum_{j=1}^i(\iota^j-\iota^{j-1})(a_i,0,\ldots)=(\iota-\id)\sum_{i\geq 1}\sum_{j=1}^i\iota^{j-1}(a_i,0,\ldots)\in\Im(\iota-\id).
\end{align*}
 
It remains to check that $\copinf F$ is a projective object in $R\Mod$, i.e., that the functor
$
R\Mod(\copinf F,\mathord{-}): R\Mod\to\Ab
$
is exact. However, by adjunction, $$R\Mod(\copinf F,\mathord{-})\simeq\forg{[F,\I_\infty(\mathord{-})]},$$
and the latter is exact since $F$ is assumed $\diffinf$projective.
\end{proof}

\begin{proposition}\label{inj-res-diffinf-inj}
If $I$ is a $\diffinf$injective $R$-module, then
$$
0\to I\stackrel{\delta}{\longrightarrow}\prodinf I\stackrel{s-\id}{\longrightarrow}\prodinf I\to 0
$$
is an injective resolution in $R\Mod$, where
$$
\delta(x)=(x,x,x,\ldots)\, \, \text{ and }\, \, s(x_0,x_1,x_2,\ldots)=(x_1,x_2,\ldots).
$$
\end{proposition}
\begin{proof}
The sequence is seen to be exact by routine verifications. It remains to check that $\prodinf I$ is an injective object in $R\Mod$, i.e., that the functor $R\Mod(\mathord{-},\prodinf I):R\Mod\to\Ab$ is exact. By adjunction, we have that
$$
R\Mod(\mathord{-},\prodinf I)=\forg{[\I_\infty(\mathord{-}),I]},
$$
and the latter functor is exact by the assumption that $I$ is $\diffinf$injective.
\end{proof}

\subsection{Extensions of difference modules}

\begin{theorem}\label{calc-ext}
Let $F$ and $F'$ be difference modules over a difference ring $R$, such that either 
\begin{enumerate}
\item $F$ is $\diffinf$projective, or
\item $F'$ is $\diffinf$injective.
\end{enumerate}
Then 
$$
\Ext_{R\Mod}^i(F,F')=
\begin{cases}
R\Mod(F,F'),  &  i=0,\\
[F,F']_s,  & i=1,\\
0, & i>1,
\end{cases}
$$
where $[F,F']_s=[F,F']/\Im(s-\id)$ is the module of $s$-coinvariants of $[F,F']$.
\end{theorem}

\begin{proof}
For (1), if $F$ is $\diffinf$projective, we can use the projective resolution of $F$ from  \ref{proj-res-diffinf-proj}, which allows us to compute $\Ext^*(F,F')$ as the cohomology of the complex
$$
0\to R\Mod(\copinf F,F')\stackrel{\iota^*-\id}{\longrightarrow}R\Mod(\copinf F,F')\to 0,
$$
which, by adjunction, is isomorphic to the complex
$$
0\to\forg{[F,F']}\stackrel{s-\id}{\longrightarrow}\forg{[F,F']}\to 0.
$$
Hence
$$
\Ext^0(F,F')=\Ker(s-\id)=\Fix([F,F'])=R\Mod(F,F'),
$$
as expected, and
$$
\Ext^1(F,F')=\coker(s-\id)=[F,F']_s.
$$
Moreover, $\Ext^i(F,F')=0$ for $i>1$.

Suppose we have the condition (2), i.e., that $F'$ is $\diffinf$injective. We can use the injective resolution of $F'$ given by \ref{inj-res-diffinf-inj}, which allows us to compute $\Ext^*(F,F')$ as the cohomology of the complex 
$$
0\to R\Mod(F,\prodinf F')\stackrel{s_*-\id}{\longrightarrow} R\Mod(F,\prodinf F')\to 0,
$$
which, by adjunction, becomes the complex
$$
0\to\forg{[F,F']}\stackrel{s-\id}{\longrightarrow}\forg{[F,F']}\to 0,
$$
so we conclude as above.
\end{proof}

\begin{proposition}\label{lin-diff-cl-ext}
Let $F$ and $F'$ be finite free \'etale difference modules over a linearly difference closed $R$. Then, for $i>0$,
$$
\Ext^i(F,F')=0.
$$
\end{proposition}
\begin{proof}
Using \ref{etale-is-diffinf-proj} and \ref{calc-ext}, it is enought to show that
$$
\Ext^1(F,F')=[F,F']_s=0.
$$
Since $F$ is \'etale, \ref{inthom-etale} gives that $[F,F']=\forg{R}\Mod(\forg{F},\forg{F'})$ with the shift $s:\varphi_0\mapsto \bar{\sigma}_{F'}\circ(\varphi_0)_\varsigma\circ\bar{\sigma}_F^{-1}.$ We need to show that for any $\varphi_1\in [F,F']$ there exists a $\varphi_0$ with
$s\varphi_0-\varphi_0=\varphi_1$.

Upon a choice of bases for $F$ and $F'$, let $A$ and $A'$ be the matrices of $\sigma_F$ and $\sigma_{F'}$, and let $B_i$ be the matrices of $\varphi_i$, $i=0,1$. Hence, given $B_1$, we need to find $B_0$ such that 
$$B_1=A'\varsigma(B_0)A^{-1}-B_0.$$
Writing $\vect(X)$ for the `vectorisation' of a matrix $X$, the above matrix difference equation can be written as the system
$$
(A'\otimes (A^{-1})^T)\varsigma(\vect(B_0))=\vect(B_0)+\vect(B_1).
$$
Using that $F'$ is also \'etale and thus $A'$ is also invertible, we can write it as the linear difference system in $\vect(B_0)$
$$
\varsigma(\vect(B_0))=({A'}^{-1}\otimes A^T)\vect(B_0)+\vect({A'}^{-1}B_1A),
$$
where the Kronecker product ${A'}^{-1}\otimes A^T$ is regular since $A'$ and $A$ are, and such a system can be solved since $R$ is linearly difference closed.
\end{proof}

\begin{corollary}\label{RFix}
The derived functors of the functor $\Fix:R\Mod\to\Ab$ are given by 
$$
R^i\Fix(M)=
\begin{cases}
\Fix(M)=M^\sigma  &  i=0,\\
M/\Im(\sigma-\id)=M_\sigma,  & i=1,\\
0, & i>1.
\end{cases}
$$
\end{corollary}
\begin{proof}
We have that $\Fix(\mathord{-})\simeq R\Mod(R,\mathord{-})$, so
$$
R^i\Fix(\mathord{-})\simeq \Ext^i(R,\mathord{-}).
$$
Since $R$ is free \'etale $R$-module, \ref{etale-is-diffinf-proj} yields that it is 
$\diffinf$projective, so \ref{calc-ext} applies. In particular, it gives that, for $M\in R\Mod$,
$$
R^1\Fix(M)\simeq\Ext^1(R,M)=[R,M]_s\simeq M_\sigma.
$$
\end{proof}
The following is immediate from the above and \ref{lin-diff-cl-ext}.
\begin{corollary}
If $R$ is linearly difference closed, the functor $\Fix$ 
is exact.
\end{corollary}

\begin{corollary}
Let $F\in R\Mod$, and suppose either:
\begin{enumerate}
\item $F$ is \'etale, or
\item $R$ is inversive with $\forg{R}$ a self-injective ring.
\end{enumerate}
Then 
$$
\Ext_{R\Mod}^i(F,R)=
\begin{cases}
R\Mod(F,R),  &  i=0,\\
(F^\vee)_s,  & i=1,\\
0, & i>1.
\end{cases}
$$
In case (1), the above further simplifies to 
$$
\Ext^1(F,R)=\forg{F}^\vee/\Im(s-\id).
$$
\end{corollary}
%

\begin{corollary}
Let $F$ be a finite free \'etale module over a linearly difference closed $R$. Then, for $i>0$,
$$\Ext^i(F,R)=0.$$
\end{corollary}

\section{Difference algebraic geometry}\label{diff-AG}

\subsection{The Zariski spectrum of a difference ring}\label{speczar-diff}

Let $A$ be a difference ring, i.e., a ring object in the topos $\cS=\diff\Set$. The spectrum of $A$ is defined as Hakim's Zariski spectrum (cf.\ \ref{spec-zar}) of the ringed topos $(\cS,A)$,
$$
\spec(A)=\speczar(\cS,A).
$$
We describe its construction as a locally ringed topos, 
following the construction \ref{tierney-spec-zar}.

The object 
$$
R\mono PA
$$
of radical ideals of $A$ in this case, using the description of power objects from  \ref{diff-powerset}, turns out to be
$$
R=\{\A=(\A_n)\in PA \, | \, \A_n \text{ is a radical ideal in }\forg{A}\text{ for all }n\in\N\}.
$$
As in the general case, $R$ is an internal frame (a complete Heyting algebra), and the retraction
$r:PA\to R$
is given by 
$$
r(X)_n=\sqrt{X_n},
$$
the radical computed in $\forg{A}$, for $n\in\N$.

The composite
$\rho:A\xrightarrow{\{\}}PA\xrightarrow{r}R$
is explicitly given by
$$
\rho(f)_n=\sqrt{\sigma^nf}.
$$

The poset $$\bbA=(A_1,A)$$ is given by $f\leq g$ whenever $\rho(f)\leq\rho(g)$. More explicitly, $f\leq g$ whenever $\sqrt{f}\subseteq\sqrt{g}$ in $\forg{A}$, i.e., if and only if there exists an element $a\in\forg{A}$ and $i\in\N$ so that $f^i=ag$. 

An internal coverage, in the sense of \ref{int-cov-poset},
$$
T\xhookrightarrow{(b,c)}A\times \mathop{\rm Idl}(\bbA)
$$
is given by the rule
$$
(f,I)\in T\text{ if }\rho(f)\leq r(I),
$$
i.e., whenever $\sigma^nf\in\sqrt{I_n}$,  and, for all $g\in I_n$, $g\in \sqrt{\sigma^nf}$, for all $n\in\N$, the latter condition being implicit from the definition of a coverage.


%

We will write $(\cA,\sigma_\cA)$ for the category with an endofunctor associated to $\bbA$ by \ref{int-cat-diff-set}. The set of objects of $\cA$ is $\forg{A}$, and we have a unique morphism $f\to g$ if and only if $f\leq g$, for $f,g\in\forg{A}$. The endofunctor is defined by $\sigma_\cA(f)=\sigma_A(f)$.

We define a coverage $\cT\mono \forg{A}\times\mathop{\rm Idl}(\cA)$ on $\cA$ by stipulating that 
$$
(f,I)\in \cT \text{ whenever }f\in\sqrt{I} \text{ and, for all }h\in I, h\in\sqrt{f}.
$$

The ring
$$
\bar{A}\in[\bbA^\op,\cS]
$$
from the construction can be defined
through the $\sigma_\cA$-equivariant presheaf $(\tilde{\bA},\sigma_{\bar{\bA}})$ on $\cA$ as in \ref{internal-diff-presh} as follows. 

Given $f\in\cA=\forg{A}$, we let
$$
\bar{\bA}(f)=\forg{A}_f,
$$
the localisation of $\forg{A}$ with respect to the multiplicative set generated by $f$.

For $f\leq g$, i.e., $\sqrt{f}\subseteq\sqrt{g}$ in $\forg{A}$, we let
$$
\bar{\bA}(f\to g): \forg{A}_g\to \forg{A}_f
$$
be the usual change of the multiplicative system map.

The natural transformation $\sigma_{\bar{\bA}}: \bar{\bA}\to \bar{\bA}\circ \sigma_\cA$
is given, for $f\in\cA$, by the maps
$$
\sigma_{\bar{\bA},f}:\forg{A}_f\to \forg{A}_{\sigma_Af}
$$ 
naturally induced by $\sigma_A$. 

The construction tells us that
$$
\speczar(\cS,A)=(\tilde{\cE},\tilde{A}),
$$
where $\tilde{\cE}=\sh_\cS(\bbA,T)$, and the local ring $\tilde{A}\in\tilde{\cE}$ is the internal sheaf associated to $\bar{A}$.

We proceed to explain its connection to the classical spectrum of the underlying ring, so we introduce the appropriate notation.

The classical spectrum  $$(X,\cO_X)=\spec(\forg{A})$$ is a locally ringed space admitting
 the morphism of locally ringed spaces
 $$({\lexp{a}{\sigma}},\tilde{\sigma}):(X,\cO_X)\to(X,\cO_X)$$
 induced by $\sigma:\forg{A}\to\forg{A}$, where $\tilde{\sigma}:\cO_X\to{\lexp{a}{\sigma}}_*\cO_X$.

We write $$D(f)$$ for the basic open subset of $X$ associated with $f\in\forg{A}$

\begin{lemma}\label{diff-spec-class-spec}
We have equivalences of categories
$$
\tilde{\cE}=\sh_\cS(\bbA,T)\simeq\sh(\cA,\cT)^{\sigma_\cA}\simeq\sh(X)^{\lexp{a}{\sigma}}.
$$
\end{lemma}
\begin{proof}
The coverages $T$ and $\cT$ are such that $(f,I)\in T$ if and only if, for all $n$, $(\sigma^n f,I_n)\in\cT$, and that $(f,I)\in\cT$ implies $(\sigma f,\sigma(I))\in\cT$, so \ref{int-sh-equiv-sh} gives the first equivalence.

The equivalence 
$$
\sh(\cA,\cT)\simeq \sh(X)
$$
assigns to $\bF\in \sh(\cA,\cT)$ the unique Zariski sheaf $\cF\in\sh(X)$ determined by its values on the basic open sets by
$$
\cF(D(f))=\bF(f),
$$
see \cite[\href{https://stacks.math.columbia.edu/tag/009H}{Tag 009H}]{stacks-project}.

Note that $\cF$ is well-defined since the coverage $\cT$ is compatible with the Zariski topology on $X$, given that $(f,I)\in\cT$ if and only if
$$
D(f)=\bigcup_{g\in I}D(g).
$$ 

The actions of $\lexp{a}{\sigma}_*$ on $\sh(X)$ and $\sigma_\cA^*$ on $\sh(\cA,\cT)$ are compatible, since
$$
\lexp{a}{\sigma}_*\cF(D(f))=\cF(\lexp{a}{\sigma}^{-1}D(f))=\cF(D(\sigma f))=\bF(\sigma_\cA f)=\sigma_\cA^*\bF(f),
$$
so we conclude that $\sigma_\cA$-equivariant sheaves on $\cA$ correspond to $\lexp{a}{\sigma}$-equivariant sheaves on $X$.
\end{proof}

\begin{remark}
In the equivalence of \ref{diff-spec-class-spec}, $\bar{A}$ corresponds to the structure sheaf $(\cO_X,\tilde{\sigma})\in\sh(X)^{\lexp{a}{\sigma}}$, so it is actually a sheaf and we conclude that the structure local ring of $\spec(A)$ is
$$
\tilde{A}=\bar{A}\in\sh_\cS(\bbA,T).
$$
\end{remark}

\begin{proposition}
Let $A$ be a difference ring. 
With the above notation, the structure morphism of the Zariski spectrum of $A$ can be identified with
$$
\spec(A)=\speczar(\diff\Set,A)\simeq(\sh(X)^{\lexp{a}{\sigma}}, (\cO_X,\tilde{\sigma}))\xrightarrow{\pi}(\diff\Set,A),
$$
where the global sections functor is given by
$$
\pi_*(\cF)=\cF(X),
$$
and $\cF(X)$ is a difference set because $\cF$ is $\lexp{a}{\sigma}$-equivariant, 
while 
$$
\pi^*(E)
$$
is the constant sheaf associated to a difference set $E$, which is naturally equivariant.
\end{proposition}

\begin{corollary}
If $\spec(A)=(\tilde{\cE},\cA)\xrightarrow{\pi}(\diff\Set,A)$ is the structure morphism of the Zariski spectrum a difference ring, then
$$
\pi_*(\tilde{A})=A.
$$
\end{corollary}
In other words, $A$ can be recovered as the global sections of the structure local ring $\tilde{A}$ of its Zariski spectrum.

\subsection{Externalisations of the Zariski spectrum of a difference ring}

Let $A$ be a difference ring, i.e., a ring object in $\cS=\diff\Set$. We offer two explicit externalisations of $\spec(A)$.

The first is parallel to Hakim's description of $(\tilde{\cE},\tilde{A})=\speczar(\cS,A)$, where we take a standard site $(\cD,J)$ of definition for $\cS$, see \ref{std-site-diff-set}. Thus
$$
\cS=\sh(\cD,J).
$$
Let $(\bbA,T)$ be the internal site in $\cS$ constructed in \ref{speczar-diff}. By construction \ref{semidir}, the underlying topos of the Zariski spectrum can be expressed as
$$
\tilde{\cE}=\sh_\cS(\bbA,T)\simeq \sh(\cD\rtimes\bbA,J\rtimes T).
$$ 
The category 
$$
\cD\rtimes\bbA
$$
consists of pairs $(U,s)$, where $U\in\cD$ is a finitely presented difference set, and $s\in A(U)$, and we have a morphism $(V,t)\xrightarrow{\varphi}(U,s)$ when $V\xrightarrow{\varphi}U$ is a morphism in $\diff\Set$ such that $t\to \bbA(\varphi)(s)$ is a morphism in $\bbA(V)$, i.e.,
$$
t\in \sqrt{s|\varphi},
$$
where $s|\varphi=\bbA(\varphi)(s)$ is the image of $s\in A(U)$ in $A(V)$ by the natural map. 

The coverage 
$$
J\rtimes T
$$
is generated by families 
$$
\{(U_\lambda,s_{\lambda,i})\xrightarrow{\varphi_\lambda}(U,s):i\in I_\lambda, \lambda\in\Lambda\}
$$
such that $\{U_\lambda\xrightarrow{\varphi_\lambda}:\lambda\in\Lambda\}$ is a $J$-covering, i.e., the family is jointly epimorphic, and, for every $\lambda\in\Lambda$,
$$
s|\varphi_\lambda\in \sqrt{\{s_{\lambda,i}:i\in I_\lambda\}},
$$
and, implicitly, each $s_{\lambda,i}\in\sqrt{s|\varphi_\lambda}$.

In the above equivalence of the categories of sheaves, the local ring $\tilde{A}\in\sh_\cS(\bbA,T)$ corresponds to 
$$
\tilde{\bA}\in\sh(\cD\rtimes\bbA,J\rtimes T), \ \ \tilde{\bA}(U,s)=A(U)_s.
$$

The second, perhaps more interesting externalisation, is obtained by viewing 
$$
\cS=[\bsi^\op,\Set]=\sh(\bsi,J_0),
$$
where $\bsi$ is the category associated to the monoid $(\N,+)$ with the trivial coverage $J_0$. We take the same internal poset $\bbA$ in $\cS$ as above, and we obtain
$$
\tilde{\cE}=\sh_\cS(\bbA,T)\simeq \sh(\bsi\rtimes \bbA,J_0\rtimes T).
$$
The set of objects of $\bsi\rtimes\bbA$ can be identified with $\forg{A}$, and morphisms  $t\to s$ are triples $(n,t,s)$ with $t\in\sqrt{\sigma^ns}$, where the radical is taken in $\forg{A}$.

The coverage
$$
J_0\rtimes T
$$
specifies that a sieve $S$ on $s$ in $\bsi\rtimes\bbA$ covering, provided, writing $I_n=\{t: (n,t,s)\in S\}$, we have that for every $n$,
$$
\sigma^ns\in\sqrt{I_n}.
$$
In this case, the internal sheaf $\tilde{A}$ corresponds to
$$
\bA\in\sh(\bsi\rtimes\bbA,J_0\rtimes T),
$$
given by $\bA(s)=\forg{A}_s$, and $\bA(n,t,s):\bA(s)\to \bA(t)$ is the composite
$$
\forg{A}_s\xrightarrow{\bar{\sigma}^n}\forg{A}_{\sigma^n s}\rightarrow \forg{A}_t,
$$
where $\bar{\sigma}$ is the morphism induced by $\sigma$ on the localisation, and the second morphism is the usual change of the multiplicative system map.

\subsection{Points of the Zariski spectrum of a difference ring}

Let $(\tilde{\cE},\tilde{A})=\speczar(\cE,A)\to (\cE,A)$ be the Zariski spectrum of a Grothendieck ringed topos $(\cE,A)$. Hakim has shown (\cite[IV.2.1]{hakim}) that
any point $\tilde{x}:\Set\to\tilde{\cE}$ above a point $x:\Set\to \cE$ naturally identifies with a point 
$\p\in\spec(x^*A)$, and 
$$
\tilde{x}^*\tilde{A}\simeq (x^*A)_\p.
$$

Let $\cS=\diff\Set$, and let $A$ be a difference ring. 
By \ref{points-diff-set}, there are only two points of $\cS$.

The point $\tau(\Nsucc):\Set\to\cS$ corresponds to the adjoint pair
$$
\forg{\,}\dashv \psig{\,},
$$
so a point $\tilde{x}$ of $\spec(A)$ above it can be identified with a point 
$$
\p\in\spec(\forg{A}), 
$$
and the corresponding stalk is  $\tilde{x}^*\tilde{A}\simeq \forg{A}_\p$.

The inverse image functor (i.e., the `fibre functor') associated with the point $\tau(\Zsucc):\Set\to\cS$ maps $A$ to the ring
$$
\varinjlim_\sigma A=\varinjlim(\cdots\xrightarrow{\sigma}\forg{A}\xrightarrow{\sigma} \forg{A}\xrightarrow{\sigma}\cdots),
$$
so a point $\tilde{x}$ over $\tau(\Zsucc)$ corresponds to a point of 
$$
\spec(\varinjlim_\sigma A)\simeq\varprojlim_{\lexp{a}{\sigma}}\spec(\forg{A})\simeq \varprojlim (\cdots\xleftarrow{\lexp{a}{\sigma}}\spec(\forg{A})\xleftarrow{\lexp{a}{\sigma}} \spec(\forg{A})\xleftarrow{\lexp{a}{\sigma}}\cdots),
$$
i.e., to a sequence $\p=(\p_i)_{i\in\Z}$, with $\p_i=\lexp{a}{\sigma}(\p_{i+1})=\sigma_A^{-1}(\p_{i+1})$. 

The corresponding stalk
$$
\tilde{x}^*\tilde{A}\simeq (\varinjlim_\sigma A)_\p\simeq \varinjlim_i \forg{A}_{\p_i}
$$
is a local ring, as a colimit of local rings.

Given that $\speczar(\cS,A)=(\tilde{\cE},\tilde{A})$ is an $\cS$-topos, we must also describe $\cS$-points, i.e., geometric morphisms
$$
\tilde{x}:\cS\to \tilde{\cE}\simeq \sh_\cS(\bbA,T).
$$
over $\cS$.

By  Diaconescu's Theorem \ref{tors-diacon}, \cite[B3.2.7]{elephant1} and \cite[C2.3.9]{elephant2}, such a morphism corresponds to a continuous $\bbA^\op$-torsor $S$. Since $\bbA^\op$ is a poset,  \cite[B3.2.4(d)]{elephant1} states that $S$ can be identified with a subobject of $A$ which is an ideal in $\bbA^\op$ (downward closed and upward directed), i.e., $S$ is a filter in $\bbA$ (upward closed and downward directed).

Taking into account that $f\leq g$ if $f\in\sqrt{g}$ in $\forg{A}$, the above properties of $S$ imply that $1\in S$, $0\notin S$, and $fg\in S$ whenever $f\in S$ and $g\in S$. 

Using the fact that $\{f, g\}$ is a $T$-covering of $f+g$, continuity of $S$ gives that $f+g\in S$ implies $f\in S$ or $g\in S$.

We conclude that $S\mono A$ is a coprime, i.e., the complement of a fixed prime 
$$
\p\in \spec(\forg{A})^{\lexp{a}{\sigma}}=\Fix(\spec(\forg{A},\lexp{a}{\sigma}).
$$ 
Writing $\gamma:\bbS\to\bbA$ for the discrete fibration corresponding to $S$, the fibre functor $\tilde{x}^*$ is the restriction to sheaves of the composite
$$
[\bbA^\op,\cS]\xrightarrow{\gamma^*}[\bbS^\op,\cS]\xrightarrow{\textstyle{\varinjlim_\bbS}}\cS,
$$
so 
$$
\tilde{x}^*\tilde{A}\simeq S^{-1}A\simeq A_\p,
$$
where the last ring is a difference ring since $\p$ is a fixed prime. 

Intuitively, reasoning in $X=\spec(\forg{A})$, the basic open sets $D(f)$, $f\in S$ form a directed system of neighbourhoods of $\p$, and the stalk $\cO_{X,\p}=\varinjlim_{f\in S}\cO_X(D(f))=\varinjlim_{f\in S}\forg{A}_f\simeq \forg{S}^{-1}\forg{A}\simeq\forg{A}_\p$, and the difference structure can be imposed because $\p$ is fixed.

\subsection{\'Etale and other spectra of a difference ring}\label{tau-spectra-diff}

Let $A$ be a difference ring, i.e., a ring in the topos $\cS=\diff\Set$, and let $\tau$ be a Grothendieck topology on $\Sch$ as in \ref{ag-sites}.

We define the \emph{internal big $\tau$-site in $\cS$}
$$
(\bbS_A, T_\tau)
$$
where $\bbS_A\in\cat(\cS)$ is associated with the category 
$$(\Sch_{\ov \spec\forg{A}},\stig_A),$$
where the endofunctor $\stig_A$ is the base change functor $(\lexp{a}{\sigma_A})^*$ via the scheme morphism $\lexp{a}{\sigma_A}:\spec\forg{A}\to\spec\forg{A}$,
and the internal coverage $T_\tau$ is defined by setting 
$$
(P,S)\in T_\tau \ \text{ if } \ (\stig^n_A(P), S_n)\in \tau \text{ for all }n\in\N.
$$
The second property required for \ref{int-sh-equiv-sh} is automatic, because $\tau$ is invariant under base change, so the \emph{big $\tau$-spectrum of $(\cS,A)$} is
$$
\sh_\cS(\bbS_A,T_\tau)\simeq\sh(\Sch_{\ov \spec\forg{A}}, \tau)^{\stig_A},
$$
the topos of $\stig_A$-equivariant sheaves on the big $\tau$-site of $\spec\forg{A}$, equipped with the $\stig_A$-equivariant structure ring $\cO_{\spec\forg{A},\tau}$.

Similarly, we define the \emph{internal small $\tau$-site in $\cS$}
$$
(\bbS_{A,\tau},T_\tau)
$$
where $\bbS_{A,\tau}\in\cat(\cS)$ is associated with the category
$$
(S_{\forg{A},\tau},\stig_A),
$$
where $S_{\forg{A},\tau}=(\spec\forg{A})_\tau$ is the small $\tau$-site on $\spec\forg{A}$ and $\stig_A$ is the base change endofunctor as above, and $T_\tau$ is defined  analogously.

Again by \ref{int-sh-equiv-sh}, the \emph{small $\tau$-spectrum of $(\cS,A)$} is
$$
\sh_\cS(\bbS_{A,\tau},T_\tau)\simeq\sh(S_{\forg{A},\tau}, \tau)^{\stig_A},
$$
the topos of $\stig_A$-equivariant sheaves on the small $\tau$-site of $\spec\forg{A}$, equipped with the appropriate structure ring $\cO_{\spec\forg{A},\tau}$.

By externalising the definition of the small $\tau$-spectrum of $(\cS,A)$ as
$$
\sh(\cD\rtimes \bbS_{A,\tau}, J\rtimes T_\tau),
$$
where $(\cD,J)$ is a standard site for $\cS$, we obtain precisely Hakim's $\tau$-spectrum from \ref{external-spectra}, so we deduce that, as $\tau$-locally ringed topoi,
$$
\spec.\tau(\cS,A)\simeq  \spec.\tau(\Set,\forg{A})^{\stig_A}.
$$

\subsection{Affine difference schemes}

In line with our general definition \ref{rel-affine}, given a difference ring $A$, the associated affine difference scheme is the locally ringed topos
$$
(X,\cO_X)=\spec(A)=\speczar(\diff\Set,A).
$$
Given a scheme topology $\tau$ that refines the Zariski topology, we define the $\tau$-topos of $X$ as the ringed topos
$$
(X_\tau,\cO_\tau)=\spec.\tau(X,\cO_X),
$$ 
whose structure is expounded in \ref{tau-spectra-diff}.

In particular, we obtain the \'etale and flat topos
$$
X_\text{\'et}, \ \  X_\text{fppf}
$$
of the affine difference scheme $X$.

\subsection{Difference schemes}

By a \emph{difference scheme}, we will mean a relative scheme over the topos $\diff\Set$ as in  \ref{rel-schemes}, and a \emph{difference quasi-scheme} is a ringed topos associated to a difference scheme.

\section{Difference Galois theory}\label{diff-gal}

The relevance of various category-theoretic and topos-theoretic Galois theories has been pointed to us by Olivia Caramello and her useful paper \cite{olivia-gal}. We also used ideas from  \cite{borceux-janelidze}, \cite{janelidze}, as well as \cite{bunge-92}, \cite{bunge-moerdijk}, \cite{bunge-04}.

\subsection{Categorical Galois theory of difference rings}

We apply Janelidze's categorical Galois theory to the context of difference rings, in the spirit of Magid's separable Galois theory of commutative rings \cite{magid}.

Recall that the \emph{Pierce spectrum} of a commutative ring $R$ is defined as the spectrum of its boolean algebra of idempotents,
$$
\Sp(R)=\spec(B(R)).
$$
The total space of its structure sheaf of indecomposable rings is the morphism 
$$
\pi_R:\coprod_{M\in\Sp(R)}R/M\to \Sp(R),
$$
topologised in a natural way, see \cite[4.2.12]{borceux-janelidze} and \cite{johnstone-stone}.

Let us write
$$
\cA=\diff\Rng^\op\ \ \text{ and } \ \  \cP=\diff\text{Prof}
$$
for the opposite category of the category of difference rings and for the category of difference profinite sets. 

The \emph{difference Pierce spectrum} is the functor
$$
S:\cA\to \cP, \ \ \ S(A)=(\Sp(\forg{A}), \Sp(\forg{\sigma_A})).
$$
Given a difference ring $k\in\diff\Rng$, it naturally restricts to the functor
$$
S_k:\cA_{\ov k}\to \cP_{\ov S(k)}
$$
on the category of difference $k$-algebras.

Let
$$
C_k:\cP_{\ov S(k)}\to \cA_{\ov k}
$$
be the functor assigning to $E\xrightarrow{e}S(k)$ the $k$-algebra of continuous functions 
$$
\mathcal{C}_{\ov \Sp(\forg{k})}(\forg{e},\pi_{\forg{k}}),
$$
and the difference operator is 
$\sigma_E^*\bar{\sigma}_{k*}$, where  $\bar{\sigma}_k$ is the map induced on the total space by $\sigma_k$.

When $k$ is indecomposable, this can be viewed as the continuous version of the cotensor construction  $\llbracket E,k\rrbracket$ from \ref{difuinf-cotens}.

We obtain the adjunction
$$
S_k\dashv C_k.
$$
Thus, in order to apply Janelidze's Galois theory as outlined in \ref{janelidze}, we need to identify an interesting class of morphisms $A\to k$ in $\cA_k$ of relative Galois descent.

\begin{lemma}
If a morphism $f:\bar{k}\to k$ in $\cA_k$ is such that $\forg{\bar{k}}$ is a \emph{separable closure} of $\forg{k}$ in the sense of Magid \cite{magid}, then it is  of relative Galois descent.
\end{lemma}

\begin{theorem}
Let $f:\bar{k}\to k\in\cA_k$ be as in the Lemma. Then we have an equivalence of categories
$$
\Split_k(f)\simeq[\Gal[f],\cP_k],
$$
where
\begin{enumerate}
\item  $\Split_k(f)$ is the category of difference $k$-algebras $A$ such that $\forg{A\otimes_k\bar{k}}$ is componentially locally strongly separable extension of $\forg{\bar{k}}$;
\item $\Gal[f]$ is a difference groupoid with space of objects $S(\bar{k})$ and space of arrows $S(\bar{k}\otimes_k\bar{k})$;
\item  (one side of) the equivalence is given by the functor
$$
A\mapsto S(A\otimes_k\bar{k}).
$$
\end{enumerate}
\end{theorem}

An immediate consequence of this fact is the following, as yet unpublished, theorem of Wibmer and the author, who previously (\cite[2.20]{ive-tgs} and \cite[1.23]{ive-mich-babb}) identified the difference absolute Galois group of a difference field $k$ as the difference group
$$
G_k=(\Gal(\forg{\bar{k}}/\forg{k}),\lexp{\sigma_{\bar{k}}}{(\,)}),
$$
where $\bar{k}$ is a separable closure of $k$ with some choice of lifting $\sigma_{\bar{k}}$ of $\sigma_k$.

\begin{corollary}\label{diff-galois-gp}
With the above notation,
$$
\Gal[\bar{k}/k]\simeq G_k.
$$
Moreover, in the above Galois correspondence, the formally \'etale difference $k$-algebras of finite $\sigma$-type correspond to continuous $G_k$-actions on difference profinite sets of finite $\sigma$-type. 
\end{corollary}

\subsection{Difference fundamental groupoid}

Let $A$ be a difference ring, and let $X=\spec(A)$ be its Zariski spectrum, and let us write $\forg{X}=\spec(\forg{A})$ and $\sigma=\lexp{a}{\sigma_A}:\forg{X}\to\forg{X}$.

 As stipulated by \ref{rel-pi1}, the \emph{difference fundamental groupoid}
$$
\pi_1^\et(X)
$$
is the prodiscrete localic fundamental groupoid of $X_\et\xrightarrow{\gamma}\cS=\diff\Set$.

Let us assume that we have a point $\omega:\cS\to X_\et$. By arguments of \ref{rel-pi1}, it comes from a difference homomorphism
$$
A\to\Omega, 
$$
where $\forg{\Omega}$ is a separably closed field. Let 
$$
\bar{x}:W=\spec\forg{\Omega}\to\forg{X}
$$
be the associated geometric point in $\forg{X}$.

Recall that $X_\et=\sh_\cS(\bbS_\et,T_\et)$, where $\bbS$ is the internal category associated with $((\forg{X})_\et,\stig)$. The  continuous torsor
$$
F\in[\C_\et,\cS]
$$
 corresponding to $\omega$  
 is associated with the $\stig$-invariant diagram
$$
(\bF,\sigma_\bF)
$$
where $\bF:(\forg{X})_\et\to \Set$ is the classical fibre functor 
$$
\bF(P\to \forg{X})=\Sch_{\ov\forg{X}}(\bar{x},P)\simeq P_{\bar{x}}(\Omega),
$$
and $\sigma_\bF:\bF\to\bF\circ\stig$ is obtained by virtue of the diagram

$$
 \begin{tikzpicture} 
 [cross line/.style={preaction={draw=white, -,
line width=3pt}}]
\matrix(m)[matrix of math nodes, minimum size=1.7em,
inner sep=0pt, 
row sep=1.5em, column sep=1em, text height=1.5ex, text depth=0.25ex]
 { 	    			
		& |(u)|{W} 		&			& |(3)|{W}		\\
 |(2)|{P} 	& 			& |(P)| {P_{(\sigma)}} &     		    	\\[1em]
 |(h)|{\forg{X}}     &			& |(1)|{\forg{X}}  		&			\\};
\path[->,font=\scriptsize,>=to, thin]
(1) edge node[below]{$\sigma$}  (h)
(2) edge (h)
(u) edge node[above left]{$\bar{y}$} (2) edge  node[right,pos=0.7]{$\bar{x}$}(h)
(3) edge node[right]{$\bar{x}$}(1) edge node[above]{$\lexp{a}{\omega}$} (u) edge[dashed] node[pos=0.6,above left=-4pt]{$\sigma_\bF(\bar{y})$} (P)
(P) edge (1) edge[cross line] (2)
;
\end{tikzpicture}
$$
for $\bar{y}\in \bF(P\to\forg{X})$, where $P_{(\sigma)}=\stig(P)$ is the base change of $P$ by $\sigma$.

The prodiscrete localic group 
$$
\pi_1^\et(X,\omega)
$$
is obtained as a prodiscrete localic group in $\cS$ of automorphisms of $F$, whose underlying group is
\begin{multline*}
\gamma_*\uAut(F)=\uAut(\cF)(1)=\{ \varphi\in\Aut(\bF)^\N=\pi_1^\et(\forg{X},\bar{x})^\N: \sigma_\bF\varphi_n=(\varphi_{n+1}\circ\stig)\sigma_\bF\}.
\end{multline*}

\subsection{Difference fundamental group of a difference field}\label{pi1-diff-field}

If $k$ is a difference field, and we choose a difference embedding $\omega:k\to \bar{k}$ into its separable closure, then 
$$
\pi_1^\et(\spec{k},\omega)\simeq G_k,
$$
as in \ref{diff-galois-gp}. Indeed, $\sigma:\spec\forg{k}\to\spec\forg{k}$ is faithfully flat, so $\stig$ is an exact functor, and, for each $n$, $\bF\circ\stig$ is again a fibre functor, hence $\sigma_\bF$ are isomorphisms by Grothendieck Galois theory, so the whole sequence $(\varphi_n)$ in the above is determined by $\varphi_0\in\pi_1^\et(\forg{k},\forg{\omega})\simeq \Gal(\forg{\bar{k}}/\forg{k})$.

It is topologised as follows. For a Galois difference field extension $L/k$ of finite $\sigma$-type,  the localic automorphism group of $L/k$ is the group of internal automorphisms
$$
\uAut(L/k)=\{f\in \Aut(\forg{L}/\forg{k})^\N: f_{n+1}\sigma_L=\sigma_Lf_n\},
$$ 
where a subbasis for the topology is given by sets
$$
\langle a,b,n\rangle=\{f\in\uAut(L/k): f_n(a)=b\},
$$ 
for $a, b\in L$, $n\in\N$.

\subsection{\'Etale spectrum of a difference field}\label{etspec-diff-field}

Let $\omega:k\to\bar{k}$ be a choice of difference embedding of difference field into its separable closure, and let $X=\spec(k)$. From \ref{pi1-diff-field}, it follows that 
$\pi_1^\et(X,\omega)\simeq G_k$, so  
$$
\Split(X_\et)\simeq BG_k, 
$$
the category of continuous (difference) $G_k$-actions. On the other hand, we have that $X_\et\simeq \forg{X}_\et^{\stig_k}$, so,  using the Galois theory of $\forg{k}$ and the fact that every sheaf in $\forg{X}_\et$ is split (by $\bar{k}$), we deduce that $\Split(X_\et)=X_\et$, and thus the \'etale topos of a difference field $k$ is equivalent to the category of continuous actions of the difference group $G_k$, i.e.,
$$
\spec.\et(k)\simeq BG_k.
$$

\section{Cohomology of difference schemes}\label{diff-coh}

\subsection{Difference cohomology}\label{diff-cohom}

Let $A$ be a difference ring, i.e., a ring in $\cS=\diff\Set$, and let $(X,\cO_X)\xrightarrow{\pi}(\cS,A)$ be a quasi-difference scheme. Let $\tau$ be a scheme topology refining Zariski topology, and let $(X_\tau,\cO_\tau)\xrightarrow{\pi_\tau}(\cS,\cO_X)$ be the $\tau$-topos associated with $X$.

In \ref{coh-rel-sch}, given abelian groups $M, N$ in $X_\tau$, we defined the abelian groups
$$
\Ext(X_\tau,M,N) \ \ \text{ and }\ \ \tH^n(X_\tau,N),
$$
and the difference abelian groups
$$
\Ext(X_\tau/\cS,M,N)\ \ \text{ and }\ \ \tH^n(X_\tau/\cS,N)
$$
as suitable instances of the classical topos cohomology.

The  \emph{Leray spectral sequence} 
$$
\tH^p(\cS,H^q(X_\tau/\cS,N))\Rightarrow \tH^{p+q}(X_\tau,N).
$$ 
connecting the above cohomology groups becomes explicit in this case. By \ref{RFix}, we have that $\tH^1(\cS,\mathord{-})$ 
is the functor of difference coinvariants, and $\tH^p(\cS,\mathord{-})=0$ for $p>1$, so the spectral sequence degenerates (as in \cite[Exercise~5.2.1]{weibel}) and we obtain exact sequences
$$
0\to \tH^{n-1}(X_\tau/\cS,N)_\sigma\to \tH^n(X_\tau,N)\to \tH^n(X_\tau/\cS,N)^\sigma\to 0.
$$

\subsection{Classifying difference torsors}\label{diff-tors}

Although we can speak about torsors in a very general setting of \ref{tors-diacon} and \ref{group-torsors}, in algebraic geometry one is usually interested in torsors on the flat site of a scheme, and their interaction with principal homogeneous spaces of a group scheme, as in \cite[III.4]{milne-etale}.
 
If $X$ is a difference scheme, $G$ a difference group scheme over $X$, and $P\to X$ a faithfully flat difference scheme over $X$ with an action $\mu_P:P\times_XG\to P$ of $G$ over $X$, we say that $P$ is principal homogeneous space for $G$, if the morphism
$$
P\times_XG\xrightarrow{(\id,\mu_P)}P\times_XP
$$
is an isomorphism. 

Let $\tilde{G}\in X_\text{fppf}$ be the group object represented by $G$. If $G$ is abelian, then \ref{H1-class-tors} tells us that 
$$
\Tors(X_\text{fppf},\tilde{G})\simeq \tH^1(X_\text{fppf},\tilde{G}).
$$
Using \cite[III.4.3]{milne-etale}, we deduce that the left hand side coincides with isomorphism classes of principal homogeneous spaces of $G$. This aspect has been of interest in \cite{michael-anette} and \cite{piotr2}, see Section~\ref{appendix1}.

\subsection{Generalised difference torsors}\label{gen-diff-tors}

Let $\gamma:\cE\to\cS=\diff\Set$ be a topos over difference sets, and let $G$ be a group in $\cE$. 

\emph{Generalised difference $G$-torsors} are defined as elements of the difference groupoid
$$
\uTors(\cE,G)=\Tors(\cE_{\ov \gamma^*\Nsucc}, (\gamma^*\Nsucc)^* G),
$$
with the difference structure inherited from $\Nsucc$. 

We also have the difference abelian group
$$
\uTors^1(\cE,G)=\Tors^1(\cE_{\ov \gamma^*\Nsucc}, (\gamma^*\Nsucc)^* G)
$$
of isomorphism classes of generalised difference $G$-torsors, which consists of connected components of the above difference groupoid.

The following statement compares generalised difference torsors to ordinary ones.

\begin{proposition}
If $G$ is an abelian group in $\cE$, we have a short exact sequence
$$
0\to (\gamma_*G)_\sigma\to \Tors^1(\cE,G)\to \uTors^1(\cE,G)^\sigma\to 0.
$$
\end{proposition}
\begin{proof}
The exact sequence of low degrees in the Leray spectral sequence is
$$
0\to \tH^1(\cS,\gamma_*G)\to \tH^1(\cE,G)\to \Gamma_\cS(R^1\gamma_*(G))\to H^2(\cS,\gamma_*G).
$$
Taking into account that $\Gamma_\cS=\Fix$, as well as \ref{RFix}, it simplifies to
$$
0\to (\gamma_*G)_\sigma\to \tH^1(\cE,G)\to (R^1\gamma_*(G))^\sigma \to 0.
$$
Using the description $\cS=[\bsi^\op,\Set]$ as in \ref{topos-diff-sets}, and the explicit description of $R^1\gamma_*$ from \ref{geom-m-coh}, we obtain that the difference set corresponding to the presheaf $R^1\gamma_*(G)$ is
$$
R^1\gamma_*(G)(o)=H^1(\cE,\gamma^*\h_o, G)\simeq H^1(\cE_{\ov \gamma^*\Nsucc}, (\gamma^*\Nsucc)^* G),
$$
so we obtain the required sequence by applying \ref{H1-class-tors}.
\end{proof}

\subsection{Comparison to cohomology of schemes}\label{comp-diff}

Using the notation from \ref{tau-spectra-diff}, for a difference scheme $X=\speczar(\cS,A)$, we have
$$
X_\tau=\sh_\cS(\bbS_{X,\tau},T_\tau)\simeq \sh(S_{\forg{X},\tau})^{\stig_X}=(\forg{X}_\tau)^{\stig_X}.
$$

Let 
$$
f:\Set\to \cS
$$
be the point from \ref{cat--diff-cat}, with $f^*=\forg{\,}$, $f_!=\ssig{\,}$ and $f_*=\psig{\,}$. 

We apply the base change construction \ref{bc-int-psh} to $f$ and the internal category $\bbS_{X,\tau}$ from \ref{bc-int-psh} to obtain a geometric morphism
$$
\tilde{f}:\forg{X}_\tau\to X_\tau\simeq (\forg{X}_\tau)^{\stig_X},
$$
where $\tilde{f}^*$ is simply the forgetful functor, forgetting the structure map of a $\stig_X$-equivariant sheaf. Its right adjoint is
$$
\tilde{f}_*(\bF)=\prod_{i\in\N} \bF\circ\stig^n,
$$
whose $\stig$-equivariant structure is given by the left shift, and its left adjoint is
$$
\tilde{f}_!(\bF)(P)=\coprod \bF(Q),
$$
where the coproduct is indexed by all $Q$ with $\stig^n(Q)=P$ for some $n$.

Clearly, $\tilde{f}_!$ is an exact left adjoint to $\tilde{f}^*$, so the forgetful functor $\tilde{f}^*$ preserves injectives. 

Given an abelian group $N$ in $X_\tau$, the relative cohomology $\tH(X_\tau/\cS,N)$
is computed using a $\stig$-equivariant injective resolution of $N$, which we now know also gives an injective resolution of $\forg{N}$, and $H^0(\forg{X}_\tau,\forg{N})=H^0(X_\tau,N)$,
so we deduce that
$$
\tH^n(X_\tau/\cS,N)\simeq \tH^n(\forg{X}_\tau,\forg{N}).
$$

The short exact sequences from \ref{diff-cohom} become
$$
0\to \tH^{n-1}(\forg{X}_\tau,\forg{N})_\sigma\to \tH^n(X_\tau,N)\to \tH^n(\forg{X}_\tau,\forg{N})^\sigma\to 0.
$$

\subsection{Cohomology of difference quasi-coherent sheaves}\label{coh-quasicoh-diff}

Let $A$ be a difference ring, and let $\varphi:B\to C$ be a morphism of difference $A$-algebras. Let $M$ be a difference $C$-module, and let $N$ be a difference $A$-module.

Let $f:\spec(C)\to \spec(B)$ be the morphism of difference schemes assocated with $\varphi$ and let $\pi_A:\spec(A)\to(\diff\Set,A)$ and $\pi_C:\spec(C)\to(\diff\Set,C)$ be structure morphisms.

Writing $\tilde{M}=M\times_C\cO_{\spec C}$ for the pullback of $M$ via $\pi_C$ and 
$\tilde{N}=N\times_A\cO_{\spec A}$ for the pullback of $A$ via $\pi_A$, the results of \ref{coh-quasicoh} give that, for $i>0$,
$$
R^if_* \tilde{M}=0,
$$
and
$$
\tH^i(X/\cS,\tilde{N})=R^i\pi_{A*}\tilde{N}=0.
$$

\subsection{Hilbert's theorem 90 for difference schemes}\label{h90-diff}

Let $X$ be a difference quasi-scheme. Hilbert's Theorem 90 for relative schemes from \ref{hilbert90} applies, so we obtain
$$
\tH^1(X,\cO_X^\times)\simeq \tH^1(X_\et,\bbG_m)\simeq \tH(X_\text{fppf},\bbG_m).
$$
Let us give a direct proof in the difference case. Using \ref{picard},
we obtain natural maps
$$
\tH^1(X,\cO_X)\simeq \Pic(\cO_X)\to \Pic(\cO_\et)\simeq \tH^1(X_\et,\bbG_m),
$$
so it is enough to show that the middle map is surjective. Using the description 
$$
X_\et=\forg{X}_\et^\stig\ \ \text{ and } X_\text{Zar}=\forg{X}_\text{Zar}^\stig
$$
as categories of $\stig$-equivariant sheaves on the classical sites from \ref{tau-spectra-diff}, if $(\bF,\sigma_\bF)$ is an invertible sheaf in $X_\et$, then $\bF$ is in particular an invertible sheaf on $\forg{X}_\et$, which happens to be $\stig$-equivariant. Hence, $\bF$ is quasi-coherent and there is a quasi-coherent sheaf $\bF_0\in \forg{X}_\text{Zar}$ such that $\bF_0^a=\bF$. It follows automatically that $\bF_0$ is $\stig$-equivariant and invertible, i.e., $\bF_0\in X_\text{Zar}$, and we are done.

The comparison theorem \ref{comp-diff} gives that, writing $\cA(X)=\cO_X(X)$, the sequence
$$
0\to \cA(X)^\times_\sigma \to \Pic(X)\to \Pic(\forg{X})^\sigma\to 0
$$
is exact.

\subsection{Difference Kummer and Artin-Schreier theory}

\begin{notation}
 for an abelian group $E$ and an integer $n$, we write
$$
{}_nE=\Ker(E\xrightarrow{n}E), \ \ \ E_n=\Coker(E\xrightarrow{n}E).
$$
\end{notation}

Let $X$ be a quasi-difference scheme. 

If $n$ is invertible in $X$,  Kummer theory from \ref{rel-kummer} in this case gives the exact sequence  
$$
0\to (\cO_X^\times(X)^\sigma)_n\to \tH^1(X_\et,\mu_n)\to{}_n\Pic(X)\to 0.
$$

If $X$ is of characteristic $p$, Artin-Schreier theory from \ref{rel-AS} gives the exact sequence  
$$
0\to \cO_X(X)^\sigma/(F-\id)\cO_X(X)^\sigma\to\tH^1(X_\et,\Z/n\Z)\to \tH^1(X_\et,\cO_\et)^F\to 0.
$$

\subsection{\'Etale cohomology of a difference field}\label{etcoh-diff-field}

Let $k$ be a difference field. We abbreviate $\tH^n(k,\bbG_m)=\tH^n(\spec(k)_\et,\bbG_m)$ and  $\tH^n(k,\mu_n)=\tH^n(\spec(k)_\et,\mu_n)$.

The comparison theorem from \ref{h90-diff} yields an exact sequence
$$
0\to (k^\times)_\sigma\to \Pic(k)\to \Pic(\forg{k})^\sigma\to0,
$$
and, since $\Pic(\forg{k})=0$, we obtain that
$$
\tH^1(k,\bbG_m)\simeq\Pic(k)\simeq (k^\times)_\sigma.
$$
In degree 2, comparison theorem \ref{comp-diff} gives an exact sequence
$$
0\to H^1(\forg{k},\bbG_m)_\sigma\to \tH^1(k,\bbG_m)\to H^2(\forg{k},\bbG_m)^\sigma\to 0,
$$
and, since the first term vanishes, we obtain that
$$
\tH^2(k,\bbG_m)\simeq \tH^2(\forg{k},\bbG_m)^\sigma.
$$
In other words, we have a relation between the Brauer groups of $k$ and its underlying field,
$$
\text{Br}(k)\simeq\text{Br}(\forg{k})^\sigma.
$$ 
The long exact sequence for cohomology associated to Kummer theory of $k$ becomes
\begin{align*}
0 & \to  \mu_n(k^\sigma)\to (k^\times)^\sigma \to (k^\times)^\sigma  \to\\
 & \to \tH^1(k,\mu_n)\to (k^\times)_\sigma\to (k^\times)_\sigma\to\\
 & \to \tH^2(k,\mu_n)\to \text{Br}(\forg{k})^\sigma\to \text{Br}(\forg{k})^\sigma,
\end{align*}
whence we extract short exact sequences  
$$
0\to ((k^\times)^\sigma)_n\to \tH^1(k,\mu_n)\to {}_n((k^\times)_\sigma)\to 0,
$$
and 
$$
0\to ((k^\times)_\sigma)_n\to \tH^2(k,\mu_n)\to {}_n\text{Br}(\forg{k})^\sigma\to 0.
$$

On the other hand, the comparison theorem yields exact sequences 
$$
0\to ({}_n(k^\times))_\sigma\to \tH^1(k,\mu_n)\to ((k^\times)_n)^\sigma\to 0
$$
and
$$
0\to ((k^\times)_n)_\sigma\to \tH^2(k,\mu_n)\to ({}_n\text{Br}(\forg{k}))^\sigma\to 0,
$$
the latter being equivalent to the sequence given above.

\subsection{Difference Galois cohomology}\label{diff-galcoh}

If $k$ is a difference field, \ref{etspec-diff-field} yields an equivalence
$$
\Ab(\spec.\et(k))\simeq \Ab(B G_k)
$$
between abelian \'etale sheaves on $k$ and continuous $G_k$-modules. Through this equivalence, we have that section functors correspond to $G_k$-invariants in the sense 
$$
\tH^0(\spec.\et(k)/\cS,\mathord{-})\simeq (\mathord{-})^{G_k} \ \ \text{ and } \ \ 
\tH^0(\spec.\et(k),\mathord{-})\simeq \Fix\circ (\mathord{-})^{G_k}.
$$
Hence, if $\cG\in\Ab(\spec.\et(k))$, we can compute the \'etale cohomology as continuous difference group cohomology
$$
H^n(\spec.\et(k)/\cS,\cG)\simeq H^n(B G_k/\cS,\cG), \ \ \text{ and } \ \ 
H^n(\spec.\et(k),\cG)\simeq H^n(B G_k,\cG),
$$
as outlined in \ref{rel-gp-coh}. 

This is of particular interest when the abelian group $\cG$ is associated to an abelian difference algebraic group $\bbG$ over $k$, in which case we can dub the above groups
$$
\tH^n((\bar{k}/k)/\cS,\bbG)\ \ \text{ and } \ \ \tH^n(\bar{k}/k,\bbG)
$$ 
and think of them as difference analogues of Galois cohomology.

\subsection{\'Etale cohomology of a difference curve}\label{etcoh-diff-curve}

Let $X$ be a smooth difference curve over a difference field $k$ with $\forg{k}$ algebraically closed, and let 
$$
\cR_X\simeq (\cR_{\forg{X}},\sigma_\cR), \ \ \cD_X\simeq (\cD_{\forg{X}},\sigma_\cD) \ \in X_\et
$$
be $\stig_X$-equivariant sheaves associated with the sheaf of rational functions and the sheaf of divisors on $\forg{X}$. We have a divisor short exact sequence in $X_\et$,
$$
0\to \bbG_m\to \cR_X^\times\to\cD_X\to 0.
$$

All cohomology groups in this subsection will be \'etale, so, to simplify notation, we write $\tH^i(X,\mathord{-})$ in place of $\tH^i(X_\et,\mathord{-})$.  

Classically (\cite[I.5.2]{freitag-kiehl}, \cite[Sect.~14]{milne-lec}, \cite[10.3]{tamme}) we know the following.
\begin{enumerate}
\item For $i\geq 2$, $$\tH^i(\forg{X},\bbG_m)=0.$$

\item For $i>0$, $$\tH^i(\forg{X},\cR^\times)=0\ \ \text{ and }\ \ \tH^i(\forg{X},\cD)=0.$$
\item Kummer theory for $\forg{X}$ gives a long exact sequence
\begin{align*}
0 &\to \mu_n(\cO^\times_{\forg{X}}(\forg{X})\to \cO^\times_{\forg{X}}(\forg{X})\xrightarrow{n}\cO^\times_{\forg{X}}(\forg{X)}\\
& \to \tH^1(\forg{X},\mu_n)\to \Pic(\forg{X})\xrightarrow{n}\Pic(\forg{X})\\
& \to \tH^2(\forg{X},\mu_n)\to 0,
\end{align*}
whence 
$$
 \tH^2(\forg{X},\mu_n)=\Pic\forg{X}_n.
$$
When $\forg{X}$ is proper, we have that $\cO_{\forg{X}}(\forg{X})=\forg{k}$, so

$$\tH^1(\forg{X},\mu_n)={}_n\Pic\forg{X}.$$  
 Moreover, in view of the exact sequence
$$
0\to\Pic^0\forg{X}\to \Pic\forg{X}\xrightarrow{d}\Z\to 0,
$$
where $d$ is the degree map, and the fact that $\Pic^0\forg{X}$ has the structure of an abelian variety, we know that $\Pic^0\forg{X}\xrightarrow{n}\Pic^0\forg{X}$ is surjective, and we obtain the more precise relations
$$
\tH^1(\forg{X},\mu_n)={}_n\Pic^0\forg{X}\ \ \text{ and } \ \ \tH^2(\forg{X},\mu_n)=\Coker(\Z\xrightarrow{n}\Z)=\Z/n\Z.
$$
\end{enumerate}

Using (1), the comparison theorem in degree 1 gives a short exact sequence
$$
0\to(\cO_X^\times(X))_\sigma\to \Pic(X)\to\Pic(\forg{X})^\sigma\to 0,
$$
in degree 2,
$$
\tH^2(X,\bbG_m)\simeq\Pic(\forg{X})_\sigma,
$$
and, for $i\geq 3$,
$$
\tH^i(X,\bbG_m)=0.
$$

Using (2), we obtain
$$
\tH^1(X,\cR^\times)\simeq (k(X)^\times)_\sigma,
$$
and, for $i\geq 2$, 
$$
\tH^i(X,\cR^\times)=0,
$$
and similarly
$$
\tH^1(X,\cD)\simeq \cD(\forg{X})_\sigma, 
$$
and, for $i\geq 2$,
$$
\tH^i(X,\cD)=0.
$$

With the above information, the long cohomology exact sequence associated to the divisor short exact sequence can be written as
\begin{multline*}
0\to (\cO_X(X)^\times)^\sigma\to (k(X)^\times)^\sigma\to (\cD(\forg{X}))^\sigma\to\\
\to\Pic(X)\to
(k(X)^\times)_\sigma\to (\cD(\forg{X}))_\sigma\to\Pic(\forg{X})_\sigma\to 0,
\end{multline*}
whence we split off a short exact sequence for $\tH^1(X,\bbG_m)=\Pic(X)$,
$$
0\to\Coker\left((k(X)^\times)^\sigma\to\cD(\forg{X})^\sigma\right)\to\Pic(X)\to \Ker\left((k(X)^\times)_\sigma\to\cD(\forg{X})_\sigma\right)\to 0,
$$
while in degree 2 we get nothing new.

Using (3), comparison theorem in degree 1 gives that an exact sequence
$$
0\to ({}_n(\cO^\times_X(X)))_\sigma\to \tH^1(X,\mu_n)\to \tH^1(\forg{X},\mu_n)^\sigma\to 0.
$$
When $X$ is proper, this simplifies to 
$$
0\to \mu_n(k)_\sigma\to \tH^1(X,\mu_n)\to ({}_n\Pic^0\forg{X})^\sigma\to 0.
$$
In degree 2,  we obtain
$$
0\to \tH^1(\forg{X},\mu_n)_\sigma\to \tH^2(X,\mu_n)\to ((\Pic\forg{X})_n)^\sigma\simeq (\Z/n\Z)^\sigma,
$$
where $\sigma$ acts on $\bZ$ by multiplication by the generic degree of $\sigma_X$. When $X$ is proper, this simplifies to
$$
0\to ({}_n\Pic^0\forg{X})_\sigma\to \tH^2(X,\mu_n)\to (\Z/n\Z)^\sigma.
$$
In degree 3, we obtain
$$
\tH^3(X,\mu_n)\simeq\tH^2(\forg{X},\mu_n)_\sigma\simeq(\Z/nZ)_\sigma.
$$

Taking into account all of the above information, Kummer theory of $X_\et$ yields a long exact cohomology sequence
\begin{align*}
0 &\to \mu_n(\cO^\times_{X}(X))^\sigma \to \cO^\times_{X}(X)^\sigma \xrightarrow{n}\cO^\times_{X}(X)^\sigma \\
& \to \tH^1(X,\mu_n)\to \Pic(X)\xrightarrow{n}\Pic(X)\\
& \to \tH^2(X,\mu_n)\to \Pic(\forg{X})_\sigma\to\Pic(\forg{X})_\sigma\\
&\to \tH^3(X,\mu_n)\to 0.
\end{align*}
We can split off the short exact sequences
$$
0\to (\cO^\times_X(X)^\sigma)_n\to \tH^1(X,\mu_n)\to {}_n\Pic(X)\to 0,
$$
which, for proper $X$, simplifies to 
$$
0\to ((k^\times)^\sigma)_n\to \tH^1(X,\mu_n)\to {}_n\Pic(X)\to 0,
$$
and
$$
0\to \Pic(X)_n\to \tH^2(X,\mu_n)\to {}_n(\Pic(\forg{X})_\sigma),
$$
while in degree 3 we get no new information.


\section{Cohomology of difference algebraic groups}\label{diff-gp-coh}

\subsection{Difference algebraic groups}

For an affine difference scheme $X$, let us write $\cA(X)=\cO_X(X)$ for its associated difference ring. 

Let $S\in\diff\Sch$ be an affine difference scheme, and let $G\to S$ be an affine $S$-group, determined by $\diff\Sch\ov S$ morphisms
\begin{enumerate}
\item  $G\times_S G\to G$ (multiplication/product);
\item  $G\to G$ (inverse element);
\item $S\to G$ (identity section).
\end{enumerate}
Taking the associated algebras, we obtain $\cA(S)$-algebra morphisms 
\begin{enumerate}
\item $\Delta:\cA(S)\to\cA(G)\otimes_{\cA(S)}\cA(G)$ (coproduct);
\item $\tau:\cA(G)\to\cA(G)$ (antipode);
\item $\epsilon:\cA(G)\to \cA(S)$ (counit),
\end{enumerate}
endowing $\cA(G)$ with a difference Hopf algebra structure.

\subsection{Constructing difference subgroups}\label{find-diff-subgp}

Let us write $\Do$ for the difference ring $(\Z[T],\id)$. Equivalently, considering the difference set $(\N,i\mapsto i+1)$ and the object $\jj_\infty(\Z)=(\Z,\id)\in\diffinf\Rng$, we obtain
$$
\Do\simeq \N\otimes \jj_\infty(\Z)=\N_{\jj_\infty(\Z)}.
$$
The ring $\Do$ acts on any commutative difference group $G$ by difference group endomorphisms  via
$$
(a_0+a_1T+\cdots+a_nT^n).g=g^{a_0}\sigma g^{a_1}\cdots\sigma^ng^{a_n}.
$$
If $\bG$ is a commutative difference  group scheme, we have an action
$$
\Do\times \bG\to\bG,
$$
where, for each $S$, the action $\Do\times\bG(S)\to\bG(S)$ is as above. Hence we obtain a morphism
$$
\rho:\Do\to\uEnd(\bG).
$$
Given an $f\in\Do$, $$\bG^f=\Ker(\rho(f))$$ is a presheaf of difference groups that can be considered as the difference subgroup of $\bG$ defined by the equation `$f=e$' in $\bG$.
Equivalently, using the properties of the constant object $(\Do/(f))_e$,
$$
\bG^f=\uHom((\Do/(f))_e,\bG).
$$
\subsection{Additive groups}\label{additive-gps}

Classically, the additive group functor is the $\widehat{\Sch}$-group $\bG_a$ given by 
$$
\bG_a(S)=\cO_S(S),
$$
considered with the additive structure of the ring $\cO_S(S)$. It is represented by $\bbG_a=\spec(\Z[T])$. More generally, for an affine $S\in\Sch$, we have
$$
\bG_{a,S}\simeq\h_{\spec(\cA(S)[T])},
$$
where $\cA(S)[T]$ is given the Hopf algebra structure via $\Delta(T)=T\otimes 1+1\otimes T$, $\epsilon(T)=0$, $\tau(T)=-T$.

In the difference context, we let
$$
\diff\bG_a:\diff\Sch^\circ\to\Set,\ \ \ \ \bG_a(S)=\forg{\cO_S(S)}.
$$
Hence $\diff\bG_a(S)=\bG_a(\forg{S})$, so
$$
\diff\bG_a=\bG_a\circ\forg{\,}=\psig{\bG_a}
$$
and $\diff\bG_a$ is represented by $\diff\bbG_a=\psig{\bbG_a}=\spec(\Z\{T\})$,
where $\Z\{T\}=\Z[T_0,T_1,\ldots]$, together with $\sigma:T_i\mapsto T_{i+1}$.

Working over an affine base difference scheme $S$, 
$$
\diff\bG_{a,S}=\h_{\spec(\cA(S)\{T\})},
$$ 
where $\cA(S)\{T\}=\Z\{T\}\otimes\cA(S)=\cA(S)[T_0,T_1,\dots]$, $\sigma{\restriction}{\cA(S)}=\sigma_{\cA(S)}$, and $\sigma:T_i\mapsto T_{i+1}$. This difference algebra is equipped with a $\sigma$-Hopf algebra structure through $\Delta(T_i)=T_i\otimes 1+1\otimes T_i$, $\tau(T_i)=-T_i$, $\epsilon(T_i)=0$.

We can also consider the $\diff\Set$-functor
$$
\diffst\bG_a:(\diffst\Sch)^\circ\to\diffst\Set,
$$
where 
$$
\diffst\bG_a(S)=\cO_S(S)
$$
with the structure of a difference group. Since $\diffst\bG_a(S)=(\bG_a(\forg{S}),\bG_s(\forg{\sigma_S}))$, we see that 
$$
\diffst\bG_a=\psig{\bG_a}^{*},
$$
so, by \ref{psh-to-diffpsh}, it is represented by $\diff\bbG_a$ in $\diffst\Sch$.

The $\diff\Set$-functor
$$
\difuinf\bG_a:(\difuinf\Sch)^\circ\to\diff\Set
$$
is given by $\difuinf\bG_a(S)=\cO_S(S)=(\bG_a(\forg{S}),\bG_a(\forg{\sigma_S}))$ as a difference group, and by $\difuinf\bG_a(f_i)=(\bG_a(f_i))_i$. Thus, 
$$\difuinf\bG_a=\psig{\bG_a}^\infty,$$ so by \ref{create-gen-dif-psh}, 
$\difuinf\bG_a$ is represented by $\jj^\infty\bbG_a=\spec(\Z[T],\id)$.

The template from \ref{find-diff-subgp} allows us to find a myriad of difference algebraic subgroups of either of $\diffst\bG_a$, $\difuinf\bG_a$. In particular, for a polynomial $f=a_0+a_1T+\cdots+a_nT^n\in\Do$, we have that
$$
(\diffst\bG_a)^f(S)=\{u\in\cO_S(S):a_0 u+a_1\sigma(u)+\cdots+a_n\sigma^n(u)=0\},
$$
which is represented by the spectrum of $\Z\{T\}/\langle f\rangle$.

\subsection{Multiplicative groups}\label{mult-gps}

The classical multiplicative group functor $\bG_m$ is the $\widehat{\Sch}$-group given by
$$
\bG_m(S)=\cO_S(S)^{\times},
$$
with the structure of the multiplicative group of invertible elements of the ring $\cO_S(S)$. It is represented by the scheme
$$
\bbG_m=\spec(\Z[T,T^{-1}]),
$$
given that $\bbG_m(S)\simeq\Hom_{\text{Alg}}(\Z[T,T^{-1}], \cO_S(S))\simeq\cO_S(S)^{\times}$.

More generally, if $S$ in an affine scheme, $\bG_{m,S}$ is represented by $\bbG_{m,S}$, the spectrum of the $\cA(S)$-algebra $\cA(S)[T,T^{-1}]$, with Hopf algebra structure given by
$\Delta(T)=T\otimes T$, $\tau(T)=T^{-1}$, $\epsilon(T)=1$.

The functor 
$$
\diff\bG_m:(\diff\Sch)^\circ\to\Set
$$
is defined by
$$
\diff\bG_m(S)=\forg{\cO_S(S)^{\times}}=\bG_m(\forg{S}),
$$
so $\diff\bG_m=\psig{\bG_m}$ and it is therefore represented by
$$
\diff\bbG_m=\spec(\Z[T_i,T_i^{-1}:i\in\N],\sigma),
$$
where $\sigma(T_i)=T_{i+1}$ and $\sigma(T_i^{-1})=T_{i+1}^{-1}$. This is a difference Hopf algebra with $\Delta(T_i)=T_i\otimes T_i$, $\tau(T_i)=T_i^{-1}$, $\epsilon(T_i)=1$.

The $\diff\Set$-functor
$$
\diffst\bG_m:(\diffst\Sch)^\circ\to\diffst\Set,\ \ \ S\mapsto\cO_S(S)^{\times}
$$
coincides with $\psig{\bG_m}^{*}$, so it is represented by $\diff\bbG_m$, while the $\diff\Set$-functor
$$
\difuinf\bG_m:(\difuinf\Sch)^\circ\to\diff\Set,\ \ \ S\mapsto\cO_S(S)^{\times}
$$
is represented by $\jj^\infty(\bbG_m)=\spec(\Z[T,T^{-1}],\id)$.

Using the template from \ref{find-diff-subgp}, we can construct difference sub-tori, i.e.,  difference algebraic subgroups of either $\diffst\bG_m$ or $\difuinf\bG_m$. In particular, for a polynomial $f=a_0+a_1T+\cdots+a_nT^n\in\Do$, we have that
$$
(\diffst\bG_m)^f(S)=\{u\in\cO_S(S):u^{a_0}\sigma(u)^{a_1}\cdots\sigma^n(u)^{a_n}=1\},
$$
which is represented by the spectrum of

\subsection{The ring $\OO$}\label{ringO}

Classically, the $\widehat{\Sch}$-ring $\OO_\text{sch}$ is given by
$$
\OO_\text{sch}(S)=\cA(S)=\cO_S(S).
$$
It is represented by the scheme $\bbO=\spec(\Z[T])$. 

If $S$ is an affine scheme, the $\widehat{\Sch_{\ov S}}$-ring $\OO_S$ is represented by an $S$-scheme $$\bbO_S=S\times_{\spec(\Z)}\bbO=\spec(\cA(S)[T]).$$ 

In the difference context, we define the $\diff\Set$-ring presheaf
$$
\OO:(\difuinf\Sch)^\circ\to\diffinf\Set 
$$
as follows. Given $S,S'\in \difuinf\Sch$, we assign
 $$\OO(S)=\cA(S),$$
  and we define the difference morphism 
 $$
\OO_{S'S}:\difuinf\Sch(S',S)\to \diffinf\Rng(\cA(S),\cA(S'))
$$
to be the global section morphism. Note that  $\OO=\psig{\OO_\text{sch}}^\infty$, so it is represented by $\jj^\infty(\bbO)=\spec(\Z[T],\id)\in\difuinf\Sch$.

Its underlying functor is the difference ring presheaf
$$
\OO_0:\diff\Sch\to\diff\Set, \ \ \ \ S\mapsto \cA(S),
$$
and its associated presheaf is the ring presheaf
$$
\assoc{\OO}:\diff\Sch\to\Set, \ \ \ \ S\mapsto\Fix(\cA(S)),
$$
and it is represented by $\I(\bbO)=\spec(\Z[T],\id)\in\diff\Sch$.

We sometimes consider
$$
\diff\OO:(\diff\Sch)^\circ\to\Set,\ \ \  S\mapsto\forg{\cO_S(S)}=\OO(\forg{S}).
$$
Thus $\diff\OO=\psig\OO$, so it is represented by 
$$\diff\bbO=\spec(\Z\{T\}).$$

Additionally, we consider the $\diff\Set$-functor
$$
\diffst\OO:(\diffst\Sch)^\circ\to\diffst\Set,\ \ \ S\mapsto\cO_S(S),
$$
with the difference ring structure. Hence $\diffst\OO=\psig{\OO}^{*}$, so it is represented by $\diff\bbO$.


\subsection{Modules}\label{diff-modules}

Classically, if $S$ is an affine scheme and $\cF$ is an $\cA(S)$-module (i.e., an $\OO(S)$-module), there are two natural ways of associating an $\OO_S$-module to it. We consider the contravariant functors $\V_\text{sch}(\cF)$ and $\W_\text{sch}(\cF)$ on $\Sch_{\ov S}$ defined by
\begin{align*}
\V_\text{sch}(\cF)(S') & = \Hom_{\OO(S')}(\cF\otimes_{\OO(S)}\OO(S'),\OO(S'))=\Hom_{\OO(S)}(\cF,\OO(S')),\\
\W_\text{sch}(\cF)(S') &= \cF\otimes_{\OO(S)}\OO(S').
\end{align*} 

We can also view $\V$ and $\W$ as functors from the category of $\cA(S)$-modules to the category of $\OO_S$-modules, where $\V$ is contravariant, and $\W$ covariant.

Let $S$ be an affine difference scheme, so $\cA(S)=\OO(S)$ is a difference ring. Let $\cF$ be an $\cA(S)$-module. We define $\OO_S$-modules $\V(\cF)$ and $\W(\cF)$ as the $\diff\Set$-functors $\difuinf\Sch\to\diffinf\Set$ given by
\begin{align*}
\V(\cF)(S') & = [\cF\otimes_{\OO(S)}\OO(S'),\OO(S')]_{\OO(S')}=[\cF,\OO(S')]_{\OO(S)},\\
\W(\cF)(S') &= \cF\otimes_{\OO(S)}\OO(S').
\end{align*} 


Their underlying functors are the $\OO_0$-modules
$$
\V(\cF)_0,\W(\cF)_0:\diff\Sch^\circ\to\diff\Set,
$$ 
given by the same formulae on objects.

Their associated presheaves are $\assoc{\OO}$-modules
$\diff\Sch^\circ\to\Set$ given by 
\begin{align*}
\assoc{\V}(\cF)(S') & = \OO(S')\Mod(\cF\otimes_{\OO(S)}\OO(S'),\OO(S'))=\OO(S)\Mod(\cF,\OO(S')),\\
\assoc{\W}(\cF)(S') &= \Fix(\cF\otimes_{\OO(S)}\OO(S')).
\end{align*}

We let
$$
\diff\V(\cF)=\forg{\,}\circ\V_0(\cF)\ \ \ \text{and}\ \ \ \diff\W(\cF)=\forg{\,}\circ\W_0(\cF).
$$
We obtain $\cA(S)\Mod$-functors
$$
\V:(\cA(S)_\infty\Mod)^\circ\to \OO_{S,\infty}\Mod\ \ \text{and}\ \  \W:\cA(S)_\infty\Mod\to \OO_{S,\infty}\Mod,
$$
although we will mostly be interested in their underlying functors (which we denote by the same letters to simplify notation)
$$
\V:(\cA(S)\Mod)^\circ\to \OO_{S}\Mod\ \ \text{and}\ \  \W:\cA(S)\Mod\to \OO_{S}\Mod.
$$
We also have the functors arising from the associated presheaves
$$
\assoc{\V}:(\cA(S)\Mod)^\circ\to \assoc{\OO}_S\Mod\ \ \text{and}\ \  \assoc{\W}:\cA(S)\Mod\to \assoc{\OO}_{S}\Mod.
$$

%
%

\begin{proposition}\label{V-representable}
Let $S$ be an affine difference scheme, and let $\cF$ be an $\cA(S)$-module. The functor $\V(\cF)$ is representable on $\difuinf\Sch$ (and $\assoc{\V}(\cF)$ is represented on $\diff\Sch$) by the affine difference $S$-scheme
$$
\bbV(\cF)=\spec(\sym(\cF)),
$$
where $\sym(\cF)$ denotes the symmetric difference $\cA(S)$-algebra associated to $\cF$, as in \ref{diff-sym-alg}. Consequently, $\diff\V(\cF)$ is representable on  $\diff\Sch$ by 
$$
\produinf\bbV(\cF).
$$ 
\end{proposition}
\begin{proof}
For $S'$ over $S$, we have
\begin{align*}
\V(\cF)(S')& =[\cF\otimes_{\OO(S)}\OO(S'),\OO(S')]_{\OO(S')\Mod}\simeq
[\cF\otimes,\OO(S')]_{\OO(S)\Mod}\\  & \simeq[\sym(\cF),\OO(S')]_{\OO(S)\Alg}
 \simeq\difuinf\Sch_{\ov S}(S',\bbV(\cF)).
\end{align*}
The last statement follows by \ref{diff-vs-gendif-enr}, since, for $S'\to S$ in $\diff\Sch$,
$$
\diff\V(\cF)(S')=\forg{\difuinf\Sch_{\ov S}(\I^\infty(S'),\bbV(F))}\simeq \diff\Sch_{\ov S}(S',\produinf\bbV(\cF)).
$$
\end{proof}

\begin{corollary}
Let $S_0$ be an affine scheme and $\cF_0$ an $\cA(S_0)$-module. Writing $\cF=\jj_{\infty}(\cF_0)$ and $S=\jj^{\infty}(S_0)$, we have that $\cF$ is an $\cA(S)$-module.
Then $\V(\cF)$ is represented on $\difuinf\Sch$ by $$\jj^\infty(\bbV(\cF_0))$$ and
$\diff\V$ is represented on $\diff\Sch$ by 
$$
\produinf\jj^\infty(\bbV(\cF_0))=\psig{\bbV(\cF_0)}.
$$

\end{corollary}

\begin{proposition}\label{enr-mod-psh-ff}
Let $S$ be an affine difference scheme, and let $\cF$ and $\cF'$ be two $\cA(S)$-modules.
\begin{enumerate}
\item The functors $\V$ and $\W$ 
commute with base change. For $S'\to S$ affine, we have
$$
\V(\cF\otimes\cA(S'))\simeq \V(\cF)_{S'}\ \ \ \text{and}\ \ \ 
\W(\cF\otimes\cA(S'))\simeq \W(\cF)_{S'}.
$$
\item\label{dva} The enriched functors $\V$ and $\W$ are fully faithful, i.e., the canonical difference maps
\begin{align*}
[\cF,\cF']_{\cA(S)} & \to \OO_S\Mod(\V(\cF'),\V(\cF))\\
[\cF,\cF']_{\cA(S)} & \to \OO_S\Mod(\W(\cF),\W(\cF')])
\end{align*}
are bijective. In particular, we have bijections
\begin{align*}
\Hom_{\cA(S)}(\cF,\cF') & \to \Hom_{\OO_S}(\V(\cF'),\V(\cF))\\
\Hom_{\cA(S)}(\cF,\cF') & \to \Hom_{\OO_S}(\W(\cF),\W(\cF')).
\end{align*}
\item The functors $\V$ and $\W$ are additive,
$$
\V(\cF\oplus\cF')\simeq\V(\cF)\times_S\V(\cF')\ \ \text{and}\ \ 
\W(\cF\oplus\cF')\simeq\W(\cF)\times_S\W(\cF').
$$
\end{enumerate}
\end{proposition}
\begin{proof}
Only claim \ref{dva} requires proof. The case of $\W$ is trivial, because, given a morphism $\phi:\W(\cF)\to\W(\cF')$ we recover the original map as $\phi(S):\cF\to\cF'$.

In view of \ref{V-representable}, we show that $\cF$ can be reconstructed from the $\OO_S$-module structure on the $S$-scheme $\bbV(\cF)$. Indeed, for any difference scheme $X$ over $S$, the underlying difference set of the $\OO(X)$-module $\V(\cF)(X)=[\cF,\OO(X)]_{\OO(S)}$ canonically identifies to 
$[\sym(\cF),\OO(X)]_{\OO(S)\Alg}$.
This isomorphism allows us to evaluate an element $h\in\V(\cF)(X)$ on an element of $\sym(\cF)$. More explicitly, if $h,h'\in [\cF,\OO(X)]_{\OO(S)}$, $s_i\in\cF$, $t\in\OO(X)$, we consider $s_1s_1\cdots s_n$ as an element of $\sym(\cF)$, and, writing $h(s_i)$ for the evaluation of $h$ at $s_i$, 
\begin{align*}
(h+h')(s_1s_2\cdots s_n) & =\prod_i(h(s_i)+h'(s_i))\\
(t.h)(s_1s_2\cdots s_n) & =t^n\prod_i h(s_i).
\end{align*}
Hence, for every $z\in\sym(\cF)$, and every $X$ over $S$, we obtain a map
$$
\text{ev}_z:\V(\cF)(X)\simeq[\sym(\cF),\OO(X)]_{\OO(S)\Alg}\to\OO(X).
$$
We claim that 
$$
\cF=\{z\in\sym(\cF): \text{for every }X\text{ over }S,\ \text{ev}_z \text{ is an }\OO(X)\text{-module homomorphism}\}.
$$
The left to right inclusion is clear from the above formulae. Conversely, suppose $z$ satisfies the defining property of the set on the right hand side, and write $z=\sum_n z_n$, where $z_n\in\sym_n(\cF)$. Let us choose $X=\spec(\sym(\cF)[T])$, the spectrum of the difference ring extending $\sym(\cF)$ determined by $\sigma(T)=T$. Then, for 
$$h\in[\sym(\cF),\sym(\cF)[T]]_{\OO(S)\Alg},$$ we have $(T.h)(z)=\sum_n T^n h(z_n)$. By hypothesis on $z$, $(T.h)(z)=T.(h(z))$, i.e., $\sum_n T^n h(z_n)=T.\sum_n h(z_n)$. Taking for $h$ the canonical injection, we get $\sum_nT^n.z_n=T.\sum_n z_n$, which implies that $z_n=0$ for $n\neq 1$, and $z\in\cF$.

\end{proof} 
 
 \begin{corollary} With the notation of \ref{enr-mod-psh-ff}, we have the following.
 \begin{enumerate}
 \item The functors $\assoc{\V}$ and $\assoc{\W}$ commute with base change.
 \item The functor $\assoc{\V}$ is fully faithful, and $\assoc{\W}$ is not in general.
 \item The functors $\assoc{\V}$ and $\assoc{\W}$ are additive.
 \end{enumerate}
  \end{corollary}
 \begin{proof}
 The properties (1) and (3) follow by applying the functor $\Fix$ to the corresponding properties from \ref{enr-mod-psh-ff}. For (2), the fact that $\assoc{\V}$ is fully faithful follows by the same proof as for $\V$. To see that $\assoc{\W}$ is not fully faithful in general, let $\cF$ be a free difference module on a single generator. Then $\assoc{\W}(\cF)(S')=0$ for any $S'\to S$, so it is impossible to reconstruct $\cF$ from its module presheaf.
\end{proof}

\begin{proposition}\label{anti-isom}
We have canonical morphisms
 \begin{center}
 \begin{tikzpicture} 
\matrix(m)[matrix of math nodes, row sep=1.2em, column sep=.7em, text height=1.5ex, text depth=0.25ex]
 {
  |(a)|{\OO_S\Mod[\W(\cF),\W(\cF')]} & & |(b)|{\OO_S\Mod[\V(\cF'),\V(\cF)]}  	\\
 & |(k)|{\W([\cF,\cF']_{\cA(S)})} &  	\\
 }; 
\path[-,font=\scriptsize,>=to, thin]
([yshift=1pt]a.east) edge 
([yshift=1pt]b.west) 
([yshift=-1pt]a.east) edge 
([yshift=-1pt]b.west) 
;
\path[->,font=\scriptsize,>=to, thin]
(k) edge (a)
(k) edge (b) 
;
\end{tikzpicture}
\end{center}
Moreover, if $\cF$ and $\cF'$ are finite \'etale $\cA(S)$-modules, then all the above arrows are isomorphisms. 
\end{proposition}
\begin{proof}
The isomorphism in the top row follows using \ref{enr-mod-psh-ff}(1).  If we write the value of all these presheaves on some $S'\to S$, and use \ref{enr-mod-psh-ff}(2), the existence of the left morphism is a consequence of the natural morphism
$$
[\cF,\cF']_{\cA(S)}\otimes_{\cA(S)} \cA(S')\to[\cF\otimes\cA(S'),\cF'\otimes\cA(S')]_{\cA(S')},
$$
and similarly for the morphism on the right. Note that this morphism is an isomorphism when $\cF$ and $\cF'$ are finite \'etale.
\end{proof}

\begin{corollary}
If $\cF$ is a finite \'etale $\cA(S)$-module, then
\begin{align*}
\W(\cF^\vee) & \simeq \OO_S\Mod[\W(\cF),\OO_S]\simeq\V(\cF),\\
\V(\cF^\vee) & \simeq \OO_S\Mod[\V(\cF),\OO_S]\simeq\W(\cF).
\end{align*}
Moreover, 
$$
\assoc{\W}(\cF^\vee)\simeq \assoc{\V}(\cF).
$$
\end{corollary}
 
 \begin{proposition}\label{diff-map-into-uhom}
 Let $\cB$ be an $\cA(S)$-algebra, and let $\cF$, $\cF'$ be $\cA(S)$-modules. Then we have a natural isomorphism
 $$
 \widehat{\difuinf\Sch_{\ov S}}(\spec(\cB),\OO_S\Mod[\W(\cF'),\W(\cF)])\simeq[\cF',\cF\otimes_{\cA(S)}\cB]_{\cA(S)}.
 $$
\end{proposition}
\begin{proof}
Writing $X=\spec{\cB}$, using enriched Yoneda, the left hand side becomes
\begin{multline*}
\OO_S\Mod[\W(\cF'),\W(\cF)](X)=\OO_X\Mod(\W(\cF')_X,\W(\cF)_X)\\
=\OO_X\Mod(\W(\cF'\otimes\cB),\W(\cF\otimes\cB))\simeq [\cF'\otimes\cB,\cF\otimes\cB]_{\cA(S)}
=[\cF',\cF\otimes\cB]_{\cA(S)}.
\end{multline*}
\end{proof}

\begin{proposition}\label{map-into-uhom-v}
With the above notation, we have a natural isomorphism
$$
\widehat{\diff\Sch_{\ov S}}(\spec{B},\uHom_{\assoc{\OO}_S}(\assoc{\V}(\cF'),\assoc{\V}(\cF))\simeq \cA(S)\Mod(\cF,\cF'\otimes_{\cA(S)}\cB).
$$
\end{proposition}
\begin{proof}
Writing $X=\spec(\cB)$, using ordinary Yoneda, the left hand side equals
\begin{multline*}
\uHom_{\assoc{\OO}_S}(\assoc{\V}(\cF'),\assoc{\V}(\cF))(X)=
\assoc{\OO}_X\Mod(\assoc{\V}(\cF')_X,\assoc{\V}(\cF)_X)\\
=\assoc{\OO}_X\Mod(\assoc{\V}(\cF'\otimes\cB),\assoc{\V}(\cF\otimes\cB))\simeq \cA(S)\Mod(\cF\otimes\cB,\cF'\otimes\cB)\\
=\cA(S)\Mod(\cF,\cF'\otimes\cB).
\end{multline*}
\end{proof}

\subsection{Difference group modules}

Let $S$ be a difference scheme, let $G$ be an $S$-group, and let $\cF$ be an $\cA(S)$-module. 

\begin{definition}\label{def-G-mod}
A structure of an enriched $G\da\cA(S)_\infty$-module on $\cF$ is given by a structure of a $\h_G\da\OO_S$-module on $\W(\cF)$. The internal hom of two $G\da\cA(S)_\infty$-modules is their internal hom in the category $\h_G\da\OO_S\Mod$. Thus we obtain a full $\diff\Set$-subcategory
$$
G\da\cA(S)_\infty\Mod
$$
of $\h_G\da\OO_S\Mod$.
\end{definition}

Equivalently, taking the fixed points of the isomorphism in \ref{enr-gp-mod},  a structure of a $G\da\cA(S)_\infty$-module on $\cF$ is determined by a morphism of $\widehat{\difuinf\Sch_{\ov S}}$-groups
$$
\rho:\h_G\to \uAut_{\OO_S}[\W(\cF)].
$$
This is, using \ref{diff-map-into-uhom}, equivalent to a choice of an $\cA(S)$-module map
$$
\mu:\cF\to \cF\otimes\cA(G),
$$
which makes $\cF$ into a difference comodule for the difference Hopf algebra $\cA(G)$. Given two $G\da\cA(S)_\infty$-comodules $\cF$, $\cF'$, the internal hom object $G\da\cA(S)_\infty\Mod(\cF,\cF')$ corresponds to the internal hom object $\cA(G)_\infty\Comod(\cF,\cF')$, which is obtained as the equaliser
\begin{center}
 \begin{tikzpicture} 
\matrix(m)[matrix of math nodes, row sep=2em, column sep=1.5em, text height=1.5ex, text depth=0.25ex]
 {
  	&	& |(2)|{[\cF\otimes\cA(G),\cF'\otimes\cA(G)]}  & 	\\
 |(l1)|{\cA(G)_\infty\Comod(\cF,\cF')} & |(l2)|{[\cF,\cF']}	& 	& |(l3)|{[\cF,\cF'\otimes\cA(G)]} 	\\
 }; 
\path[->,font=\scriptsize,>=to, thin]
(l1) edge (l2)
(l2) edge node[above,pos=.4]{$$} (2) edge node[below]{$\mu_{\cF',*}$}   (l3)
(2) edge node[above,pos=.7]{$\mu_{\cF}^*$} (l3) 
;
\end{tikzpicture}
\end{center}
so we conclude the following.

\begin{proposition}\label{enr-eq-G-mod}
There is an equivalence of $\cA(S)\Mod$-categories
$$
G\da\cA(S)_\infty\Mod\simeq \cA(G)_\infty\Comod.
$$
\end{proposition}

 
\begin{remark}
If $\cF$ is an $\cA(S)$-module, a homomorphism 
$\rho:\h_G\to\uAut_{\OO_S}[\W(\cF)]$ gives rise to the diagram
 \begin{center}
 \begin{tikzpicture} 
\matrix(m)[matrix of math nodes, row sep=2em, column sep=3em, text height=1.5ex, text depth=0.25ex]
 {
 |(1)|{\h_G}		& |(2)|{\uAut_{\OO_S}[\W(\cF)]} 	\\
 |(l1)|{\h_G}		& |(l2)|{\uAut_{\OO_S}[\V(\cF)]} 	\\
 }; 
\path[->,font=\scriptsize,>=to, thin]
(1) edge node[above]{$\rho$} (2) 
(l1) edge node[left]{$$}   (1)
(2) edge node[right]{$$} (l2) 
(l1) edge  node[above]{$\rho^\vee$} (l2);
\end{tikzpicture}
\end{center}
where the left vertical arrow is the inverse of $G$, and the right vertical arrow is the anti-isomorphism obtained using \ref{anti-isom}. The contragredient representation $\rho^\vee$ makes $\V(\cF)$ into a $\h_G\da\OO_S$-module whenever $\rho$ makes $\W(\cF)$ into one. 

Given that the functor $\cV$ is contravariant, the category of $\cA(S)$-modules which are given the $\h_G\da\OO_S$-module structure through $\V(\cF)$ is equivalent to
$$(\cA(G)_\infty\Comod)^\circ.$$
\end{remark} 
 
\begin{definition}
A structure of a $G\da\cA(S)$-module on $\cF$ is given by a structure of a $\assoc{\h}_G\da\assoc{\OO}_S$-module on $\assoc{\W}(\cF)$. On the other hand, a structure of a $G\da\cA(S)^\circ$-module on $\cF$ is given by a structure of a $\assoc{\h}_G\da\assoc{\OO}_S$-module on $\assoc{\V}(\cF)$. Thus we obtain (ordinary) categories
$$
G\da\cA(S)\Mod \ \ \ \text{ and } \ \ \ \ G\da\cA(S)^\circ\Mod.
$$
\end{definition}

\begin{proposition}
There is an equivalence of categories 
$$
G\da\cA(S)^\circ\Mod\simeq \cA(G)\Comod^\circ.
$$
\end{proposition}
\begin{proof}
The proof is parallel to the proof of \ref{enr-eq-G-mod}, using \ref{map-into-uhom-v} in place of \ref{diff-map-into-uhom}.
\end{proof}

\begin{corollary}\label{flat-comod-abelian}
Suppose that the affine difference algebraic group $G$ is flat over $S$ (i.e., $\forg{\cA(G)}$ is flat over $\forg{\cA(S)}$. Then the category $G\da\cA(S)_\infty\Mod$ (equivalent to $\cA(G)_\infty\Comod$) is enriched abelian. Moreover, the category $G\da\cA(S)^\circ\Mod$ (equivalent to $\cA(G)\Comod^\circ$) is abelian.
\end{corollary}
\begin{proof}
It is a classical fact that $\forg{\cA(G)}\Comod$ is abelian when $\forg{\cA(G)}$ is flat over $\forg{\cA(S)}$. We simply note that the difference structure does not interfere with the abelian structure.
\end{proof}

\subsection{Induced difference group modules}

Let $S$ be a difference scheme and let $G$ be a difference $S$-group. Let us write $\Delta:\cA(G)\to\cA(G)\otimes\cA(G)$ for the comultiplication, and $\eta:\cA(G)\to\cA(S)$ for the counit of the difference Hopf algebra $\cA(G)$. 

Given an $\cA(S)$-module $\cP$, we let
$$
\ind(\cP)=\cP\otimes_{\cA(S)}\cA(G),
$$
with the $\cA(G)$-comodule structure given by 
$$
\id_\cP\otimes\Delta:\cP\otimes_{\cA(S)}\cA(G)\to \cP\otimes_{\cA(S)}\cA(G)\otimes_{\cA(S)}\cA(G).
$$
This defines a $\diff\Ab$-functor 
$$
\ind:\cA(S)_\infty\Mod\to G\da\cA(S)_\infty\Mod.
$$
This construction is related to the more general construction from \ref{ind-gp-mod} via
$$
\W(\ind(\cP))\simeq E(\W(\cP))=[\h_G,\W(\cP)].
$$
Through this identification, the morphism $\varepsilon:E(\W(\cP))\to\W(\cP)$ corresponds to the morphism $\id_\cP\otimes\eta:\ind(\cP)\to\cP$. 

We know by \ref{enr-mod-psh-ff} that the enriched functor $\W:\cA(S)_\infty\Mod\to \OO_S\Mod$ is fully faithful. On the other hand, by definition \ref{def-G-mod}, its restriction to $G\da\cA(S)_\infty\Mod$ is also fully faithful, i.e., for any $\cM,\cM'\in G\da\cA(S)_\infty\Mod$, there is an enriched natural isomorphism
$$
G\da\cA(S)_\infty\Mod(\cM,\cM')\simeq \h_G\da\OO_S\Mod(\W(\cM),\W(\cM')).
$$
It would be straightforward to prove it directly, but we can now draw the following conclusions using \ref{E-radj-forg} and \ref{enough-G-enr-inj}.
\begin{corollary}
The functor $\ind$ is an enriched right adjoint to the forgetful $\diff\Ab$-functor $G\da\cA(S)_\infty\Mod\to \cA(S)_\infty\Mod$. More precisely, we have isomorphisms
$$
G\da\cA(S)_\infty\Mod(\cM,\ind(\cP))\simeq \cA(S)_\infty\Mod(\cM,\cP),
$$
enriched natural in $\cM\in G\da\cA(S)_\infty\Mod$ and $\cP\in\cA(S)_\infty\Mod$.
\end{corollary}

\begin{corollary}\label{G-mod-enough-enr-inj}
If $\cI$ is an enriched injective object of $\cA(S)_\infty\Mod$, then $\ind(\cI)$ is an enriched injective object of $G\da\cA(S)_\infty\Mod$. Consequently, $G\da\cA(S)_\infty\Mod$ has enough enriched injectives.
\end{corollary}
\begin{proof}
We note that, by \ref{diff-mod-enough-enr-inj-proj}, $\cA(S)_\infty\Mod$ has enough enriched injectives. 
\end{proof}

\subsection{Cohomology of difference algebraic groups}

The most general theory of cohomology of (enriched) difference algebraic groups, is obtained by specialising the context of \ref{enr-gp-coh} to the following:
\begin{enumerate}
\item $\cV=\diff\Set$, the cartesian closed category of difference sets, with internal homs coming from $\diffinf\Set$ (when we think of $\cV$ as enriched over itself, it may be useful to think of $\cV$ as $\diffinf\Set$, with the underlying category $\cV_0=\diff\Set$);
\item $\Vab=\diff\Ab$, the category of difference abelian groups;
\item $\cC=\difuinf\Sch$, the $\cV$-category of enriched difference schemes, whose underlying category is $\cC_0=\diff\Sch$;
\end{enumerate}

Hence, if  $\bG:(\difuinf\Sch)^\circ\to \diffinf\Set$ is a group $\diff\Set$-presheaf, 
$\bO:(\difuinf\Sch)^\circ\to \diffinf\Set$ is a ring $\diff\Set$-presheaf, and $\bF$ is a $\bG\da\bO$-module, \ref{enr-gp-coh} gives meaning to enriched cohomology difference groups
$$
\HH^n(\bG,\bF)\in\diff\Ab\ \ \ \text{and}\ \ \ \uH^n(\bG,\bF)\in\bO\Mod.
$$

If  $\bar{\bG}:(\diff\Sch)^\circ\to \Set$ is a group presheaf, 
$\bar{\bO}:(\diff\Sch)^\circ\to \Set$ is a ring presheaf, and $\bar{\bF}$ is a $\bar{\bG}\da\bar{\bO}$-module, the classical context of \ref{group-coh} gives meaning to  cohomology groups
$$
\HH^n(\bar{\bG},\bar{\bF})\in\Ab\ \ \ \text{and}\ \ \ \uH^n(\bar{\bG},\bar{\bF})\in\bO\Mod.
$$

The spectral sequence relating the enriched  cohomology difference groups $\HH^n(\bG,\bF)$ (resp.\ $\uH^n(\bG,\bF)$) and the cohomology groups of the associated difference group presheaves
$\HH^n(\assoc{\bG},\assoc{\bF})$ (resp.\ $\uH^n(\assoc{\bG},\assoc{\bF})$) is given in \ref{coh-assoc}.

\subsection{Cohomology of enriched difference group modules}

Let $S$ be a difference scheme, and let $G$ be a difference group scheme over $S$, and let $\cF\in G\da\cA(S)_\infty\Mod$. We define the enriched difference cohomology groups of $G$ with values in $\cF$ by
$$
\HH^n(G,\cF)=\HH^n(\h_G,\W(\cF)).
$$
Using \ref{diff-map-into-uhom}, these are the cohomology groups of the complex $\Ch^*(G,\cF)$ given by
$$
\Ch^n(G,\cF)=\W(\cF)(G^n)=\cF\otimes\cA(G)^{\otimes n}.
$$
For $f\in\cF$ and $a_i\in\cA(G)$, the boundary operator is given in terms of the difference coalgebra structure of $\cA(G)$, with comultiplication $\Delta:\cA(G)\to\cA(G)\otimes\cA(G)$, and the difference $\cA(G)$-comodule structure $\mu_\cF:\cF\to\cF\otimes\cA(G)$ by
\begin{align*}
\partial(f\otimes a_1\otimes\cdots a_n)&=\mu_\cF(f)\otimes a_1\otimes\cdots\otimes a_n\\
&+\sum_{i=1}^n(-1)^i f\otimes a_1\otimes\cdots\otimes\Delta a_i\otimes\cdots\otimes a_n\\
&+(-1)^{n+1} f\otimes a_1\otimes\cdots\otimes a_n\otimes 1.
\end{align*}

In particular, 
\begin{align*}
\HH^0(G,\cF)&=\W(\cF)^{\h_G}=\ker(\partial_0)=\{f\in\cF:\mu_\cF(f)=f\otimes 1\}\\&=\cA(G)_\infty\Comod(\cA(S),\cF).
\end{align*}

\begin{theorem}
Let $S$ be an affine difference scheme, and let $G$ be an affine flat difference $S$-group. The functors $\HH^n(G,\mathord{-})$ are the enriched derived functors of $\HH^0(G,\mathord{-})$ on the enriched abelian category $G\da\cA(S)_\infty\Mod$, i.e.,
$$
\HH^n(G,\cF)=\Ext^n_{\cA(G)_\infty\Comod}(\cA(S),\cF).
$$
\end{theorem}
\begin{proof}
By \ref{enr-eq-G-mod}, the $\diff\Ab$-category category $G\cA(S)_\infty\Mod$ is 
enriched equivalent to $\cA(G)_\infty\Comod$, and, since $\bA(G)$ is flat over $\bA(S)$, \ref{flat-comod-abelian} shows that it is enriched abelian, while \ref{G-mod-enough-enr-inj} shows that it has enough enriched injectives. Again by flatness, each enriched functor 
$$
\cF\mapsto\cF\otimes_{\cA(S)}\cA(G)^{\otimes n}
$$
is exact, so $\Ch^*(G,\mathord{-})$ is an enriched exact functor on $G\da\cA(S)_\infty\Mod$. Hence $\HH^*(G,\mathord{-})$ is an enriched cohomological functor, so it suffices to prove that it is effaceable.

Indeed, if $\cF$ is a $G\da\cA(S)$-module, the comodule coaction $\mu_\cF:\cF\to \ind(\cF)$ is a monomorphism, and, by the proof \ref{enr-der-of-fix},
$$
\HH^n(G,\ind(\cF))=\HH^n(\h_G,\W(\ind(\cF)))=\HH^n(\h_G,E(\W(\cF)))=0,
$$  
for $n>0$.
\end{proof}

\subsection{Cohomology of associated difference group modules}

Let $S$ be a difference scheme, let $G$ be a flat difference group scheme over $S$, and let $\cF\in G\da\cA(S)_\infty\Mod$, i.e., $\W(\cF)$ is given a structure of the $\h_G\da\OO_S$-module. Equivalently, we can view $\cF$ as an object in $\cA(G)_\infty\Comod$. 

In particular, $\assoc{\W}(\cF)$ is the associated presheaf of $\W(\cF)$, and, when needed, we may consider $\cF$ as an object of the underlying category $\cA(G)\Comod$ of difference $\cA(G)$-comodules.

We write
$$
\HH^{\sigma,n}(G,\cF)=\HH^n(\assoc{\h}_G,\assoc{\W}(\cF)).
$$
Expanding the definition, these are the cohomology groups of the difference invariants complex 
$\Fix(\Ch^*(G,\cF))$,
i.e.,
$$
\HH^{\sigma,n}(G,\cF)=\coh^n(\Fix(\Ch^*(G,\cF))).
$$
We also introduce the cohomology groups of the difference coinvariants complex
$$
\HH^{n}_\sigma(G,\cF)=\coh^n(\Quo(\Ch^*(G,\cF))).
$$
From a  purely algebraic perspective, we let
$$
\cH^n(G,\cF)=\Ext^n_{\cA(G)\Comod}(\cA(S),\cF),
$$
as also considered in \cite{piotr}. 
\begin{remark}\label{hyperderived}
By the construction following \cite[2.4.2]{groth-tohoku}, the above groups can be expressed using the the hypercohomology functors of $\Fix$. Indeed, considering $\coh^0$ as a functor defined on chains, and using the fact that
$\Ch^*(G,\mathord{-})$ is exact, we see that
\begin{align*}
\cR^n\Fix(\Ch^*(G,\cF))&=\der^n(\Fix\circ\coh^0)(\Ch^*(G,F))\\
&=\der^n(\Fix\circ\coh^0\circ\Ch^*(G,\mathord{-})(\cF)\\
&=\der^n(\Fix\circ\HH^0(G,\mathord{-}))(\cF)=\cH^n(G,\cF).
\end{align*}
\end{remark}
The two hypercohomology spectral sequences \cite[2.4.2]{groth-tohoku} yield that
$$
\lexp{II}{E}_2^{p,q}=(\der^p\Fix)(\coh^q(\Ch^*(G,\cF)))\Rightarrow \cR^{p+q}\Fix(\Ch^*(G,\cF)),
$$
and
$$
\lexp{I}{E}_2^{p,q}=\coh^p(\der^q\Fix(\Ch^*(G,\cF)))\Rightarrow\cR^{p+q}\Fix(\Ch^*(G,\cF)).
$$
According to \ref{RFix}, the derived functors of $\Fix=(\mathord{-})^\sigma$ are $\der^1\Fix=\Quo=(\mathord{-})_\sigma$, and $\der^p\Fix=0$ for $p>1$.  Thus, $\lexp{II}{E}_2^{p,q}$ is a two-column spectral sequence, whence, for $n>0$, we obtain the exact sequence
$$
0\to\der^1\Fix(\HH^{n-1}(G,\cF))\to\cR^n\Fix(\Ch^*(G,\cF))\to\Fix(\HH^n(G,\cF))\to 0.
$$
In view of the notation we introduced, as well as \ref{hyperderived}, we can write it as follows.
\begin{proposition}
For every $n>0$, we have an exact sequence
$$
0\to\HH^{n-1}(G,\cF)_\sigma\to\cH^n(G,\cF)\to\HH^n(G,\cF)^\sigma\to 0.
$$
\end{proposition}
We could have obtained the same result through the Grothendieck spectral sequence for the composite functor
$$
\cH^0(G,\mathord{-})=\Fix\circ\HH^0(G,\mathord{-}).
$$

On the other hand, $\lexp{I}{E}_2^{p,q}$ is a two-row spectral sequence, which gives an exact sequence
\begin{multline*}
\cdots\to\cR^{n}\Fix(\Ch^*(G,\cF))\to\coh^{n-1}(\der^1\Fix(\Ch^*(G,\cF))\\\to\coh^{n+1}(\Fix(\Ch^*(G,\cF)))\to \cR^{n+1}\Fix(\Ch^*(G,\cF))\to\cdots
\end{multline*}
We can rewrite it more conveniently as follows.
\begin{proposition}
With the above notation, we have an exact sequence
$$
\cdots\to\cH^{n}(G,\cF)\to\HH^{n-1}_\sigma(G,\cF)\to\HH^{\sigma,n+1}(G,\cF)\to\cH^{n+1}(G,\cF)\to\cdots
$$
\end{proposition}

\section{Comparison to literature}\label{appendix1}

\subsection{Work by Cha{\l}upnik-Kowalski}

In their recent paper \cite{piotr2}, the authors independently identified a number of objects that naturally appeared in our work. 

In particular, their notion of site (Assumption 1.1) $\bC(X)$ with respect to some scheme topology $\tau$, equipped with a base change by $\sigma$ endofunctor, is known to us as the internal $\tau$-site in $\diff\Set$, and we think that their notion of left difference sheaves (Definition 2.1) corresponds to our objects of the $\tau$-spectrum $X_\tau$ from \ref{tau-spectra-diff}, sometimes referred to as $\stig$-equivariant.

What the authors call a difference site in 2.2, is obtained in various guises in our paper through the externalisation procedure we denote by $\rtimes$, but we do a variety of externalisations with respect to several sites for $\diff\Set$, depending on the intended application. 

Consequently, the difference sheaf cohomology $H_\sigma^n(X,\cF)$ from 3.1 of \cite{piotr2} agrees with our $\tH^n(X_\tau,\cF)$, and we obtain analogous results on the comparison with the cohomology of the underlying scheme, and various calculations of the cohomology groups. Our computations of the Picard group of a difference field agree.

The authors address the classification of principal homogeneous spaces for difference algebraic groups in more detail than us, we mostly dealt with it at the level of objects (sheaves) in the flat spectrum and for abelian groups, and did not consider the case of actual difference algebraic groups in more detail. We did not have enough time to analyse that aspect of their paper, and we aim to include a more thorough analysis in the next draft of this manuscript.

 In \cite{piotr}, the authors study cohomology of difference algebraic groups by computing $\Ext$ functors in the category of difference comodules for a difference Hopf algebra. In our Section~\ref{diff-gp-coh}, we wanted to do difference algebraic group cohomology in the style of Demazure's article in \cite{sga3.1}. We do this in the context of enriched category theory over $\diff\Set$, and derive a comparison theorem to the cohomology from \cite{piotr}. We use internal homs to replace the assumption of invertibility of difference operators.

\subsection{Work of Bachmayr-Wibmer}

In \cite{michael-anette}, given a difference algebraic group $G$ over a difference ring $k$ and a $k$-algebra $A$, the authors introduce the cohomology set
$$
H^1_\sigma(A/k,G),
$$
which classifies up to isomorphism the $G$-torsors which are trivialised by $A$. We think that the system of these cohomology sets for varying $A$ should play a role comparable, in the case of abelian $G$, to our cohomology groups
$$
\tH^1(\spec.\text{fppf}(k),G),
$$
or, in the case of a difference field $k$, to our Galois cohomology \ref{diff-galcoh}. 
We will provide a more precise comparison in the next draft of this manuscript.

\backmatter

\bibliographystyle{plain}
\bibliography{pib}
\end{document}